\documentclass[twoside, 11pt, a4paper]{amsart}

\usepackage[T1]{fontenc}
\usepackage[lf]{Baskervaldx} 
\usepackage[cal=boondoxo]{mathalfa} 

\usepackage[utf8]{inputenc}

\usepackage{fullpage}
\usepackage{latexsym}
\usepackage{amscd}
\usepackage{amsmath}
\usepackage{amssymb}
\usepackage{amsthm}
\usepackage{graphicx}
\usepackage{mathrsfs}
\usepackage{mathtools}
\usepackage{pict2e}
\usepackage{stackrel}
\usepackage{wasysym}
\usepackage[all]{xy}
\usepackage{a4wide}
\usepackage[dvipsnames]{xcolor}
\usepackage[colorlinks,final,hyperindex]{hyperref}
\usepackage[noabbrev,capitalize]{cleveref}
\usepackage{tikz}
\usepackage{tikz-cd}
\usetikzlibrary{decorations.pathmorphing,decorations.markings,arrows,calc,shapes.geometric,arrows.meta,positioning}
\tikzset{math3d/.style=
    {x= {(-0.353cm,-0.353cm)}, z={(0cm,1cm)},y={(1cm,0cm)}}}

\usepackage{etoolbox}
\apptocmd{\sloppy}{\hbadness 10000\relax}{}{}
\apptocmd{\sloppy}{\vbadness 10000\relax}{}{}

\definecolor{Chocolat}{rgb}{0.36, 0.2, 0.09}
\definecolor{BleuTresFonce}{rgb}{0.215, 0.215, 0.36}
\definecolor{BleuMinuit}{RGB}{0, 51, 102}

\hypersetup{citecolor=BleuMinuit, linkcolor=Chocolat, filecolor=black, urlcolor=BleuMinuit}

\allowdisplaybreaks

\newtheorem{lemma}{Lemma}[section]
\newtheorem{theorem}[lemma]{Theorem}
\newtheorem{corollary}[lemma]{Corollary}
\newtheorem{proposition}[lemma]{Proposition}
\newtheorem{conjecture}[lemma]{Conjecture}

\newtheorem*{ThmIntroRep}{\cref{thm:isomorphism of models}} 
\newtheorem*{thmfillhorns}{\cref{thm:KanExt}}
\newtheorem*{thmBerglund}{\cref{thm:Berglund}}
\newtheorem*{thmDR}{\cref{thm:HoInvariance}}
\newtheorem*{ThmIntroRHT}{\cref{thm: rationalization}}

\theoremstyle{definition}
\newtheorem{definition}[lemma]{Definition}
\newtheorem{remark}[lemma]{\sc Remark}
\newtheorem{example}[lemma]{\sc Example}

\definecolor{red}{RGB}{230,97,0}
\definecolor{blue}{RGB}{93,58,155}

\newcommand{\ch}{\mathsf{Ch}}

\newcommand{\sSe}{\mathsf{sSet}}
\newcommand{\sLialg}{\ensuremath{\mathcal{sL}_\infty\text{-}\,\mathsf{alg}}}
\newcommand{\isLialg}{\ensuremath{\infty\text{-}\,\mathcal{sL}_\infty\text{-}\,\mathsf{alg}}}
\newcommand{\sLiealg}{\ensuremath{\mathcal{s}\lie\,\text{-}\,\mathsf{alg}}}
\def\cD{\Delta}

\newcommand{\ucomnalg}{\ensuremath{\mathrm{uCom}_{\leqslant 0}\text{-}\,\mathsf{alg}}}
\newcommand{\fnQsp}{\ensuremath{\mathsf{ft}\mathbb{Q}\text{-}\mathsf{ho}(\mathsf{sSet})}}

\newcommand{\fQalg}{\ensuremath{\mathsf{ft}\text{-}\mathsf{ho}(\mathrm{uCom}_{\leqslant 0}\text{-}\,\mathsf{alg})}}

\newcommand{\Ran}{\operatorname{Ran}}
\newcommand{\ho}{\operatorname{ho}}

\newcommand{\sLi}{\ensuremath{\mathcal{sL}_\infty}}

\newcommand{\sLie}{\ensuremath{\mathcal{s}\mathrm{Lie}}}
\newcommand{\com}{\ensuremath{\mathrm{Com}}}

\newcommand{\Cobar}{\ensuremath{\Omega}}
\newcommand{\hatCobar}{\ensuremath{\widehat{\Omega}}}
\renewcommand{\Bar}{\ensuremath{\mathrm{B}}}

\newcommand{\lie}{\ensuremath{\mathrm{Lie}}}

\newcommand{\g}{\ensuremath{\mathfrak{g}}}
\newcommand{\wsLi}{\ensuremath{\widehat{\mathcal{sL}_\infty}}}

\newcommand{\R}{\ensuremath{\mathrm{R}}}
\newcommand{\Rh}{\ensuremath{\mathrm{R^h}}}
\renewcommand{\L}{\ensuremath{\mathscr{L}}}
\newcommand{\Li}{\mathfrak{L}}
\newcommand{\APL}{\mathrm{A_{PL}}}
\newcommand{\CPL}{\mathrm{C_{PL}}}

\def\QQ{\mathbb{Q}}
\newcommand{\antishriek}{\text{\raisebox{\depth}{\textexclamdown}}}
\newcommand{\RT}{\mathrm{RT}}
\newcommand{\LRT}{\mathrm{LRT}}
\newcommand{\Sy}{\mathbb{S}}
\renewcommand{\d}{\ensuremath{\mathrm{d}}}
\newcommand{\PP}{\ensuremath{\mathrm{P}}}
\def\Ho#1#2{\Lambda^{#2}_{#1}}
\def\De#1{\Delta^{#1}}
\newcommand{\NN}{\mathbb{N}}
\newcommand{\RR}{\mathbb{R}}
\def\BCH{\mathrm{BCH}}
\newcommand{\wPT}{\ensuremath{\mathrm{wPT}}}

\newcommand{\PaRT}{\ensuremath{\mathrm{PaRT}}}
\newcommand{\PaPT}{\ensuremath{\mathrm{PaPT}}}
\newcommand{\PaPRT}{\ensuremath{\mathrm{PaPRT}}}
\def\rmC{\mathrm{C}}

\def\hot{\widehat{\otimes}} 
\def\whk{\widehat{k}}

\def\colim{\mathop{\mathrm{colim}}}

\newcommand{\coloneq}{\coloneqq}

\newcommand{\F}{\ensuremath{\mathrm{F}}}
\newcommand{\h}{\ensuremath{\mathfrak{h}}}

\newcommand{\C}{\ensuremath{\mathscr{C}}}

\renewcommand{\P}{\ensuremath{\mathscr{P}}}

\newcommand{\VdL}{\ensuremath{\mathrm{VdL}}}

\newcommand{\End}{\ensuremath{\mathrm{End}}}
\newcommand{\eend}{\ensuremath{\mathrm{end}}}

\newcommand{\T}{\ensuremath{\mathcal{T}}}

\newcommand{\Tw}{\ensuremath{\mathrm{Tw}}}
\newcommand{\Hom}{\ensuremath{\mathrm{Hom}}}
\renewcommand{\S}{\ensuremath{\mathbb{S}}}
\renewcommand{\k}{\ensuremath{\mathbb{K}}}
\newcommand{\id}{\ensuremath{\mathrm{id}}}
\newcommand{\MC}{\ensuremath{\mathrm{MC}}}
\newcommand{\mc}{\ensuremath{\mathfrak{mc}}}
\newcommand{\mclie}{\ensuremath{\overline{\mathfrak{mc}}}}

\newcommand{\B}{\mathcal{B}}

\newcommand{\ad}{\operatorname{ad}}

\newcommand{\PT}{\ensuremath{\mathrm{PT}}}

\makeatletter
\newcommand{\whiteleaf}{\mathord{\mathpalette\nicoud@YESNO\relax}}
\newcommand{\blackleaf}{\mathord{\mathpalette\nicoud@YESNO{\nicoud@path{\fillpath}}}}
\newcommand{\nicoud@YESNO}[2]{%
	\begingroup
	\settoheight{\unitlength}{$#1X$}%
	\begin{picture}(0.7,1)
	\linethickness{\variable@rule{#1}}%
	\roundcap\roundjoin
	\nicoud@path{\strokepath}
	#2
	\Line(0.35,0)(0.35,0.5)
	\end{picture}%
	\endgroup
}
\newcommand{\nicoud@path}[1]{%
	\moveto(0.1,0.5)
	\lineto(0.1,1)\lineto(0.6,1)\lineto(0.6,0.5)
	\closepath
	#1
}
\newcommand{\variable@rule}[1]{%
	\fontdimen8  
	\ifx#1\displaystyle\textfont3\else
	\ifx#1\textstyle\textfont3\else
	\ifx#1\scriptstyle\scriptfont3\else
	\scriptscriptfont3\relax
	\fi\fi\fi
}
\makeatletter

\author{Daniel Robert-Nicoud and Bruno Vallette}
\title{Higher Lie theory}
\date{\today}
\address{Daniel Robert-Nicoud}
\email{\href{mailto:daniel.robertnicoud@gmail.com}{daniel.robertnicoud@gmail.com}}
\address{Bruno Vallette, Universit\'e Sorbonne Paris Nord, Laboratoire de G\'eom\'etrie, Analyse et Applications, LAGA, CNRS, UMR 7539, F-93430, Villetaneuse, France}
\email{\href{mailto:vallette@math.univ-paris13.fr}{vallette@math.univ-paris13.fr}}

\makeindex

\thanks{
\indent$\,\! $  The first author was supported by grants from R\'egion Ile-de-France, the second author is supported by the Institut Universitaire de France.}


\begin{document}
	
\begin{abstract}
	We present a novel approach to the problem of integrating homotopy Lie algebras by representing the Maurer--Cartan space functor with a universal cosimplicial object. 
	This recovers Getzler's original functor but allows us to prove the existence of additional, previously unknown, structures and properties. Namely, we introduce a well-behaved left adjoint functor, we establish functoriality with respect to infinity-morphisms, and we construct a coherent hierarchy of higher Baker--Campbell--Hausdorff formulas. Thanks to these tools, we are able to establish the most important results of higher Lie theory. We use the recent developments of the operadic calculus, which leads us to explicit tree-wise formulas at all stage. We conclude by applying this theory to rational homotopy theory: the left adjoint functor is shown to provide us with homotopy Lie algebra models for topological spaces which faithfully capture their rational homotopy type.
\end{abstract}

\maketitle

\makeatletter
\def\@tocline#1#2#3#4#5#6#7{\relax
	\ifnum #1>\c@tocdepth 
	\else
	\par \addpenalty\@secpenalty\addvspace{#2}%
	\begingroup \hyphenpenalty\@M
	\@ifempty{#4}{%
		\@tempdima\csname r@tocindent\number#1\endcsname\relax
	}{%
		\@tempdima#4\relax
	}%
	\parindent\z@ \leftskip#3\relax \advance\leftskip\@tempdima\relax
	\rightskip\@pnumwidth plus4em \parfillskip-\@pnumwidth
	#5\leavevmode\hskip-\@tempdima
	\ifcase #1
	\or\or \hskip 1em \or \hskip 2em \else \hskip 3em \fi%
	#6\nobreak\relax
	\hfill\hbox to\@pnumwidth{\@tocpagenum{#7}}\par
	\nobreak
	\endgroup
	\fi}
\makeatother

\setcounter{tocdepth}{2}
\tableofcontents

\section*{Introduction}

\subsubsection*{\bf Higher Lie theory} At the center of classical Lie theory lies Lie's third theorem, which tells us how to integrate a finite dimensional real Lie algebra in order to obtain a simply connected real Lie group, providing an equivalence between the respective categories. The Baker--Campbell-Hausdorff (BCH) formula plays a fundamental role, inducing the group structure. In this paper, we  study \emph{higher Lie theory}, i.e.\ the extension of this theory to the derived world of differential graded Lie algebras and --- further still --- to homotopy Lie algebras.

\medskip

Replacing classical Lie algebras by their derived versions is mandatory in deformation theory, whose fundamental theorem, due to Pridham \cite{pri10} and Lurie \cite{lur11}, shows that every deformation problem is controlled by a differential graded Lie algebra  in characteristic zero, formalizing a heuristic principle that dates back to Deligne, see  \cite{Toen17}. 
In this domain, given an underling ``space'' together with a type of structure, one would like to classify all the possible structures present on that space up to some equivalence relations.
For example, one can study the classification of associative algebras structures on a chain complex up to isomorphisms \cite{Gerstenhaber64}, the classification of complex structures on Riemannian manifolds up to diffeomorphisms \cite{KodairaSpencer58}, or the classification of Poisson structures on manifolds up to diffeomorphisms
 \cite{Kontsevich03}. In each of these cases, there is a differential graded Lie algebra whose  Maurer--Cartan elements are in one-to-one correspondence with structures and whose action of the gauge group, obtained via the BCH formula, models the equivalence relation of interest. 
 
\medskip

This deformation theoretical information is contained in the \emph{Deligne groupoid}, whose points are the Maurer--Cartan elements of the differential graded Lie algebra and whose morphisms are the gauge equivalences. However, in the world of homotopy we are interested in comparing these equivalences, and then the equivalences between equivalences, and so on. This raises the  question of finding a Deligne ``$\infty$-groupoid'' which would faithfully encode this higher data. The picture that comes to mind is that of a topological space with points related by paths subject to homotopies and then homotopies of homotopies, etc. Indeed, according to Grothendieck's homotopy hypothesis, this is exactly the type of object that one should be looking for in order to get a suitable notion of an $\infty$-groupoid.

\medskip

The study of the rational homotopy theory of spaces is yet another domain 
where differential graded Lie algebras play a key role: as shown by Quillen \cite{Quillen69} with homotopical methods, they faithfully model the rational homotopy type. Using geometrical methods, Sullivan \cite{Sullivan77} settled a Koszul dual version of rational homotopy theory with quasi-free differential graded commutative algebras as models. Although the notion was not yet defined at the time, such a Sullivan model is equivalent to the structure of a \emph{homotopy Lie algebra} on the dual of its generators. This latter notion, which generalizes that of a differential graded Lie algebra by relaxing relations up to homotopy, has come to play a ubiquitous role in recent years.

\medskip

In the world of homotopy theory, where one considers objects up to quasi-isomorphisms, the notions of a differential graded Lie algebra and of a homotopy Lie algebra are equivalent. For instance, the Lurie--Pridham theorem can equally be stated using homotopy Lie algebras as it is formulated in terms of an equivalence of  $\infty$-categories. 
But in the algebraic world, where one works up to isomorphisms, this does not hold true and the notion of a homotopy Lie algebra is mandatory to encode some deformation problems in the way explained above. For instance, this is the case for the deformation theory of algebra structures over non-Koszul operads or of morphisms between algebras
 \cite{MerkulovVallette09I, MerkulovVallette09II}.
 Homotopy Lie algebras also give rise to a natural notion of ``morphisms up to homotopy'' --- commonly called $\infty$-morphisms --- that also play a very important role in deformation theory, see the deformation quantization of Poisson manifolds by Kontsevich \cite{Kontsevich03}. If one wants to do deformation theory using homotopy Lie algebras, the main issue is to coin the right generalization of the gauge group: what is the nature of the object that integrates homotopy Lie algebras?

\subsubsection*{\bf Present results} This article lies at the cornerstone of Lie theory, deformation theory, and rational homotopy theory. The entry door is the aforementioned quest for a deformation $\infty$-groupoid naturally associated to homotopy Lie algebras. 
A first answer to this question was given by
Hinich \cite{Hinich97bis} with the construction
\[
\MC_\bullet(\g)\coloneqq\MC\left(\g \, \widehat{\otimes}\, \Omega_\bullet\right),
\]
where $\Omega_\bullet$ stands for the simplicial commutative algebra of polynomial differential forms on the geometrical simplices of \cite{Sullivan77}. This object is a very good candidate: it forms a Kan complex, i.e.\ the simplicial incarnation of an $\infty$-groupoid, and its set of $0$-simplices corresponds to the set of Maurer--Cartan element of $\g$~. However, the set of $1$-simplices is too big to correspond to the gauge equivalences in the differential graded Lie case, although it encodes an equivalent equivalence relation.

\medskip

In his pioneering work, Getzler \cite{Getzler09} completely solved this problem by introducing a homotopically equivalent Kan sub-complex $\gamma_\bullet(\g)\subset\MC_\bullet(\g)$ constructed as follows. He considered the simplicial contraction
\[
\hbox{
	\begin{tikzpicture}
	\def\upshift{0.075}
	\def\downshift{0.075}
	\pgfmathsetmacro{\midshift}{0.005}
	
	\node[left] (x) at (0, 0) {$\Omega_\bullet$};
	\node[right=1.5 cm of x] (y) {$\mathrm{C}_\bullet$~,};
	
	\draw[-{To[left]}] ($(x.east) + (0.1, \upshift)$) -- node[above]{\mbox{\tiny{$p_\bullet$}}} ($(y.west) + (-0.1, \upshift)$);
	\draw[-{To[left]}] ($(y.west) + (-0.1, -\downshift)$) -- node[below]{\mbox{\tiny{$i_\bullet$}}} ($(x.east) + (0.1, -\downshift)$);
	
	\draw[->] ($(x.south west) + (0, 0.1)$) to [out=-160,in=160,looseness=5] node[left]{\mbox{\tiny{$h_\bullet$}}} ($(x.north west) - (0, 0.1)$);
	
	\end{tikzpicture}}
\]
introduced by Dupont in his simplicial version of the de Rham theorem \cite{Dupont76}, where $\rmC_n$ is the normalized chain complex of the geometric $n$-simplex, and defined
\[
\gamma_\bullet(\g)\coloneqq \MC_\bullet(\g)\cap \ker h_\bullet\ .
\]
Restricting to the kernel of $h_\bullet$ can be interpreted as the physicist's notion of a ``gauge condition''. Applied to a differential graded Lie algebra $\g$~, the Kan complex $\gamma_\bullet(\g)$ is nothing but the nerve
of its gauge group, that is the one associated to the BCH formula, see \cite{Bandiera14}.

\medskip

Since its introduction, Getzler's $\infty$-groupoid was rarely studied or used directly, contrarily to Hinich's version, see e.g.\ \cite{Henriques08, ber15, DolgushevRogers15, Bandiera17}. The reason for this lies in the complicated form of its intrinsic definition and the hard combinatorics generated by Dupont's contraction. This is unfortunate since, in the authors' opinion, this object should be central in both deformation theory and rational homotopy theory. Driven by this, the starting goal of the this article is to present an alternative definition of Getzler's $\infty$-groupoid with which it is easier to work and for which we can give explicit formulas.

\medskip

In order to do that, we start by using the operadic calculus \cite{LodayVallette12} in the form of the homotopy transfer theorem for a new type of homotopy commutative algebras applied to the Dupont contraction. This produces a   cosimplicial (complete shifted) homotopy Lie algebra 
\[
\mc^\bullet\coloneqq \widehat{\sLi}\left( \mathrm{C}^\bullet\right)
\]
which is freely generated by the normalized chain complex of the geometric simplices. Then an idea going back to Kan \cite{Kan58bis} tells us how to use it to obtain a couple of adjoint functors
\[
\hbox{
	\begin{tikzpicture}
	\def\upshift{0.185}
	\def\downshift{0.15}
	\pgfmathsetmacro{\midshift}{0.005}
	
	\node[left] (x) at (0, 0) {$\Li\ \ :\ \ \sSe$};
	\node[right] (y) at (2, 0) {$\sLialg\ \ :\ \ \R\ .$};
	
	\draw[-{To[left]}] ($(x.east) + (0.1, \upshift)$) -- ($(y.west) + (-0.1, \upshift)$);
	\draw[-{To[left]}] ($(y.west) + (-0.1, -\downshift)$) -- ($(x.east) + (0.1, -\downshift)$);
	
	\node at ($(x.east)!0.5!(y.west) + (0, \midshift)$) {\scalebox{0.8}{$\perp$}};
	\end{tikzpicture}}
\]
Our first result says that the right adjoint functor, called the \emph{integration functor}, corepresented by $\mc^\bullet$, 
\[
\mathrm{R}(\g)\coloneqq\Hom_{\sLialg}(\mc^\bullet, \g)
\]
recovers Getzler's construction.

\begin{ThmIntroRep}
	Getzler's functor is naturally isomorphic to the integration functor:
	\[
	\mathrm{R}(\g)\cong\gamma_\bullet(\g)\ .
	\]
\end{ThmIntroRep}

This new description has many advantages. For example, Hinich's space $\MC_\bullet(\g)$ is functorial with respect to $\infty$-morphisms, while Getzler's space was known to be functorial with respect to strict morphisms only. With the present new construction, we are able to prove functoriality of $\R$ with respect of a refined version of $\infty$-morphisms, which we dub $\infty_\pi$-morphisms. Moreover, this approach automatically gives us a left adjoint functor $\Li$ which provides us with explicit homotopy Lie algebra models for spaces. We dedicate the second part of the article to the study of this functor.

\medskip

Recall that the Kan property for a simplicial set is defined by the existence of horn fillers.
In the case of Getzler's functor, the set of horn fillers admits the following elegant description.

\begin{thmfillhorns}[{\cite{Getzler09}}]
	Given a horn $\Ho{k}{n} \to \R(\g)$~, the set of horn fillers is canonically and naturally in bijection with the set of  elements of degree $n$ of the (complete shifted)  $\mathcal{L}_\infty$-algebra $\g$~:
	\[
	\left\{\vcenter{\hbox{
			\begin{tikzpicture}
			\node (a) at (0, 0) {$\Ho{k}{n}$};
			\node (b) at (2, 0) {$\R(\g)$};
			\node (c) at (0, -1.5) {$\De{n}$};
			
			\draw[->] (a) to (b);
			\draw[right hook->] (a) to node[below=-0.05cm, sloped]{$\sim$} (c);
			\draw[->, dashed] (c) to (b);
			\end{tikzpicture}
	}}\right\} 
	\cong \g_n\ .
	\]
\end{thmfillhorns}

The Kan property is immediate, since $0\in\g_n$ so that a filler always exists. 
But this statement tells us more: $0$ provides us with a canonical way of filling horns, so that we actually get a \emph{canonical algebraic $\infty$-groupoid}. This latter notion \cite{Nickolaus11} is defined by a structure (the canonical horn fillers) and not by a property (existence); as such, it behaves much better than generic Kan complexes with respect to algebraic properties. This represents a \emph{change of paradigm} which makes Getzler's construction leave homotopy theory and enter the world of algebra.

\medskip

Although the aforementioned result is not new, we give it a new and \emph{constructive} proof using the left adjoint functor $\Li$~, leading to \emph{explicit formulas} for the horn fillers. Using this, we give another proof, after \cite{Bandiera14}, that the classical BCH formula is obtained by filling a $2$-simplex.
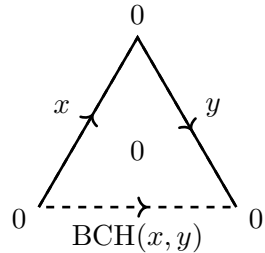
\[
\begin{tikzpicture}

\coordinate (v0) at (210:1.5);
\coordinate (v1) at (90:1.5);
\coordinate (v2) at (-30:1.5);

\draw[line width=1] (v0)--(v1)--(v2);
\draw[dashed, line width=1]  (v0)--(v1)--(v2)--cycle;

\begin{scope}[decoration={
	markings,
	mark=at position 0.55 with {\arrow{>}}},
line width=1
]

\path[postaction={decorate}] (v0) -- (v1);
\path[postaction={decorate}] (v1) -- (v2);
\path[postaction={decorate}] (v0) -- (v2);

\end{scope}


\node at ($(v0) + (30:-0.3)$) {$0$};
\node at ($(v1) + (0,0.3)$) {$0$};
\node at ($(v2) + (-30:0.3)$) {$0$};

\node at ($(v0)!0.5!(v1) + (-30:-0.4)$) {$x$};
\node at ($(v1)!0.5!(v2) + (30:0.4)$) {$y$};
\node at ($(v0)!0.5!(v2) + (0,-0.4)$) {$\BCH(x,y)$};

\node at (0,0) {$0$};

\end{tikzpicture}
\]
This requires us to give a new universal characterization of the BCH formula which is of independent interest (\cref{prop:uniqueness of classical BCH}) and provides us with a new expression for it as a sum of labeled trees. The real new interest lies in the fillers for the higher dimensional horns: they allow us to introduce a whole hierarchy of \emph{higher Baker--Campbell--Hausdorff products} for homotopy Lie algebras. Such a notion was just sketched at the very end of \cite{Getzler09}; we give it here the central position it deserves and describe these products with closed formulas.

\medskip

These higher BCH products provide us with a powerful new tool to study the homotopy theory of Maurer--Cartan spaces. 
We study their properties and apply them to get a short new proof of  Berglund's Hurewicz type theorem, which we show to be functorial with respect to $\infty_\pi$-morphisms.

\begin{thmBerglund}[{\cite{ber15}}]
	For each $n\geqslant1$~, there is a group isomorphism
	\begin{align*}
	\mathcal{B}_n\ :\ {H}_n(\g^\alpha){}&\ \stackrel{\cong}{\longrightarrow}\ \pi_n\left(\R(\g), \alpha\right)
	\end{align*}
	which is natural with respect to $\infty_\pi$-morphisms. The group structure on the right-hand side is given by the Baker--Campbell--Hausdorff formula, for $n=1$~,  and by the sum,  for $n\geqslant 2$~.
\end{thmBerglund}

This result shows how amenable is the deformation $\infty$-groupoid: its weak homotopy type can be computed via the homology groups of the original homotopy Lie algebra.

\medskip

We establish another structural result in this domain:  the homotopy invariance property of Maurer--Cartan spaces. This  is the end point of a long series of generalizations starting from Goldman--Millson \cite{gm88}, although the result therein is attributed to Schlessinger--Stasheff \cite{SS12} by the authors, and reaching a modern formulation with Dolgushev--Rogers \cite{DolgushevRogers15} on the level of Hinich's functor.

\begin{thmDR}
	Any filtered $\infty_\pi$-quasi-isomorphism $f : \g\stackrel{\sim}{\rightsquigarrow} \h$ of (complete shifted) $\mathcal{L}_\infty$-algebras induces a weak equivalence of simplicial sets $\R(f):\R(\g)\stackrel{\sim}{\to} \R(\h)$~.
\end{thmDR}

Finally, the higher Lie theory developed in this article admits a salient application in rational homotopy theory: the adjoint functors $\Li$ and $\R$ directly relate homotopy Lie algebras and simplicial sets, i.e.\ spaces, faithfully  preserving their rational homotopy type. Since the homotopy Lie algebra functor $\Li$ creates an ``artificial'' new point, corresponding to the canonical Maurer--Cartan element $0$~, we get rid of it by working with pointed spaces and ``reduced'' versions
\[
\hbox{
	\begin{tikzpicture}
	\def\upshift{0.185}
	\def\downshift{0.15}
	\pgfmathsetmacro{\midshift}{0.005}
	
	\node[left] (x) at (0, 0) {$\widetilde{\Li}\ \ :\ \ \sSe_*$};
	\node[right] (y) at (2, 0) {$\sLialg\ \ :\ \ \widetilde{\R}$};
	
	\draw[-{To[left]}] ($(x.east) + (0.1, \upshift)$) -- ($(y.west) + (-0.1, \upshift)$);
	\draw[-{To[left]}] ($(y.west) + (-0.1, -\downshift)$) -- ($(x.east) + (0.1, -\downshift)$);
	
	\node at ($(x.east)!0.5!(y.west) + (0, \midshift)$) {\scalebox{0.8}{$\perp$}};
	\end{tikzpicture}}
\]
of the original pair of functors. 

\begin{ThmIntroRHT}
For $X_\bullet$ a pointed connected finite simplicial set, the unit 
\[
X_\bullet \longrightarrow \widetilde{\R}\widetilde{\Li}(X_\bullet)
\]
of the $\widetilde{\Li}\dashv\widetilde{\R}$-adjunction is homotopically equivalent to the $\QQ$-completion of Bousfield--Kan. In particular, it a rationalization when  $X_\bullet$ is nilpotent.
\end{ThmIntroRHT}

This simplifies the Lie side of rational homotopy theory: recall that Quillen's original construction, associating a Lie algebra to a space, is done via the composite of many consecutive adjunctions.  A similar approach --- which was a source of inspiration for the present work --- was developed by Buijs--F\'elix--Tanr\'e--Murillo \cite{BFMT18}, but with a construction working only for differential graded Lie algebras.
The present proof of the aforementioned result relies on Sullivan's side of rational homotopy theory, after Bousfield--Guggenheim \cite{BousfieldGugenheim} and linear dualization, so it still carries with it a finiteness assumption. However, the way the present theory is established makes us believe that it should hold in a wider context. 

\subsubsection*{\bf ``Discourse on the method''}
A few words are in order about the methods and the range of applications. 
First, one might think that working with homotopy Lie algebras is much more difficult than working with differential graded Lie algebras due to their more complex definition. This is not the case: working with homotopy Lie algebras is \textit{easier and produces much cleaner formulas}. Conceptually, this comes from the fact that this higher notion is encoded by a (quasi)-free operad whereas its classical counterpart is governed by an intricate operadic quotient. In other words, the free Lie algebra over a chain complex is a rather involved object, due to the Jacobi relation, whereas the free homotopy Lie algebra is a simple construction, with a canonical basis given by  rooted trees. This is the key property to obtain explicit formulas for our constructions.

\medskip

Second, for the first time in higher Lie theory, we systematically apply the operadic calculus, including its most recent developments such as \cite{rn17tensor, Wierstra19, rnw17}. 
In order to achieve some of the present results, we need to improve the operadic calculus even further, see e.g.\ \cref{thm:convolution algebras and HTT} which establishes the compatibility of the homotopy transfer theorem with tensor products, providing results of independent interest. The main surprise for the authors was to realize that the suitable homotopy commutative structure one should consider on the normalized cochain level (Dupont forms) is the one governed by the ``unusual'' bar-cobar construction $\Omega \B \com$ of the operad encoding commutative algebras  instead of the more ``economical'' Koszul model $\Omega \com^{\antishriek}$~.

\medskip

Last but not least, we provide \emph{closed explicit formulas for all the notions and results present in this paper}. Since these are coming from a joint application of the operadic calculus, fixed-point equations or formal differential equations, cf.\ \cref{sec:app}, these formulas are all given in terms of trees (rooted, planar, partitioned, etc.). 
This brings us close to the original work of Cayley \cite{Cayley}, numerical analysis (B-series) \cite{Butcher, HLW10}, and to the more recent work in renormalization theory of Bruned--Hairer--Zambotti \cite{BHZ19}, associated to pre-Lie algebra structures. 
This is not a surprise since the totalization of an operad produces a pre-Lie algebra. For example, we fully describe in \cref{prop:Diffmc1} the universal (complete shifted) homotopy Lie algebra $\mc^1$, which represents gauges between Maurer--Cartan elements, and which is shown to recover the Lawrence--Sullivan differential graded Lie algebra model of the interval. We also completely characterize the notion of gauge equivalences for $\infty$-morphisms of homotopy Lie algebras in \cref{prop:GaugeInfMor}.

\subsubsection*{\bf Structure of the paper}
Our goal was to make this paper accessible to all researchers with basic knowledge of operad theory and homotopy theory. In this optic, \cref{sect:recollections} is dedicated to the recollection of the main concepts and structures underlying all of the other constructions in this work. Namely, we remind the reader what a complete shifted homotopy Lie algebra is, the basic notions of simplicial sets and algebraic $\infty$-groupoids, and a non-trivial amount of operadic calculus. The core of the article begins with \cref{sect:integration of Loo}, where we construct the universal homotopy Lie algebra representing Maurer--Cartan elements and we discuss the integration of homotopy Lie algebras to Maurer--Cartan spaces, including relations with the notions already existing in the literature. In \cref{sec:Func}, we introduce a refined notion of homotopy morphisms, called $\infty_\pi$-morphisms, with respect to which the Maurer--Cartan spaces are functorial.
\cref{sec:gauges} deals with the notion of gauges between Maurer--Cartan elements: we describe the universal complete shifted homotopy algebra $\mc^1$ as well as the notion of gauge equivalence between $\infty$-morphisms. 
Horn fillers as well as classical and higher Baker--Hausdorff products are the subject of \cref{sec:HighBCH}. 
\cref{sect:homotopy theory} treats the homotopical properties of Maurer--Cartan spaces: Berglund's Hurewicz type theorem, homotopy invariance of the integration functor, and model category structures. 
In \cref{sec: rational models}, we apply the theory developed in the previous sections to rational homotopy theory in order to recover the Bousfield--Kan $\QQ$-completion as the unit of our adjunction. 
Finally, \cref{sec:app} provides us with formulas for the solutions of formal fixed-point equations and formal differential equations.

\subsubsection*{\bf Conventions}
Throughout the article, we work over a field $\k$ of characteristic $0$~, except for \cref{sec: rational models}, where we specialize to the field of rational numbers $\k=\mathbb{Q}$~. 
We work with chain complexes under the homological degree convention, where the differentials have degree $-1$~; cochain complexes are viewed as chain complexes with opposite degree convention. 
We denote the (degree wise, arity wise) linear duals by $V^\vee$. 
The symbol $s$ denotes a formal element of degree $1$ and is used to suspend chain complexes under the convention $sV\coloneqq V\otimes\k s$~, that is  $(sV)_n\cong V_{n-1}$~. The symbol $s^{-1}$ denotes the (formal) dual of $s$ and receives degree $-1$~. 
The symbol $\S_m$ stands for the symmetric group on $m$ elements.
We denote by  $[n]$  the set of natural numbers up to $n$~, i.e.\ $[n]\coloneqq\{0, 1, \ldots, n\}$~. When $n$ is clear from the context, we simply denote
$
\widehat{k}\coloneqq [n]\backslash\{k\}$~.

\subsubsection*{\bf Acknowledgements}
We would like to express our appreciation to Benjamin Enriquez for sharing with us the arguments of \cref{lemma:Center}. It is a pleasure to thank 
Anton Alekseev, 
Alexander Berglund, 
Urtzi Buijs, 
Damien Calaque, 
Vladimir Dotsenko, 
Ezra Getzler,
Geoffroy Horel,
Aniceto Murillo,
Victor Roca Lucio, 
Pavol Severa, and
Felix Wierstra
for interesting discussions.

\section{Recollections}\label{sect:recollections}

Classical Lie theory consists of an equivalence between the categories of real Lie algebras and of simply connected real Lie groups. Its ``up to homotopy'' version, \emph{higher Lie theory}, revolves around a pair of Quillen adjoint functors between the categories of complete shifted homotopy Lie algebras and algebraic $\infty$-groupoids. We start this section by recalling the definitions and main properties of these notions, after which we provide the reader with some recollections on the operadic calculus, which is the main toolbox in the present work. Additionally, this section contains two news statements in the theory of operads which are of general interest and might be fruitfully used elsewhere. Namely, these are an extension of the properties of the bar and cobar constructions for algebras and coalgebras in the complete setting and the compatibility of the homotopy transfer theorem with tensor products. 

\subsection{Complete shifted homotopy Lie algebras}\label{subsec:sLooalg}

\begin{definition}[Complete chain complexes]
	A \emph{complete chain complex}\index{complete!chain complex} $(V, d, \mathrm{F})$ is a chain complex, i.e.\ a $\mathbb{Z}$-graded vector space equipped with a degree $-1$ differential $d$ that squares to zero, endowed with a degree-wise filtration 
	\[
	V_n=\F_0 V_n \supseteq \F_1 V_n \supseteq \F_2 V_n \supseteq \cdots \supseteq \F_k V_n \supseteq \F_{k+1}V_n \supseteq \cdots
	\]
	of sub-vector spaces that are preserved by the differential, i.e.\ $d(\F_k V_n)\subseteq \F_k V_n$~, and such that the associated topology is complete, i.e. $V_n\cong \lim_{k\to\infty} V_n/\F_k  V_n$ through the canonical map.
\end{definition}

A morphism of complete chain complexes is a chain map preserving the respective filtrations. The resulting category is denoted by $\ch$~. Any chain complex $V$ can be viewed as a complete chain complex endowed with the discrete filtration $V=\F_0 V \supseteq \F_1  V= \F_2 V= \F_3 V=  \cdots=0$~. 

\medskip

Equipped with the complete tensor product defined by $V\widehat{\otimes}W\coloneqq \widehat{V\otimes W}$~, where $\widehat{\ }$ denotes the completion functor, the category $\ch$ becomes a bicomplete closed symmetric monoidal category, see \cite[Section~1.2]{DotsenkoShadrinVallette18} for instance. The internal hom's are denoted by $\hom(V,W)\coloneqq \F_0 \Hom(V,W)$ and is made up of maps which respect the filtrations, its is equipped with the complete filtration given by
\[
\F_k \hom(V, W)\coloneqq\left\{
f :V \to W\ | \ f(\F_n V)\subset \F_{n+k} W\ , \ \forall n\in \NN
\right\}\ .
\]

\begin{definition}[Complete $\sLi$-algebra]
	A \emph{complete shifted $\L_\infty$-algebra structure}\index{complete!homotopy Lie algebra}\index{homotopy Lie algebra} (or a \emph{complete $\sLi$-algebra structure}) on a complete chain complex $\g$ is a collection $\ell_m\in\hom\big(\g^{\odot m}, \g\big)$ of symmetric maps of degree $-1$ for $m\geqslant2$ which respect the filtration and which satisfy 
	\[
	\partial\left(\ell_m\right)+\sum_{\substack{p+q=m\\ 2\leqslant p,q \leqslant m}}
	\sum_{\sigma\in \mathrm{Sh}_{p,q}^{-1}}
	(\ell_{p+1}\circ_{1} \ell_q)^{\sigma}=0\ ,
	\]
	for any $m\geqslant2$~, where $ \mathrm{Sh}_{p,q}^{-1}$ denotes the set of the inverses of $(p,q)$-shuffles and where $\partial\left(\ell_m\right)=d_\g\ell_m +\ell_m d_{\g^{\odot m}}$~. Here, the notation $\g^{\odot m}$ stands for the space of coinvariants $(\g^{\otimes m})_{\Sy_m}$~.
\end{definition}

In the rest of the article, we will sometimes us the convention $\ell_1\coloneqq d_{\g}$ to simplify some formulas. By a slight abuse of notation, we will usually use the same letter $\g$ for the underlying chain complex and the full data of complete $\sLi$-algebra.

\medskip

In the present paper, unless otherwise specified, we restrict ourselves to the case $\F_0 \g= \F_1 \g$ for complete $\sLi$-algebras. This induces no loss of generality: if one is confronted with a complete $\sLi$-algebra with $\F_0 \g\supset \F_1 \g$ then it is always possible to restrict to $\F_1\g$~, which forms a complete $\sLi$-algebra by itself. 

\medskip

These algebras are actually algebras over an operad, denoted by $\sLi$~, in the symmetric mo\-noi\-dal category of complete chain complexes, see \cref{subsec:Op}. As such, they are equipped with \emph{strict morphisms}, which are maps of degree $0$ commuting strictly with the structure operations. The resulting category is denoted by $\sLialg$~. This operadic interpretation allows us to apply to complete $\sLi$-algebras all the classical operadic constructions and results of \cite{LodayVallette12}, such as the following.
 
\begin{proposition}\label{prop:CatCocomp}
	The category of complete $\sLi$-algebras is locally small, complete, and cocomplete. 
\end{proposition}
 
\begin{proof}
	All the arguments of \cite[Lemma~4]{DotsenkoShadrinVallette18} hold true for complete chain complexes whose filtration satisfies $\F_0=\F_1$~. Therefore, the category of complete chain complexes is locally small, complete, and cocomplete, and so is the category of complete $\sLi$-algebras as a category of algebras over an operad, by \cite[Theorem~1]{DotsenkoShadrinVallette18}.
\end{proof}
 
\begin{remark} 
	Under the desuspension isomorphism of the underlying chain complexes: $\g \mapsto s^{-1}\g$~, the category of complete $\sLi$-algebras is isomorphic to the ``classical'' category of complete $\L_\infty$-algebras. In many instances, $\sLi$-algebras appear more naturally than their unsifted analogues and they have the great advantage of having a much simpler sign convention.
\end{remark}

For $m\geqslant2$~, we denote by $\RT_m$ the set of \emph{rooted trees}\index{trees!rooted}\index{$\RT$} with $m$ indexed leaves. These are graph theoretical trees with $m+1$ leaves, of which one is marked (the \emph{root} of the tree) and the others are numbered by unique labels from $1$ to $m$~. Each vertex of the tree is required to have at least two incoming edges and a unique outgoing edge. Homeomorphic trees that give the identity on the labels are identified. We do not require that these trees can be embedded into $2$-dimensional space in such a way that the leaves are in the order defined by their labels; as such, we can view these trees as objects living in $3$-dimensional space. The rooted tree with a single vertex and $m$ leaves is called the $m$-corolla. In general, the number of leaves of the tree is called the \emph{arity} of the tree.
The set $\RT_1\coloneqq\{|\}$ consists only of the trivial rooted tree. We denote the number of vertices of a rooted tree $\tau\in\RT$ by $|\tau|$~. We denote by $\overline{\RT}$ the set of \emph{reduced} rooted trees\index{trees!reduced rooted}, that is all rooted trees with the exception of the trivial one.

\begin{proposition}\label{prop:FreesLi}
Let $V$ be a chain complex with basis $\{v_i\}_{i\in I}$~. The \emph{free complete $\sLi$-algebra}\index{free complete homotopy Lie algebra} generated by $V$ is isomorphic to
\[
\wsLi(V)\cong \prod_{m\geqslant 1}
\k \RT_m \otimes_{\Sy_m} V ^{\otimes m}  
\ ,
\]
which is made up of series, indexed by  $m\geqslant 1$~, of linear combinations of rooted trees $\tau$~, of degree $-|\tau|$~, with $n$ leaves labeled by elements $v_i$~, such as
\[
\tau\left(v_{i_1}, v_{i_2}, v_{i_3}, v_{i_4}, v_{i_5}, v_{i_6}\right)=
\vcenter{\hbox{
		\begin{tikzpicture}
		\def\scale{0.6};
		\pgfmathsetmacro{\diagcm}{sqrt(2)};
		
		\def\xangle{35};
		\pgfmathsetmacro{\xcm}{1/sin(\xangle)};
		
		\coordinate (r) at (0,0);
		\coordinate (v11) at ($(r) + (0,\scale*1)$);
		\coordinate (v21) at ($(v11) + (180-\xangle:\scale*\xcm)$);
		\coordinate (v22) at ($(v11) + (\xangle:\scale*\xcm)$);
		\coordinate (v31) at ($(v22) + (45:\scale*\diagcm)$);
		\coordinate (l1) at ($(v21) + (135:\scale*\diagcm)$);
		\coordinate (l2) at ($(v21) + (0,\scale*1)$);
		\coordinate (l3) at ($(v21) + (45:\scale*\diagcm)$);
		\coordinate (l4) at ($(v22) + (135:\scale*\diagcm)$);
		\coordinate (l5) at ($(v31) + (135:\scale*\diagcm)$);
		\coordinate (l6) at ($(v31) + (45:\scale*\diagcm)$);
		
		\draw[thick] (r) to (v11);
		\draw[thick] (v11) to (v21);
		\draw[thick] (v11) to (v22);
		\draw[thick] (v21) to (l1);
		\draw[thick] (v21) to (l2);
		\draw[thick] (v21) to (l3);
		\draw[thick] (v22) to (l4);
		\draw[thick] (v22) to (v31);
		\draw[thick] (v31) to (l5);
		\draw[thick] (v31) to (l6);
		
		\node[above] at (l1) {$v_{i_1}$};
		\node[above] at (l2) {$v_{i_3}$};
		\node[above] at (l3) {$v_{i_4}$};
		\node[above] at (l4) {$v_{i_5}$};
		\node[above] at (l5) {$v_{i_2}$};
		\node[above] at (l6) {$v_{i_6}$};
		\end{tikzpicture}}}
\text{for}\quad \tau=
\vcenter{\hbox{
		\begin{tikzpicture}
		\def\scale{0.75};
		\pgfmathsetmacro{\diagcm}{sqrt(2)};
		
		\def\xangle{35};
		\pgfmathsetmacro{\xcm}{1/sin(\xangle)};
		
		\coordinate (r) at (0,0);
		\coordinate (v11) at ($(r) + (0,\scale*1)$);
		\coordinate (v21) at ($(v11) + (180-\xangle:\scale*\xcm)$);
		\coordinate (v22) at ($(v11) + (\xangle:\scale*\xcm)$);
		\coordinate (v31) at ($(v22) + (45:\scale*\diagcm)$);
		\coordinate (l1) at ($(v21) + (135:\scale*\diagcm)$);
		\coordinate (l2) at ($(v21) + (0,\scale*1)$);
		\coordinate (l3) at ($(v21) + (45:\scale*\diagcm)$);
		\coordinate (l4) at ($(v22) + (135:\scale*\diagcm)$);
		\coordinate (l5) at ($(v31) + (135:\scale*\diagcm)$);
		\coordinate (l6) at ($(v31) + (45:\scale*\diagcm)$);
		
		\draw[thick] (r) to (v11);
		\draw[thick] (v11) to (v21);
		\draw[thick] (v11) to (v22);
		\draw[thick] (v21) to (l1);
		\draw[thick] (v21) to (l2);
		\draw[thick] (v21) to (l3);
		\draw[thick] (v22) to (l4);
		\draw[thick] (v22) to (v31);
		\draw[thick] (v31) to (l5);
		\draw[thick] (v31) to (l6);
		
		\node[above] at (l1) {$1$};
		\node[above] at (l2) {$3$};
		\node[above] at (l3) {$4$};
		\node[above] at (l4) {$5$};
		\node[above] at (l5) {$2$};
		\node[above] at (l6) {$6$};
		\end{tikzpicture}}}
,
\]
and where the complete filtration is given by 
\[
\F_k \wsLi(V)\coloneqq \prod_{m\geqslant k} \k\RT_m\otimes_{\Sy_m} V^{\otimes m}.
\]
The differential is the sum of the internal differential on $V$ and the operation on trees that sums over every vertex of a tree and splits its corollas into two in all possible ways. 
The structure operations $\ell_m$ amount to grafting the $m$ rooted trees given as arguments by their root to the $m$ leaves of a single $m$-corolla.
\end{proposition}

\begin{proof}
	This result can be proved by hand, but we will not do so; a proof using operadic calculus is given in \cref{ex:OpsLI}.
\end{proof}

\begin{definition}[Maurer--Cartan elements]
A \emph{Maurer--Cartan element}\index{Maurer--Cartan!element} of an $\sLi$-algebra $\g$ is a degree $0$ element $\alpha\in\g_0$ satisfying the \emph{Maurer--Cartan equation}\index{Maurer--Cartan!equation}:
\[
d(\alpha)+\sum_{m\geqslant 2} {\textstyle \frac{1}{m!}}\ell_m(\alpha, \ldots, \alpha)=0 \ .
\]
We denote the set of all Maurer--Cartan elements of $\g$ by $\MC(\g)$~.
\end{definition}

\begin{proposition}[Twisting procedure]\label{prop:TwiProc}\index{twisting!of a homotopy Lie algebra}
Given a Maurer--Cartan element $\alpha$ of a complete $\sLi$-algebra $\g$~, the differential and operations 
\[
d^\alpha\coloneqq \sum_{k\geqslant 0} {\textstyle \frac{1}{k!}} \ell_{k+1}\big(\alpha^k, -\big) \qquad\text{and} \qquad 
\ell_m^\alpha\coloneqq \sum_{k\geqslant 0} {\textstyle \frac{1}{k!}} \ell_{k+m}\big(\alpha^k, -,  \ldots,  -\big)\qquad\text{for }m\geqslant2
\]
define a complete $\sLi$-algebra structure with the same filtration. This new structure is called the complete $\sLi$-algebra $\g$ \emph{twisted} by the Maurer--Cartan element $\alpha$ and it is denoted by $\g^\alpha$~.
\end{proposition}

\begin{proof}
This well-known result can be shown by a straightforward computation. We refer the interested reader to \cite[Chapter~3, Proposition~5.1]{DotsenkoShadrinVallette18} for a more conceptual proof by means of a gauge group action. 
\end{proof}

These considerations show that the tangent space of the Maurer--Cartan ``set'' at a point $\alpha$ is given by the kernel of the twisted differential:
\[
\mathrm{T}_\alpha\MC(\g)=\ker d^\alpha\ .
\]
This legitimates the following definition.

\begin{definition}[Gauge equivalence]\label{def:Gauge}
	Given a Maurer--Cartan element $\alpha\in\MC(\g)$ and a degree $1$ element $\lambda\in\g_1$ in a complete $\sLi$-algebra, we define the differential equation
	\[
	\frac{d}{dt}\gamma(t) = d^{\gamma(t)}(\lambda)
	\]
	with initial condition $\gamma(0) = \alpha$~. We define the \emph{gauge action}\index{gauge!action} of $\lambda$ on $\alpha$ by
	\[
	\lambda\cdot\alpha\coloneqq\gamma(1)\ .
	\]
	We say that two Maurer--Cartan elements $\alpha, \beta\in\MC(\g)$ are \emph{gauge equivalent}\index{gauge!equivalence} if there exists $\lambda\in\g_1$ such that $\lambda\cdot\alpha = \beta$~. The element $\lambda$ is called a \emph{gauge} from $\alpha$ to $\beta$ in this case.
\end{definition}

Since we are in the complete setting, the differential equation underlying the notion of gauges always admits a solution. Since the twisted differential squares to zero, we have
\[
d^{x}(\lambda)\in \mathrm{T}_{x}\MC(\g)
\]
for any $x\in\MC(\g)$ and $t\in\k$~, and it follows that  $\gamma(t)\in \MC(\g)$ if the path starts in the Maurer--Cartan set. It is well known  that the gauge equivalence relation is an equivalence relation; this is a consequence of \cref{lem:mc1} and \cref{thm:KanExt}. We denote it by $\sim$ and we write
\[
\mathcal{MC}(\g)\coloneqq \MC(\g)/\sim
\]
for the moduli space of Maurer--Cartan elements\index{Maurer--Cartan!moduli space} up to gauge equivalence.

\medskip

The differential equation of \cref{def:Gauge} can be solved explicitly using the methods of \cref{App:FormDE} as follows. The set $\PT$ of \emph{planar rooted trees}\index{trees!planar rooted}\index{$\PT$} consists of graph theoretical trees with one marked leaf (the \emph{root} of the tree) together with a specified embedding into $2$-dimensional space. This embedding induces a unique labeling of the leaves by going around the tree in clockwise direction starting from the root. The trivial tree $|$ is included in $\PT$~. The number of vertices of a planar rooted tree is denoted by $|\tau|$~. All vertices of a planar rooted tree have exactly one outgoing edge, but we allow for vertices with zero or one incoming edges.

\medskip

The decomposition of a planar rooted tree $\tau = \mathrm{c}_m\circ(\tau_1,\ldots,\tau_m)$ as the grafting of $m$ planar rooted trees $\tau_1,\ldots,\tau_m$ along the leaves of 
the root corolla $\mathrm{c}_m$ allows us to recursively define the coefficients $C(\tau)\in \mathbb{N}$ by
\[
C(\, |\, ) \coloneqq 1\ ,\qquad C\left(\mathrm{c}_m\right) \coloneqq m!\ ,\qquad C(\tau) \coloneqq m! |\tau|\prod_{i=1}^m C(\tau_i)\ ,
\]
and the elements $\tau^\lambda(\alpha)\in \g_0$ by
\[
|^\lambda(\alpha)\coloneqq \alpha\ ,\qquad \mathrm{c}_m^\lambda(\alpha)\coloneqq\ell_{m+1}(\alpha, \ldots, \alpha, \lambda)\ ,\qquad \tau^\lambda(\alpha) \coloneqq \ell_{m+1}\left(\tau_1^\lambda(\alpha),\ldots, \tau_m^\lambda(\alpha), \lambda \right).
\]

\begin{example}
For the planar rooted tree 
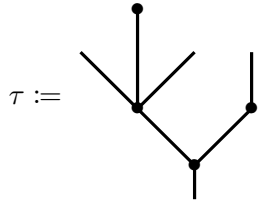
\[
\tau\coloneqq \vcenter{\hbox{
	\begin{tikzpicture}
		\def\scale{0.75}
		\pgfmathsetmacro{\diagcm}{sqrt(2)};
		
		\coordinate (r) at (0,0);
		
		\coordinate (v11) at ($(r) + (0,\scale*0.6)$);
		
		\coordinate (v21) at ($(v11) + (135:\scale*\diagcm)$);
		\coordinate (v22) at ($(v11) + (45:\scale*\diagcm)$);
		
		\coordinate (v31) at ($(v21) + (0,\scale*1.75)$);
		
		\coordinate (l1) at ($(v21) + (135:\scale*\diagcm)$);
		\coordinate (l2) at ($(v21) + (45:\scale*\diagcm)$);
		\coordinate (l3) at ($(v22) + (0,\scale*1)$);
		
		\node at (v11) {$\bullet$};
		\node at (v21) {$\bullet$};
		\node at (v22) {$\bullet$};
		\node at (v31) {$\bullet$};
		
		\draw[very thick] (r) to (v11);
		\draw[very thick] (v11) to (v21);
		\draw[very thick] (v11) to (v22);
		\draw[very thick] (v21) to (v31);
		\draw[very thick] (v21) to (l1);
		\draw[very thick] (v21) to (l2);
		\draw[very thick] (v22) to (l3);
	\end{tikzpicture}}}
\]
the coefficient $C(\tau)$ is equal to $96$ and the element $\tau^\lambda(\alpha)$ is equal to 
\[
\tau^\lambda(\alpha) = \ell_3(\ell_4(\alpha, d(\lambda), \alpha, \lambda), \ell_2(\alpha, \lambda), \lambda)\ ,
\] 
which corresponds graphically to 
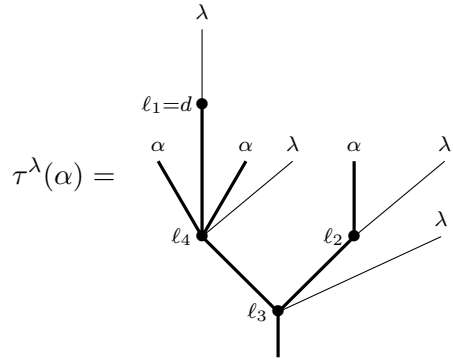
\[\tau^\lambda(\alpha)= \vcenter{\hbox{
\begin{tikzpicture}
	\def\scale{1}
	\pgfmathsetmacro{\diagcm}{sqrt(2)};
	\pgfmathsetmacro{\xcm}{1/sin(60)};
	
	\def\angle{40}
	\pgfmathsetmacro{\ycm}{1/sin(\angle)};
	
	\def\anglebis{25}
	\pgfmathsetmacro{\zcm}{1/sin(\anglebis)};
	
	\coordinate (r) at (0,0);
	
	\coordinate (v11) at ($(r) + (0,\scale*0.6)$);
	
	\coordinate (v21) at ($(v11) + (135:\scale*\diagcm)$);
	\coordinate (v22) at ($(v11) + (45:\scale*\diagcm)$);
	
	\coordinate (v31) at ($(v21) + (0,\scale*1.75)$);
	
	\coordinate (l1) at ($(v21) + (120:\scale*\xcm)$);
	\coordinate (l2) at ($(v21) + (60:\scale*\xcm)$);
	\coordinate (l3) at ($(v22) + (0,\scale*1)$);
	
	\coordinate (lam1) at ($(v31) + (0,\scale*1)$);
	\coordinate (lam2) at ($(v21) + (\angle:\scale*\ycm)$);
	\coordinate (lam3) at ($(v22) + (\angle:\scale*\ycm)$);
	\coordinate (lam4) at ($(v11) + (\anglebis:\scale*\zcm)$);
	
	\node at (v11) {$\bullet$};
	\node at (v21) {$\bullet$};
	\node at (v22) {$\bullet$};
	\node at (v31) {$\bullet$};
	
	\draw[very thick] (r) to (v11);
	\draw[very thick] (v11) to (v21);
	\draw[very thick] (v11) to (v22);
	\draw[very thick] (v21) to (v31);
	\draw[very thick] (v21) to (l1);
	\draw[very thick] (v21) to (l2);
	\draw[very thick] (v22) to (l3);
	
	\draw[thin] (v31) to (lam1);
	\draw[thin] (v21) to (lam2);
	\draw[thin] (v22) to (lam3);
	\draw[thin] (v11) to (lam4);
	
	\node[above] at (l1) {$\scriptstyle{\alpha}$};
	\node[above] at (l2) {$\scriptstyle{\alpha}$};
	\node[above] at (l3) {$\scriptstyle{\alpha}$};
	
	\node[above] at (lam1) {$\scriptstyle{\lambda}$};
	\node[above] at (lam2) {$\scriptstyle{\lambda}$};
	\node[above] at (lam3) {$\scriptstyle{\lambda}$};
	\node[above] at (lam4) {$\scriptstyle{\lambda}$};
	
	\node[left] at (v11) {$\scriptstyle{\ell_3}$};
	\node[left] at (v21) {$\scriptstyle{\ell_4}$};
	\node[left] at (v22) {$\scriptstyle{\ell_2}$};
	\node[left] at (v31) {$\scriptstyle{\ell_1 = d}$};
\end{tikzpicture}}}
\]
\end{example}

The resulting formula is similar (and gives the same output) to the one of \cite[Proposition~5.9]{Getzler09}.

\begin{proposition}\label{prop:gaugeformula}
	Let $\g$ be a complete $\sLi$-algebra, let $\alpha\in\MC(\g)$~, and let $\lambda\in\g_1$~. Then
	\[
	\lambda\cdot\alpha = \sum_{\tau\in \PT}{\textstyle \frac{1}{C(\tau)}}\tau^\lambda(\alpha)
\ .\]
\end{proposition}

\begin{proof}
	This is an application of \cref{prop:formula for formal ODE} to the family of functions 
	\[
	f_{m}^{0}(x_1,\ldots,x_m)\coloneqq{\textstyle \frac{1}{m!}}\, \ell_{m+1}(x_1,\ldots,x_m,\lambda)\ , \quad \text{for} \quad m\geqslant 0\ ,
	\]
	and $f^k_m=0$~, for $k\geqslant 1$~. Notice that even if the finiteness assumption of \cref{prop:formula for formal ODE} is not satisfied here, the completeness property of the underling chain complex ensures that the resulting formula makes sense. 
\end{proof}

\subsection{Algebraic \texorpdfstring{$\infty$}{infinity}-groupoids}

\begin{definition}[Simplex category]
The \emph{simplex category}\index{simplex category}, denoted by $\cD$~, is the category whose objects are the totally ordered sets $[n]\coloneqq\{0<  \cdots< n\}$~, for  $n\in \NN$~, and whose morphisms are the (non-strictly) order-preserving maps.
\end{definition}

\begin{definition}[Simplicial sets]
The category of \emph{simplicial sets}\index{simplicial set}, denoted by $\sSe$~, is the category of presheaves $\sSe\coloneqq\mathsf{Fun(}\cD^\mathsf{op}, \mathsf{Set)}$ on the simplex category. 
\end{definition}

Presheaves on the opposite category are called \emph{cosimplicial sets}\index{cosimplicial set}. To avoid confusion, we denote simplicial sets by $X_\bullet$ and cosimplicial sets by $X^\bullet$~, respectively. We use the classical convention $X_n\coloneqq X([n])$ for the \emph{set of $n$-simplices}. A simplicial set is equivalent to a collection of sets $\{X_n\}_n\geqslant0$ together with \emph{face} and \emph{degeneracy maps} $d_{n, i}:X_n\to X_{n-1}$ and $s_{n, i}:X_n\to X_{n+1}$ for $0\leqslant i\leqslant n$ satisfying certain relations called the \emph{simplicial identities}, see e.g.\ \cite[Section~I.1.]{GoerssJardine09}.

\begin{example}
For $n\in \NN$~, the \emph{standard $n$-simplex} is the simplicial set 
$\De{n}\coloneqq\Hom_{\cD}(-,[n])$ represented by $[n]$~. This combinatorial object encodes the cellular decomposition of the \emph{geometric $n$-simplex}, which is the convex hull of the unit vectors in $\RR^{n+1}$~:
\[
|\De{n}|\coloneqq\left\{
(t_0, \ldots, t_n) \in [0,1]^{n+1}\, | \, t_0+\cdots+t_n=1
\right\} \ .
\]
Taken all together, the geometric simplices form a cosimplicial topological space $|\De{\bullet}|$~. 
\end{example}

\begin{definition}
	The \emph{$k{\text{th}}$ horn $\Ho{k}{n}$ of dimension $n$}\index{horn}, for $n\geqslant 2$ and $0\leqslant k \leqslant n$~, is the union of all of the faces of $\Delta^n$ except for the $k$th one.
\end{definition}

In other words, the horn $\Ho{k}{n}$ is obtained from the standard $n$-simplex by removing the top dimensional face and the  facet opposite to the $k^{\text{th}}$ vertex. We call an \emph{$n$-horn} in a simplicial set $X_\bullet$ a morphism of simplicial sets 
$\Ho{k}{n}\to X_\bullet$ and a \emph{horn filler}\index{horn!filler} a morphism of simplicial sets $\De{n}\to X_\bullet$ which lifts the given horn:
\[
\begin{tikzpicture}
\node (a) at (0, 0) {$\Ho{k}{n}$};
\node (b) at (1.75, 0) {$X_\bullet$};
\node (c) at (0, -1.5) {$\De{n}$};

\draw[->] (a) to (b);
\draw[right hook->] (a) to node[below=-0.05cm, sloped]{$\sim$} (c);
\draw[->, dashed] (c) to (b);
\end{tikzpicture}
\]

\begin{definition}[Kan complex]
A \emph{Kan complex}\index{Kan complex} is simplicial set $X_\bullet$ in which any horn $\Ho{k}{n}\to X_\bullet$~, for $n\geqslant 2$ and $0\leqslant k \leqslant n$~, admits a filler.
\end{definition}

A classical and very important example of Kan complexes is the following.

\begin{proposition}
The simplicial set $\Hom_{\mathsf{Top}}(|\De{\bullet}|, X)$ of singular chains associated to any topological space $X$ is a Kan complex.
\end{proposition}

\begin{proof}
We refer to \cite[Lemma~3.3]{GoerssJardine09} for a proof of this classical result. 
\end{proof}

Moreover, the singular chains functor is faithful (but not full) and Kan complexes are known to share the same homotopy theory than compactly generated Hausdorff spaces, notably after the works of D.\ Quillen \cite{Quillen67}. For this reason, one often refers to Kan complexes simply as \emph{spaces}. 

\medskip

In the context of $\infty$-categories, see e.g.\ \cite{Lurie09}, Kan complexes play the role of the higher analogue of categories where every morphism is invertible, i.e.\ groupoids. Thus, one also refers to Kan complexes as \emph{$\infty$-groupoids}\index{infinity-groupoids}.

\medskip

The notion of a Kan complex is  defined by a \emph{property}; one make it into a \emph{structure}, which is actually an algebra over a monad, by the following definition due to T.\ Nickolaus \cite{Nickolaus11}, called \emph{algebraic Kan complex} in \textit{loc.\ cit.} 

\begin{definition}[Algebraic $\infty$-groupoid]
An \emph{algebraic $\infty$-groupoid}\index{algebraic infinity-groupoid} is a simplicial set $X_\bullet$ with a given filler for each $\Ho{k}{n}\to X_\bullet$~, for $n\geqslant 2$ and $0\leqslant k \leqslant n$~.
\end{definition}

\begin{proposition}\label{prop:NerveGroupoid}
The nerve 
\[\mathrm{N}\mathsf{G}_n\coloneqq\left\{\vcenter{\hbox{
\begin{tikzcd}
c_0 \arrow[r, "f_1"]& 
c_1 \arrow[r, "f_2"]&
\cdots \arrow[r, "f_n"] &
c_n\end{tikzcd}
}}\right\}
\]
of a groupoid $\mathsf{G}$~,  where the face and degeneracy maps are given respectively by
\begin{align*}
	d_i(f_1, \ldots, f_n)\coloneqq{}&\begin{cases}
		(f_2, \ldots, f_n) &\text{for}\ i=0\ , \\
		(f_1, \ldots, f_{i+1}f_i, \ldots, f_n) & \text{for}\ 1\leqslant i\leqslant n-1\ , \\
		(f_1, \ldots, f_{n-1})&\text{for}\ i=n\ ,
	\end{cases}\\
	s_i(f_1, \ldots, f_n)\coloneqq{}&(f_1, \ldots, f_i, \mathrm{id}, f_{i+1}, \ldots, f_n)
\end{align*}
is an algebraic $\infty$-groupoid. 
\end{proposition}

\begin{proof}
In this case, there is no choice: every horn admits exactly one filler. 
\end{proof}

We refer the reader to \cite{GoerssJardine09} for more details on simplicial sets and their homotopy theory.

\subsection{Operadic calculus}\label{subsec:Op}

\begin{definition}[$\Sy$-modules]
An \emph{$\S$-module}\index{symmetric module} $\P$ is a collection $\{\P(m)\}_{m\geqslant 0}$ of right $\S_m$-modules. If $p\in\P(m)$~, we say that $p$ has \emph{arity}\index{arity} $m$~.
\end{definition}

The following assignment  defines a functor, called the \emph{Schur functor}\index{Schur functor}, from the category of $\Sy$-modules to the category of endofunctors of vector spaces:
\[
\P(V)\coloneqq \bigoplus_{m\geqslant 0} \P(m)\otimes_{\Sy_m} V^{\otimes m}\ .
\]
The category of $\Sy$-modules admits a monoidal product $\circ$ which reflects the composition of Schur functors, that is $(\P\circ \mathcal{Q})(V)\cong \P(\mathcal{Q}(V))$~, see \cite[Corollary~5.1.4]{LodayVallette12} for a detailed formula.
The unit for this monoidal product is the $\Sy$-module $\mathcal{I}$ concentrated in arity $1$~, where $\mathcal{I}(1)=\k\id$~. 
 The category of vector spaces can be viewed as the sub-category 
of $\Sy$-modules concentrated in arity $0$~; under this interpretation,  we get $\P \circ V =\P(V)$~. 

\begin{definition}[Operads and cooperads]
In the monoidal category of $\Sy$-modules with $\circ$~, a monoid is called an \emph{operad} and a comonoid is called a \emph{cooperad}. 
\end{definition}

\begin{example}
Given any chain complex $V$~, the $\Sy$-module 
$\End_V\coloneqq \left\{\Hom(V^{\otimes m}, V)\right\}_{m\geqslant 0}$
given by all of the multilinear operations on $V$ carries an operad structure given by the usual composite of functions. This operad is called the \emph{endomorphism operad of $V$}\index{endomorphism operad}. In the filtered case, the endomorphism operad made up of the internal hom's is denoted by $\eend_V\coloneqq \left\{\hom(V^{\otimes m}, V)\right\}_{m\geqslant 0}$~.
\end{example}

The data of an operad is thus equivalent to a monad structure on the associated Schur functor, and dually for a cooperad. 

\begin{definition}[Algebras over an operad and coalgebras over a cooperad]
An \emph{algebra over an operad} is an algebra over the associated Schur monad.
A \emph{coalgebra over a cooperad} is a coalgebra over the associated Schur comonad.
\end{definition}

An immediate consequence of this definition is the fact that the free algebra\index{free!algebra} on a chain complex is obtained simply by applying the Schur functor to it. Dually, the cofree coalgebra is also obtained by an application of the Schur functor.

\medskip

An equivalent definition of a $\P$-algebra structure on a chain complex $A$ is a morphism of operads $\P\to \End_A$~.

\begin{example}
We consider the $\Sy$-module $\com$ given in arity $m\geqslant1$ by the trivial representation of the symmetric group $\Sy_m$~, that is $\com(m)\coloneqq \k \mu_m$ where the generating element $\mu_m$ is viewed as a corolla with $m$ symmetric leaves. We endow it with the operad structure defined by $(\mu_k; \mu_{i_1}, \ldots, \mu_{i_k})\mapsto \mu_{i_1+\cdots +i_k}$~. The associated Schur monad is the monad of polynomials without constant term
\[
\com(V)=\bigoplus_{m\geqslant 1} (V^{\otimes m})_{\Sy_m}=
\bigoplus_{m\geqslant 1} V^{\odot m}
\ .
\]
Thus, the category of $\com$-algebras is the category of commutative algebras (without unit).
The arity-wise linear dual $\com^\vee$  forms a cooperad whose category of coalgebras is the category of  cocommutative coalgebras (without counit).
\end{example}

The above definitions extend \emph{mutatis mutandis} to the underlying category of complete chain complexes, see \cite[Chapter~1, Section~2]{DotsenkoShadrinVallette18}. 
For instance, given a differential graded operad $\P$~, the free complete $\P$-algebra\index{free!complete algebra} on a chain complex $V$ is given by 
\[
\widehat{\P}(V)\coloneqq\P\, \widehat{\circ}\, V= \prod_{m\geqslant 0} \P(m)\otimes_{\Sy_m} V^{\otimes m}\ .
\]
In the present article, we will only consider differential graded (co)operads but viewed as discrete complete and whose underlying $\Sy$-module satisfies $\mathcal{M}(0)=0$ and $\mathcal{M}(1)=\k\id$~. For instance, they will all be canonically (co)augmented. 

\medskip

Given a cooperad $\C$ and an operad $\P$~, the mapping $\Sy$-module $\{\Hom(\C(m), \P(m))\}_{m\geqslant 0}$ is naturally endowed with a canonical operad structure denoted by $\Hom(\C, \P)$ and called the \emph{convolution operad}. It further induces a Lie algebra structure on the totalization 
\[
\Hom_{\Sy}(\C, \P)\coloneqq \prod_{m\geqslant 1} \Hom_{\Sy_m}(\C(m), \P(m))\ .\]
We refer the reader to \cite[Section~6.4]{LodayVallette12} for more details.

\begin{definition}[Operadic twisting morphism]
A Maurer--Cartan element in the Lie  algebra 
\[
\big(\Hom_{\Sy}(\C, \P), \partial, [\, , ]\big)\ ,
\]
which vanishes respectively on the augmentation of $\P$ and on the coaugmentation of $\C$~,  
is called an \emph{operadic twisting morphism}\index{twisting morphism}. The set of all such operadic twisting morphisms is denoted by $\Tw(\C,\P)$~. 
\end{definition}

The bifunctor of twisting morphisms is (co)represented by the bar the cobar construction respectively. Recall that the \emph{bar construction}\index{bar construction!of an operad} of an augmented operad $\P$ is defined by the cofree cooperad on the suspension of the augmentation ideal of $\P$~: 
\[
\Bar \P\coloneqq \left(\T^c\big(s \overline{\P}\big), d_1+d_2 
\right)
\]
where $d_1$ is the unique coderivation which extends the internal differential of $\P$ and where 
$d_2$ is the unique coderivation which extends the infinitesimal composition maps of $\P$~.
The \emph{cobar construction}\index{cobar construction!of an operad} of a coaugmented cooperad $\C$ is 
defined by the free operad on the desuspension of the coaugmentation coideal of $\C$~: 
\[\Cobar \C\coloneqq \left(\T\big(s^{-1} \overline{\C}\big), d_1+d_2 
\right)\]
where $d_1$ is the unique derivation which extends the internal differential of $\C$ and where 
$d_2$ is the unique derivation which extends the infinitesimal decomposition maps of $\C$~.
We refer the reader to \cite[Section~6.5]{LodayVallette12} for more details. 

\begin{proposition}[{\cite[Theorem~6.5.7]{LodayVallette12}}]\label{prop:Rosetta}
There exist natural bijections
\[\Hom_{\mathsf{Op}}(\Omega \C, \P)\cong \Tw(\C, \P)\cong \Hom_{\mathsf{coOp}}(\C, \Bar \P) \ .\]
\end{proposition}

The bar-cobar adjunction gives rise to the universal operadic twisting morphisms 
\[\iota: \C\to \Cobar \C \qquad \text{and}\qquad \pi : \Bar\P \to \P  \ , \]
associated respectively to the unit and the counit of adjunction 
\[\upsilon: \C \to \Bar \Cobar \C   \qquad \text{and} \qquad  \varepsilon: \Cobar \Bar \P \to \P\ .\]
Recall from \cite[Theorem~6.6.3]{LodayVallette12} that these latter two maps are quasi-isomorphisms. 

\medskip

\cref{prop:Rosetta}  shows that the data of an $\Omega \C$-algebra structure on $A$ is equivalent to an operadic twisting morphism $\C \to \End_A$~. 

\begin{example}\label{ex:OpsLI}
The cobar construction of the cooperad $\com^\vee$ produces the operad encoding shifted homotopy Lie algebras: 
\[\sLi\coloneqq \Cobar \com^\vee=
\Big(\T\Big(s^{-1} \overline{\com}^\vee\Big), d_2\Big)\ .\]
Since the free operad is made up of rooted trees, the set $\RT$ provides us with a basis of the operad $\sLi$~, where vertices carry degree $-1$~. The form of the operad structure on $\com$ given above shows that the differential $d_2$ amounts to the operation summing summing over each vertex of a tree and all possible splittings of the respective corolla into two. 
These facts together with the form of the free complete algebra over an operad prove \cref{prop:FreesLi}.
\end{example}

The operad encoding shifted homotopy Lie algebras actually plays the following universal role with respect to operadic twisting morphisms. 

\begin{lemma}[{\cite[Section~7]{Wierstra19} and \cite[Theorem~3.1]{rn17tensor}}]\label{lem:sLiTW}
There are natural bijections 
\[
\Tw(\C, \P)\cong \Hom_{\mathsf{Op}}(\sLi, \Hom(\C,\P))\cong \Hom_{\mathsf{Op}}\Big(\sLi, \P\otimes\C^\vee\Big)\ , 
\]
where $\C$ is required to be arity-wise finite dimensional for the second bijection to hold, and where $\otimes$  stands for the arity-wise tensor product of operads.
\end{lemma}

Given any $\C$-coalgebra $C$ and any complete $\P$-algebra $A$~, one can endow the complete mapping space $\hom(C,A)$ with a canonical complete $\Hom(\C, \P)$-algebra structure. Then \cref{lem:sLiTW} above implies the following.

\begin{definition}[Convolution algebra]\label{def:convolution algebra}
Given any operadic twisting morphism $\alpha : \C \to \P$~, \cref{lem:sLiTW} endows 
 $\hom(C,A)$ with a complete $\sLi$-algebra structure called the \emph{convolution algebra}\index{convolution algebra} of $C$ and $A$~, and denoted by $\hom^\alpha(C,A)$~.
 \end{definition}

An analogous result holds for the tensor product: the $\sLi$-algebra structure produced by \cref{lem:sLiTW} on the complete tensor product of a complete $\P$-algebra $A$ and a complete $\C^\vee$-algebra $B$ is denoted by 
$A\, \widehat{\otimes}^\alpha B$~. 
Under finite dimensionality hypotheses on the cooperad $\C$ and a $\C$-coalgebra $C$~, the linear duality isomorphism induces an isomorphism of $\sLi$-algebras:
\[
\hom^\alpha(C,A)\cong A\, \widehat{\otimes}^\alpha C^\vee~.
\]
Notice that since $C$ is discrete and since $\F_0 A =\F_1 A$~, the induced filtration on the mapping space satisfies 
$\F_0 \hom(C, A)=\F_1 \hom(C, A)$~. 

\begin{definition}[Twisting morphism with respect to $\alpha$]\label{def:alpha-oo-morphism}
	A Maurer--Cartan element in the convolution $\sLi$-algebra is called a \emph{twisting morphism with respect to $\alpha$}\index{twisting morphism!of algebras}
	; their set is denoted by 
	\[
	\Tw_\alpha(C,A)\coloneqq \MC(\hom^\alpha(C,A))\ .
	\]
\end{definition}

One can again this bifunctor gives rise to bar and cobar constructions. 
The \emph{bar construction with respect to $\alpha$}\index{bar construction!of an algebra}\index{relative!bar construction} of a complete $\P$-algebra $A$ is 
defined by the conilpotent cofree $\C$-coalgebra on $A$ after forgetting its filtration,
\[\Bar_\alpha A\coloneqq \big({\C}(A),  d_1+d_2 \big)\ , \]
where $d_1$ is the unique coderivation which extends the internal differential $d$ of $A$ and where 
$d_2$ is the unique coderivation which extends the following composite
\[
{\C}(A) \xrightarrow{\alpha(\id_A)} \P(A)\subseteq
\widehat{\P}(A)
 \xrightarrow{\gamma}
A\ .
\]
The \emph{complete cobar construction with respect to $\alpha$}\index{cobar construction!of an algebra}\index{relative!cobar construction} of a $\C$-coalgebra $C$ is 
defined by the  free complete $\P$-algebra on $C$
\[
\hatCobar_\alpha C\coloneqq \big(\widehat{\P}(C),  d_1+d_2 \big)
\]
equipped with the trivial and complete filtration $C=\F_0\supseteq \F_1 C= \F_2 C= \cdots=0$~, 
where $d_1$ is the unique derivation which extends the internal differential $d$ of $C$ and where 
$-d_2$ is the unique derivation which extends the following composite
\[ 
C \xrightarrow{\delta} 
\widehat{\C}(C) \xrightarrow{\alpha(\id_C)}
\widehat{\P}(C)
\ .\]
This is a direct generalization of \cite[Section~11.2.5]{LodayVallette12} to the complete case. 
We denote simply 
by $\C\text{-}\,\mathsf{coalg}$ the category of $\C$-coalgebras 
and 
by $\P\text{-}\,\mathsf{alg}$ the category of complete $\P$-algebras, and we have the two functors
\[
\hbox{
	\begin{tikzpicture}
	\def\upshift{0.075}
	\def\downshift{0.075}
	\pgfmathsetmacro{\midshift}{0.005}
	
	\node[left] (x) at (0, 0) {$\hatCobar_\alpha\ \ :\ \ \C\text{-}\mathsf{coalg}$};
	\node[right] (y) at (1.5, 0) {$\P\text{-}\mathsf{alg}\ \ :\ \ \Bar_\alpha\ .$};
	
	\draw[-{To[left]}] ($(x.east) + (0.1, \upshift)$) -- ($(y.west) + (-0.1, \upshift)$);
	\draw[-{To[left]}] ($(y.west) + (-0.1, -\downshift)$) -- ($(x.east) + (0.1, -\downshift)$);
	\end{tikzpicture}}
\]
Unlike the non-complete case, these functors are \emph{not} adjoint.

\begin{proposition}\label{thm:MC elements of convolution algebras}
Let $\alpha : \C \to \P$ be an operadic twisting morphism. There exist a natural bijection
\[
	\hom_{\P\text{-}\mathsf{alg}}\Big(\widehat{\Cobar}_\alpha C,A\Big)\cong \Tw_\alpha(C,A)\ .
\]
\end{proposition}

\begin{proof}
The discrete case, when the $\C$-coalgebra $C$ is nilpotent, was settled in \cite[Theorem~2.7]{rnw17}, where it is proved that the Maurer--Cartan equation associated to the convolution algebra $\hom^\alpha(C,A)$ is equal to Maurer--Cartan equation given in \cite[Section~11.1.1]{LodayVallette12}. All the computations performed in \emph{loc.\ cit.} hold in the complete case as well. 
\end{proof}

\begin{remark}
In order to get a right adjoint to the complete cobar construction, one would need to refine the definition of the bar construction. This is more subtle since it requires to consider a non-trivial pullback of the completion of $\C(A)$ in order to get the cofree coalgebra, see \cite{Smith03, Anel14}. We will not go into such a refinement here since we do  not need such a construction.

\medskip

In the context of algebras over cooperads and coalgebras over operads, a similar generalization of the classical bar-cobar adjunction to the complete case was developed systematically in \cite[Section~8]{LGL18}.
\end{remark}

\begin{proposition}\label{CobarColimits}
	The complete cobar construction $\hatCobar_\alpha$ preserves colimits. 
\end{proposition}

\begin{proof}
	For this proof, let $D:\mathsf{S}\to\C\text{-}\mathsf{coalg}$ be a diagram of $\C$-coalgebras and let $A$ be a complete $\P$-algebra.
	
	\medskip
	
	First we show that
	\[
	\hom^\alpha\left(\colim_\mathsf{S}D, A\right)\cong\lim_\mathsf{S}\hom^\alpha(D, A)\ .
	\]
	The statement is obviously true at the level of chain complexes, since the chain complex of a colimit of coalgebras corresponds with the one of the colimit of the underlying chain complexes, and the same is true for limits of algebras. The canonical inclusions $i_s:D(s)\to\colim_\mathsf{S}D$ induce algebra morphisms
	\[
	i_s^* : \hom^\alpha\left(\colim_\mathsf{S}D, A\right)\longrightarrow\hom^\alpha\left(D(s), A\right)
	\]
	for each $s\in\mathsf{S}$~. By universal property, there exists a unique morphism of algebras
	\[
	\hom^\alpha\left(\colim_\mathsf{S}D, A\right)\longrightarrow\lim_\mathsf{S}\hom^\alpha(D, A)\ ,
	\]
	but since it must agree with the isomorphism of chain complexes, it is an isomorphism, proving the claim.
	
	\medskip
	
	Any morphism of complete $\P$-algebras $\hatCobar_\alpha D(s) \to A$ is equivalent to a Maurer--Cartan element in
	$\MC\left(\hom^\alpha(D(s), A)
	\right)$
	by \cref{thm:MC elements of convolution algebras}. So it corresponds to a morphism of $\sLi$-algebras 
	$\mc^0\to \hom^\alpha(D(s), A)$~, where $\mc^0$ is the quasi-free $\sLi$-algebra on a Maurer--Cartan element, see \cref{prop:mc0}. 
	By the universal property of limits and what we said above, there exists a unique morphism of $\sLi$-algebras 
	\[\mc^0\to \lim_{\mathrm{S}} \hom^\alpha(D, A)\cong  \hom^\alpha\left(\colim_{\mathrm{S}}\mathrm{D}, A\right), \]
	which uniquely corresponds to a morphism $\hatCobar_\alpha\left(\colim_{\mathrm{S}} D\right)\to A$~. Thus, $\hatCobar_\alpha\left(\colim_{\mathrm{S}} D\right)$ satisfies the universal property of colimits, implying that
	\[
	\colim_{\mathrm{S}}\hatCobar_\alpha D\cong
	\hatCobar_\alpha\left(\colim_{\mathrm{S}} D\right).\]
\end{proof}

Let us establish the following homological property, which generalizes the discrete case, since we will need it later on. 

\begin{proposition}\label{prop:EpsiResFonc}
Let $\alpha : \C \to \P$ be an operadic Koszul morphism. 
The map 
\[
\epsilon_A \ \colon\  \hatCobar_\alpha \Bar_\alpha A \stackrel{\sim}{\longrightarrow} A 
\]
defined on the underlying space by 
\[
\widehat{\P}(\C(A))\twoheadrightarrow \widehat{\P}(A) \xrightarrow{\gamma_A} A
\]
is a  quasi-isomorphism which is natural in the category of complete $\P$-algebras such that $\F_1 A=A$~. 
\end{proposition}

\begin{proof}
Notice first that the map $\epsilon_A$ is a morphism of complete $\P$-algebras. It is defined by the same formula as the counit of the classical bar-cobar adjunction associated to $\alpha$~, see \cite[Section~11.3.2]{LodayVallette12}; so it is a morphism of $\P$-algebras. The complete filtration on the left-hand side is given by the arity, that is 
\[\F_k \widehat{\P}(\C(A))= \prod_{m\geqslant k} \P(m)\otimes_{\Sy_m} (\C(A))^{\otimes m}\ .\]
The condition $\F_1 A=A$ ensures that the image 
of any element in $\P(m)\otimes_{\Sy_m} A^{\otimes m}$ 
under the $\P$-algebra structure map $\gamma_A$  lives in $\F_m A$~. This shows that
\[
\epsilon_A\big(\F_k \widehat{\P}(\C(A))\big)\subset \F_k A\ .
\]
Since $\C(0)=0$ by assumption, the underlying space of the left-hand side is isomorphic to 
\begin{equation}\label{eq:isoCounitAdj}\tag{$\ast$}
\widehat{\P}(\C(A))\cong \prod_{m\geqslant 1} \P(m)\otimes_{\Sy_m} (\C(A))^{\otimes m}
\cong \prod_{p\geqslant 1}\left(\P\circ \C\right)(p)\otimes_{\Sy_p} A^{\otimes p}
\ .
\end{equation}
On it , we consider the descending filtration $\mathcal{F}$
defined by 
\[
\mathcal{F}_k\coloneqq \prod_{p\geqslant 1}\left(\P\circ \C\right)(p)\otimes_{\Sy_p} \F_{k+p}\left(A^{\otimes p}\right).
\]
The condition $\F_1 A=A$ ensures that $\mathcal{F}_0=\widehat{\P}(\C(A))$~; it is straightforward to see that  this filtration is complete. The summand of the differential coming from $d_A$ and the summand of the differential coming from the part $d_2$ of the complete cobar construction both send $\mathcal{F}_k$ to $\mathcal{F}_k$~.The summand of the differential coming from the part $d_2$ of the differential of the bar construction $\Bar_\alpha A$ sends $\mathcal{F}_k$ to $\mathcal{F}_{k+1}$ since $\C(0)=0$ and $\C(1)=\id$ by assumption. Therefore, the zeroth page $E^0$ of the associated spectral sequence is made up of the same underlying space but equipped with a differential which is the sum of the one coming from $d_A$ and the 
one corresponding to the differential of $\P\circ_\alpha \C$~. 
Equipped with this differential, each term 
$\left(\P\circ \C\right)(p)\otimes_{\Sy_p} A^{\otimes p}$ forms a chain sub-complex of $E^0$, which is acyclic, for $p>1$~, by Masche and operadic K\"unneth theorems \cite[Proposition 6.2.3]{LodayVallette12}, and whose homology is isomorphic to $H(A)$~, for $p=1$~, 
since $\alpha$ is a Koszul morphism. 
Using the isomorphism~\eqref{eq:isoCounitAdj}, we conclude that $E^1\cong H(A)$~.

\medskip

Again, the condition $\F_1 A=A$ ensures that the map $\epsilon_A$ sends $\mathcal{F}_k$ onto $\F_k A$~. The arguments mentioned above show that one gets an isomorphism on the level of the first page $E^1$~. Since both  filtrations are exhaustive and complete, the Eilenberg--Moore Comparison \cite[Theorem~5.5.11]{WeibelBook} implies that $\epsilon_A$ is a quasi-isomorphism. 
 \end{proof}

\begin{remark}
Applying the composite functor $\hatCobar_\alpha\Bar_\alpha$~, when $\alpha$ a Koszul morphism, one recovers a version of the \emph{homotopy completion} of Harper--Hess \cite{HarperHess13}. 
In the present context, \cref{prop:EpsiResFonc} implies that any complete $\P$-algebra $A$ satisfying $\F_1 A=A$ is \emph{homotopy complete} according to \cite[Definition~5.13]{CPRNW19}.
\end{remark}

The bar constructions allow us to define higher notions of morphisms. In the case of homotopy Lie algebras, they happen to play a crucial role in deformation theory, see \cite{Kontsevich03} for a seminal example.

\begin{definition}[$\infty_\alpha$-morphism]\label{def:InfMor}
Let $\alpha : \C \to \P$ be an operadic twisting morphism. An \emph{$\infty_\alpha$-morphism}\index{infinity-morphism!relative to a twisting morphism} $f : A \rightsquigarrow B$ between two complete $\P$-algebras $A$ and $B$ is a morphism of $\C$-coalgebras 
$f : \Bar_\alpha A \to \Bar_\alpha B$ between the bar constructions such that 
\[
f\left(\C(p)\otimes_{\Sy_p} \F_k\left(A^{\otimes p}\right)\right)\subset 
\bigoplus_{m\geqslant 1}\C(m)\otimes_{\Sy_m} \F_k\left(B^{\otimes m}\right)
\]
for any $k,p \geqslant 1$~. Composition is given by the usual composition of morphisms of coalgebras.
\end{definition}

Such an $\infty_\alpha$-morphism amounts to a degree $0$ map $f : {\C}(A) \to B$ satisfying the relations coming from imposing commutation with respective differentials.
The class of strict morphisms $f : A \to B$  of $\P$-algebras is equal to the class $\infty_\alpha$-morphisms of the form 
\[
{\C}(A) \xrightarrow{\varepsilon_{\C}(A)} A \xrightarrow{f} B\ , 
 \]
where $\varepsilon_{\C}$ stands for the counit of the cooperad $\C$~.

\medskip

The category of complete $\P$-algebras with the $\infty_\alpha$-morphisms as morphisms is denoted by
$\infty_\alpha\text{-}\,\P\text{-}\,\mathsf{alg}$~.

\medskip

An $\infty_\alpha$-morphism 
is called an \emph{$\infty_\alpha$-isomorphism} (respectively an \emph{$\infty_\alpha$-quasi-isomorphism}) 
when its first component 
\[
A\cong \mathcal{I}(A) \xrightarrow{\eta(A)} {\C}(A) \xrightarrow{f} B\ ,
\]
where $\eta$ stands for the coaugmentation of the cooperad $\C$~, 
is an isomorphism (respectively a quasi-isomorphism). 
The refined notions of $\infty_\alpha$-morphisms carry properties that strict morphisms lack. 
The class of $\infty_\alpha$-isomorphism is the class of invertible $\infty_\alpha$-morphisms: the computations of 
\cite[Theorem~10.4.1]{LodayVallette12} extend \emph{mutatis mutandis}. The $\infty_\alpha$-quasi-isomorphisms are homotopy invertible, see \cite[Theorem~10.4.4]{LodayVallette12}. 

\medskip

For $\Cobar \C$-algebras, we often consider the universal twisting morphism $\iota : \C \to \Cobar \C$ and the associated notion of an  $\infty_\iota$-morphism, which are simply called \emph{$\infty$-morphisms}\index{infinity-morphism}.

\begin{example}\label{ex:Loomorph}
In the case of the cooperad $\com^\vee$, an $\infty$-morphism between two complete $\sLi$-algebras $\g$ and $\h$ 
amounts to a collection of symmetric maps $f_m : \g^{\hat{\odot} m} \to \h$  of degree $0$~, for any $m\geqslant 1$~, preserving the respective filtrations, and satisfying the  relations
\[
\partial(f_m)=\sum_{\substack{p+q=m\\ q\geqslant 2}}
\sum_{\sigma\in \mathrm{Sh}_{p,q}^{-1}}
 (f_{p+1}\circ_{1} \ell_q)^{\sigma}-
 \sum_{\substack{i_1+\cdots+i_k=m\\ k\geqslant 2}}
\sum_{\sigma\in \mathrm{Sh}_{i_1, \ldots, i_k}^{-1}}
 \mathcal{k}_k\circ (f_{i_1}, \ldots, f_{i_k})^{\sigma}\ ,
\]
where $\{\ell_m\}_{m\geqslant 2}$ are the structural operations of $\g$ and 
$\{\mathcal{k}_m\}_{m\geqslant 2}$ the ones of $\h$~, 
see \cite[Section~10.2.6]{LodayVallette12} for more details. 
This notion plays a key role in the proof of the deformation quantization of Poisson manifolds by M.\ Kontsevich \cite{Kontsevich03}. In the literature, often one refers to this notion as ``filtered $\sLi$-morphism,'' making explicit the operad and the fact that they need to respect the filtrations.
A strict morphism of complete $\sLi$-algebras is equivalent to an $\infty$-morphism with trivial components $f_m=0$~, for $m\geqslant 2$~. 
The category of complete $\sLi$-algebras with their $\infty$-morphisms is denoted by $\isLialg$~.
A fundamental property of $\infty$-morphisms of complete $\sLi$-algebras is that the ``image'' 
\[
f(\alpha)\coloneqq \sum_{m\geqslant 1} {\textstyle \frac{1}{m!}} f_m(\alpha, \ldots, \alpha)\in \MC(\h)
\]
of  any Maurer--Cartan element $\alpha\in \MC(\g)$ is again a Maurer--Cartan element of $\h$~. This defined the action of $\MC : \isLialg \to \textsf{Set}$ on morphisms, making it into a functor.
\end{example}

The two dual constructions $\hom^\alpha(C,A)$ and $A\, \widehat{\otimes}^\alpha C^\vee$ are functorial 
with respect to $\infty_\alpha$-mor\-phi\-sms 
on the left-hand side and on the right-hand side separately, see \cite{rnw17}; we will recall such structures in the body of the text when needed.

\medskip

We now briefly recall the \emph{homotopy transfer theorem} --- a centerpiece of the homotopy theory of operadic algebras --- before studying its compatibility with the twisted tensor product. 

\begin{theorem}[Homotopy transfer theorem, {see \cite[Section~10.3]{LodayVallette12}}]\label{thm:HTT}\index{homotopy transfer theorem}
Given a contraction of complete chain complexes
\[
\hbox{
	\begin{tikzpicture}
	\def\upshift{0.075}
	\def\downshift{0.075}
	\pgfmathsetmacro{\midshift}{0.005}
	
	\node[left] (x) at (0, 0) {$X$};
	\node[right=1.5 cm of x] (y) {$Y$};
	
	\draw[-{To[left]}] ($(x.east) + (0.1, \upshift)$) -- node[above]{\mbox{\tiny{$p$}}} ($(y.west) + (-0.1, \upshift)$);
	\draw[-{To[left]}] ($(y.west) + (-0.1, -\downshift)$) -- node[below]{\mbox{\tiny{$i$}}} ($(x.east) + (0.1, -\downshift)$);
	
	\draw[->] ($(x.south west) + (0, 0.1)$) to [out=-160,in=160,looseness=5] node[left]{\mbox{\tiny{$h$}}} ($(x.north west) - (0, 0.1)$);
	
	\end{tikzpicture}}
\]
and a complete $\Cobar \C$-algebra structure on $X$~, there is a complete $\Cobar \C$-algebra structure on $Y$ such that the maps $i$ and $p$ extends to $\infty$-quasi-isomorphisms $i_\infty:Y\rightsquigarrow X$ and $p_\infty:X\rightsquigarrow Y$~. 
\end{theorem}

 If we denote by 
\[
\varphi_X:\C\longrightarrow\End_X
\]
 the algebra structure of $X$ as an operadic twisting morphism,  then  an explicit formula for such a transferred structure on $Y$ is given by
\[
\varphi_Y\coloneqq\left(\C\xrightarrow{\Delta_\C^\mathrm{mon}}\T(\C)\xrightarrow{\T(s\varphi_X)}\T(s\End_X)\xrightarrow{\VdL_h}\End_Y\right)\ ,
\]
where $\Delta_\C^\mathrm{mon}$ denotes the comonadic definition of the cooperad $\C$ and where
$\VdL_h$ is the \emph{Van der Laan} map associated to the contraction $h$~. This latter map  
amounts to composing the rooted trees with the leaves labeled by $i$~, the internal edges by $h$ and the root by $p$~.
For an more conceptual approach to the homotopy transfer theorem, we refer the reader to \cite{HLV19}, where e.g.\ the functorial property with respect to morphisms of cooperads is established. 
 
\medskip

Let us finish this section with the following compatibility between the tensor product and the homotopy transfer theorem. 
We consider an arity-wise finite dimensional operad $\P$ and the universal operadic twisting morphism 
$\pi : \Bar \P \to \P$~. We write $\C\coloneqq\P^\vee$ for the dual cooperad. Given a contraction as above with $X$ equipped with a complete $\Omega \C$-algebra structure and a complete $\P$-algebra $A$~. 
There are two natural ways to construct a $\sLi$-algebra structure on $A\, \hot\,  Y$~:
\begin{enumerate}
	\item one can first apply the homotopy transfer theorem \ref{thm:HTT} to endow $Y$ with a complete $\Cobar\C$-algebra structure and then  consider the complete $\sLi$-algebra $A\, \widehat{\otimes}^\pi Y$~, or
	\item one can consider the induced contraction
	\[
	\hbox{
		\begin{tikzpicture}
		\def\upshift{0.075}
		\def\downshift{0.075}
		\pgfmathsetmacro{\midshift}{0.005}
		
		\node[left] (x) at (0, 0) {$A\,\hot \,X$};
		\node[right=1.5 cm of x] (y) {$A\,\hot \,Y$};
		
		\draw[-{To[left]}] ($(x.east) + (0.1, \upshift)$) -- node[above]{\mbox{\tiny{$\id\,\hot\, p$}}} ($(y.west) + (-0.1, \upshift)$);
		\draw[-{To[left]}] ($(y.west) + (-0.1, -\downshift)$) -- node[below]{\mbox{\tiny{$\id\,\hot\, i$}}} ($(x.east) + (0.1, -\downshift)$);
		
		\draw[->] ($(x.south west) + (0, 0.1)$) to [out=-160,in=160,looseness=5] node[left]{\mbox{\tiny{$\id\,\hot\, h$}}} ($(x.north west) - (0, 0.1)$);
		
		\end{tikzpicture}}
	\]
	and  transfer the complete $\sLi$-algebra structure of $A\,\hot^\pi X$ to obtain a complete $\sLi$-algebra structure on $A\,\hot\, Y$~.
\end{enumerate}

\begin{proposition}\label{thm:convolution algebras and HTT}
	The two complete $\sLi$-algebra structures on $A\,\hot\, Y$ described above are equal.
\end{proposition}

\begin{proof}
	We compare the two structures explicitly. Let us recall that the morphism of operads of \cref{lem:sLiTW} 
	\[
	\mathsf{M}_\pi:\sLi\longrightarrow\P \otimes \Cobar\C
	\]
associated to the operadic twisting morphism $\pi$ is given by 
	\[
	\mathsf{M}_\pi\big(s^{-1}\mu_m^\vee\big) = \sum_{i\in I(m)}(-1)^{c_i}p_i\otimes s^{-1}c_i\ ,
	\]
	where $\mu_m\in \com(m)$ is the canonical generator, $\{p_i\}_{i\in I(m)}$ is a basis of $\P(m)$~, and $c_i\coloneqq p_i^\vee$ is the dual basis
	of $\C(m)=\P(m)^\vee$. Notice that $(-1)^{|c_i|}s^{-1}c_i = (sp_i)^\vee$. 
	We encode the complete $\P$-algebra structure of $A$ by the morphism of operads
	\[
	\rho_A:\P\longrightarrow\eend_A\ .
	\]
	We encode the $\Cobar\C$-algebra structures on $X$ in terms of the operadic twisting morphism 
	\[
	\varphi_X:\C\longrightarrow\eend_X\ ,
	\]
	and similarly for the other algebras, like the one on $Y$ obtained by the homotopy transfer theorem and the complete $\sLi$-algebra structure on $A\hot X$~. Given two $\S$-modules $M$ and $N$~, the map
	\[
	\Phi:\T\left(M \otimes N\right)\longrightarrow\T(M)\otimes\T(N)
	\]
	is given by sending a tree with vertices indexed by pure tensors in $M{\otimes} N$ to the tensor product of two copies of the underlying tree with the vertices of the first one indexed by the respective elements of $M$ and the vertices of the second one 
	indexed by the respective elements of $N$~, all multiplied by the appropriate Koszul sign, see \cite[Section~3.1]{Vallette08}.
	
	\medskip	
	
		The first complete $\sLi$-algebra structure is given by 
	\begin{align*}
	\ell_m =&\ \left(\rho_A\otimes\varphi_Y\right)s\mathsf{M}_\pi s^{-1}\big(\mu_m^\vee\big)\\
	=&\ \left(\rho_A\otimes\VdL_h\T(s\varphi_X)\Delta_\C^\mathrm{mon}\right)s\mathsf{M}_\pi s^{-1}\big(\mu_m^\vee\big)\\
	=&\ \left(\rho_A\otimes\VdL_h\T(s\varphi_X)\right)\left(\id_\P\otimes\Delta_\C^\mathrm{mon}\right)s\mathsf{M}_\pi s^{-1}\big(\mu_m^\vee\big)\ ,
	\end{align*}
	where in the third line we see $s\mathsf{M}_\pi s^{-1}$ as a map $\com^\vee\to\P\otimes\C$~, given by
	\[
	s\mathsf{M}_\pi s^{-1}\big(\mu_m^\vee\big) = \sum_{i\in I(m)}p_i\otimes c_i\  .
	\]
	The second complete $\sLi$-algebra structure on $A\hot Y$ is given by
	\begin{align*}
	\ell_m =&\ \VdL_{1 \hot h}\T(s\varphi_{A\hot^\pi X})\Delta_{\com^\vee}^\mathrm{mon}\big(\mu_m^\vee\big)\\
	=&\ \left(\gamma_{\eend_A}\otimes\VdL_h\right)\Phi\T(s\varphi_{X\hot^\pi A})\Delta_{\com^\vee}^\mathrm{mon}\big(\mu_m^\vee\big)\\
	=&\ \left(\gamma_{\eend_A}\otimes\VdL_h\right)\Phi\T\left((\rho_A\otimes s\varphi_X)s\mathsf{M}_\pi s^{-1}\right)\Delta_{\com^\vee}^\mathrm{mon}\big(\mu_m^\vee\big)\\
	=&\ \left(\gamma_{\eend_A}\otimes\VdL_h\right)\left(\T(\rho_A)\otimes\T(s\varphi_X)\right)\Phi\T(s\mathsf{M}_\pi s^{-1})\Delta_{\com^\vee}^\mathrm{mon}\big(\mu_m^\vee\big)\\
	=&\ \left(\rho_A\otimes\VdL_h\T(s\varphi_X)\right)\left(\gamma_\P\otimes\id_{\T(\C)}\right)\Phi\T(s\mathsf{M}_\pi s^{-1})\Delta_{\com^\vee}^\mathrm{mon}\big(\mu_m^\vee\big)\ ,
	\end{align*}
where, in the last line, we used the fact that
	\[
	\gamma_{\eend_A}\T(\rho_A) = \rho_A\gamma_\P\ .
	\]
	Therefore, to conclude we need to show that
\begin{equation}\label{eq:Interm}
	\left(\gamma_\P\otimes\id_{\T(\C)}\right)\Phi\T(s\mathsf{M}_\pi s^{-1})\Delta_{\com^\vee}^\mathrm{mon}\big(\mu_m^\vee\big) = \left(\id_\P\otimes\Delta_\C^\mathrm{mon}\right)s\mathsf{M}_\pi s^{-1}\big(\mu_m^\vee\big)\ .
\end{equation}
We introduce the following notation
	\[
	\Delta_{\com^\vee}^\mathrm{mon}(\mu_m^\vee) = \sum_{\tau\in\RT_m}\tau\left(v\mapsto\mu_{|v|}^\vee\right),
	\]
where the right-hand side represents the sum of rooted trees with vertices labeled by the elements $\mu_m^\vee$ of according arity.	With a similar notation, the left-hand side of \cref{eq:Interm}  is equal to
	\[
	\sum_{\tau\in\RT_m}(-1)^\epsilon\sum_{\substack{i_v\in I(|v|)\\ \text{ for }v\in V_\tau}}\gamma_\P(\tau(v\mapsto p_{i_v}))\otimes\tau(v\mapsto c_{i_v})\ ,
	\]
	where $\epsilon$ is the Koszul sign coming from $\Phi$~. At the same time, the right-hand side of \cref{eq:Interm} is given by
	\[
	\sum_{i\in I(m)}p_i\otimes\Delta_\C^\mathrm{mon}(c_i)\ .
	\]
	Both are expression for the map
	\[
	\gamma_\P:\T(\P)\longrightarrow\P
	\]
	seen as an element of $\P \otimes\T(\P)^\vee$, and thus they are equal.
\end{proof}

\begin{remark}
Similarly, there are two natural ways to extend the chain maps $\id \,\hot\, i$ and $\id\, \hot\, p$ into $\infty$-quasi-isomorphisms. 
With the same method as above, one can prove that these two extensions are equal.  We refer the reader to \cite[Chapter~9]{rnPHD} for more details. 
\end{remark}
\section{Integration theory of complete shifted homotopy Lie algebras}\label{sect:integration of Loo}

In this section, we introduce the \emph{universal Maurer--Cartan algebra}, which is a cosimplicial complete $\sLi$-algebra that encodes all of the properties of Maurer--Cartan elements: the Maurer--Cartan equation, but also the gauge equivalence relation and its higher homotopy properties. We then use it to define a functor $\R$ that integrates complete $\sLi$-algebras into simplicial sets, and to introduce a left adjoint functor $\Li$ which provides us with complete $\sLi$-algebra models for simplicial sets. We show that our newly defined functor generalizes the Dold--Kan adjunction and study its relation with the Maurer--Cartan spaces of Deligne--Hinich and of Getzler, proving that they are all equivalent.

\subsection{Adjunctions with simplicial sets}
We start with the following key observation, which goes back to D.\ M.\ Kan.

\begin{lemma}[{\cite[Section~3]{Kan58bis}}]\label{lem:AdjCosimpliObj}
Let $\mathsf{C}$ be a locally small cocomplete category. 
The data of a pair of adjoint functors whose left adjoint has domain in simplicial sets
\[
\hbox{
	\begin{tikzpicture}
	\def\upshift{0.185}
	\def\downshift{0.15}
	\pgfmathsetmacro{\midshift}{0.005}
	
	\node[left] (x) at (0, 0) {$\Li\ \ :\ \ \sSe$};
	\node[right] (y) at (2, 0) {$\mathsf{C}\ \ :\ \ \mathrm{R}$};
	
	\draw[-{To[left]}] ($(x.east) + (0.1, \upshift)$) -- ($(y.west) + (-0.1, \upshift)$);
	\draw[-{To[left]}] ($(y.west) + (-0.1, -\downshift)$) -- ($(x.east) + (0.1, -\downshift)$);
	
	\node at ($(x.east)!0.5!(y.west) + (0, \midshift)$) {\scalebox{0.8}{$\perp$}};
	\end{tikzpicture}}
\]
is equivalent to the data of a cosimplicial object
$\mathrm{C}^\bullet : \De{} \to \mathsf{C}$
in $\mathsf{C}$ under the restriction
\[
\vcenter{\hbox{
\begin{tikzcd}
\mathrm{C}^\bullet=\mathrm{L}\mathrm{Y}\ : \ \cD
\arrow[r, "\mathrm{Y}"]
&
\sSe \arrow[r, "\mathrm{L}"]
&
\mathsf{C}
\end{tikzcd}
}}
\]
of the left adjoint functor to  the Yoneda embedding $\mathrm{Y}$~. 
\end{lemma}

\begin{proof}
Any cosimplicial object $\mathrm{C}^\bullet : \cD \to \mathsf{C}$ induces a functor
\begin{align*}
	\mathrm{R}\ :\ \mathsf{C}{}&\ \longrightarrow\ \sSe\\
	c{}&\ \longmapsto\ \Hom_{\mathsf{C}}(\mathrm{C}^\bullet, c)
\end{align*}
which admits a left adjoint given by the following left Kan extension
\[
\vcenter{\hbox{
	\begin{tikzpicture}
		\node (a) at (0,0){$\cD$};
		\node (b) at (3.5,0){$\mathsf{C}$};
		\node (c) at (1.75,-1.5){$\sSe$};
		
		\draw[->] (a) to node[above]{$\rmC^\bullet$} (b);
		\draw[right hook->] (a) to node[below left]{$\mathrm{Y}$} (c);
		\draw[->] (c) to node[below right]{$\mathrm{Lan}_\mathrm{Y} \rmC^\bullet$} (b);
	\end{tikzpicture}
}}
\]
Conversely, notice that left adjoint functors preserve colimits and that there is a unique  colimit preserving functor $\mathrm{L} : \sSe \to \mathsf{C}$ with prescribed restriction to $\cD$~, by the density theorem: any simplicial set $X_\bullet$ is isomorphic to the colimit 
$$X_\bullet\cong \mathop{\rm colim}_{\mathsf{E}(X_\bullet)} \mathrm{Y}\circ \Pi $$
over its category of elements, where 
  $\Pi : \mathsf{E}(X_\bullet) \to \cD$ is the canonical projection which sends $x\in X_n$ onto $[n]$~. 
\end{proof}

This fact shows that if one wants to introduce a functor from a category $\mathsf{C}$ to simplicial sets, an efficient way to do so is to consider a cosimplicial object in $\mathsf{C}$~. 

\begin{example}\index{Dold--Kan adjunction}
Let $\mathsf{C}=\ch$ be the category of (complete) chain complexes. We consider the cosimplicial chain complex
\[\mathrm{C}^\bullet\coloneqq C(\De{\bullet})\]
given by the cellular chain complexes $C(\De{n})$ of the geometric $n$-simplices $|\De{n}|$~, or equivalently the normalized chain complex of the standard $n$-simplices $\De{n}$~ (with the discrete filtration). The associated adjunction
\[
\hbox{
	\begin{tikzpicture}
	\def\upshift{0.185}
	\def\downshift{0.15}
	\pgfmathsetmacro{\midshift}{0.005}
	
	\node[left] (x) at (0, 0) {$\mathrm{L}_{\mathrm{DK}}\ \ :\ \ \sSe$};
	\node[right] (y) at (2, 0) {$\ch\ \ :\ \ \mathrm{R}_{\mathrm{DK}}$};
	
	\draw[-{To[left]}] ($(x.east) + (0.1, \upshift)$) -- ($(y.west) + (-0.1, \upshift)$);
	\draw[-{To[left]}] ($(y.west) + (-0.1, -\downshift)$) -- ($(x.east) + (0.1, -\downshift)$);
	
	\node at ($(x.east)!0.5!(y.west) + (0, \midshift)$) {\scalebox{0.8}{$\perp$}};
	\end{tikzpicture}}
\]
is nothing but the one involved in the Dold--Kan correspondence, giving an equivalence of categories between simplicial abelian groups and non-negatively gra\-ded chain complexes of abelian groups (without filtrations). 
\end{example}

\begin{example}
In the category $\mathsf{C}=\mathsf{Top}$ of  topological spaces, we  consider the cosimplicial object $|\De{\bullet}|$ built from the geometric $n$-simplices $|\De{n}|$ . The associated adjunction
\[
\hbox{
	\begin{tikzpicture}
	\def\upshift{0.185}
	\def\downshift{0.15}
	\pgfmathsetmacro{\midshift}{0.005}
	
	\node[left] (x) at (0, 0) {$|\ \  |\ \ :\ \ \sSe$};
	\node[right] (y) at (2, 0) {$\mathsf{Top}\ \ :\ \ \mathrm{Sing}$};
	
	\draw[-{To[left]}] ($(x.east) + (0.1, \upshift)$) -- ($(y.west) + (-0.1, \upshift)$);
	\draw[-{To[left]}] ($(y.west) + (-0.1, -\downshift)$) -- ($(x.east) + (0.1, -\downshift)$);
	
	\node at ($(x.east)!0.5!(y.west) + (0, \midshift)$) {\scalebox{0.8}{$\perp$}};
	\end{tikzpicture}}
\]
is the classical adjunction given by geometric realization and the singular chain functor. 
\end{example}

\begin{example}
In the case of the category $\mathsf{C}=\mathsf{Cat}$ of small categories, the  cosimplicial object
\begin{align*}
\mathsf{C}^\bullet\ :\ \cD{}&\ \longrightarrow\ \mathsf{Cat}\\
[n]{}&\ \longmapsto\ \mathsf{Cat}[n]:=\{0\to1\to \cdots \to n\}\ ,
\end{align*}
where $\mathsf{Cat}[n]$ is the category associated to the totally ordered set $[n]$~,  
induces the nerve functor
$\mathrm{N} \mathsf{D}:=\Hom_{\mathsf{Cat}}(\mathsf{C}^\bullet, \mathsf{D})$~, mentioned in \cref{prop:NerveGroupoid}. 
\end{example}

\subsection{The universal Maurer--Cartan algebra}\label{subsec:MCcosimpsLI}

In the present paper, we focus on the category $\mathsf{C}=\sLialg$ of complete $\sLi$-algebras. To define  a suitable cosimplicial object in this category, we consider the following ingredients involved in Dupont's simplicial version of the de Rham Theorem. 

\begin{proposition}[\cite{Dupont76}]\label{prop:DupontContr}
There exists a simplicial contraction\index{Dupont contraction}
\[
\hbox{
	\begin{tikzpicture}
	\def\upshift{0.075}
	\def\downshift{0.075}
	\pgfmathsetmacro{\midshift}{0.005}
	
	\node[left] (x) at (0, 0) {$\Omega_\bullet$};
	\node[right=1.5 cm of x] (y) {$\mathrm{C}_\bullet$};
	
	\draw[-{To[left]}] ($(x.east) + (0.1, \upshift)$) -- node[above]{\mbox{\tiny{$p_\bullet$}}} ($(y.west) + (-0.1, \upshift)$);
	\draw[-{To[left]}] ($(y.west) + (-0.1, -\downshift)$) -- node[below]{\mbox{\tiny{$i_\bullet$}}} ($(x.east) + (0.1, -\downshift)$);
	
	\draw[->] ($(x.south west) + (0, 0.1)$) to [out=-160,in=160,looseness=5] node[left]{\mbox{\tiny{$h_\bullet$}}} ($(x.north west) - (0, 0.1)$);
	
	\end{tikzpicture}}
\]
between the simplicial commutative algebras of piece-wise polynomial differential forms on the geometric $n$-simplices 
\[
\Omega_n\coloneqq \frac{\k[t_0, \ldots, t_n, d t_0, \ldots, d t_n]}{(t_0+\cdots+t_n-1, d t_0+\cdots+ d t_n)}
\]
and the simplicial chain complex made up of the linear dual of the cellular chain complexes of the geometric $n$-simplices 
\[\mathrm{C}_\bullet\coloneqq C(\De{\bullet})^\vee\ ,\]
with dual basis denoted by  $\omega_I$~, for non-empty sub-sets $I=\{i_0,\ldots, i_k\}\subset[n]$~, of degree $|\omega_I|=-|I|+1=-k$~.
\end{proposition}

Let us recall that a contraction amounts to the data of two chain maps $i_n : \mathrm{C}_n \to \Omega_n$~, $p_n : \Omega_n\to \mathrm{C}_n$~, and a degree $1$ linear map $h_n : \Omega_n \to \Omega_n$~,  satisfying the relations 
\[
p_ni_n=\id_{\mathrm{C}_n}, \quad 
i_np_n -\id_{\Omega_n}=d_n h_n+h_nd_n, \quad 
h_n^2=0\ , \quad p_nh_n=0\ , \quad \text{and} \quad h_ni_n=0 \ , 
\]
where $d_n$ stands for the differential map of $\Omega_n$~. In a simplicial contraction, the above three collections of maps  are required to commute with the respective simplicial maps.

\medskip

For a full definition of the maps involved in the Dupont contraction we refer the reader to the paper \cite{lunardon18}. A Python package for computations using these objects has been developed and made freely available by the first author, see \cite{RN20}.

\begin{remark}
Since we work with the homological degree convention, the elements $d t_i$ have all degree $-1$ and the chain complex $C(\De{n})^\vee$ is concentrated in degree range $-n$ to $0$ with differential of degree $-1$~. 
\end{remark}

For each $n\in \NN$~, the chain complex $\Omega_n$ carries a commutative algebra structure, i.e. 
an algebra structure over the operad $\com$~. This latter operad admits a canonical cofibrant replacement by means of the bar-cobar construction $\Cobar \Bar \com \xrightarrow{\sim} \com$~. By pulling back, each $\Omega_n$ 
is endowed with a $\Cobar \Bar \com$-algebra structure that we can transfer onto 
$C(\De{n})^\vee$ using the homotopy transfer theorem --- \cref{thm:HTT}.
Since all the $C(\De{n})^\vee$ are degree-wise finite dimensional  and since $\Cobar \Bar \com$ is degree-wise and arity-wise finite dimensional, every  cellular chain complex $\mathrm{C}^n=C(\De{n})$ admits a $\Bar \Cobar \com^\vee\cong \left(
\Cobar \Bar \com\right)^\vee$-coalgebra structure. The images $\hatCobar_\pi \mathrm{C}^n$ of these coalgebras under the complete cobar construction  associated to the the operadic twisting morphism 
\[
\pi \colon\Bar \Cobar \com^\vee \to \Cobar \com^\vee\cong \sLi
\]
produces the wanted object. 

\begin{proposition}\label{prop:Simpli}
The collection $\hatCobar_\pi \mathrm{C}^\bullet$ forms a cosimplicial complete $\sLi$-algebra. 
\end{proposition}

\begin{proof}
It remains to study the behavior of the original simplicial maps. It is immediate that $\Omega_\bullet$ becomes a simplicial $\Cobar \Bar \com$-algebra. Since Dupont's contraction is simplicial, it is straightforward to check from the explicit formula 
of the transferred structure in the homotopy transfer theorem that the simplicial chain maps of the simplicial chain complex $\mathrm{C}_\bullet$ respect strictly the transferred $\Cobar \Bar \com$-algebra structure operations. Linear dualization produces a cosimplicial conilpotent $\Bar \Cobar \com^\vee$-coalgebra structure on $\mathrm{C}^\bullet$~, and finally the cobar functor gives a cosimplicial complete $\sLi$-algebra structure on $\hatCobar_\pi \mathrm{C}^\bullet$~. 
\end{proof}

\begin{definition}[The Maurer--Cartan cosimplicial $\sLi$-algebra]\label{def:MCcosimp}
	The cosimplicial complete $\sLi$-algebra
	\[
	\mc^\bullet\coloneqq \hatCobar_\pi \mathrm{C}^\bullet\ .
	\]
	is called the \emph{universal Maurer--Cartan algebra}\index{universal Maurer--Cartan algebra}.
\end{definition}

We will now make this new object a bit more explicit by studying the first cosimplicial degree $\mc^0$. First, recall that the underlying construction of the cobar functor $\hatCobar_\pi$ is the free complete $\sLi$-algebra. 
We denote by $a_I:=\omega_I^\vee$ the cellular basis on $C(\De{n})$~.
Forgetting the differential, the Maurer--Cartan $\sLi$-algebra $\mc^n$
is isomorphic to 
\[
\widehat{\sLi}\left(\{a_I\}_{I\subseteq[n]}\right),
\]
whose elements can be represented as sums of rooted trees with leaves labeled by the $\{a_I\}_{I\subseteq[n]}$~, as we have proved in \cref{prop:FreesLi}. The proof of the next proposition will provide the reader with more details about the differential.

\begin{proposition}\label{prop:mc0}
The complete $\sLi$-algebra $\mc^0$ is quasi-free on one Maurer--Cartan element
\[
\mc^0\cong\widehat{\sLi}(a_0)
\]
with differential given by
\[
\d(a_0)=-\sum_{m\geqslant 2} {\textstyle \frac{1}{m!}}\ell_m(a_0, \ldots, a_0)\ .
\]
\end{proposition}

\begin{proof}
We start by studying the transferred $\Omega\Bar\com$-algebra structure on $\rmC_0$ and the dual structure on $\rmC^0$.

\medskip
	
The chain complex $C(\De{n})^\vee$ can be seen as a sub-chain complex of $\Omega_n$ 
under the following identifications 
\[
\omega_I=\omega_{i_0\cdots i_k}\coloneqq k!\sum_{j=0}^k (-1)^jt_{i_j} dt_{i_0}\cdots \widehat{dt_{i_j}}\cdots dt_{i_k} \ ,
\]
for $\emptyset\neq I=\{i_0,\ldots, i_k\}\subseteq[n]$~. For $n=0$~, the commutative algebra $\Omega_0\cong\k$ is just the field with its multiplicative structure and the chain complex 
$C(\De{0})^\vee=\k \omega_0$ is one-dimensional.

\medskip

Since the operad $\Cobar \Bar \com$ is quasi-free on $s^{-1} \Bar \com=s^{-1}\T^c\big(s\overline{\com}\big)$~, the transferred $\Cobar \Bar \com$-algebra structure on $C(\De{n})^\vee$ amounts to a collection of operations $\{\mu_\tau\}_{\tau \in \RT}$ indexed by rooted trees. The operation $\mu_{|}=\id$ is the identity. Given a rooted tree $\tau\in \RT_m$~, the associated operation
\[
\mu_\tau : \left(C(\De{n})^\vee\right)^{\otimes m} \to C(\De{n})^\vee
\]
has degree $|\tau|-1$ and is (graphically) given by the formula
\begin{equation}\label{Eq:HTTonCn}
\mu_\tau\left(\omega_{I_1}, \ldots, \omega_{I_6}\right)=
\vcenter{\hbox{
	\begin{tikzpicture}
		\def\scale{0.75};
		\pgfmathsetmacro{\diagcm}{sqrt(2)};
		
		\def\xangle{35};
		\pgfmathsetmacro{\xcm}{1/sin(\xangle)};
		
		\coordinate (r) at (0,0);
		\coordinate (v11) at ($(r) + (0,\scale*1)$);
		\coordinate (v21) at ($(v11) + (180-\xangle:\scale*\xcm)$);
		\coordinate (v22) at ($(v11) + (\xangle:\scale*\xcm)$);
		\coordinate (v31) at ($(v22) + (45:\scale*\diagcm)$);
		\coordinate (l1) at ($(v21) + (135:\scale*\diagcm)$);
		\coordinate (l2) at ($(v21) + (0,\scale*1)$);
		\coordinate (l3) at ($(v21) + (45:\scale*\diagcm)$);
		\coordinate (l4) at ($(v22) + (135:\scale*\diagcm)$);
		\coordinate (l5) at ($(v31) + (135:\scale*\diagcm)$);
		\coordinate (l6) at ($(v31) + (45:\scale*\diagcm)$);
		
		\draw[thick] (r) to (v11);
		\draw[thick] (v11) to (v21);
		\draw[thick] (v11) to (v22);
		\draw[thick] (v21) to (l1);
		\draw[thick] (v21) to (l2);
		\draw[thick] (v21) to (l3);
		\draw[thick] (v22) to (l4);
		\draw[thick] (v22) to (v31);
		\draw[thick] (v31) to (l5);
		\draw[thick] (v31) to (l6);
		
		\node[above] at (l1) {$\omega_{i_1}$};
		\node[above] at (l2) {$\omega_{i_3}$};
		\node[above] at (l3) {$\omega_{i_4}$};
		\node[above] at (l4) {$\omega_{i_5}$};
		\node[above] at (l5) {$\omega_{i_2}$};
		\node[above] at (l6) {$\omega_{i_6}$};
		
		\node[right] at (r) {$p_n$};
		\node[below left] at ($(v11)!0.5!(v21)$) {$h_n$};
		\node[below right] at ($(v11)!0.5!(v22)$) {$h_n$};
		\node[below right] at ($(v22)!0.5!(v31)$) {$h_n$};
	\end{tikzpicture}}}
\text{for}\quad \tau=
\vcenter{\hbox{
	\begin{tikzpicture}
		\def\scale{0.75};
		\pgfmathsetmacro{\diagcm}{sqrt(2)};
		
		\def\xangle{35};
		\pgfmathsetmacro{\xcm}{1/sin(\xangle)};
		
		\coordinate (r) at (0,0);
		\coordinate (v11) at ($(r) + (0,\scale*1)$);
		\coordinate (v21) at ($(v11) + (180-\xangle:\scale*\xcm)$);
		\coordinate (v22) at ($(v11) + (\xangle:\scale*\xcm)$);
		\coordinate (v31) at ($(v22) + (45:\scale*\diagcm)$);
		\coordinate (l1) at ($(v21) + (135:\scale*\diagcm)$);
		\coordinate (l2) at ($(v21) + (0,\scale*1)$);
		\coordinate (l3) at ($(v21) + (45:\scale*\diagcm)$);
		\coordinate (l4) at ($(v22) + (135:\scale*\diagcm)$);
		\coordinate (l5) at ($(v31) + (135:\scale*\diagcm)$);
		\coordinate (l6) at ($(v31) + (45:\scale*\diagcm)$);
		
		\draw[thick] (r) to (v11);
		\draw[thick] (v11) to (v21);
		\draw[thick] (v11) to (v22);
		\draw[thick] (v21) to (l1);
		\draw[thick] (v21) to (l2);
		\draw[thick] (v21) to (l3);
		\draw[thick] (v22) to (l4);
		\draw[thick] (v22) to (v31);
		\draw[thick] (v31) to (l5);
		\draw[thick] (v31) to (l6);
		
		\node[above] at (l1) {$1$};
		\node[above] at (l2) {$3$};
		\node[above] at (l3) {$4$};
		\node[above] at (l4) {$5$};
		\node[above] at (l5) {$2$};
		\node[above] at (l6) {$6$};
	\end{tikzpicture}}}
,
\end{equation}
where the corollas mean the (iterated) commutative product in $\Omega_n$~, and where the underlying composition scheme is given by the rooted tree $\tau$~.
For $n=0$~, the two chain maps $i_0$ and $p_0$ are the identity on $\Omega_0\cong\k\cong\rmC_0$ and the contraction $h_0=0$ is trivial. So only the operations indexed by corollas $c_m$ do not collapse and they are equal to  $\mu_{c_m}(\omega_0, \ldots, \omega_0)=\omega_0$~.

\medskip

The linear dual of the operadic algebra structure corresponding to
\[
s^{-1} \Bar \com\left(C(\De{n})^\vee \right)\to C(\De{n})^\vee
\]
produces the cooperadic coalgebra structure corresponding to
\[
C(\De{n})\to s \Cobar \com^\vee\left(C(\De{n})\right).
\]
Notice that, since this linear dual includes the isomorphism between invariants and coinvariants under the actions of the symmetric groups $\Sy_m$~, it carries a coefficient ${\textstyle \frac{1}{m!}}$~. 
In general, we will use the notation 
\begin{equation}\label{eq:StructureCoef}
\mu_\tau\left(
\omega_{I_1}, \ldots, \omega_{I_m} 
\right)=\sum_{\emptyset\neq J\subset [n]}
\lambda^{\tau(I_1, \ldots, I_m)}_J
\omega_J\ 
\end{equation}
for the coefficients of the  transferred $\Cobar \Bar \com$-algebra structure on $C(\De{n})^\vee$~. 
If $\lambda\coloneqq\lambda^{\tau(I_1, \ldots, I_m)}_J$ is not $0$~, this means that the term $\omega_J$ appears in the product 
$\mu_\tau\left(\omega_{I_1}, \ldots, \omega_{I_m}\right)$~, then the image under 
the coproduct map of the element $a_J$ includes the term ${\textstyle \frac{1}{\lambda\, m!}} s\tau(  
a_{I_1}, \ldots, a_{I_m})$~. 
Notice that, due to degree reasons, the coefficient $\lambda$ can only be non-trivial when 
$|J|-2=|I_1|+\cdots+|I_m|-|\tau|-m$~. 
For $n=0$~, the image under the coproduct map of $a_0$ is equal to 
\[
a_0\mapsto \sum_{m\geqslant 1} {\textstyle \frac{1}{m!}} s c_m(a_0, \ldots, a_0)\ .
\]
Finally, the differential of the cobar construction $\hatCobar_\pi C(\De{n})$  is produced by the difference between the internal differential of the chain complex $C(\De{n})$ and the desuspension of the above coalgebra structure, without the primitive term, that is:
\begin{equation}\label{eq:DiffMCn}
\d(a_J)= 
\sum_{l=0}^{k} (-1)^{l} a_{j_0\ldots \widehat{j_l} \ldots j_k}
-\sum_{m\geqslant 2}\sum_{\substack{\tau \in \overline{\RT}_m\\ I_1,\ldots, I_m\subseteq [n], \ , I_l\neq \emptyset\\ \lambda^{\tau(I_1, \ldots, I_m)}_J\neq 0}} \frac{1}{\lambda^{\tau(I_1, \ldots, I_m)}_J\,m!}\, \tau(a_{I_1}, \ldots, a_{I_m})
\ ,
\end{equation}
for $J=\{j_0, \ldots, j_k\}$~.
In the case $n=0$~, this gives 
\[\d(a_0)=-\sum_{m\geqslant 2} {\textstyle \frac{1}{m!}}\ell_m(a_0, \ldots, a_0)\ ,\]
which concludes the proof. 
\end{proof}

\subsection{Integration and \texorpdfstring{$\sLi$}{sL-infinity}-algebra functors}\label{subsec:RepReaFun}
Following \cref{lem:AdjCosimpliObj}, the Maurer--Cartan algebra 
$\mc^\bullet$ canonically induces the following pair of adjoint functors. 

\begin{definition}[Integration functor]\label{def:SimpRep}
The  \emph{integration functor} of complete $\sLi$-algebras is defined by 
\begin{align*}
\mathrm{R}\ :\ \sLialg{}&\ \longrightarrow\ \sSe\\
\g{}&\ \longmapsto\ \Hom_{\,\sLialg}\left(\mc^\bullet, \g\right).
\end{align*}
\end{definition}

As a direct corollary of  \cref{prop:mc0}, we have
\[
\R(\g)_0=\hom_{\,\sLialg}\left(\mc^0, \g\right)\cong \MC(\g)\ .
\]

\begin{definition}[$\sLi$-algebra of a simplicial set]
The functor of \emph{$\sLi$-algebra of a simplicial set}\index{$\sLi$-algebra of a simplicial set} is defined by 
\[\Li\coloneqq \mathrm{Lan}_\mathrm{Y} \mc_\bullet\ : \ \sSe \to \sLialg \ .
\] 
\end{definition}

One can compute it as the colimit of the functor $\mc^\bullet \circ \Pi$ 
from the category $\mathsf{E}(X_\bullet)$ of elements of $X_\bullet$~, where 
 $\Pi : \mathsf{E}(X_\bullet) \to \cD$ stands for  the canonical projection sending $x\in X_n$ to $[n]$~. 
We denote this colimit simply by 
\[\Li(X_\bullet)
 \cong 
\mathop{\mathrm{colim}}_{\mathsf{E}(X_\bullet)} \mc^\bullet~.\]

\begin{theorem}\label{thm:MainAdjunction}
The integration functor is right adjoint to the $\sLi$-algebra functor:
\[
\hbox{
	\begin{tikzpicture}
	\def\upshift{0.185}
	\def\downshift{0.15}
	\pgfmathsetmacro{\midshift}{0.005}
	
	\node[left] (x) at (0, 0) {$\Li\ \ :\ \ \sSe$};
	\node[right] (y) at (2, 0) {$\sLialg\ \ :\ \ \R\ .$};
	
	\draw[-{To[left]}] ($(x.east) + (0.1, \upshift)$) -- ($(y.west) + (-0.1, \upshift)$);
	\draw[-{To[left]}] ($(y.west) + (-0.1, -\downshift)$) -- ($(x.east) + (0.1, -\downshift)$);
	
	\node at ($(x.east)!0.5!(y.west) + (0, \midshift)$) {\scalebox{0.8}{$\perp$}};
	\end{tikzpicture}}
\]
\end{theorem}

\begin{proof}
This is a direct corollary of \cref{lem:AdjCosimpliObj} and \cref{prop:Simpli}. 
\end{proof}

 \subsection{Non-abelian Dold--Kan adjunction}\label{subsec:NAbDKadj}
 
The sub-category of complete abelian $\sLi$-al\-ge\-bras, where all the operations $\ell_m$~, i.e.\ the $\sLi$-algebras where all operations $\ell_m$ are trivial, for $m\geqslant2$~, is the same as the category of complete chain complexes. Assuming \cref{thm:isomorphism of models}, the following proposition coincides with \cite[Proposition~5.1]{Getzler09}.

\begin{proposition}
The two simplicial sets which are obtained from a complete chain complex under the Dold--Kan functor $\R_{\mathrm{DK}}$ and the integration functor $\R$ are naturally isomorphic.
\end{proposition}

\begin{proof}
Let $\g$ be a complete chain complex, viewed as a complete abelian $\sLi$-algebra. Its image under the integration  functor is equal to
\[
\R(\g)=\Hom_{\,\sLialg}\left(\mc^\bullet, \g\right)\cong\Hom_{\,\sLialg}\left(\wsLi\left(\mathrm{C}^\bullet\right), \g\right)\cong
\Hom_{\ch}\left(\mathrm{C}^\bullet, \g\right)=\R_{\mathrm{DK}}(\g)\ .
\]
The second isomorphism comes from the fact that the non-linear part of the differential of the cobar construction involves brackets, so that commuting with it is a void condition since $\g$ is abelian.
\end{proof}

One level above, one can consider the full subcategory of complete $\sLi$-algebras so that $\ell_m$ is trivial for all $m\geqslant3$ but $\ell_2$ can be non-trivial. This is isomorphic to the category of complete shifted i.e.\ algebras, which is controlled by the operad
\[
\sLie\coloneqq\End_{\k s} \otimes \lie\ .
\]
This fact can be conceptually expressed via the canonical quasi-isomorphism of operads
\[
\rho\colon\sLi \stackrel{\sim}{\longrightarrow} \sLie\ .
\]
If ${\g}$ is an $\sLie$-algebra, we denote by $\rho^*{\g}$ the canonical $\sLi$-algebra structure obtained by pulling back by the operad morphism $\rho$ and which amounts to considering trivial higher operations $\ell_m$~, for $m\geqslant 3$~.

\medskip

In the same way as above, one can consider the cofibrant Koszul resolution 
$\Cobar  \com^{\antishriek} \xrightarrow{\sim} \com$~, and then endow  $C(\De{\bullet})^\vee$ with a simplicial $\Cobar  \com^{\antishriek}$-algebra structures through the homotopy transfer theorem. Its linear dual $C(\De{\bullet})$ obtains a cosimplicial 
$\Bar\,  \sLie\cong \big(
\Cobar  \com^{\antishriek}\big)^\vee$-coalgebra structure. Finally, 
its image under  the complete cobar construction  associated to the canonical operadic twisting morphism 
$\overline{\pi} \colon\Bar\,  \sLie \to  \sLie$ 
produces the following cosimplicial complete $\sLie$-algebra:
\[
\overline{\mc}^\bullet\coloneqq \hatCobar_{\overline{\pi}}  \mathrm{C}^\bullet\ .
\]
The associated pair of adjoint functors
\[
\hbox{
	\begin{tikzpicture}
	\def\upshift{0.185}
	\def\downshift{0.15}
	\pgfmathsetmacro{\midshift}{0.005}
	
	\node[left] (x) at (0, 0) {$\overline{\mathrm{L}}\ \ :\ \ \sSe$};
	\node[right] (y) at (2, 0) {$\sLiealg\ \ :\ \ \overline{\R}\ .$};
	
	\draw[-{To[left]}] ($(x.east) + (0.1, \upshift)$) -- ($(y.west) + (-0.1, \upshift)$);
	\draw[-{To[left]}] ($(y.west) + (-0.1, -\downshift)$) -- ($(x.east) + (0.1, -\downshift)$);
	
	\node at ($(x.east)!0.5!(y.west) + (0, \midshift)$) {\scalebox{0.8}{$\perp$}};
	\end{tikzpicture}}
\]
coincides with the one introduced in \cite{BFMT18} since $\overline{\mc}^\bullet$ is a ``sequence of compatible models of $\Delta$'' by \cite[Theorem~7.3~(i)]{BFMT15}.

\begin{proposition}\label{prop:lie=LiReal}
The two simplicial sets obtained from a complete shifted Lie algebra respectively under the 
above functor $\overline{\mathrm{R}}$ and the integration functor $\R$  are isomorphic:
\[
\overline{\mathrm{R}}({\g})\cong \mathrm{R}(\rho^*{\g})~,
\]
for ${\g}$ a $\sLie$-algebra.
\end{proposition}

\begin{proof}
Let ${\g}$ be a complete $\sLie$-algebra. Recall that
\[
\overline{\mathrm{R}}({\g})\cong \MC\left({\g}\, \widehat{\otimes}\mathrm{C}_\bullet\right)
\]
by \cite[Theorem~5.2]{rn17cosimplicial}, where the cosimplicial $\sLi$-algebra structure on ${\g}\, \widehat{\otimes}\mathrm{C}_\bullet$ is obtained by applying the homotopy transfer theorem \ref{thm:HTT} to the contraction
\[
\hbox{
	\begin{tikzpicture}
	\def\upshift{0.075}
	\def\downshift{0.075}
	\pgfmathsetmacro{\midshift}{0.005}
	
	\node[left] (x) at (0, 0) {${\g}\,\widehat{\otimes}\,\Omega_\bullet$};
	\node[right=1.5 cm of x] (y) {${\g}\,\widehat{\otimes}\,\mathrm{C}_\bullet\ .$};
	
	\draw[-{To[left]}] ($(x.east) + (0.1, \upshift)$) -- node[above]{\mbox{\tiny{$\id\,\hot\,p_\bullet$}}} ($(y.west) + (-0.1, \upshift)$);
	\draw[-{To[left]}] ($(y.west) + (-0.1, -\downshift)$) -- node[below]{\mbox{\tiny{$\id\,\hot\,i_\bullet$}}} ($(x.east) + (0.1, -\downshift)$);
	
	\draw[->] ($(x.south west) + (0, 0.1)$) to [out=-160,in=160,looseness=5] node[left]{\mbox{\tiny{$\id\,\hot\,h_\bullet$}}} ($(x.north west) - (0, 0.1)$);
	
	\end{tikzpicture}}
\]
We now  consider the 
complete $\sLi$-algebra structure $\rho^*{\g}$~. \cref{thm:convolution algebras and HTT} gives the isomorphism
\[
{\mathrm{R}}(\rho^*{\g})\cong \MC\left(\rho^*{\g}\, \widehat{\otimes}\mathrm{C}_\bullet\right) , 
\]
see the proof of \cref{thm:isomorphism of models} for more details. 
We conclude by noticing that the homotopy transfer theorem starting from a complete $\sLie$-algebra structure or from its induced complete $\sLi$-algebra structure produces the exact same complete $\sLi$-algebra structure.
\end{proof}

Summarizing, the integration functor $\R$ extends both the Dold--Kan functor and the integration functor for complete $\sLie$-algebras functor of Buijs--Felix-Murillo--Tanr\'e.\index{Dold--Kan adjunction!non-abelian}
\[
\hbox{
\begin{tikzpicture}
	\node (ss) at (0, 0){$\sSe$};
	\node (ch) at (3.5, -1.5){$\ch$};
	\node (lie) at (3.5, 0){$\sLiealg$};
	\node (loo) at (3.5, 1.5){$\sLialg$};
	
	\draw[->] (ch) to [out=-180,in=-45] node[below left]{\mbox{\small{$\mathrm{R_{DK}}$}}} (ss);
	\draw[->] (lie) to node[above]{\mbox{\small{$\overline{\R}$}}} (ss);
	\draw[->] (loo) to [out=-180,in=45] node[above left]{\mbox{\small{$\R$}}} (ss);
	
	\draw[right hook->] (ch) to node[right]{\mbox{\small{abelian $\sLie$-algebras}}} (lie);
	\draw[right hook->] (lie) to node[right]{\mbox{\small{$\rho^*$}}} (loo);
\end{tikzpicture}
}
\]

For simplicity, we only represented the right adjoint functors in the above displayed diagram but the respective left adjoint functors commute as well. 

\medskip

One can also directly compare the two universal Maurer--Cartan cosimplicial algebras. 

\begin{proposition}\label{prop:ResolutionMC}
The morphism 
\[
{\mc}^\bullet \longrightarrow
\rho^* \,\overline{\mc}^\bullet
\]
given by $\rho\circ\id_{\mathrm{C}^\bullet}$
\[ 
\widehat{\sLi}\left(\mathrm{C}^\bullet \right)
\longrightarrow
\widehat{\sLie}\left(\mathrm{C}^\bullet \right)
\]
is a quasi-isomorphism of cosimplicial complete $\sLi$-algebras. 
\end{proposition}

\begin{proof}
By its operadic definition, it is straightforward to see that such a map preserves the brackets. Since the cosimplicial maps are given by that of $\mathrm{C}^\bullet$ in both cases, it is immediate to see that the map $\rho$ preserves the cosimplicial structures.

\medskip

We consider the canonical morphism of cooperads $\rho^\vee : \com^{\antishriek} \to \Bar \com$~. 
There are two ways to get a $\Omega \com^{\antishriek}$-algebra structure on $\mathrm{C}_\bullet$:
\begin{enumerate}
\item as explained above, one can see ${\Omega}_\bullet$ as a  $\Omega \com^{\antishriek}$-algebra  and transfer this type of structure onto $\mathrm{C}_\bullet$ under the Dupont's contraction,

\item alternatively, as described in \cref{subsec:MCcosimpsLI}, one can see ${\Omega}_\bullet$ as a  $\Omega \Bar \com$-algebra,  transfer this type of structure onto $\mathrm{C}_\bullet$ under the Dupont's contraction, and then pull it back by $\Omega \rho^\vee$.
\end{enumerate}

Using the explicit formula for the homotopy transfer theorem given in \cref{thm:HTT}, it is straightforward to show that these two structures are equal.  More generally, the homotopy transfer theorem is functorial with respect to morphisms of cooperads, see \cite[Theorem~4.14]{HLV19}. 

This shows that the map $\rho(\id_{\mathrm{C}^\bullet})$ preserves the differentials: it can be seen on the following  commutative diagram
\begin{center}
		\begin{tikzpicture}
			\node (a) at (0,0){$\mathrm{C}^\bullet$};
			\node (b) at (4,0){$\Bar(\sLi)(\mathrm{C}^\bullet)$};
			\node (c) at (9,0){$\sLi(\mathrm{C}^\bullet)$};
			\node (d) at (0,-2.5){$\mathrm{C}^\bullet$};
()			\node (e) at (4,-2.5){$\Bar(\sLie)(\mathrm{C}^\bullet)$};
			\node (f) at (9,-2.5){$\sLie(\mathrm{C}^\bullet)$\ ,};
			
			\draw[double equal sign distance] (a) to (d);
			\draw[->] (a) to node[above]{$\Delta_{\Bar(\sLi)}$} (b);
			\draw[->] (b) to node[above]{$\pi\left(\id_\mathrm{C}^\bullet\right)$} (c);
			\draw[->] (d) to node[above]{$\Delta_{\Bar(\sLie)}$} (e);
			\draw[->] (e) to node[above]{$\overline{\pi}\left(\id_\mathrm{C}^\bullet\right)$} (f);
			\draw[->] (b) to node[right]{$\Bar (\rho)\left(\id_{\mathrm{C}^\bullet}\right)$} (e);
			\draw[->] (c) to node[right]{$\rho\left(\id_{\mathrm{C}^\bullet}\right)$} (f);
		\end{tikzpicture}
	\end{center}
	where the map $\Delta_{\Bar(\sLi)}$ stands for the $\Bar(\sLi)$-coalgebra structure on $\mathrm{C}^\bullet$  and where the map $\Delta_{\Bar(\sLie)}$  stands for  its $\Bar(\sLie)$-coalgebra structure.	
Indeed, the linear part of the differentials of ${\mc}^\bullet$ and $\overline{\mc}^\bullet$ on the generators is the same. 
The rest of the differential of ${\mc}^\bullet$ is given by the top part of the diagram, while the rest of the differential of $\overline{\mc}^\bullet$ is given by the bottom part of the diagram. 
The left square commutes by  the facts that $\Bar \rho\cong(\Cobar \rho^\vee)^\vee$ and that the two 
$\Omega \com^{\antishriek}$-algebra structures on $\mathrm{C}_\bullet$ described at the beginning of the proof are equal and thus so are the two dual  
$\Bar \com$-coalgebra structures on $\mathrm{C}^\bullet$.
The right square commutes by the functorial properties of the twisting morphisms, see \cref{prop:Rosetta}.

\medskip

In order to prove that these maps of chain complexes are quasi-isomorphisms, we consider on both side the arity filtration which makes each algebra complete:
\[
\F_k \widehat{\P}(\mathrm{C}^\bullet)\coloneqq \prod_{m\geqslant k} \P(m)\otimes_{\Sy_m} (\mathrm{C}^\bullet)^{\otimes m}\ ,
\]
where $\P$ is either $\sLi$ or $\sLie$~. 
On the zeroth page $E^0$  of the associated spectral sequences, the differential amounts only to the internal differential of the operad $\sLi$ and to the internal differential of $\mathrm{C}^\bullet$. Since the morphism of operads $\rho$ is a quasi-isomorphism, we get the isomorphism $E^1(\rho) \colon \sLie(\k) \cong \sLie(\k)$ on the first page. 
These filtrations are exhaustive and complete, so we can apply the Eilenberg--Moore Comparison Theorem \cite[Theorem~5.5.11]{WeibelBook} and conclude that $\rho\circ\id_{\mathrm{C}^\bullet}$ 
is quasi-isomorphism. 
\end{proof}

We can be more precise about the homotopy type of the 
Maurer--Cartan cosimplicial $\sLi$-algebra and the Maurer--Cartan cosimplicial $\sLie$-algebra.

\begin{proposition}
	For all $n\geqslant0$~, we have
	\[
	H(\mc^n) = 0\qquad\text{and}\qquad H(\overline{\mc}^n) = 0\ .
	\]
\end{proposition}

\begin{proof}
	Thanks to \cref{prop:ResolutionMC}, it is enough to treat the case of $\overline{\mc}^\bullet$. 
	For $n\geqslant0$~, the map of chain complexes
	\[
	\k\cong\mathrm{C}_0\longrightarrow\mathrm{C}_n
	\]
	given by sending $\omega_0\in\mathrm{C}_0$ to $\omega_0\in\mathrm{C}_n$ is a quasi-isomorphism of $\Cobar \com^{\antishriek}$-algebras. We consider the twisting morphism
	\[
	\overline{\iota}:\sLie^\vee\longrightarrow\Cobar(\sLie^\vee)
	\]
	dual to $\overline{\pi}$~. Then, \cite[Proposition~11.2.3]{LodayVallette12} shows that the associated bar construction  preserve quasi-isomorphisms, and thus the maps
	\[
	\Bar_{\overline{\iota}}\, \mathrm{C}_0\stackrel{\sim}{\longrightarrow}\Bar_{\overline{\iota}}\, \mathrm{C}_n
	\]
	are quasi-isomorphisms. Since we are working over a field, their duals
	\[
	\overline{\mc}^n = \left(\Bar_{\overline{\iota}}\, \mathrm{C}_n\right)^\vee\stackrel{\sim}{\longrightarrow}\overline{\mc}^0 = \left(\Bar_{\overline{\iota}}\, \mathrm{C}_0\right)^\vee	\]
	also are. We are only left to prove that
	\[
	H(\mc^0) = 0\ , 
	\]
which is straightforward since 
	\[
	\overline{\mc}^0 = \k\alpha\oplus\k[\alpha,\alpha]
	\]
	with $|\alpha| = 0$ and $d\alpha = -\tfrac{1}{2}[\alpha,\alpha]$~.
\end{proof}

\subsection{Relation with Maurer--Cartan spaces}

There are several other ways to define Maurer--Cartan spaces\index{Maurer--Cartan!spaces} associated to a complete $\sLi$-algebra $\g$~. All of these have been shown to be Kan complexes. We will prove the same --- actually in a strengthened version --- for $\R$ from first principles in \cref{subsec:fill horns}.

\medskip

One can first consider the \emph{Deligne--Hinich space}\index{Deligne--Hinich space} \cite{Hinich97bis} made up of the solutions to the Maurer--Cartan equation of the complete $\sLi$-algebra structure on the tensor product with the simplicial commutative algebra $\Omega_\bullet$:
\[\MC_\bullet(\g)\coloneqq\MC\left(\g \, \widehat{\otimes}\, \Omega_\bullet\right).\]
This structure coincides with $\g \, \widehat{\otimes}^\iota\, \Omega_\bullet$ under definitions of \cref{subsec:Op}.

\medskip

One can also consider the \emph{Getzler space}\index{Getzler space} \cite{Getzler09} given by the kernels of Dupont's homotopy:
\[
\gamma_\bullet(\g)\coloneqq\left\{
\alpha \in \MC_\bullet(\g)\ | \ 
\big(\id_\g\,\widehat{\otimes}\,h_\bullet\big)(\alpha)=0
\right\}\ .
\]
Finally, one can also use Dupont's simplicial contraction to transfer the complete $\sLi$-al\-ge\-bra structure of $\g \, \widehat{\otimes}\, \Omega_\bullet$ onto $\g \, \widehat{\otimes}\, \rmC_\bullet$~, and then consider $\MC\left(\g\, \widehat{\otimes}\, \rmC_\bullet\right)$~.

\medskip

All of these constructions were shown to provide us with equivalent models for Maurer--Cartan spaces: they are weakly equivalent as Kan complexes to the Deligne--Hinich space.
In \cite{Getzler09}, E.\ Getzler proved that there is a  weak equivalence of simplicial sets
\[
\MC_\bullet(\g)\simeq\gamma_\bullet(\g)\ ,
\]
which is natural with respect to strict morphisms. 
In \cite{Bandiera14, Bandiera17}, R. Bandiera proved that there is a natural isomorphism between the  last two constructions
\[
\gamma_\bullet(\g)\cong\MC\left(\g\, \widehat{\otimes}\, \rmC_\bullet\right).
\]
The first named author also exhibited in \cite{rn17cosimplicial} a direct natural weak equivalence
\[
\MC_\bullet(\g)\simeq\MC\left(\g\, \widehat{\otimes}\, \rmC_\bullet\right) \ .
\]
In this last article, the cosimplicial complete $\sLie$-algebra $\mclie^\bullet$ was also considered, and it was proven that
\[
\MC\left(\overline{\g}\, \widehat{\otimes}\, \rmC_\bullet\right)\cong\Hom_{\sLiealg}(\mclie^\bullet,\overline{\g})=\overline{\mathrm{R}}(\overline{\g})
\]
for any complete $\sLie$-algebra $\overline{\g}$~. This statement also holds true in our context for complete $\sLi$-algebras.

\begin{theorem}\label{thm:isomorphism of models}
	For any $\sLi$-algebra $\g$ there is a canonical isomorphism of simplicial sets
	\[
	\R(\g)\cong \MC\left(\g\, \widehat{\otimes}\, \mathrm{C}_\bullet\right),
	\]
	given by the restriction of the linear duality isomorphism $\g\, \widehat{\otimes}\, \mathrm{C}_\bullet \cong \hom(\mathrm{C}^\bullet, \g)$ to the Maurer--Cartan set, which is natural with respect to strict morphisms. As a consequence,  we get the following natural weak equivalence and isomorphisms: 
	\[
	\mathrm{R}(\g) \cong \MC\left(\g\, \widehat{\otimes}\, \rmC_\bullet\right)\cong \gamma_\bullet(\g) \simeq 
	\MC_\bullet(\g)\ .
	\]
\end{theorem}

\begin{proof}[Proof of \cref{thm:isomorphism of models}]
We apply \cref{thm:convolution algebras and HTT} 
 to $\P=\sLi=\Cobar \com^\vee$, $X=\Omega_\bullet$~, $Y=\mathrm{C}_\bullet$~, $A=\g$~, and the operadic twisting morphism 
$\pi:\Bar(\sLi)=\Bar\Cobar\com^\vee\to\sLi$~.
Recall that in $\MC(\g\, \widehat{\otimes}\, \mathrm{C}_\bullet)$~, the chain complex $\g\, \widehat{\otimes}\, \mathrm{C}_\bullet$ was endowed with the complete $\sLi$-algebra structure obtained by applying the homotopy transfer theorem to the contraction
\[
\hbox{
	\begin{tikzpicture}
	\def\upshift{0.075}
	\def\downshift{0.075}
	\pgfmathsetmacro{\midshift}{0.005}
	
	\node[left] (x) at (0, 0) {$\g\,\widehat{\otimes}\,\Omega_\bullet$};
	\node[right=1.5 cm of x] (y) {$\g\,\widehat{\otimes}\,\mathrm{C}_\bullet$};
	
	\draw[-{To[left]}] ($(x.east) + (0.1, \upshift)$) -- node[above]{\mbox{\tiny{$\id\,\hot\,p_\bullet$}}} ($(y.west) + (-0.1, \upshift)$);
	\draw[-{To[left]}] ($(y.west) + (-0.1, -\downshift)$) -- node[below]{\mbox{\tiny{$\id\,\hot\,i_\bullet$}}} ($(x.east) + (0.1, -\downshift)$);
	
	\draw[->] ($(x.south west) + (0, 0.1)$) to [out=-160,in=160,looseness=5] node[left]{\mbox{\tiny{$\id\,\hot\,h_\bullet$}}} ($(x.north west) - (0, 0.1)$);
	
	\end{tikzpicture}}
\]
induced by Dupont's contraction.

\medskip

Since $\Omega_\bullet$ is a simplicial commutative algebra, it is in particular a simplicial $\Cobar\Bar\com$-algebra. The homotopy transfer theorem applied to the Dupont's contraction gives a $\Cobar\Bar\com$-algebra structure to $\mathrm{C}_\bullet$~, then the convolution tensor product provides us with  the $\sLi$-algebras $\g\otimes^\pi \mathrm{C}_\bullet$~.
By \cref{thm:convolution algebras and HTT}, the $\sLi$-algebras $\g\, \hot^\pi \mathrm{C}_\bullet$ and 
$\g\,\widehat{\otimes}\,\mathrm{C}_\bullet$ are equal. Therefore, we have the isomorphisms
\begin{align*}
	\MC\left(\g\, \widehat{\otimes}\, \mathrm{C}_\bullet\right)\cong&\ \MC\left(\g\, \widehat{\otimes}^\pi \mathrm{C}_\bullet\right)
	\cong \MC\big(\hom^\pi(\mathrm{C}^\bullet,\g)\big)
	\cong \hom_{\sLialg}\Big(\hatCobar_\pi \mathrm{C}^\bullet,\g\Big)
	= \R(\g)\ ,
\end{align*}
where in the second bijection we used the fact that $\mathrm{C}_\bullet$ is finite dimensional and in the third we used \cref{thm:MC elements of convolution algebras}.
\end{proof}

\begin{remark} 
As mentioned in \cref{ex:Loomorph},  the ``image''
	\[
	f(\alpha)\coloneqq \sum_{m\geqslant 1} {\textstyle \frac{1}{m!}} f_m(\alpha, \ldots, \alpha)\in \MC(\h)
	\]
of any Maurer--Cartan element $\alpha\in \MC(\g)$	
under an $\infty$-morphism is again a Maurer--Cartan element of $\h$~. This assignment makes $\MC$ into a functor with respect to $\infty$-morphisms, and by compatibility with the simplicial structure it also extends the Deligne--Hinich space $\MC_\bullet$ to a functor
\[
\MC_\bullet : \isLialg\longrightarrow\sSe\ .
\]
This additional functoriality is not preserved by the Getzler sub-space since the kernels of the Dupont's homotopy do not form commutative sub-algebras of $\Omega_\bullet$~. We will study and fix this issue in \cref{sec:Func} considering a refinement of the notion of $\infty$-morphisms. 
\end{remark}

\subsection{Geometric interpretation of the integration functor}\label{subsec:GeoRep}
Let $\g$ be a complete $\sLi$-al\-ge\-bra. By the Yoneda lemma and \cref{thm:isomorphism of models}, an $n$-simplex $x:\Delta^n\to\R(\g)$ is nothing else than an element of
\[
\R(\g)_n\cong\MC\left(\g\, \widehat{\otimes}\, \mathrm{C}_n\right), 
\]
that is a Maurer--Cartan element of the form
\[
x = \sum_{\emptyset\neq I\subseteq[n]}x_I\otimes\omega_I\ , 
\]
where $|x_I|=|I|-1$ by the fact that $x$ must have degree $0$~.

\medskip

Graphically, we draw such an element of $\R(\g)$ as a geometrical $n$-simplex with the sub-simplex corresponding to $\emptyset\neq I\subseteq[n]$ labeled by the element $x_I\in\g$~. The fact that $\R(\g)$ is a simplicial set implies that whenever two simplices share a face, their respective faces have the same element of $\g$ labeling them. Thus, this graphical representation can be used for any combination of simplices in $\R(\g)$~. 
We find this way of visualizing the simplicial set $\R(\g)$ extremely useful and we adopt it throughout the rest of the article, see the proof of \cref{thm:Berglund} for instance. 

\begin{example}
	For $n=2$~, we draw a $2$-simplex $x$ of $\R(\g)$ as follows. 
	
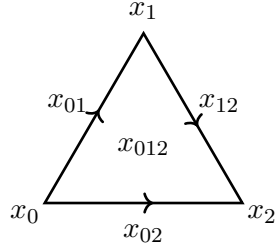
\[
	\begin{tikzpicture}
	
		\coordinate (v0) at (210:1.5);
		\coordinate (v1) at (90:1.5);
		\coordinate (v2) at (-30:1.5);
		
		\draw[line width=1] (v0)--(v1)--(v2)--cycle;
		
		\begin{scope}[decoration={
			markings,
			mark=at position 0.55 with {\arrow{>}}},
			line width=1
		]
			
			\path[postaction={decorate}] (v0) -- (v1);
			\path[postaction={decorate}] (v1) -- (v2);
			\path[postaction={decorate}] (v0) -- (v2);
		
		\end{scope}
		
		
		\node at ($(v0) + (30:-0.3)$) {$x_0$};
		\node at ($(v1) + (0,0.3)$) {$x_1$};
		\node at ($(v2) + (-30:0.3)$) {$x_2$};
		
		\node at ($(v0)!0.5!(v1) + (-30:-0.4)$) {$x_{01}$};
		\node at ($(v1)!0.5!(v2) + (30:0.4)$) {$x_{12}$};
		\node at ($(v0)!0.5!(v2) + (0,-0.4)$) {$x_{02}$};
		
		\node at (0,0) {$x_{012}$};
		
	\end{tikzpicture}
	\]
\end{example}
\section{Functoriality of the integration}\label{sec:Func}

Our conceptual approach based on the operadic calculus allows us to go further, establishing functoriality of $\R$ with respect to a class of higher morphisms, called $\infty_\pi$-morphisms, that refines the classical notion of $\infty$-morphisms. We are able to settle functoriality with respect to $\infty$-morphisms as well, but only on the level of the homotopy category. Since the  notion of $\infty_\pi$-morphism is new, we study it in detail, describing it explicitly and giving a precise comparison with the classical notion. 

\subsection{Functoriality with respect to \texorpdfstring{$\infty_\pi$}{infinity-pi}-morphisms}\label{subsec:FuncinftyPiMorph}

\cref{def:InfMor} applied to the operadic twisting morphism $\pi \colon \Bar \Cobar \com^\vee \to \Cobar \com^\vee\cong \sLi$ produces the notion of an \emph{$\infty_\pi$-morphism} of complete $\sLi$-algebras $f : \g \rightsquigarrow \h$~, which amounts to a  morphism of  $\Bar \Cobar \com^\vee$-coalgebras between the bar constructions $f : \Bar_\pi \g \to \Bar_\pi \h$ that preserves the respective filtrations. 

\begin{proposition}\label{prop:Functoriality}
	The integration functor $\R$ extends canonically to $\infty_\pi$-morphisms: 
	\[
	\R\colon\infty_\pi\text{-}\,\sLi\text{-}\,\mathsf{alg}\longrightarrow\sSe\ .
	\]
\end{proposition}

\begin{proof}
We use the convolution algebra description obtained by \cref{def:alpha-oo-morphism,thm:MC elements of convolution algebras}:
	\[
	\R(\g) \cong \MC\left(\hom^\pi(\mathrm{C}^\bullet,\g)\right)~,
	\]
which shows that $\R$ is the composite of two functors:
\[
\R = \left(\infty_\pi\text{-}\,\sLi\text{-}\,\mathsf{alg}\xrightarrow{\hom^\pi(\mathrm{C}^\bullet,-)}\infty\text{-}\,\sLi\text{-}\,\mathsf{alg}\xrightarrow{\MC}\sSe\right),
\]
where the fact that the first functor is functorial with respect to $\infty_\pi$-morphisms, sending them to $\infty$-morphisms, was proven in \cite[Corollary~5.4]{rnw17}, and where $\MC$ acts on $\infty$-morphisms by \cref{ex:Loomorph}. This concludes the proof.
\end{proof}

Let us spend a moment to make this functoriality more explicit. Any $\infty_\pi$-morphism $f:\g\rightsquigarrow\h$ induces an $\infty$-morphism between the convolution algebras
\[
\hom^\pi_r(\id, f):\hom^\pi(\mathrm{C}^\bullet,\g)\rightsquigarrow\hom^\pi(\mathrm{C}^\bullet,\h)
\]
by the formula
\[
\hom_r^\pi(\id, f)_m(x_1, \ldots,  x_m) = f\left(\id\otimes_{\Sy_m}(x_1\otimes\cdots\otimes x_m)\right)\Delta_m\ ,
\]
on $x_1,\ldots,x_n\in\hom^\pi(\mathrm{C}^\bullet,\g)$~, where the right-hand side is the composite
\[
\rmC^\bullet\xrightarrow{\Delta_m}\Bar \Cobar \com^\vee(m)\otimes_{\S_m}(\mathrm{C}^\bullet)^{\otimes m}\xrightarrow{\id \otimes_{\Sy_m}(x_1\otimes\cdots\otimes x_m)}\Bar \Cobar \com^\vee(m)\otimes_{\S_m}\g^{\otimes m}\xrightarrow{f}\h\ .
\]
This formula is the same as the one described in \cite[Section\ 5.1]{rnw17}; the proof of Theorem 5.1(b) in \emph{op. cit.} goes through unchanged despite $\mathrm{C}^\bullet$ not being conilpotent. The morphism of simplicial sets $\R(f)$ is then defined to be the induced map on Maurer--Cartan elements
\[
\R(f)\coloneqq\MC(\hom^\pi_r(\id, f))\ , 
\]
which is explicitly given by 
\begin{equation}\label{Eq:FuncR}
\R(f)(x) = \sum_{m\geqslant 1}\tfrac{1}{m!}f\left(\id \otimes_{\Sy_m} x^{\otimes m}\right)\Delta_m\ .
\end{equation}
Notice that everything is indeed compatible with the simplicial maps, so that this is a well defined morphism of simplicial sets.

\medskip

This functorial property is a motivation to unfold the compact definition of an $\infty_\pi$-mor\-phism. 

\subsection{Explicit description of \texorpdfstring{$\infty_\pi$}{infinity-pi}-morphisms}\label{subsec:inftyPiMorph}
We consider the set of \emph{partitioned rooted trees} $\PaRT$\index{trees!partitioned rooted}\index{$\PaRT$} made up of \emph{rooted trees}, with vertices of arity greater or equal to 2, endowed with a complete \emph{partition} into sub-trees called \emph{blocks}. The subset of partitioned rooted trees with $m$ leaves is denoted by $\PaRT_m$~, and, by convention, $\PaRT_1$ is made up of the trivial tree $|$~. 
Each vertex has degree $-1$ and each partition has degree $+1$~, so the degree of a partitioned rooted tree is equal to the number of partitions minus the number of vertices. Therefore, the degree is non-positive and partitioned rooted trees of degree $0$ are the ones where each block of the partition contains exactly one vertex. 

\begin{example}
Here is an example of a partitioned rooted tree
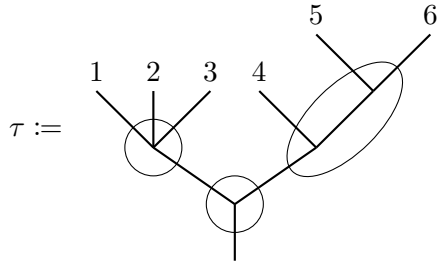
\[
\tau \coloneqq \vcenter{\hbox{
	\begin{tikzpicture}
		\def\scale{0.75};
		\pgfmathsetmacro{\diagcm}{sqrt(2)};
		
		\def\xangle{35};
		\pgfmathsetmacro{\xcm}{1/sin(\xangle)};
		
		\coordinate (r) at (0,0);
		\coordinate (v11) at ($(r) + (0,\scale*1)$);
		\coordinate (v21) at ($(v11) + (180-\xangle:\scale*\xcm)$);
		\coordinate (v22) at ($(v11) + (\xangle:\scale*\xcm)$);
		\coordinate (v31) at ($(v22) + (45:\scale*\diagcm)$);
		\coordinate (l1) at ($(v21) + (135:\scale*\diagcm)$);
		\coordinate (l2) at ($(v21) + (0,\scale*1)$);
		\coordinate (l3) at ($(v21) + (45:\scale*\diagcm)$);
		\coordinate (l4) at ($(v22) + (135:\scale*\diagcm)$);
		\coordinate (l5) at ($(v31) + (135:\scale*\diagcm)$);
		\coordinate (l6) at ($(v31) + (45:\scale*\diagcm)$);
		
		\draw[thick] (r) to (v11);
		\draw[thick] (v11) to (v21);
		\draw[thick] (v11) to (v22);
		\draw[thick] (v21) to (l1);
		\draw[thick] (v21) to (l2);
		\draw[thick] (v21) to (l3);
		\draw[thick] (v22) to (l4);
		\draw[thick] (v22) to (v31);
		\draw[thick] (v31) to (l5);
		\draw[thick] (v31) to (l6);
		
		\node[above] at (l1) {$1$};
		\node[above] at (l2) {$2$};
		\node[above] at (l3) {$3$};
		\node[above] at (l4) {$4$};
		\node[above] at (l5) {$5$};
		\node[above] at (l6) {$6$};
		
		\draw (v11) circle[radius=\scale*0.5];
		\draw (v21) circle[radius=\scale*0.5];
		\draw ($(v22)!0.5!(v31)$) ellipse[x radius=\scale*1.3, y radius=\scale*0.6, rotate=45];
	\end{tikzpicture}
}}
\]
of degree $|\tau|=3-4=-1$~. 
\end{example}

The graded space spanned by partitioned rooted trees is endowed with a differential $\d_e$ which is the sum of two terms. The first one sums over all edges linking two blocks of the partition and acts by merging the two blocks into a single one. The second term sums over all vertices of the tree and acts by splitting the vertex into two in all possible ways.

\begin{example}
The image under the differential $\d_e$ of the above displayed example of a partitioned rooted tree is equal to 
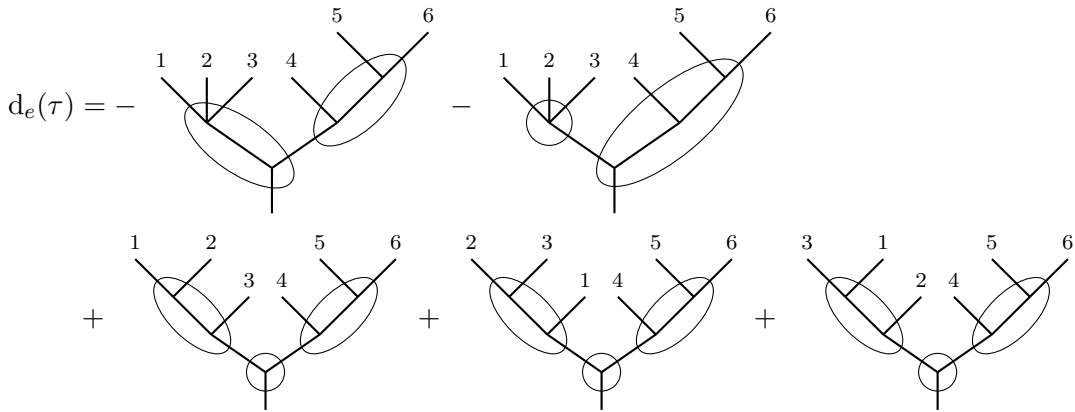
\begin{align*}
\d_e(\tau){}&=-
\vcenter{\hbox{
\begin{tikzpicture}
	\def\scale{0.6};
	\pgfmathsetmacro{\diagcm}{sqrt(2)};
	\def\xangle{35};
	\pgfmathsetmacro{\xcm}{1/sin(\xangle)};
	\coordinate (r) at (0,0);
	\coordinate (v11) at ($(r) + (0,\scale*1)$);
	\coordinate (v21) at ($(v11) + (180-\xangle:\scale*\xcm)$);
	\coordinate (v22) at ($(v11) + (\xangle:\scale*\xcm)$);
	\coordinate (v31) at ($(v22) + (45:\scale*\diagcm)$);
	\coordinate (l1) at ($(v21) + (135:\scale*\diagcm)$);
	\coordinate (l2) at ($(v21) + (0,\scale*1)$);
	\coordinate (l3) at ($(v21) + (45:\scale*\diagcm)$);
	\coordinate (l4) at ($(v22) + (135:\scale*\diagcm)$);
	\coordinate (l5) at ($(v31) + (135:\scale*\diagcm)$);
	\coordinate (l6) at ($(v31) + (45:\scale*\diagcm)$);
	\draw[thick] (r) to (v11);
	\draw[thick] (v11) to (v21);
	\draw[thick] (v11) to (v22);
	\draw[thick] (v21) to (l1);
	\draw[thick] (v21) to (l2);
	\draw[thick] (v21) to (l3);
	\draw[thick] (v22) to (l4);
	\draw[thick] (v22) to (v31);
	\draw[thick] (v31) to (l5);
	\draw[thick] (v31) to (l6);
	\node[above] at (l1) {$\scriptstyle1$};
	\node[above] at (l2) {$\scriptstyle2$};
	\node[above] at (l3) {$\scriptstyle3$};
	\node[above] at (l4) {$\scriptstyle4$};
	\node[above] at (l5) {$\scriptstyle5$};
	\node[above] at (l6) {$\scriptstyle6$};
	\draw ($(v22)!0.5!(v31)$) ellipse[x radius=\scale*1.3, y radius=\scale*0.6, rotate=45];
	\draw ($(v11)!0.5!(v21)$) ellipse[x radius=\scale*1.4, y radius=\scale*0.6, rotate=-\xangle];
\end{tikzpicture}}}
-
\vcenter{\hbox{
\begin{tikzpicture}
	\def\scale{0.6};
	\pgfmathsetmacro{\diagcm}{sqrt(2)};
	\def\xangle{35};
	\pgfmathsetmacro{\xcm}{1/sin(\xangle)};
	\coordinate (r) at (0,0);
	\coordinate (v11) at ($(r) + (0,\scale*1)$);
	\coordinate (v21) at ($(v11) + (180-\xangle:\scale*\xcm)$);
	\coordinate (v22) at ($(v11) + (\xangle:\scale*\xcm)$);
	\coordinate (v31) at ($(v22) + (45:\scale*\diagcm)$);
	\coordinate (l1) at ($(v21) + (135:\scale*\diagcm)$);
	\coordinate (l2) at ($(v21) + (0,\scale*1)$);
	\coordinate (l3) at ($(v21) + (45:\scale*\diagcm)$);
	\coordinate (l4) at ($(v22) + (135:\scale*\diagcm)$);
	\coordinate (l5) at ($(v31) + (135:\scale*\diagcm)$);
	\coordinate (l6) at ($(v31) + (45:\scale*\diagcm)$);
	\draw[thick] (r) to (v11);
	\draw[thick] (v11) to (v21);
	\draw[thick] (v11) to (v22);
	\draw[thick] (v21) to (l1);
	\draw[thick] (v21) to (l2);
	\draw[thick] (v21) to (l3);
	\draw[thick] (v22) to (l4);
	\draw[thick] (v22) to (v31);
	\draw[thick] (v31) to (l5);
	\draw[thick] (v31) to (l6);
	\node[above] at (l1) {$\scriptstyle1$};
	\node[above] at (l2) {$\scriptstyle2$};
	\node[above] at (l3) {$\scriptstyle3$};
	\node[above] at (l4) {$\scriptstyle4$};
	\node[above] at (l5) {$\scriptstyle5$};
	\node[above] at (l6) {$\scriptstyle6$};
	\draw (v21) circle[radius=\scale*0.5];
	\draw ($(v11)!0.5!(v31)$) ellipse[x radius=\scale*2, y radius=\scale*0.75, rotate=(\xangle + 45)/2];
\end{tikzpicture}}}\\
&+ \vcenter{\hbox{
\begin{tikzpicture}
	\def\scale{0.5};
	\pgfmathsetmacro{\diagcm}{sqrt(2)};
	\def\xangle{35};
	\pgfmathsetmacro{\xcm}{1/sin(\xangle)};
	\coordinate (r) at (0,0);
	\coordinate (v11) at ($(r) + (0,\scale*1)$);
	\coordinate (v21) at ($(v11) + (180-\xangle:\scale*\xcm)$);
	\coordinate (v22) at ($(v11) + (\xangle:\scale*\xcm)$);
	\coordinate (v31) at ($(v21) + (135:\scale*\diagcm)$);
	\coordinate (v32) at ($(v22) + (45:\scale*\diagcm)$);
	\coordinate (l1) at ($(v31) + (135:\scale*\diagcm)$);
	\coordinate (l2) at ($(v31) + (45:\scale*\diagcm)$);
	\coordinate (l3) at ($(v21) + (45:\scale*\diagcm)$);
	\coordinate (l4) at ($(v22) + (135:\scale*\diagcm)$);
	\coordinate (l5) at ($(v32) + (135:\scale*\diagcm)$);
	\coordinate (l6) at ($(v32) + (45:\scale*\diagcm)$);
	\draw[thick] (r) to (v11);
	\draw[thick] (v11) to (v21);
	\draw[thick] (v11) to (v22);
	\draw[thick] (v21) to (v31);
	\draw[thick] (v31) to (l1);
	\draw[thick] (v31) to (l2);
	\draw[thick] (v21) to (l3);
	\draw[thick] (v22) to (v32);
	\draw[thick] (v22) to (l4);
	\draw[thick] (v32) to (l5);
	\draw[thick] (v32) to (l6);
	\node[above] at (l1) {$\scriptstyle1$};
	\node[above] at (l2) {$\scriptstyle2$};
	\node[above] at (l3) {$\scriptstyle3$};
	\node[above] at (l4) {$\scriptstyle4$};
	\node[above] at (l5) {$\scriptstyle5$};
	\node[above] at (l6) {$\scriptstyle6$};
	\draw (v11) circle[radius=\scale*0.5];
	\draw ($(v21)!0.5!(v31)$) ellipse[x radius=\scale*1.3, y radius=\scale*0.6, rotate=-45];
	\draw ($(v22)!0.5!(v32)$) ellipse[x radius=\scale*1.3, y radius=\scale*0.6, rotate=45];
\end{tikzpicture}}}
+ \vcenter{\hbox{
\begin{tikzpicture}
	\def\scale{0.5};
	\pgfmathsetmacro{\diagcm}{sqrt(2)};
	\def\xangle{35};
	\pgfmathsetmacro{\xcm}{1/sin(\xangle)};
	\coordinate (r) at (0,0);
	\coordinate (v11) at ($(r) + (0,\scale*1)$);
	\coordinate (v21) at ($(v11) + (180-\xangle:\scale*\xcm)$);
	\coordinate (v22) at ($(v11) + (\xangle:\scale*\xcm)$);
	\coordinate (v31) at ($(v21) + (135:\scale*\diagcm)$);
	\coordinate (v32) at ($(v22) + (45:\scale*\diagcm)$);
	\coordinate (l1) at ($(v31) + (135:\scale*\diagcm)$);
	\coordinate (l2) at ($(v31) + (45:\scale*\diagcm)$);
	\coordinate (l3) at ($(v21) + (45:\scale*\diagcm)$);
	\coordinate (l4) at ($(v22) + (135:\scale*\diagcm)$);
	\coordinate (l5) at ($(v32) + (135:\scale*\diagcm)$);
	\coordinate (l6) at ($(v32) + (45:\scale*\diagcm)$);
	\draw[thick] (r) to (v11);
	\draw[thick] (v11) to (v21);
	\draw[thick] (v11) to (v22);
	\draw[thick] (v21) to (v31);
	\draw[thick] (v31) to (l1);
	\draw[thick] (v31) to (l2);
	\draw[thick] (v21) to (l3);
	\draw[thick] (v22) to (v32);
	\draw[thick] (v22) to (l4);
	\draw[thick] (v32) to (l5);
	\draw[thick] (v32) to (l6);
	\node[above] at (l1) {$\scriptstyle2$};
	\node[above] at (l2) {$\scriptstyle3$};
	\node[above] at (l3) {$\scriptstyle1$};
	\node[above] at (l4) {$\scriptstyle4$};
	\node[above] at (l5) {$\scriptstyle5$};
	\node[above] at (l6) {$\scriptstyle6$};
	\draw (v11) circle[radius=\scale*0.5];
	\draw ($(v21)!0.5!(v31)$) ellipse[x radius=\scale*1.3, y radius=\scale*0.6, rotate=-45];
	\draw ($(v22)!0.5!(v32)$) ellipse[x radius=\scale*1.3, y radius=\scale*0.6, rotate=45];
\end{tikzpicture}}}
+ \vcenter{\hbox{
\begin{tikzpicture}
	\def\scale{0.5};
	\pgfmathsetmacro{\diagcm}{sqrt(2)};
	\def\xangle{35};
	\pgfmathsetmacro{\xcm}{1/sin(\xangle)};
	\coordinate (r) at (0,0);
	\coordinate (v11) at ($(r) + (0,\scale*1)$);
	\coordinate (v21) at ($(v11) + (180-\xangle:\scale*\xcm)$);
	\coordinate (v22) at ($(v11) + (\xangle:\scale*\xcm)$);
	\coordinate (v31) at ($(v21) + (135:\scale*\diagcm)$);
	\coordinate (v32) at ($(v22) + (45:\scale*\diagcm)$);
	\coordinate (l1) at ($(v31) + (135:\scale*\diagcm)$);
	\coordinate (l2) at ($(v31) + (45:\scale*\diagcm)$);
	\coordinate (l3) at ($(v21) + (45:\scale*\diagcm)$);
	\coordinate (l4) at ($(v22) + (135:\scale*\diagcm)$);
	\coordinate (l5) at ($(v32) + (135:\scale*\diagcm)$);
	\coordinate (l6) at ($(v32) + (45:\scale*\diagcm)$);
	\draw[thick] (r) to (v11);
	\draw[thick] (v11) to (v21);
	\draw[thick] (v11) to (v22);
	\draw[thick] (v21) to (v31);
	\draw[thick] (v31) to (l1);
	\draw[thick] (v31) to (l2);
	\draw[thick] (v21) to (l3);
	\draw[thick] (v22) to (v32);
	\draw[thick] (v22) to (l4);
	\draw[thick] (v32) to (l5);
	\draw[thick] (v32) to (l6);
	\node[above] at (l1) {$\scriptstyle3$};
	\node[above] at (l2) {$\scriptstyle1$};
	\node[above] at (l3) {$\scriptstyle2$};
	\node[above] at (l4) {$\scriptstyle4$};
	\node[above] at (l5) {$\scriptstyle5$};
	\node[above] at (l6) {$\scriptstyle6$};
	\draw (v11) circle[radius=\scale*0.5];
	\draw ($(v21)!0.5!(v31)$) ellipse[x radius=\scale*1.3, y radius=\scale*0.6, rotate=-45];
	\draw ($(v22)!0.5!(v32)$) ellipse[x radius=\scale*1.3, y radius=\scale*0.6, rotate=45];
\end{tikzpicture}}}
\end{align*}
\end{example}

\begin{lemma}\label{lem:BarPi}
The underlying space of the bar construction $ \Bar_\pi  \g$ of a complete $\sLi$-algebra $\g$ is the cofree conilpotent $\Bar \Cobar \com^\vee$-coalgebra on $\g$~, explicitly given by 
\[
\Bar \Cobar \com^\vee (\g) \cong \bigoplus_{\tau \in \PaRT} \tau(\g)\ ,  
\]
where $\tau(\g)$ stands for the partitioned rooted tree $\tau$ with leaves labeled by elements of $\g$~. It is endowed with a square-zero coderivation $\d=\d_\g+\d_e+\d_\ell$~, where $\d_\g$ is induced by the underlying differential $d$ of $\g$~,  where 
$\d_e(\tau(\g))\coloneqq \d_e(\tau)(\g)$~, and where $\d_\ell(\tau(\g))$ is equal to 
the sum over all the blocks lying at the top of the partitioned tree $\tau$~, cutting such block one-by-one out of the tree and making the operations $\ell_m$ corresponding to the vertices contained in the block act on the associated elements of $\g$~. 
\end{lemma}

\begin{proof}
Recall that the  universal operadic twisting morphism has the following explicit form
\begin{align*}
	\pi \colon \Bar \Cobar \com^\vee=
	\T^c\Big(s\overline{\T}\Big(s^{-1} \overline{\com}^\vee\Big)\Big)&
	\relbar\joinrel\twoheadrightarrow
	s\overline{\T}\Big(s^{-1} \overline{\com}^\vee\Big)\\
	&\xrightarrow{s^{-1}}
	\overline{\T}\Big(s^{-1} \overline{\com}^\vee\Big)
	= \Cobar \com^\vee\cong \sLi\ .
\end{align*}
We establish $\PaRT$ as a basis for $\Bar \Cobar \com^\vee$ by identifying the ``inner'' elements of $\overline{\T}(s^{-1} \overline{\com}^\vee)$ with the blocks and the ``outer'' tree structure with the whole tree. The statement then follows by a direct application of the definition of the bar construction with respect to $\pi$ given in \cref{subsec:Op}.
\end{proof}

\begin{example}
The image under the differential of the bar construction of the above given tree $\tau$ labeled by element $a_1, \ldots, a_6$ is equal to
\begin{align*}
\d(\tau(a_1, \ldots, a_6))=&-\sum_{i=1}^6(-1)^{|a_1|+\cdots+|a_{i-1}|}\tau(a_1, \ldots, d(a_i), \ldots, a_6) 
+\d_e(\tau)(a_1, \ldots, a_6)\\&
-\overline{\tau}^1(\ell_3(a_1, a_2, a_3), a_4, a_5, a_6)
+(-1)^{|a_4|}
\overline{\tau}^4
\big(a_1, a_2, a_3, \ell_2(a_4, \ell_2(a_5, a_6))\big)\ , 
\end{align*}
where 
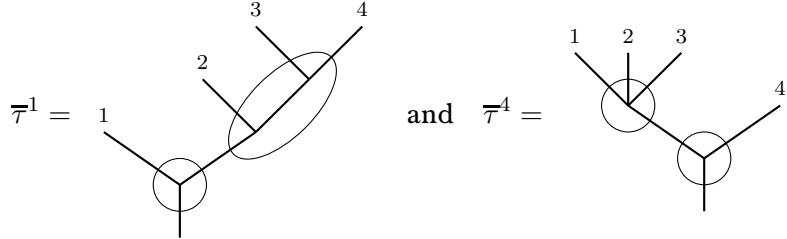
\[
\overline{\tau}^1= \vcenter{\hbox{
	\begin{tikzpicture}
		\def\scale{0.7};
		\pgfmathsetmacro{\diagcm}{sqrt(2)};
		
		\def\xangle{35};
		\pgfmathsetmacro{\xcm}{1/sin(\xangle)};
		
		\coordinate (r) at (0,0);
		\coordinate (v11) at ($(r) + (0,\scale*1)$);
		\coordinate (v21) at ($(v11) + (180-\xangle:\scale*\xcm)$);
		\coordinate (v22) at ($(v11) + (\xangle:\scale*\xcm)$);
		\coordinate (v31) at ($(v22) + (45:\scale*\diagcm)$);
		\coordinate (l4) at ($(v22) + (135:\scale*\diagcm)$);
		\coordinate (l5) at ($(v31) + (135:\scale*\diagcm)$);
		\coordinate (l6) at ($(v31) + (45:\scale*\diagcm)$);
		
		\draw[thick] (r) to (v11);
		\draw[thick] (v11) to (v21);
		\draw[thick] (v11) to (v22);
		\draw[thick] (v22) to (l4);
		\draw[thick] (v22) to (v31);
		\draw[thick] (v31) to (l5);
		\draw[thick] (v31) to (l6);
		
		\node[above] at (v21) {$\scriptstyle1$};
		\node[above] at (l4) {$\scriptstyle2$};
		\node[above] at (l5) {$\scriptstyle3$};
		\node[above] at (l6) {$\scriptstyle4$};
		
		\draw (v11) circle[radius=\scale*0.5];
		\draw ($(v22)!0.5!(v31)$) ellipse[x radius=\scale*1.3, y radius=\scale*0.6, rotate=45];
	\end{tikzpicture}}}
\quad \text{and}\quad 
\overline{\tau}^4= 
\vcenter{\hbox{
	\begin{tikzpicture}
		\def\scale{0.7};
		\pgfmathsetmacro{\diagcm}{sqrt(2)};
		
		\def\xangle{35};
		\pgfmathsetmacro{\xcm}{1/sin(\xangle)};
		
		\coordinate (r) at (0,0);
		\coordinate (v11) at ($(r) + (0,\scale*1)$);
		\coordinate (v21) at ($(v11) + (180-\xangle:\scale*\xcm)$);
		\coordinate (v22) at ($(v11) + (\xangle:\scale*\xcm)$);
		\coordinate (l1) at ($(v21) + (135:\scale*\diagcm)$);
		\coordinate (l2) at ($(v21) + (0,\scale*1)$);
		\coordinate (l3) at ($(v21) + (45:\scale*\diagcm)$);
		
		\draw[thick] (r) to (v11);
		\draw[thick] (v11) to (v21);
		\draw[thick] (v11) to (v22);
		\draw[thick] (v21) to (l1);
		\draw[thick] (v21) to (l2);
		\draw[thick] (v21) to (l3);
		
		\node[above] at (l1) {$\scriptstyle1$};
		\node[above] at (l2) {$\scriptstyle2$};
		\node[above] at (l3) {$\scriptstyle3$};
		\node[above] at (v22) {$\scriptstyle4$};
		
		\draw (v11) circle[radius=\scale*0.5];
		\draw (v21) circle[radius=\scale*0.5];
	\end{tikzpicture}}}
\]
\end{example}

Given an $\sLi$-algebra
$(\g,\allowbreak \{\ell_m\}_{m\geqslant 1})$ and a rooted tree $\tau\in\RT_k$ with $k$ leaves, we denote by
\[
\ell_\tau \colon \g^{\widehat{\otimes} k}\longrightarrow \g
\]
the operation obtained by interpreting the tree $\tau$ as a composite of the structure operations of $\g$~, with each vertex replaced by the operation of the same arity.

\begin{proposition}
	The data of an $\infty_\pi$-morphism $f : \g \rightsquigarrow \h$ between two complete $\sLi$-algebras
	$(\g,\allowbreak \{\ell_m\}_{m\geqslant 1})$ and $(\h, \{\mathcal{k}_m\}_{m\geqslant 1})$ 
	amounts to a collection of degree $0$ maps
	\[
	f_\tau : \tau(\g) \longrightarrow \h\ ,
	\]
	or equivalently to a collection of degree $|\tau|$ maps
	\[
	f_\tau : \g^{\widehat{\otimes} m} \longrightarrow \h\ ,
	\]
	labeled by the set $\PaRT$ of partitioned rooted trees, which have the same symmetries as the underlying tree and which respect the filtrations. These maps are subject to the following conditions, one for each $\tau \in \PaRT$~:
	\[
	\partial(f_\tau)=
	f_{\d_e(\tau)}+\left(\sum \pm\, f_{\overline{\tau}^i}\circ_i \ell_{\tau_i}\right)
	-\mathcal{k}_{\widetilde\tau}\left(f_{\widetilde\tau_1}, \ldots, f_{\widetilde\tau_k}\right)
	\ ,
	\]
	where
	\begin{itemize}
		\item[$\diamond$] $f_{\d_e(\tau)}$ comes from the extension of $f_{-}$ by linearity in the indexing tree,
		\item[$\diamond$] the sum runs over all the top blocks $\tau_i$ of $\tau$~, that is all the blocks whose incoming edges are all leaves, and where in the notation $f_{\overline{\tau}^i}\circ_i \ell_{\tau_i}$~, the term $\overline{\tau}^i$ means the tree $\tau$ with the block $\tau_i$ cut away and the operation $\ell_{\tau_i}$ is inserted on the leaf created by removing the block, and
		\item[$\diamond$] $\widetilde{\tau}$ indicates the block of $\tau$ linked to the root and $\widetilde{\tau}_1,\ldots,\widetilde{\tau}_k$ denote all the sub-trees above it (one for each leaf of the root block).
	\end{itemize}
\end{proposition}

\begin{proof}
This is a direct corollary of \cref{lem:BarPi} 
applied to \cref{def:InfMor}. 
\end{proof}

\begin{example}
On the example of the partitioned rooted tree $\tau$ that we are following and under the same convention, this gives
\[
\partial(f_\tau)=
f_{\d_e(\tau)}- f_{\overline{\tau}^1}\circ_1 \ell_{\tau_1}
+f_{\overline{\tau}^4}\circ_4 \ell_{\tau_4}
-\mathcal{k}_{\tilde\tau}\left(f_{\tilde\tau_1}, f_{\tilde\tau_2}\right),
\]
where
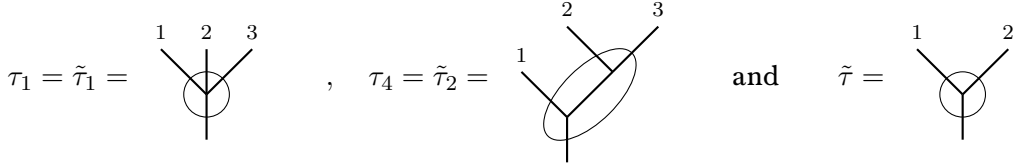
\[
\tau_1= \tilde\tau_1=\vcenter{\hbox{
	\begin{tikzpicture}
		\def\scale{0.6};
		\pgfmathsetmacro{\diagcm}{sqrt(2)};
	
		\coordinate (r) at (0,0);
		\coordinate (v1) at (0,\scale*1);
		\coordinate (l1) at ($(v1) + (135:\scale*\diagcm)$);
		\coordinate (l2) at ($(v1) + (0,\scale*1)$);
		\coordinate (l3) at ($(v1) + (45:\scale*\diagcm)$);
		
		\draw[thick] (r) to (v1);
		\draw[thick] (v1) to (l1);
		\draw[thick] (v1) to (l2);
		\draw[thick] (v1) to (l3);
		
		\node[above] at (l1) {$\scriptstyle1$};
		\node[above] at (l2) {$\scriptstyle2$};
		\node[above] at (l3) {$\scriptstyle3$};
		
		\draw (v1) circle[radius=\scale*0.5];
	\end{tikzpicture}}}
\qquad , \quad
\tau_4= \tilde\tau_2=\vcenter{\hbox{
	\begin{tikzpicture}
		\def\scale{0.6};
		\pgfmathsetmacro{\diagcm}{sqrt(2)};
		
		\coordinate (r) at (0,0);
		\coordinate (v1) at (0,\scale*1);
		\coordinate (v2) at ($(v1) + (45:\scale*\diagcm)$);
		\coordinate (l1) at ($(v1) + (135:\scale*\diagcm)$);
		\coordinate (l2) at ($(v2) + (135:\scale*\diagcm)$);
		\coordinate (l3) at ($(v2) + (45:\scale*\diagcm)$);
		
		\draw[thick] (r) to (v1);
		\draw[thick] (v1) to (v2);
		\draw[thick] (v1) to (l1);
		\draw[thick] (v2) to (l2);
		\draw[thick] (v2) to (l3);
		
		\node[above] at (l1) {$\scriptstyle1$};
		\node[above] at (l2) {$\scriptstyle2$};
		\node[above] at (l3) {$\scriptstyle3$};
		
		\draw ($(v1)!0.5!(v2)$) ellipse[x radius=\scale*1.3, y radius=\scale*0.6, rotate=45];
	\end{tikzpicture}}}
\qquad \text{and} 
\qquad 
\tilde\tau=\vcenter{\hbox{
	\begin{tikzpicture}
		\def\scale{0.6};
		\pgfmathsetmacro{\diagcm}{sqrt(2)};
		
		\coordinate (r) at (0,0);
		\coordinate (v1) at (0,\scale*1);
		\coordinate (l1) at ($(v1) + (135:\scale*\diagcm)$);
		\coordinate (l3) at ($(v1) + (45:\scale*\diagcm)$);
		
		\draw[thick] (r) to (v1);
		\draw[thick] (v1) to (l1);
		\draw[thick] (v1) to (l3);
		
		\node[above] at (l1) {$\scriptstyle1$};
		\node[above] at (l3) {$\scriptstyle2$};
		
		\draw (v1) circle[radius=\scale*0.5];
	\end{tikzpicture}}}
\]
So 
$f_{\overline{\tau}^1}\circ_1 \ell_{\tau_1}=f_{\overline{\tau}^1}\circ_1 \ell_3$~, 
$f_{\overline{\tau}^4}\circ_4 \ell_{\tau_4}=f_{\overline{\tau}^4}\circ_4 (\ell_2\circ_2\ell_2)$~,
and $\mathcal{k}_{\tilde\tau}\left(f_{\tilde\tau_1}, f_{\tilde\tau_2}\right)=\mathcal{k}_2\left(f_{\tilde\tau_1}, f_{\tilde\tau_2}\right)$~. 
\end{example}

\subsection{Relationship between \texorpdfstring{$\infty_\pi$}{infinity-pi}-morphisms and \texorpdfstring{$\infty$}{infinity}-morphisms}

This notion of an $\infty_\pi$-mor\-phi\-sm, which is very natural with respect to the integration functor, is however not the most simple and most used higher notion of morphisms of complete $\sLi$-algebras. In the literature, it is the notion of an $\infty$-morphism, that is the notion associated to the operadic twisting morphism $\iota$~, which is omnipresent. The following simple construction relates these two notions in one direction. 

\begin{proposition}\label{lemma:oo-pi induce oo-iota}
Any $\infty_\pi$-morphism
\[
\big\{f_\tau : \g^{\hot m} \longrightarrow \h\big\}_{\tau \in \PaRT}
\]
of complete $\sLi$-algebras functorially induces an $\infty$-morphism $\upsilon^*f$ under the formula:
\[
(\upsilon^*f)_m\coloneqq \sum_{\substack{\tau \in \PaRT_m \\ |\tau|=0}} f_\tau \ .
\]
\end{proposition}

\begin{proof}
The two operadic twisting morphisms $\pi$ and $\iota$ are related by the unit of the operadic bar-cobar adjunction: 
\begin{center}
		\begin{tikzpicture}
		\node (a) at (0,1){$\com^\vee$};
		\node (b) at (0,-1){$\Bar\Cobar\com^\vee$};
		\node (c) at (2,0){$\sLi$~.};
		
		\draw[->] (a) -- node[left]{$\upsilon$} (b);
		\draw[->] (a) -- node[above right]{$\iota$} (c);
		\draw[->] (b) -- node[below right]{$\pi$} (c);
		\end{tikzpicture}
	\end{center}
So  pulling back along $\upsilon$~, one sends an $\infty_\pi$-morphism to an $\infty$-morphism:
\[
\upsilon^*f \ : \ 
\com^\vee(\g) \xrightarrow{\upsilon(\id)}
\Bar\Cobar\com^\vee(\g) \xrightarrow{f}
\h \ .
\]
 The explicit form of the unit $\upsilon$ of the operadic bar-cobar adjunction gives the formula of the present statement.
\end{proof}

\begin{proposition}\label{lemma:describe R(f)_0}
	Let $f:\g\rightsquigarrow\h$ be an $\infty_\pi$-morphism. Under the identification $\R(\g)_0\cong\MC(\g)$~, we have 
	\[\R(f)_0\cong\MC(\upsilon^*f)\ .\]
\end{proposition}

\begin{proof}
	Recall that 
	\[
	\R(\g)_0\cong\MC(\hom^\pi(C^0,\g))
	\]
	and that $\R(f)_0$ acts by $\MC(\hom^\pi_r(\id, f))$ on it, see the proof of  \cref{prop:Functoriality}. 
Also, it follows directly from the proof of \cref{prop:mc0}, that the 
 $\Bar \Cobar \com^\vee$-coalgebra structure on $\rmC^0\cong \k a_0$ is given by
		\[
		\Delta(a_0) = \sum_{m\geqslant 1}\sum_{\substack{\tau\in\PaRT_m\\ |\tau|=0}}\tau\otimes_{\S_m}a_0^{\otimes m}\ .
		\]
	Thus, for any $x_1,\ldots,x_m\in\hom^\pi(C^0,\g)$~, we have 
	\begin{align*}
		\hom^\pi_r(\id, f)_m(x_1, \ldots, x_m)(a_0) ={}&f(\id \otimes_{\Sy_m}(x_1\otimes\cdots\otimes x_m))\Delta_m(a_0)\\
		={}&f(\id \otimes_{\Sy_m}(x_1\otimes\cdots\otimes x_m))\left(\sum_{\substack{\tau \in \PaRT_m \\ |\tau|=0}}\tau\otimes_{\Sy_m}  a_0^{\otimes m}\right)\\
		={}&\left(\upsilon^*f\right)_m (x_1(a_0), \ldots, x_m(a_0))
	\end{align*}
	as desired. 
\end{proof}

Recall that the set of degree 0 partitioned rooted trees is made up of all rooted trees with vertices of arity greater of equal to 2 and equipped with the finest partition, that is the one where every vertex forms a block. Unfortunately, this process cannot be reversed: the counit $\upsilon$ of adjunction is not an isomorphism. However, it is a quasi-isomorphism, so the reverse process holds true ``up to homotopy''. 
No such reverse process can be  strictly compatible with the respective composition of higher morphisms. 

\begin{theorem}\label{thm:NoUnivForm}\leavevmode
\begin{enumerate}
\item\label{pt:exist formulas} There exist universal formulas which associate $\infty$-morphisms to $\infty_\pi$-morphisms
and which are the identity on strict morphisms.
\item\label{pt:no composition} No such universal formulas which associate $\infty$-morphisms to $\infty_\pi$-morphisms and which   do not change the complete $\sLi$-algebra structures can respect the compositions.
\end{enumerate}
\end{theorem}

\begin{proof}
	Along with the proof of point (\ref{pt:exist formulas}), we will also make more explicit what we mean by such a statement. Recall from \cite{Markl04} and \cite[Section~6.2]{MerkulovVallette09I} that there exists a $2$-colored operad 
	\[\left(\T\left(
	s^{-1} \overline{\com}^\vee \oplus 
	{\com}^\vee \oplus 
	s^{-1} \overline{\com}^\vee
	\right), d\right)\]
	which encodes the data of two complete $\sLi$-algebras on $\g$ and $\h$ together with an $\infty$-morphism from $\g$ to $\h$~.
	The first generating summand lies in color $1$~, the second one lies in input color $0$ and output color $1$~, and the third one lies in color $0$~. The image of the differential on the first and the third summands is the one of the cobar construction, so it is constructed from the partial decomposition products of the cooperad $\com^\vee$ and the image of the differential on 
	second summand is given by both the partial and the monoidal decomposition product of the cooperad $\com^\vee$, see \cite{HLV2} for more details.
	In the same way, there is a $2$-colored operad 
	\[\left(\T\left(
	s^{-1} \overline{\com}^\vee \oplus 
	\Bar \Cobar \overline{\com}^\vee \oplus 
	s^{-1} \overline{\com}^\vee
	\right), d'\right)\]
	which encodes the data of two complete $\sLi$-algebras on $\g$ and $\h$ together with an $\infty_\pi$-morphism from $\g$ to $\h$~. The only difference with the previous $2$-colored operad lies in the image of the differential on the second summand 
	which is constructed from the partial and the monoidal decomposition products of the cooperad  $\Bar \Cobar \overline{\com}^\vee$ \emph{and} the canonical twisting morphism $\pi$~. 
	Applying the morphism of cooperads $\upsilon : \com^\vee \xrightarrow{\sim} \Bar \Cobar \overline{\com}^\vee$ to the second generating summand and the identity to the other ones, we get a morphism of $2$-colored operads 
	\[
	\left(\T\left(
	s^{-1} \overline{\com}^\vee \oplus 
	{\com}^\vee \oplus 
	s^{-1} \overline{\com}^\vee
	\right), d\right)
	\longrightarrow
	\left(\T\left(
	s^{-1} \overline{\com}^\vee \oplus 
	\Bar \Cobar \overline{\com}^\vee \oplus 
	s^{-1} \overline{\com}^\vee
	\right), d'\right).\]
	Pulling back along this map produces the functorial and universal formula of \cref{lemma:oo-pi induce oo-iota}.
	
	\medskip
	Since $\upsilon$ is a quasi-isomorphism one can see, by a spectral sequence argument, that this morphism of $2$-colored operads is a quasi-isomorphism. 
	In the model category of \cite{caviglia2014}, every $2$-colored operad is fibrant. 
	Considering the weight filtration, one can see that the above two $2$-colored operads are cofibrant. 
	Therefore, the above mentioned quasi-isomorphism admits a homotopy inverse.
	Pulling back along it, produces universal formulas which associate $\infty$-morphisms to $\infty_\pi$-morphisms. Since this homotopy inverse induces an inverse isomorphism on the level of homology group in arity $1$ and color $01$ part, it must be equal to the identity in this arity and color, so that it preserves strict morphisms. 
	
	\medskip
	
	For point (\ref{pt:no composition}), suppose that  
	\[
	\left(\T\left(
	s^{-1} \overline{\com}^\vee \oplus 
	\Bar \Cobar \overline{\com}^\vee \oplus 
	s^{-1} \overline{\com}^\vee
	\right)
	,d' \right)
	\stackrel{\Phi}{\longrightarrow}
	\left(\T\left(
	s^{-1} \overline{\com}^\vee \oplus 
	{\com}^\vee \oplus 
	s^{-1} \overline{\com}^\vee
	\right), d\right)
	\]
	is a morphism of $2$-colored operads which is the identity on the first and the third summands, corresponding to the fact that it leaves the algebraic structures untouched, and which sends 
	$\id \in \Bar \Cobar \overline{\com}^\vee$
to $a\, \id \in \com^\vee$~, with $a\in\k$~. We claim that the induced universal formulas associating $\infty$-morphisms to $\infty_\pi$-morphisms, obtained  
	by pulling back along $\Phi$~, do not respect their compositions. The fact that the morphism $\Phi$ preserves the respective differentials forces 
\begin{align*}
\Phi(\tau_1)=a\, \mu_2^{01}\ , \quad \Phi(\tau_2)=a\left(
\mu_2^{01}\circ_1 \left(s^{-1}\mu_2^{0}\right)-
\left(s^{-1}\mu_2^{1}\right)\circ_1 \mu_2^{01}
\right), \quad \text{and}\quad \Phi(\tau_3)=0\ ,
\end{align*}
	where $s^{-1}\mu_m^0$ and $s^{-1}\mu_m^1$ stand respectively for the generators of the first and the third summand, where $\mu_m^{01}$ stand for the generators of the second summand, and where 
	\begin{align*}
	\tau_1=\vcenter{\hbox{
		\begin{tikzpicture}
			\def\scale{0.6};
			\pgfmathsetmacro{\diagcm}{sqrt(2)};
			\coordinate (r) at (0,0);
			\coordinate (v1) at ($(r) + (0,\scale*1)$);
			\coordinate (l1) at ($(v1) + (135:\scale*\diagcm)$);
			\coordinate (l2) at ($(v1) + (45:\scale*\diagcm)$);
			\draw[thick] (r) to (v1);
			\draw[thick] (v1) to (l1);
			\draw[thick] (v1) to (l2);
			\draw (v1) circle[radius=\scale*0.5];
			\node[above] at (l1) {$\scriptstyle1$};
			\node[above] at (l2) {$\scriptstyle2$};
		\end{tikzpicture}}}
	\ , \ 
	\tau_2=\vcenter{\hbox{\begin{tikzpicture}
			\def\scale{0.6};
			\pgfmathsetmacro{\diagcm}{sqrt(2)};
			\coordinate (r) at (0,0);
			\coordinate (v1) at ($(r) + (0,\scale*1)$);
			\coordinate (v2) at ($(v1) + (135:\scale*\diagcm)$);
			\coordinate (l1) at ($(v2) + (135:\scale*\diagcm)$);
			\coordinate (l2) at ($(v2) + (45:\scale*\diagcm)$);
			\coordinate (l3) at ($(v1) + (45:\scale*\diagcm)$);
			\draw[thick] (r) to (v1);
			\draw[thick] (v1) to (v2);
			\draw[thick] (v2) to (l1);
			\draw[thick] (v2) to (l2);
			\draw[thick] (v1) to (l3);
			\node[above] at (l1) {$\scriptstyle1$};
			\node[above] at (l2) {$\scriptstyle2$};
			\node[above] at (l3) {$\scriptstyle3$};
			\draw ($(v1)!0.5!(v2)$) ellipse[x radius=\scale*1.3, y radius=\scale*0.6, rotate=-45];
		\end{tikzpicture}}}
	\ ,\ \ \text{and} \ \ 
	\tau_3=\vcenter{\hbox{\begin{tikzpicture}
			\def\scale{0.6};
			\pgfmathsetmacro{\diagcm}{sqrt(2)};
			\coordinate (r) at (0,0);
			\coordinate (v1) at ($(r) + (0,\scale*1)$);
			\coordinate (v2) at ($(v1) + (135:\scale*\diagcm)$);
			\coordinate (l1) at ($(v2) + (135:\scale*\diagcm)$);
			\coordinate (l2) at ($(v2) + (45:\scale*\diagcm)$);
			\coordinate (l3) at ($(v1) + (45:\scale*\diagcm)$);
			\draw[thick] (r) to (v1);
			\draw[thick] (v1) to (v2);
			\draw[thick] (v2) to (l1);
			\draw[thick] (v2) to (l2);
			\draw[thick] (v1) to (l3);
			\node[above] at (l1) {$\scriptstyle1$};
			\node[above] at (l2) {$\scriptstyle2$};
			\node[above] at (l3) {$\scriptstyle3$};
			\draw (v1) circle[radius=\scale*0.5];
			\draw (v2) circle[radius=\scale*0.5];
		\end{tikzpicture}}} .
	\end{align*}
Notice that the condition on $\tau_2$ implies that $a\in\{0, 1, -1\}$~. 
When $f : \g \rightsquigarrow \h$ is an $\infty$-morphism, we denote by $\tilde{f}$ the induced $\infty_\pi$-morphism obtained by pulling back  along $\Phi$~. The above computations show that, for any pair $f : \g \rightsquigarrow \h$ and 
$g : \h \rightsquigarrow \mathfrak{k}$ of $\infty$-morphisms, we have 
\[\left(\tilde{g}\circ \tilde{f}\right)_{\tau_3}=a^2 (g_2\circ_1 f_2) \neq \widetilde{\left(g \circ f\right)}_{\tau_3}=0\ . \]
\end{proof}

\subsection{Homotopy functoriality with respect to \texorpdfstring{$\infty$}{infinity}-morphisms}

As a consequence of \cref{thm:NoUnivForm}, in order to define a functorial structure on the integration functor $\R$ with respect to $\infty$-morphisms we need to refine our arguments as follows.
Recall that the universal twisting morphism $\iota$ is Koszul by \cite[Lemma~6.5.9]{LodayVallette12} and so we have a natural quasi-isomorphism
	\[
	\epsilon_\g \ \colon\  \hatCobar_\iota \Bar_\iota \g \stackrel{\sim}{\longrightarrow} \g 
	\]
in the category $\sLialg$ by \cref{prop:EpsiResFonc}. 
For any $\infty$-morphism $f : \Bar_\iota \g \to \Bar_\iota \h$ of complete $\sLi$-algebras we consider the zig-zag 
	\[
\begin{tikzcd}[column sep=large]
 \g
&
\arrow[l,, "\sim", "\epsilon_\g"']
\hatCobar_\iota \Bar_\iota \g
 \arrow[r, "\hatCobar_\iota f"]
&
\hatCobar_\iota \Bar_\iota \h
\arrow[r, "\epsilon_{\, \h}","\sim"']
&\h
\end{tikzcd}	
	\]
of strict morphisms of complete $\sLi$-algebras and its image under the integration functor
	\[
\begin{tikzcd}[column sep=large]
\R(\g)
&
\arrow[l, "\sim", "\R(\epsilon_\g)"']
\R\big(\hatCobar_\iota \Bar_\iota \g\big)
 \arrow[r, "\R(\hatCobar_\iota f)"]
&
\R\big(\hatCobar_\iota \Bar_\iota \h\big)
\arrow[r, "\R(\epsilon_{\, \h})", "\sim"']
&\R(\h) \ .
\end{tikzcd}	
	\]
Notice that the integration functor sends strict quasi-isomorphisms to weak equivalences by \cref{thm:HoInvariance}; the proof of that result is  independent of the following result. 

\begin{proposition}
The above assignment induces a functor 
		\[
		\Rh\ : \ \infty\text{-}\,\sLi\text{-}\,\mathsf{alg}\longrightarrow \mathsf{Ho(}\sSe\mathsf{)}\ , 
		\]
which extends the composite of the integration functor with the projection onto the homotopy category of simplicial sets. 		
\end{proposition}

\begin{proof}
Since the map $\R(\epsilon_\g)$ is always a weak equivalence of simplicial sets, the map $\Rh(f)$ is always  well-defined in the homotopy category of simplicial sets. Applying this to the  identity map, we get
\[
\begin{tikzcd}[column sep=large]
\R(\g)
&
\arrow[l, "\R(\epsilon_\g)"', "\sim"]
\R\big(\hatCobar_\iota \Bar_\iota \g\big)
\arrow[r, "\R(\epsilon_\g)","\sim"']
&\R(\g) \ , 
\end{tikzcd}	
\]
which is equal to the identity of $\R(\g)$ in $\mathsf{Ho(}\sSe\mathsf{)}$~.

\medskip

Given two $\infty$-morphisms 
$f : \Bar_\iota \g \to \Bar_\iota \h$ and 
$g : \Bar_\iota \h \to \Bar_\iota \mathfrak{k}$~, the composite of their images under $\Rh$ is equal to
\[
\vcenter{\hbox{
	\begin{tikzpicture}
		\node (a) at (0,0) {${\R(\g)}$};
		\node (b) at (3,0) {${\R\big(\hatCobar_\iota \Bar_\iota \g\big)}$};
		\node (c) at (6,0) {${\R\big(\hatCobar_\iota \Bar_\iota \h\big)}$};
		\node (d) at (9,0) {${\R(\h)}$};
		\node (e) at (9,-1.8) {${\R\big(\hatCobar_\iota \Bar_\iota \h\big)}$};
		\node (f) at (9,-3.6) {${\R\big(\hatCobar_\iota \Bar_\iota \mathfrak{k}\big)}$};
		\node (g) at (9,-5.4) {$\scriptstyle{\R(\mathfrak{k})}$};
		
		\draw[->] (b) to node[above]{$\scriptstyle{\R(\epsilon_\g)}$} node[below]{$\sim$} (a);
		\draw[->] (b) to node[above]{$\scriptstyle{\R(\hatCobar_\iota f)}$} (c);
		\draw[->] (c) to node[above]{$\scriptstyle{\R(\epsilon_{\, \h})}$} node[below]{$\sim$} (d);
		\draw[->] (e) to node[right]{$\scriptstyle{\R(\epsilon_{\, \h})}$} node[above, sloped]{$\sim$} (d);
		\draw[->] (e) to node[right]{$\scriptstyle{\R(\hatCobar_\iota g)}$} (f);
		\draw[->] (f) to node[right]{$\scriptstyle{\R(\epsilon_\mathfrak{k})}$} node[below, sloped]{$\sim$} (g);
		
		\draw[double equal sign distance] (c) to (e);
	\end{tikzpicture}
}}
\]
which is the image of their composite in the homotopy category of simplicial sets. This proves that we have a well-defined functor.

\medskip

For any strict morphism $f \colon \g \to \h$ of complete $\sLi$-algebras, the following diagram is commutative 
\[
\vcenter{\hbox{
	\begin{tikzpicture}
		\node (a) at (0,0) {$\hatCobar_\iota \Bar_\iota \g$};
		\node (b) at (3,0) {$\hatCobar_\iota \Bar_\iota \h$};
		\node (c) at (0,-2) {$\g$};
		\node (d) at (3,-2) {$\h$};
		
		\draw[->] (a) to node[above]{$\scriptstyle{\hatCobar_\iota\Bar_\iota f}$} (b);
		\draw[->] (a) to node[left]{$\scriptstyle{\epsilon_\g}$} node[above, sloped]{$\sim$} (c);
		\draw[->] (b) to node[right]{$\scriptstyle{\epsilon_{\, \h}}$} node[below, sloped]{$\sim$} (d);
		\draw[->] (c) to node[above]{$\scriptstyle{f}$} (d);
	\end{tikzpicture}
}}
\]
This shows that $\Rh\big(\Bar_\iota f\big)=\mathrm{proj}\circ \R(f)$~, which proves the last claim. 
\end{proof}

\subsection{Twisting procedure and the integration functor}\label{sec:TwisProcSimRep}

One of the main applications of the twisting procedure of \cref{prop:TwiProc} is to change the base point 
from $0$ to $\alpha$~, as the following well-known result shows. 

\begin{lemma}\label{lemma:MC in twisted Loo alg}
Let $\g$ be a complete $\sLi$-algebra and let $\alpha\in\MC(\g)$~. An element $x\in\g_0$ is a Maurer--Cartan element in $\g^\alpha$  if and only if $x+\alpha$ is a Maurer--Cartan element in $\g$~.
\end{lemma}

\begin{proof}
This follows from a straightforward computation. For a more conceptual explanation, we refer the reader to \cite[Chapter~3, Corollary~3.1]{DotsenkoShadrinVallette18}. 
\end{proof}

This is particularly useful when one wants to study the homotopy groups of the simplicial set $\R(\g)$ as it allows us to always reduce to the case where the base point is $0\in\g$~. We will systematically use this fact in \cref{sect:homotopy theory,sec: rational models}. However, in order to prove that some statements hold functorially with respect to $\infty_\pi$-morphisms, we need a functorial version of the twisting procedure. Such a result holds on the level of $\infty$-morphisms. 

\begin{proposition}\label{prop:TwiFuncInfMor}
Let $f : \g \rightsquigarrow \h$ be an $\infty$-morphism between complete $\sLi$-algebras and let $\alpha\in\MC(\g)$ be a Maurer--Cartan element. 
The formula 
\[f^\alpha_m \coloneqq \sum_{k\geqslant 0} {\textstyle \frac{1}{k!}} f_{k+m}\big(\alpha^k, -,  \ldots,  -   \big) \]
defines an $\infty$-morphism 
\[f^\alpha : \g^\alpha \rightsquigarrow \h^{f(\alpha)}\]
between the associated twisted $\sLi$-algebras. 
This assignment is functorial: 
we have 
\[\id^\alpha=\id : \g^\alpha \to \g^\alpha\]
and 
\[(g\circ f)^{\alpha}=g^{f(\alpha)}\circ f^{\alpha} : \g^\alpha \to \mathfrak{k}^{g(f(\alpha))}
~,\]
for any $\infty$-morphism $g : \h \rightsquigarrow \mathfrak{k}$~.
\end{proposition}

\begin{proof}
Again, this follows from a straightforward computation. 
A more conceptual proof can be found in \cite[Chapter~3, Section~5]{DotsenkoShadrinVallette18}. 
\end{proof}

We failed to establish a similar functorial twisting procedure with respect to $\infty_\pi$-morphisms on the level of $\sLi$-algebras. Instead, we will settle it ``only'' on the level of the integration functor using the following arguments. 

\medskip

Recall that for any simplicial set we can see the set of $0$-simplices as a simplicial sub-set. In the present case, for any Maurer--Cartan element 
$\alpha \in \MC(\g)$~, 
this amounts to considering, for any $n\geqslant 0$~, the element $\alpha_n\in \hom(\rmC^n, \g)$ that sends $a_i$ to $\alpha$~, for each $0\leqslant i\leqslant n$~, and everything else to zero. 
In other words, under the geometric representation of \cref{subsec:GeoRep}, the element $\alpha_n$ is obtained by labeling the vertices of the 
geometric $n$-simplex by $\alpha$ and the other faces by $0$~. 
We denote this ``simplicial'' collections by $\alpha_\bullet\in\R(\g)$~. 

\begin{lemma}\label{prop:convolution algebra and twisting}
	Let $\g$ be a complete $\sLi$-algebra and let $\alpha\in\MC(\g)$ be a Maurer--Cartan element. 
	The two cosimplicial complete $\sLi$-algebras 
	\[
	\hom^\pi(\rmC^\bullet, \g^\alpha) = \hom^\pi(\rmC^\bullet, \g)^{\alpha_\bullet}
	\]
	are identical. 
\end{lemma}

\begin{proof}
	As we already noticed in the proof of \cref{prop:Simpli}, the cosimplicial structure of $\rmC_\bullet$ is compatible with Dupont's contraction, so that the homotopy transfer theorem endows it with a simplicial $\Cobar\Bar\com$-algebra structure, and its dual $\rmC^\bullet$ obtains a cosimplicial $\Bar\Cobar\com^\vee$-coalgebra structure. It is straightforward to see that the convolution algebra is again compatible with the structure maps, so that $\hom^\pi(\rmC^\bullet, \g)$ is a simplicial complete $\sLi$-algebra. Since each element $\alpha_n\in\hom^\pi(\rmC^n, \g)$ is the image of the Maurer--Cartan element $\alpha^0\in \hom^\pi(\rmC^0, \g)$ under degeneracy maps, it is again a Maurer--Cartan element in $\hom^\pi(\rmC^n, \g)$~. Therefore, we can twist the complete $\sLi$-algebra with it to obtain $\hom^\pi(\rmC^n, \g)^{\alpha_n}$ using \cref{prop:TwiFuncInfMor} and we obtain a simplicial complete $\sLi$-algebra $\hom^\pi(\rmC^\bullet, \g)^{\alpha_\bullet}$ when applying this procedure at every simplicial level.

	\medskip

	We now fix a simplicial degree $n\geqslant 0$ and we show that the twisted $\sLi$-algebra structure 
	$\hom^\pi(\rmC^n, \g)^{\alpha_n}$ is equal to $\hom^\pi(\rmC^n, \g^\alpha)$~. 
	We will work in the dual picture
	\[
	\R(\g)_n\cong\MC\left(\g\, \widehat{\otimes}^\pi\rmC_n\right).
	\] 
	In the right-hand side, our Maurer--Cartan element is given by $\alpha\otimes 1$~, with $1=\omega_0+\cdots+\omega_n\in\rmC_n$~. Recall from \cite{rn17tensor} that the $\sLi$-algebra structure on $\g\, \widehat{\otimes}^\pi \rmC_n$ is given as follows. The set of partitioned rooted trees $\tau \in \PaRT$ forms a basis of the operad $\Cobar\Bar\com\cong(\Bar\sLi)^\vee$. 
	We denote by $\tau^\vee$ the dual basis of the cooperad $\Bar \Cobar \com^\vee\cong (\Cobar\Bar\com)^\vee$~.
	Given $x_1 \otimes \theta_1,\ldots,x_m \otimes \theta_m\in \g\, \widehat{\otimes}^\pi\rmC_n$~, we have
	\begin{align*}
		\ell_m(x_1 \otimes \theta_1,\ldots,{}&x_m \otimes \theta_m) =\\
		={}&\sum_{\tau\in\PaRT_m}(-1)^\epsilon\, \gamma_{\g}\left(\pi(\tau^\vee)\otimes x_1\otimes\cdots\otimes x_m\right)
		\otimes 
		\gamma_{C_n}\left(\tau\otimes\theta_1\otimes\cdots\otimes\theta_m\right),
	\end{align*}
	where $\epsilon$ is the Koszul sign. Since the twisting morphism $\pi$ projects to the trees containing only a single partition, this expression simplifies to
	\[
	\ell_m(x_1 \otimes \theta_1,\ldots,x_m \otimes \theta_m) = 
	\sum_{\tau\in\RT_m}(-1)^\epsilon\, \gamma_{\g}\left(\tau^\vee\otimes x_1\otimes\cdots\otimes x_m\right)
	\otimes 
	\gamma_{C_n}(\tau\otimes\theta_1\otimes\cdots\otimes\theta_m)\ ,
	\]
	where $\tau \in \RT(m)$ is viewed as a basis element of $\Bar \com(m)$ and $\tau^\vee$ as a basis element of $\Omega \com^\vee(m)$~. So the twisted structure is equal to 
	\begin{align*}
		{}& \ell^{\alpha\otimes 1}_m(x_1 \otimes \theta_1,\ldots,x_m \otimes \theta_m) =
		\sum_{k\geqslant 0}{\textstyle \frac{1}{k!}}\ell_{k+m}(\alpha\otimes 1,\ldots,\alpha\otimes 1,
		x_1 \otimes \theta_1,\ldots,x_m\otimes \theta_m)\\
		={}&\sum_{k\geqslant 0}{\textstyle \frac{1}{k!}}\sum_{\tau\in\RT_{k+m}}(-1)^\epsilon\, 
		\gamma_{\g}\big(\tau^\vee\otimes\alpha^{\otimes k}\otimes x_1\otimes\cdots\otimes x_m\big)
		\otimes
		\gamma_{C_n}\big(\tau\otimes1^{\otimes k}\otimes\theta_1\otimes\cdots\otimes\theta_m\big)\ .
	\end{align*}
	Let us consider the term $\gamma_{C_n}\big(\tau\otimes1^{\otimes k}\otimes\theta_1\otimes\cdots\otimes\theta_m\big)$~. As explained in the proof of \cref{prop:mc0}, it is given by taking the rooted tree $\tau$~, decorating each leaf with the corresponding element (either a copy of $1$ or one of the $\theta_i$), applying $i_n$ to these elements, multiplying them at the vertices, applying the homotopy $h_n$ at the inner edges, and applying the projection chain map $p_n$ at the root. Due to the application of $h_n$~, it is immediately clear that if there is a vertex of $\tau$ that is linked only to leaves and such that $1$ is at each of these leaves, then
	\[
	\gamma_{C_n}\big(\tau\otimes1^{\otimes k}\otimes\theta_1\otimes\cdots\otimes\theta_m\big) = 0\ .
	\]
	Other trees as well give a trivial contribution because of the side conditions of the Dupont contraction: if a tree has a vertex linked only to leaves and all of the leaves are labeled by $1$ except for a single one labeled by an element $\theta_i$~, or if it has a vertex that is linked only to leaves labeled by $1$ and an incoming edge linked to a sub-tree, then the conditions $h_ni_n=0$ and $h_nh_n=0$ make it so that the result of the tree is zero, respectively.

	\medskip

	Excluding these trees from $\RT_{m+k}$~, we are left with the set of trees that can be obtained by taking a tree in $\RT_m$~, shifting the enumeration of the leaves by $k$~, adding $k$ leaves at existing vertices and choosing an enumeration for $1$ to $k$ for these new leaves, and then placing $k$ copies of $1$ and the elements $\theta_1,\ldots,\theta_m$ at the leaves, in order. 
	If $V_\tau$ denotes the set of vertices of the original tree $\tau$ in $\RT_m$~, writing $k_v$ for the number of leaves added to the vertex $v\in V_\tau$~, we can forget the precise ordering of the added leaves obtaining a global multinomial coefficient
	\[
	\binom{k}{\{k_v\}_{v\in V_\tau}} = \frac{k!}{\prod_{v\in V_\tau}k_v!}
	\]
	for the tree. We write $\tau + \{k_v\}$ for the rooted tree $\tau$ with $k_v$ leaves added to the vertex $v\in V_\tau$~. Since $1$ is the unit of the Sullivan algebra, we have that
	\[
	\gamma_{C_n}\left((\tau + \{k_v\})\otimes1^{\otimes k}\otimes\theta_1\otimes\cdots\otimes\theta_m\right) = \gamma_{C_n}(\tau\otimes\theta_1\otimes\cdots\otimes\theta_m)\ .
	\]
	Putting all we have said until now together, we obtain that
	\begin{align*}
		\ell^{1\otimes\alpha}_m(x_1 \otimes \theta_1&,\ldots,x_m \otimes \theta_m) =\\
		={}&\sum_{\tau \in\RT_m}\sum_{\substack{v\in V_\tau \\k_v\geqslant 0}}\frac{(-1)^\epsilon}{\prod_{v\in V_\tau}k_v!}
		\gamma_{\g}\left((\tau + \{k_v\})^\vee\otimes\alpha^{\otimes k}\otimes x_1\otimes\cdots\otimes x_m\right)
		\otimes\\
		&\qquad\qquad\otimes
				\gamma_{C_n}\left((\tau + \{k_v\})\otimes1^{k}\otimes\theta_1\otimes\cdots\otimes\theta_m\right)
		\\
		={}&\sum_{\tau\in\RT_m}(-1)^\epsilon\left(\sum_{\substack{v\in V_\tau\\k_v\geqslant0}}\frac{1}{\prod_{v\in V_\tau}k_v!}\gamma_{\g}\left((t + \{k_v\})^\vee\otimes\alpha^{\otimes k}\otimes x_1\otimes\cdots\otimes x_m\right)\right)
		\otimes\\
		&\qquad\qquad\otimes
		\gamma_{C_n}(\tau\otimes\theta_1\otimes\cdots\otimes\theta_m)
		\\
		={}&\sum_{\tau\in\RT_m}(-1)^\epsilon
		\gamma_{\g^\alpha}\left(\tau^\vee\otimes x_1\otimes\cdots\otimes x_m\right)
		\otimes
		\gamma_{C_n}(\tau\otimes\theta_1\otimes\cdots\otimes\theta_m)\\
		={}& 	\ell_m(\theta_1\otimes x_1,\ldots,\theta_m\otimes x_m)
		\ ,
	\end{align*}
	where the last operation $\ell_m$ is the one of the $\sLi$-algebra $\hom^\pi(\rmC^n, \g^\alpha)$~. This is exactly what we needed to show and concludes the proof.
\end{proof}

Thanks to \cref{prop:convolution algebra and twisting}, we can write the following definition. 

\begin{definition}\label{def:TwRInfPi}
Let $\g, \h$ be two complete $\sLi$-algebras and let $f:\g\rightsquigarrow\h$ be an $\infty_\pi$-morphism between them. For any Maurer--Cartan element $\alpha\in\MC(\g)$ we define a morphism of simplicial sets
\[
\R(f)^\alpha\colon\R(\g^\alpha)\longrightarrow\R\left(\h^{\upsilon^*f(\alpha)}\right)
\]
by
\[
\R(f)^\alpha \coloneqq\MC\left(\hom^\pi_r(\id, f)^{\alpha_\bullet}\right).
\]
\end{definition}

\begin{proposition}\label{prop:FuncTwInfPi}
	The twisting procedure for the integration functor with respect to $\infty_\pi$-mor\-phi\-sms of \cref{def:TwRInfPi} is functorial. Namely, for any $\alpha\in \MC(\g)$ we have
	\[
	\R(\id_\g)^\alpha=\id_{\R(\g^\alpha)}  : \R(\g^\alpha) \to \R(\g^\alpha)
	\]
	and for any pair $f : \g \rightsquigarrow \h$ and $g : \h \rightsquigarrow \mathfrak{k}$ of  $\infty_\pi$-morphisms of complete $\sLi$-algebras
	\[
	\R(g\circ f)^{\alpha}=\R(g)^{\upsilon^*f(\alpha)}\R(f)^{\alpha}\ .
	\]
\end{proposition}

\begin{proof}
This is straightforward consequence of \cref{def:TwRInfPi} and \cref{prop:TwiFuncInfMor}.
\end{proof}

From \cref{lemma:MC in twisted Loo alg} and \cref{prop:convolution algebra and twisting}, we see that a simplex $x:\Delta^n\to\hom^\pi(\rmC^\bullet,\g)$ is in $\R(\g^\alpha)$ if and only if the simplex $x + \alpha_n$~,  where the vertices $\{x_i\}_{0\leqslant i\leqslant n}$ are replaced by $\{x_i + \alpha\}_{0\leqslant i\leqslant n}$ and where all the rest remains unchanged, is in $\R(\g)$~. This is exactly what we need in order to change base points when considering the homotopy groups of $\R(\g)$~.

\begin{theorem}\label{thm:isom translation by alpha}
	Let $\g$ be a complete $\sLi$-algebra and let $\alpha\in\MC(\g)$ be a Maurer--Cartan element. The translation maps
	\begin{align*}
	\hom(\rmC^n,\g^{\alpha_n}){}&\ \longrightarrow\ \hom(\rmC^n,\g)\\
	x{}&\ \longmapsto\ \left\{
	\begin{array}{ll}
	x+\alpha_n &, \ \text{when} \ |x|=0~,  \\
	x &,\ \text{when} \ |x|\neq 0~,  \\
	\end{array}
	\right.
	\end{align*}
	induce an isomorphism of simplicial sets
	\[
	\R(\g^\alpha)\cong\R(\g)~,
	\]
	which is natural in $\infty_\pi$-morphisms.
\end{theorem}

\begin{proof}
The translation map is actually a curved $\infty$-isotopy after \cite[Chapter~3, Section~5]{DotsenkoShadrinVallette18}; as such it induces a bijection on the level Maurer--Cartan simplicial sets. 
The functoriality follows from \cref{prop:convolution algebra and twisting}, \cref{prop:FuncTwInfPi}, and the interpretation of (curved) $\infty$-morphisms as gauge group elements given in \emph{loc.\ cit.}
\end{proof}

\begin{corollary}\label{cor:change base point}
	Let $\g$ be a complete $\sLi$-algebra, let $\alpha\in\MC(\g)$ be a Maurer--Cartan element. There is an isomorphism of sets
	\[
	\pi_0(\R(\g^\alpha))\cong\pi_0(\R(\g))
	\]
	and isomorphisms of groups
	\[
	\pi_n(\R(\g^\alpha), 0)	\cong\pi_n(\R(\g), \alpha)
	\]
	for each $n\geqslant1$~. They are all natural in $\infty_\pi$-morphisms.
\end{corollary}

\begin{proof}
	This is immediate from \cref{thm:isom translation by alpha}.
\end{proof}
\section{Gauges}\label{sec:gauges}

After studying some general properties of the universal Maurer--Cartan algebra and completely determining the $0$th cosimplicial level $\mc^0$~, we now turn to $\mc^1$ and show that it is the complete $\sLi$-algebra representing two Maurer--Cartan elements and a gauge between them. This generalizes the Lawrence--Sullivan Lie algebra model of the interval \cite{ls10}. As we did for $\mc^0$, we will give completely explicit formulas for $\mc^1$. We then give an application by completely characterizing the notion of a gauge homotopy between $\infty$-morphisms of $\sLi$-algebras.

\subsection{The higher Lawrence--Sullivan algebra}\label{subsec:HighLSalg}
In \cref{prop:mc0}, we computed $\mc^0$, which is the free complete $\sLi$-algebra generated by a single Maurer--Cartan element. This result can easily be recovered now using  \cref{thm:isomorphism of models} and the fact that $\rmC_0\cong\k$ since they give
\[
\R(\g)_0\cong \MC(\g\widehat{\otimes} \rmC_0)\cong \MC(\g)\ .
\]

We would like to understand the next simplicial level, that is $\mc^1$. 

\begin{proposition}\label{lem:mc1}
	The complete $\sLi$-algebra $\mc^1$ is the free complete $\sLi$-algebra on three generators $a_0$~, $a_1$~, and $a_{01}$ of respective degree $|a_0| = |a_1| = 0$ and $|a_{01}| = 1$~, with the unique differential  such that $a_0$ and $a_1$ are Maurer--Cartan elements, and $a_{01}$ is a gauge between them.
\end{proposition}

\begin{proof}
	From \cref{def:MCcosimp}, we know that the underling complete $\sLi$-algebra of $\mc^1$ is free on three generators 
	$a_0$~, $a_1$~, and $a_{01}$ with degrees as above. The cosimplicial structure of $\mc^\bullet$ 
	and \cref{prop:mc0} show that 
	$a_0$ and $a_1$ are Maurer--Cartan elements in $\mc^1$. 
	\cref{thm:isomorphism of models} shows  that $\R(\g)_1\cong\MC\big(\g\, \widehat{\otimes} \mathrm{C}_1\big)$~. 
	This latter set, once viewed in $\MC_1(\g)$~, was shown by \cite[Proposition~4.3]{rn17cosimplicial}  to be made up of pairs  Maurer--Cartan elements and a gauge between them.
	This proves that $\R(\g)_1$ is made up of the same elements and forces the image of $a_{01}$ under the differential to be such that $a_{01}$ is a gauge from $a_0$ to $a_1$~.
\end{proof}

\begin{corollary}\label{cor:pi0}
	There is a canonical bijection 
	\[\pi_0(\R(\g))\cong \mathcal{MC}(\g)\ ,\]
	which is natural in complete $\sLi$-algebras with respect to $\infty_\pi$-morphisms.
\end{corollary}

\begin{proof}
	The existence of a canonical bijection is a direct consequence of \cref{prop:mc0} and \cref{lem:mc1}; one could also prove it directly with \cref{thm:isomorphism of models}. 
	\cref{lemma:describe R(f)_0} shows that the bijection $\R(\g)_0\cong \MC(\g)$ is functorial with respect to $\infty_\pi$-morphisms: for any $\infty_\pi$-morphism $f : \g \rightsquigarrow \h$~, the following diagram is commutative
	\[
	\vcenter{\hbox{
		\begin{tikzpicture}
			\node (a) at (0,0) {$\R(\g)_0$};
			\node (b) at (3,0) {$\MC(\g)$};
			\node (c) at (0,-1.5) {$\R(\h)_0$};
			\node (d) at (3,-1.5) {$\MC(\h)$};
			
			\draw[->] (a) to node[above=-0.05]{$\scriptstyle{\cong}$} (b);
			\draw[->] (a) to node[left]{$\scriptstyle{\R(f)_0}$} (c);
			\draw[->] (b) to node[right]{$\scriptstyle{\MC(\upsilon^*f)}$} (d);
			\draw[->] (c) to node[above=-0.05]{$\scriptstyle{\cong}$} (d);
		\end{tikzpicture}
	}}
	\]
	Since $R(f)$ is a morphism of simplicial sets, the left-hand side passes to $\pi_0$~. We claim that the right-hand side also passes to the moduli space of Maurer--Cartan elements up gauge equivalence. Indeed, Deligne--Hinich's space 
	$\MC_\bullet$ is functorial with respect to $\infty$-morphisms, so any $\infty$-morphism from $\g$ to $\h$ sends two Maurer--Cartan elements which live in the same connected component of $\MC_\bullet(\g)$ to the same connected component of $\MC_\bullet(\h)$~. \cref{thm:isomorphism of models} asserts that $\MC_\bullet(\g)\simeq\MC(\g \widehat{\otimes} \mathrm{C}_\bullet)\cong\R(\g)$ and we know by \cite[Proposition~4.3]{rn17cosimplicial} that 
	$\MC\big(\g\, \widehat{\otimes} \mathrm{C}_1\big)$ is given by pairs of Maurer--Cartan element and a  gauge between them. Thus, if two Maurer--Cartan element are gauge equivalent then they live in the same connected component of $\MC_\bullet(\g)$~, so their image live in the same connected component of $\MC_\bullet(\h)$~, which means that their images are gauge equivalent. 
\end{proof}

We are left to make explicit the differential $\d(a_{01})$~, which requires to introduce the following combinatorial objects. We consider the set of \emph{partitioned planar rooted trees} $\PaPT$\index{trees!partitioned planar rooted}\index{$\PaPT$} made up of \emph{planar rooted trees}, with vertices of arity greater or equal to 1, endowed with a complete \emph{partition} into sub-trees  called \emph{blocks}, and where the leaves are labeled by two colors:  black or white. Each partition is required to contain at least one block and each 
block is required to contain at least one vertex. We will also consider the trees $\whiteleaf$ and $\blackleaf$ given by a single white or respectively black leaf, which strictly speaking do not belong to $\PaPT$~. 

\begin{example}
	\[
	\tau\coloneqq\vcenter{\hbox{
			\begin{tikzpicture}
			\def\scale{1};
			\pgfmathsetmacro{\diagcm}{sqrt(2)};
			
			\coordinate (r1) at (0, 0); 
			\coordinate (v11) at ($(r1) + (90:\scale*0.7)$);
			\coordinate (v12) at ($(v11) + (135:\scale*1*\diagcm)$);
			\coordinate (v13) at ($(v12) + (135:\scale*0.5*\diagcm)$);
			\coordinate (v14) at ($(v11) + (45:\scale*1*\diagcm)$);
			\coordinate (l11) at ($(v13) + (135:\scale*0.5*\diagcm)$);
			\coordinate (l12) at ($(v13) + (45:\scale*0.5*\diagcm)$);
			\coordinate (l13) at ($(v12) + (45:\scale*0.5*\diagcm)$);
			\coordinate (l14) at ($(v11) + (90:\scale*0.5)$);
			\coordinate (l15) at ($(v14) + (90:\scale*0.5)$);

			\draw[thick] (r1) to (l14);
			\draw[thick] (v11) to ($(l11) + (135:0.05)$);
			\draw[thick] (v13) to ($(l12) + (45:0.05)$);	
			\draw[thick] (v12) to ($(l13) + (45:0.05)$);	
			\draw[thick] (v11) to (v14);
			\draw[thick] (v14) to (l15);
			
			\fill (v11) circle [radius=2.5pt];
			\fill (v12) circle [radius=2pt];
			\fill (v13) circle [radius=2pt];
			\fill (v14) circle [radius=2pt];
			
			\fill[white] (l11)+(-2pt,0pt) rectangle +(2pt,4pt);
			\draw (l11)+(-2pt,0pt) rectangle +(2pt,4pt);
			\fill (l12)+(-2pt,0pt) rectangle +(2pt,4pt);
			\fill (l13)+(-2pt,0pt) rectangle +(2pt,4pt);
			\draw (l14)+(-2pt,0pt) rectangle +(2pt,4pt);
			\draw (l15)+(-2pt,0pt) rectangle +(2pt,4pt);

			\draw (v11) circle [radius=\scale*8pt];
			\draw ($(v12)!0.5!(v13)$) circle [radius=\scale*15pt];
			\draw (v14) circle [radius=\scale*8pt];
			
			\end{tikzpicture}
	}}\in \PaPT\ .
	\]
\end{example}

In the same way as in \cref{subsec:sLooalg}, for any $\tau \in \PaPT$~, $\alpha, \beta, \lambda\in \g$~, 
we define 
$\tau^\lambda(\alpha, \beta)\in \g$
by induction:
\[
\whiteleaf^\lambda(\alpha,\beta)\coloneqq \alpha\ , \quad 
\blackleaf^\lambda(\alpha,\beta)\coloneqq \beta\ , \quad \text{and} \quad
\tau^\lambda(\alpha, \beta) \coloneqq \ell_{m+1}\left(\tau_1^\lambda(\alpha, \beta),\ldots, \tau_m^\lambda(\alpha, \beta), \lambda \right)
\]
when the underlying planar rooted tree of $\tau$~, that is after forgetting the partition, is equal to 
$\tau=c_m\circ(\tau_1, \ldots, \tau_m)$~.

\begin{example}
	In the example of the above partitioned planar rooted tree $\tau$~, we have
	\[
	\tau^\lambda(\alpha, \beta)=\vcenter{\hbox{
			\begin{tikzpicture}
			\def\scale{1};
			\pgfmathsetmacro{\diagcm}{sqrt(2)};
			\pgfmathsetmacro{\diagcmbis}{1/sin(60)};
			
			\coordinate (r1) at (0, 0); 
			\coordinate (v11) at ($(r1) + (90:\scale*0.7)$);
			\coordinate (v12) at ($(v11) + (135:\scale*1*\diagcm)$);
			\coordinate (v13) at ($(v12) + (135:\scale*0.5*\diagcm)$);
			\coordinate (v14) at ($(v11) + (45:\scale*1*\diagcm)$);
			\coordinate (l11) at ($(v13) + (120:\scale*0.5*\diagcmbis)$);
			\coordinate (l12) at ($(v13) + (60:\scale*0.5*\diagcmbis)$);
			\coordinate (l13) at ($(v12) + (60:\scale*0.5*\diagcmbis)$);
			\coordinate (l14) at ($(v11) + (90:\scale*0.5)$);
			\coordinate (l15) at ($(v14) + (90:\scale*0.5)$);
			
			\draw[very thick] (r1) to (l14);
			\draw[very thick] (v11) to (v13);
			\draw[very thick] (v13) to ($(l11) + (135:0.05)$);
			\draw[very thick] (v13) to ($(l12) + (45:0.05)$);	
			\draw[very thick] (v12) to ($(l13) + (45:0.05)$);	
			\draw[very thick] (v11) to (v14);
			\draw[very thick] (v14) to (l15);
			
			\fill (v11) circle [radius=2.5pt];
			\fill (v12) circle [radius=2pt];
			\fill (v13) circle [radius=2pt];
			\fill (v14) circle [radius=2pt];
			
			\def\angle{30}
			\pgfmathsetmacro{\xcm}{1/sin(\angle)};
			\coordinate (llam1) at ($(v13) + (\angle:\scale*0.5*\xcm)$);
			\coordinate (llam2) at ($(v12) + (\angle:\scale*0.5*\xcm)$);
			
			\pgfmathsetmacro{\ycm}{1/sin(25)};
			\coordinate (llam3) at ($(v11) + (25:\scale*0.5*\ycm)$);
			
			\coordinate (llam4) at ($(v14) + (45:\scale*0.5*\diagcm)$);
			
			\draw[thin] (v13) to (llam1);
			\draw[thin] (v12) to (llam2);
			\draw[thin] (v11) to (llam3);
			\draw[thin] (v14) to (llam4);
			
			\draw ($(l11)-(0, 0.05)$) node[above] {$\scriptstyle{\alpha}$};
			\draw ($(l12)-(0, 0.05)$) node[above] {$\scriptstyle{\beta}$};
			\draw ($(l13)-(0, 0.05)$) node[above] {$\scriptstyle{\beta}$};		
			\draw ($(l14)-(0, 0.05)$) node[above] {$\scriptstyle{\alpha}$};
			\draw ($(l15)-(0, 0.05)$) node[above] {$\scriptstyle{\alpha}$};	
			
			\draw (v11) node[left] {$\scriptstyle{\ell_4}$};	
			\draw (v12) node[left] {$\scriptstyle{\ell_3}$};	
			\draw (v13) node[left] {$\scriptstyle{\ell_3}$};	
			\draw (v14) node[left] {$\scriptstyle{\ell_2}$};
			
			\draw ($(llam1)-(0, 0.05)$) node[above] {$\scriptstyle{\lambda}$};		
			\draw ($(llam2)-(0, 0.05)$) node[above] {$\scriptstyle{\lambda}$};	
			\draw ($(llam3)-(0, 0.05)$) node[above] {$\scriptstyle{\lambda}$};	
			\draw ($(llam4)-(0, 0.05)$) node[above] {$\scriptstyle{\lambda}$};
			
			\end{tikzpicture}
	}}=\ell_4(\ell_3(\ell_3(\alpha, \beta, \lambda),\beta ,\lambda), \alpha, \ell_2(\alpha, \lambda), \lambda)\ .
	\]
\end{example}

Given a partitioned planar rooted tree $\tau \in \PaPT$~, we denote by $\tau^1, \ldots, \tau^k$ the planar rooted sub-trees corresponding to its $k$ blocks. 
To each sub-tree $\tau^i$ we associate the planar rooted tree ${\tilde\tau}^i$ where all the leaves connected to another block or to a black leaf are substituted by a vertex of arity $0$~. 

\begin{example}
	Under the convention
	\[
	\tau=\vcenter
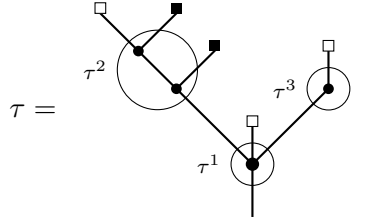
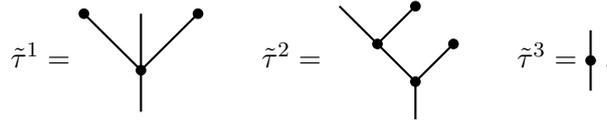
{\hbox{
			\begin{tikzpicture}
			\def\scale{1};
			\pgfmathsetmacro{\diagcm}{sqrt(2)};
			
			\coordinate (r1) at (0, 0); 
			\coordinate (v11) at ($(r1) + (90:\scale*0.7)$);
			\coordinate (v12) at ($(v11) + (135:\scale*1*\diagcm)$);
			\coordinate (v13) at ($(v12) + (135:\scale*0.5*\diagcm)$);
			\coordinate (v14) at ($(v11) + (45:\scale*1*\diagcm)$);
			\coordinate (l11) at ($(v13) + (135:\scale*0.5*\diagcm)$);
			\coordinate (l12) at ($(v13) + (45:\scale*0.5*\diagcm)$);
			\coordinate (l13) at ($(v12) + (45:\scale*0.5*\diagcm)$);
			\coordinate (l14) at ($(v11) + (90:\scale*0.5)$);
			\coordinate (l15) at ($(v14) + (90:\scale*0.5)$);
			
			\draw[thick] (r1) to (l14);
			\draw[thick] (v11) to ($(l11) + (135:0.05)$);
			\draw[thick] (v13) to ($(l12) + (45:0.05)$);	
			\draw[thick] (v12) to ($(l13) + (45:0.05)$);	
			\draw[thick] (v11) to (v14);
			\draw[thick] (v14) to (l15);
			
			\fill (v11) circle [radius=2.5pt];
			\fill (v12) circle [radius=2pt];
			\fill (v13) circle [radius=2pt];
			\fill (v14) circle [radius=2pt];
			
			\fill[white] (l11)+(-2pt,0pt) rectangle +(2pt,4pt);
			\draw (l11)+(-2pt,0pt) rectangle +(2pt,4pt);
			\fill (l12)+(-2pt,0pt) rectangle +(2pt,4pt);
			\fill (l13)+(-2pt,0pt) rectangle +(2pt,4pt);
			\draw (l14)+(-2pt,0pt) rectangle +(2pt,4pt);
			\draw (l15)+(-2pt,0pt) rectangle +(2pt,4pt);
			
			\draw (v11) circle [radius=\scale*8pt];
			\draw ($(v12)!0.5!(v13)$) circle [radius=\scale*15pt];
			\draw (v14) circle [radius=\scale*8pt];
			
			\node[left] at ($(v11) + \scale*(-8pt, 0)$) {$\scriptstyle{\tau^1}$};
			\node[left] at ($(v12)!0.5!(v13) + \scale*(-15pt, 0)$) {$\scriptstyle{\tau^2}$};
			\node[left] at ($(v14) + \scale*(-8pt, 0)$) {$\scriptstyle{\tau^3}$};
			
			\end{tikzpicture}
	}}\ ,
	\]
	we get 
	\[\tilde{\tau}^1=
	\vcenter{\hbox{\begin{tikzpicture}
			\coordinate (r2) at (0,0.2);
			\coordinate (v21) at (0,.75);
			\coordinate (l21) at (-.75,1.5);
			\coordinate (l22) at (0,1.5);
			\coordinate (l23) at (.75,1.5);
			
			\draw[thick] (r2) to (l22);
			\draw[thick] (v21) to (l21);
			\draw[thick] (v21) to (l23);
			\fill (v21) circle [radius=2pt];
			\fill (l21) circle [radius=2pt];
			\fill (l23) circle [radius=2pt];
			\end{tikzpicture}}}
	\qquad 
	\tilde{\tau}^2=\vcenter{\hbox{
			\begin{tikzpicture}
			\coordinate (r1) at (-1,1.5); 
			\coordinate (v11) at (-1,2);
			\coordinate (v12) at (-1,2);
			\coordinate (v13) at (-1.5,2.5);
			\coordinate (v14) at (1,2);
			\coordinate (l11) at (-2,3);
			\coordinate (l12) at (-1,3);
			\coordinate (l13) at (-0.5,2.5);
			\coordinate (l14) at (0,1.5);
			\coordinate (l15) at (1,2.5);
			
			\draw[thick] (v11) to (l11);
			\draw[thick] (v13) to (l12);
			\draw[thick] (v12) to (l13);	
			\draw[thick] (v11) to (r1);
			
			\fill (v11) circle [radius=2pt];	
			\fill (v13) circle [radius=2pt];	
			\fill (l12) circle [radius=2pt];	
			\fill (l13) circle [radius=2pt];				
			\end{tikzpicture}}}
	\qquad 
	\tilde{\tau}^3=
	\vcenter{\hbox{\begin{tikzpicture}
			\coordinate (r3) at (0,0.2);
			\coordinate (v31) at (0,0.6);
			\coordinate (l31) at (0,1);
			
			\draw[thick] (r3) to (l31);
			\fill (v31) circle [radius=2pt];
			\end{tikzpicture}}}\ .\]
\end{example}

We associate a coefficient $D(\tau)$ to $\tau$ by
\[
D(\tau)\coloneqq (-1)^k\prod_{i=1}^k C\left(\tilde{\tau}^i\right),
\]
where the $C\left(\tilde{\tau}^i\right)$ is the one defined in \cref{subsec:sLooalg}.

\begin{proposition}\label{prop:Diffmc1}
	The differential in the complete $\sLi$-algebra $\mc^1$ is given by 
	\[
	\d(a_{01})= a_1 - a_0 + \sum_{\tau \in\PaPT}\tfrac{1}{D(\tau)}\,\tau^{a_{01}}(a_0,a_1-a_0)\ .
	\]
\end{proposition}

\begin{proof}
	Recall from \cref{prop:gaugeformula} that 
	\[ a_1=	 \sum_{\tau\in \PT} {\textstyle \frac{1}{C(\tau)}}\tau^{a_{01}}(a_0)
	\ \]
	in $\mc^1$~. Here one should be careful, as we are using the notation $\tau^{a_{01}}$ both for a function defined by trees in $\PaPT$ ($2$ arguments) and for trees in $\PT$ ($1$ argument).
	Rearranging the terms, we obtain 
	\begin{equation}\label{eq=fixedpteq}\tag{$\ast$}
	\d(a_{01}) = a_1 - a_0 - \sum_{\tau\in\overline{\PT}} {\textstyle \frac{1}{C(\tau)}}\tau^{a_{01}}(a_0)
	- \sum_{\tau\in\widetilde{\PT}} {\textstyle \frac{1}{C(\tau)}}\tau^{a_{01}}(a_0)
	\ ,
	\end{equation}
	where $\overline{\PT}$ stands for the sub-set of planar rooted trees with at least one vertex and such that each vertex has positive arity and where 
	$\widetilde{\PT}$ stands for the sub-set of planar rooted trees with  
	at least two vertices and such that at least 
	one of them has arity $0$~. Under the convention 
	\begin{align*}
	p_0\coloneqq{}& a_1 - a_0 - \sum_{\tau\in\overline{\PT}} {\textstyle \frac{1}{C(\tau)}}\tau^{a_{01}}(a_0)\qquad\text{and}\\ 
	\PP_m\left(\d\left(a_{01}\right)^{\otimes m}\right)\coloneqq{}&  
	- \sum_{\tau\in\widetilde{\PT}^{\langle m\rangle}} {\textstyle \frac{1}{C(\tau)}}\tau^{a_{01}}(a_0)\ , \ \text{for} \ m\geqslant 1\ ,
	\end{align*}
	where $\widetilde{\PT}^{\langle m \rangle}$ is the sub-set of $\widetilde{\PT}$ made up of planar rooted trees with exactly $m$ vertices of arity $0$~, Equation~\eqref{eq=fixedpteq} is a fixed-point equation satisfied by $\d(a_{01})$~, which appears in the right-hand side at the vertices of arity $0$~. \cref{prop:FixPtEqua}, which guarantees the existence and uniqueness of a solution to \eqref{eq=fixedpteq} and provides us with an explicit formula, which becomes the one of the statement after a straightforward computation. 
\end{proof}

From this result, one recovers the Lawrence--Sullivan algebra $\overline{\mc}^1$ as follows. Under the convention $\overline{\d}$ for the differential of $\overline{\mc}^1$, the two differentials are related by the formula 
\[\overline{\d}(a_{01})=\left(\rho\otimes \id_{C^1}\right)(\d(a_{01}))\ , \]
where $\rho : \sLi \to  \sLie$ is the morphism of operads defined in \cref{subsec:NAbDKadj}. Since this morphism kills all the generators of $\sLi$ of arity greater or equal to $3$ and preserves the binary generators, all the elements of 
$\d(a_{01})$ get killed in 
$\left(\rho\otimes \id_{C^1}\right)(\d(a_{01}))$ except for the ladders, that is the linear trees of shape 
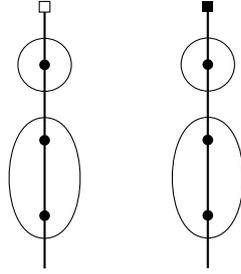
\[\vcenter{\hbox{	
		\begin{tikzpicture}
		\coordinate (r1) at (0,0.3); 
		\coordinate (v1) at (0,1);
		\coordinate (v2) at (0,2);	
		\coordinate (v3) at (0,3);
		\coordinate (v4) at (0,3.7);			
		
		\draw[thick] (r1) to (v4);
		
		\fill (v1) circle [radius=2pt];
		\fill (v2) circle [radius=2pt];
		\fill (v3) circle [radius=2pt];
		\draw (v4)+(-2pt,0pt) rectangle +(2pt,4pt);
		
		\draw (v3) circle [radius=10pt];
		
		\draw ($(v1)-(0,0.3)$) to [out=0,in=0] ($(v2)+(0,0.3)$)  to [out=180,in=180] ($(v1)-(0,0.3)$) ;	
		\end{tikzpicture}}}\qquad 
\vcenter{\hbox{	
		\begin{tikzpicture}
		\coordinate (r1) at (0,0.3); 
		\coordinate (v1) at (0,1);
		\coordinate (v2) at (0,2);	
		\coordinate (v3) at (0,3);
		\coordinate (v4) at (0,3.7);			
		
		\draw[thick] (r1) to (v4);
		
		\fill (v1) circle [radius=2pt];
		\fill (v2) circle [radius=2pt];
		\fill (v3) circle [radius=2pt];
		\fill (v4)+(-2pt,0pt) rectangle +(2pt,4pt);
		
		\draw (v3) circle [radius=10pt];
		
		\draw ($(v1)-(0,0.3)$) to [out=0,in=0] ($(v2)+(0,0.3)$)  to [out=180,in=180] ($(v1)-(0,0.3)$) ;	
		\end{tikzpicture}}}
\]
In the sum 
$\sum_{\tau \in\PaPT}\tfrac{1}{D(\tau)}\,\tau^{a_{01}}(a_0,a_1-a_0)$~, 
the terms corresponding to the ladders with a black leaf or with a white leaf are respectively equal to 
\[\sum_{k\geqslant 1} b_n \ad_{a_{01}}^n(a_0) \qquad \text{and} \qquad \sum_{k\geqslant 1} w_n \ad_{a_{01}}^n(a_1-a_0)\ ,\]
where the coefficients are equal to 
\[b_n=\sum_{l=1}^n (-1)^l \sum_{\substack{k_1+\cdots+k_l=n \\ k_i\geqslant 1}}
\prod_{i=1}^l\frac{1}{(k_i+1)!}
\]
and
\[
w_n=\sum_{l=1}^n (-1)^l \sum_{\substack{k_1+\cdots+k_l=n \\ k_i\geqslant 1}}
(k_l+1)
\prod_{i=1}^l \frac{1}{(k_i+1)!}
\ .\]
Finally, one can see  that 
\[
w_1=-1\ , \quad w_n=0\ , \ \text{for}\ n\geqslant 2\ , \quad \text{and} \quad 
b_n=\tfrac{B_n}{n!}, \ \text{for}\ n\geqslant 1\ ,
\]
where the $B_n$ denote the Bernoulli numbers. This recovers the formulas of \cite{ls10}.

\subsection{Gauge and homotopy equivalences}
From what we have seen so far, the integration functor admits the following equivalent descriptions:
\begin{align*}
\R(\g)\cong \Hom_{\,\sLialg}\left(\mc^\bullet, \g\right) \cong 
\MC\left(\g\, \widehat{\otimes}^\pi \mathrm{C}_\bullet\right)\cong 
\Hom_{\,\sLialg}\left(\mc^0, \g\, \widehat{\otimes}^\pi \mathrm{C}_\bullet\right)
~.
\end{align*}
The simple construction $\g \,\widehat{\otimes}\, \Omega_1$ provides us with a functorial path object for complete $\sLi$-algebras, see \cref{thm:MConSLi}, which has been commonly used ever since the pioneering works of D.\ Quillen \cite{Quillen67, Quillen69} and D.\ Sullivan \cite{Sullivan77}. The price to pay for such a nice construction is the size of the data of the induced notion of a homotopy since $\Omega_1$ is infinite dimensional. On the other hand, the final form given above for the integration functor $\R(\g)$ hints to the fact that 
$\g\, \widehat{\otimes}^\pi \mathrm{C}_1$ is a much more economical path object, see \cref{prop:C1PATH} for more details. The formula given in \cref{prop:Diffmc1} makes the $\sLi$-algebra structure on $\g\, \widehat{\otimes}^\pi \mathrm{C}_1$ explicit. 

\medskip

Recall from \cite[Theorem\ 4.6]{rnw17} that the Maurer--Cartan elements of the convolution algebra $\hom^\alpha(\Bar_\alpha A, B)$ are in one-to-one correspondence with the $\infty_\alpha$-morphisms from $A$ to $B$~. The gauge equivalence, under any of its forms mentioned above, gives rise to the following notion. 

\begin{definition}[Gauge homotopy equivalence of $\infty_\alpha$-morphisms]
	Let $\alpha : \C \to \P$ be an operadic twisting morphism.
	Two $\infty_\alpha$-morphisms between two complete $\P$-algebras $A$ and $B$ are called \emph{gauge homotopy equivalent} if
	their corresponding Maurer--Cartan elements in the complete $\sLi$-algebra 
	$\hom^\alpha(\Bar_\alpha A,B)$ are  gauge equivalent. 
\end{definition}

For the universal twisting morphism $\iota$~, this equivalence relation is equivalent to all the homotopy notions introduced so far among $\infty$-morphisms, see \cite{Dolgushev07, DotsenkoPoncin16, rnw18, Vallette20}, but it carries the advantage of being the most economical in term of the required data. 
The size and form for this gauge homotopy equivalence allows us to formulate the following conjecture. 

\begin{conjecture}
	The element $h_\infty=h \Psi$ defined in the last remark before the appendix of \cite{DotsenkoShadrinVallette16} is a gauge homotopy equivalence of $\infty$-morphisms between the composite $i_\infty \circledcirc p_\infty$ and the identity $\id$~. 
\end{conjecture}

The size of a gauge homotopy equivalence is optimal, ``so'' the various formulas are more involved than when working with the bigger model $\g \,\widehat{\otimes}\, \Omega_1$~. For its intrinsic interest, but also in order to develop tools to study the conjecture mentioned above, we provide an explicit formula for gauge homotopies. One can 
apply the explicit formula given in \cref{prop:gaugeformula}, which is equivalent to the formula for $\mc^1$ given in \cref{prop:Diffmc1}. Let us just treat in detail the case of $\infty$-morphisms of complete $\sLi$-algebras, the general case being similar and thus left to the interested reader. 

\medskip

For any $m\geqslant 1$~, we consider the set $\LRT_m$ of \emph{line rooted trees}\index{tree!line rooted}\index{$\LRT$} which are rooted trees with $m$  leaves bijectively indexed by $\{1, \ldots, m\}$ together with an horizontal line supporting only vertices, i.e.\ not crossing any edge. Vertices on this line have arity greater or equal to $1$ and can each have at most one other vertex directly above them, while all other vertices must have arity at least $2$~. The vertices on the line are labeled I, II, or III according to the following rules:
\begin{itemize}
	\item[$\diamond$] for each vertex below the line that has at least one children on the line, exactly one of these children is labeled III; this vertex can not have children of its own;
	\item[$\diamond$] vertices on the line with one vertex above them are labeled II; and
	\item[$\diamond$] all other vertices are labeled I or II.
\end{itemize}
These trees are not required to have a planar structure. 
For any line rooted tree, we denote by $\#\tau$ the number of vertices above the line. As an example, we consider the following tree with $\#\tau = 1$~.

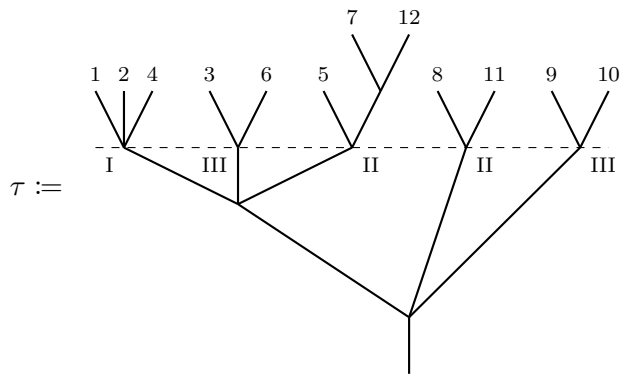
\[\tau\coloneqq\vcenter{\hbox{	
		\begin{tikzpicture}
		\def\scale{0.75}
		
		\coordinate (r1) at ($\scale*(1,-4)$);
		\coordinate (v1) at ($\scale*(1, -3)$); 
		\coordinate (v2) at ($\scale*(-2,-1)$); 	
		\coordinate (V1) at ($\scale*(-4,0)$); 	
		\coordinate (V2) at ($\scale*(-2,0)$); 	
		\coordinate (V3) at ($\scale*(0,0)$); 	
		\coordinate (V4) at ($\scale*(2,0)$); 	
		\coordinate (V5) at ($\scale*(4,0)$); 	
		\coordinate (vu1) at ($\scale*(0.5,1)$); 		
		\coordinate (l1) at ($\scale*(-4.5,1)$); 						
		\coordinate (l2) at ($\scale*(-4,1)$); 						
		\coordinate (l3) at ($\scale*(-3.5,1)$); 							
		\coordinate (l4) at ($\scale*(-2.5,1)$); 						
		\coordinate (l5) at ($\scale*(-1.5,1)$); 						
		\coordinate (l6) at ($\scale*(-0.5,1)$); 						
		\coordinate (l7) at ($\scale*(0,2)$); 							
		\coordinate (l8) at ($\scale*(1,2)$); 							
		\coordinate (l9) at ($\scale*(1.5,1)$); 								
		\coordinate (l10) at ($\scale*(2.5,1)$); 							
		\coordinate (l11) at ($\scale*(3.5,1)$); 							
		\coordinate (l12) at ($\scale*(4.5,1)$); 								
		
		\draw[thick] (r1) to (v1);	
		\draw[thick] (v1) to (v2);
		\draw[thick] (v2) to (V1);		
		\draw[thick] (v2) to (V2);		
		\draw[thick] (v2) to (V3);		
		\draw[thick] (v1) to (V4);		
		\draw[thick] (v1) to (V5);	
		\draw[thick] (V1) to (l1);							
		\draw[thick] (V1) to (l2);							
		\draw[thick] (V1) to (l3);									
		\draw[thick] (V2) to (l4);								
		\draw[thick] (V2) to (l5);									
		\draw[thick] (V3) to (l6);
		\draw[thick] (V3) to (vu1);	
		\draw[thick] (vu1) to (l7);										
		\draw[thick] (vu1) to (l8);										
		\draw[thick] (V4) to (l9);
		\draw[thick] (V4) to (l10);
		\draw[thick] (V5) to (l11);
		\draw[thick] (V5) to (l12);
		
		\draw[thin, dashed] ($\scale*(-4.5,0)$)--($\scale*(4.5,0)$);	
		
		\draw (V1) node[below left] {$\scriptstyle{{\rm I}}$};	
		\draw (V2) node[below left] {$\scriptstyle{{\rm III}}$};		
		\draw (V3) node[below right] {$\scriptstyle{{\rm II}}$};			
		\draw (V4) node[below right] {$\scriptstyle{{\rm II}}$};			
		\draw (V5) node[below right] {$\scriptstyle{{\rm III}}$};		
		
		\draw (l1) node[above] {$\scriptstyle{{1}}$};						
		\draw (l2) node[above] {$\scriptstyle{{2}}$};						
		\draw (l3) node[above] {$\scriptstyle{{4}}$};						
		\draw (l4) node[above] {$\scriptstyle{{3}}$};						
		\draw (l5) node[above] {$\scriptstyle{{6}}$};						
		\draw (l6) node[above] {$\scriptstyle{{5}}$};						
		\draw (l7) node[above] {$\scriptstyle{{7}}$};						
		\draw (l8) node[above] {$\scriptstyle{{12}}$};						
		\draw (l9) node[above] {$\scriptstyle{{8}}$};						
		\draw (l10) node[above] {$\scriptstyle{{11}}$};						
		\draw (l11) node[above] {$\scriptstyle{{9}}$};						
		\draw (l12) node[above] {$\scriptstyle{{10}}$};						
		\end{tikzpicture}}}
\]
To a line rooted tree $\tau\in\LRT$ we associate a planar tree $\overline\tau \in \PT$ as follows. First, we embed the line rooted tree into the plane according to its shuffle tree structure, see \cite{Hoffbeck10} and \cite[Section~8.2.2]{LodayVallette12}. The tree given as an example above is already drawn in this way. Then we remove everything above the line, the vertices of type I become leaves, the vertices of type II become vertices of arity $0$~, and the vertices of type III are removed completely. For the example above we obtain:
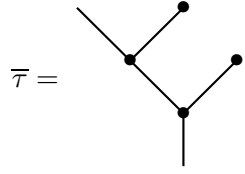
\[
\overline{\tau}= \vcenter{\hbox{
	\begin{tikzpicture}
		\def\scale{0.7};
		\pgfmathsetmacro{\diagcm}{sqrt(2)};
		
		\coordinate (r) at (0,0);
		\coordinate (v1) at ($(r) + (0,\scale*1)$);
		\coordinate (v2) at ($(v1) + (135:\scale*\diagcm)$);
		\coordinate (l1) at ($(v2) + (135:\scale*\diagcm)$);
		\coordinate (l2) at ($(v2) + (45:\scale*\diagcm)$);
		\coordinate (l3) at ($(v1) + (45:\scale*\diagcm)$);
		
		\draw[thick] (r) to (v1);
		\draw[thick] (v1) to (v2);
		\draw[thick] (v1) to (l3);
		\draw[thick] (v2) to (l1);
		\draw[thick] (v2) to (l2);
		
		\node at (v1) {$\bullet$};
		\node at (v2) {$\bullet$};
		\node at (l2) {$\bullet$};
		\node at (l3) {$\bullet$};
	\end{tikzpicture}}}
\]
For any planar rooted tree $\overline{\tau}\in\PT$ obtained this way, we consider the coefficient $A(\overline{\tau})$ given by the product of the factorial of the arities of all the vertices of the tree. This counts the number of distinct planar embeddings of the underlying topological tree and represents the intrinsic symmetries of the tree. In the example above, $A(\overline{\tau}) = 4$~.

\begin{proposition}\label{prop:GaugeInfMor}
	Let 
	$(\g, \{\ell_m\}_{m\geqslant 2})$ and $(\h, \{\mathcal{k}_m\}_{m\geqslant 2})$ 
	be  two complete $\sLi$-algebras and let 
	$\{f_m : \g^{\hat{\odot} m} \to \h\}_{m\geqslant 1}$ and $\{g_m : \g^{\hat{\odot} m} \to \h\}_{m\geqslant 1}$
	be two $\infty$-morphisms between them. 
	The data of a gauge homotopy between $f$ and $g$ amounts to a collection
	$\{\lambda_m : \g^{\hat{\odot} m} \to \h \}_{m\geqslant 1}$ of degree 1 maps satisfying, for any $m\geqslant 1$:
	\[
	g_m=\sum_{\tau \in \LRT_m} (-1)^{\#\tau} {\textstyle \frac{A(\overline{\tau})}{C(\overline{\tau})}} \, \tau(\mathcal{k}, f, \lambda, \ell)
	~, 
	\]
	where the vertices of $\tau$ below the line are labeled according to their arity by $\mathcal{k}_m$~, the vertices of type I by $f_m$~, the vertices of type II with a vertex above by $\lambda_m$ and the ones without a vertex above by $\partial(\lambda_m)$~, the vertices of type III by $\lambda_m$~, and the vertices above the line by $\ell_m$~.
\end{proposition}

\[\tau(\mathcal{k}, f, \lambda, \ell)=\vcenter{\hbox{	
		\begin{tikzpicture}
		\def\scale{0.75}
		
		\coordinate (r1) at ($\scale*(1,-4)$);
		\coordinate (v1) at ($\scale*(1, -3)$); 
		\coordinate (v2) at ($\scale*(-2,-1)$); 	
		\coordinate (V1) at ($\scale*(-4,0)$); 	
		\coordinate (V2) at ($\scale*(-2,0)$); 	
		\coordinate (V3) at ($\scale*(0,0)$); 	
		\coordinate (V4) at ($\scale*(2,0)$); 	
		\coordinate (V5) at ($\scale*(4,0)$); 	
		\coordinate (vu1) at ($\scale*(0.5,1)$); 		
		\coordinate (l1) at ($\scale*(-4.5,1)$); 						
		\coordinate (l2) at ($\scale*(-4,1)$); 						
		\coordinate (l3) at ($\scale*(-3.5,1)$); 							
		\coordinate (l4) at ($\scale*(-2.5,1)$); 						
		\coordinate (l5) at ($\scale*(-1.5,1)$); 						
		\coordinate (l6) at ($\scale*(-0.5,1)$); 						
		\coordinate (l7) at ($\scale*(0,2)$); 							
		\coordinate (l8) at ($\scale*(1,2)$); 							
		\coordinate (l9) at ($\scale*(1.5,1)$); 								
		\coordinate (l10) at ($\scale*(2.5,1)$); 							
		\coordinate (l11) at ($\scale*(3.5,1)$); 							
		\coordinate (l12) at ($\scale*(4.5,1)$); 								
		
		\draw[thick] (r1) to (v1);	
		\draw[thick] (v1) to (v2);
		\draw[thick] (v2) to (V1);		
		\draw[thick] (v2) to (V2);		
		\draw[thick] (v2) to (V3);		
		\draw[thick] (v1) to (V4);		
		\draw[thick] (v1) to (V5);	
		\draw[thick] (V1) to (l1);							
		\draw[thick] (V1) to (l2);							
		\draw[thick] (V1) to (l3);									
		\draw[thick] (V2) to (l4);								
		\draw[thick] (V2) to (l5);									
		\draw[thick] (V3) to (l6);
		\draw[thick] (V3) to (vu1);	
		\draw[thick] (vu1) to (l7);										
		\draw[thick] (vu1) to (l8);										
		\draw[thick] (V4) to (l9);
		\draw[thick] (V4) to (l10);
		\draw[thick] (V5) to (l11);
		\draw[thick] (V5) to (l12);
		
		\draw (v1) node[left] {$\scriptstyle{{\mathcal{k}_3}}$};	
		\draw (v2) node[left] {$\scriptstyle{{\mathcal{k}_3}}$};	
		\draw (vu1) node[right] {$\scriptstyle{{\ell_2}}$};		
		
		\draw (V1) node[ left] {$\scriptstyle{{f_3}}$};	
		\draw (V2) node[ left] {$\scriptstyle{{\lambda_2}}$};		
		\draw (V3) node[ right] {$\scriptstyle{{\lambda_2}}$};			
		\draw (V4) node[ right] {$\scriptstyle{{\partial(\lambda_2)}}$};			
		\draw (V5) node[ right] {$\scriptstyle{{\lambda_2}}$};		
		
		\draw (l1) node[above] {$\scriptstyle{{1}}$};						
		\draw (l2) node[above] {$\scriptstyle{{2}}$};						
		\draw (l3) node[above] {$\scriptstyle{{4}}$};						
		\draw (l4) node[above] {$\scriptstyle{{3}}$};						
		\draw (l5) node[above] {$\scriptstyle{{6}}$};						
		\draw (l6) node[above] {$\scriptstyle{{5}}$};						
		\draw (l7) node[above] {$\scriptstyle{{7}}$};						
		\draw (l8) node[above] {$\scriptstyle{{12}}$};						
		\draw (l9) node[above] {$\scriptstyle{{8}}$};						
		\draw (l10) node[above] {$\scriptstyle{{11}}$};						
		\draw (l11) node[above] {$\scriptstyle{{9}}$};						
		\draw (l12) node[above] {$\scriptstyle{{10}}$};						
		\end{tikzpicture}}}\]

\begin{proof}
	This a direct application of \cref{prop:gaugeformula} to the
	convolution algebra 
	$\hom^\iota(\Bar_\iota \g, \h)$~. 
	In this case, the cooperad $\C=\com^\vee$ is the linear dual of the operad $\com$~, so the convolution algebra is isomorphic to 
	\[\hom_\Sy\big(\com^\vee, \eend^\g_\h\big)
	\cong 
	\prod_{m\geqslant 1} \hom
	\left(\g^{\hat{\odot} m}, \h\right)\]
	with  structure operations 
	\[d(f)_m=\partial(f_m)-\sum_{\substack{p+q=m\\ q\geqslant 2}}
	\sum_{\sigma\in \mathrm{Sh}_{p,q}^{-1}}
	(f_{p+1}\circ_{1} \ell_q)^{\sigma}\]
	and
	\[
	\mu_k\big(f^1, \ldots, f^k\big)_m={\textstyle \frac{1}{k!}}
	\sum_{\tau \in \Sy_k}
	\sum_{\substack{i_1+\cdots+i_k=m}}
	\sum_{\sigma\in \mathrm{Sh}_{i_1, \ldots, i_k}^{-1}}
	\mathcal{k}_k\circ \left(f^{\tau(1)}_{i_1}, \ldots, f^{\tau(k)}_{i_k}\right)^{\sigma}\ .
	\]
	The formula
	\[
	g=\sum_{\overline{\tau}\in \PT}{\textstyle \frac{1}{C(\overline{\tau})}}{\overline{\tau}}^\lambda(f)
	\]
	of \cref{prop:gaugeformula} produces the composition of maps along line rooted trees given in the statement together with the coefficient ${\textstyle \frac{1}{C(\overline{\tau})}}$~. The sign comes from the minus sign on the right-hand term of the differential given above for the convolution algebra. Finally, any contribution along a line rooted tree $\tau$ actually comes from $A(\overline{\tau})$ planar rooted trees. 
\end{proof}
\section{Higher Baker--Campbell--Hausdorff formulas}\label{sec:HighBCH}

It is known that the Deligne--Hinich and the Getzler Maurer--Cartan spaces are actually Kan complex. In this section, we show this fact for our model $\R(\g)$ from first principles. We recover much more, fully characterizing all fillings of horns in $\R(\g)$ with explicit formulas. We apply this newfound knowledge to first recover the classical Baker--Campbell--Hausdorff formula and then to introduce natural higher generalizations.

\medskip

It should be noted that the existence of such products was already noticed by Getzler \cite[Definition~5.5]{Getzler09} and that Bandiera proved in \cite[Proposition~5.2.36]{Bandiera14} that the first one coincides with the classical Baker--Campbell--Hausdorff formula. The higher ones have not yet been studied in any additional depth. 
In this section, we first prove the existence of the higher BCH products using the new methods introduced in the present paper (left adjoint functor $\L$ of the integration functor $\R$ and fixed-point equations \cref{Sec:FixedPtEq}). This allows us to give explicit formulas to calculate them. We establish a new characterization of the classical BCH formula which leads to another proof that the first of these products is equal to the classical  BCH formula for Lie algebras. 
Finally, we study some of their properties. We will  apply these higher BCH products in \cref{sect:homotopy theory} to study the homotopical behavior of the integration functor $\R$~.

\subsection{Filling horns in Maurer--Cartan spaces}\label{subsec:fill horns}

The purpose of this section is to show that, given a horn $\Ho{k}{n} \to \R(\g)$~, the set of horn fillers is in canonical bijection with the set of  elements of degree $n$ of the complete $\sLi$-algebra:
\begin{equation}\label{eq:HornFillers}
\left\{\vcenter{\hbox{
	\begin{tikzpicture}
		\node (a) at (0, 0) {$\Ho{k}{n}$};
		\node (b) at (2, 0) {$\R(\g)$};
		\node (c) at (0, -1.5) {$\De{n}$};
		
		\draw[->] (a) to (b);
		\draw[right hook->] (a) to node[below=-0.05cm, sloped]{$\sim$} (c);
		\draw[->, dashed] (c) to (b);
	\end{tikzpicture}
}}\right\} 
\cong \g_n\ .
\end{equation}

The following result is the key ingredient of this section.

\begin{lemma}\label{lemm:Fundamental}
	For $n\geqslant 2$~, there is an isomorphism of complete $\sLi$-algebras 
	\[
	\mc^n=\Li\big(\De{n}\big)\cong \Li\big(\Ho{k}{n}\big)\,\widehat{\sqcup}\, \widehat{\sLi}(u, du)\ ,
	\]
	where $u$ is a generator of degree $n$ and where the differential on the right-hand side is given by $d(u)=du$~.
\end{lemma}

\begin{proof}
	We consider the following morphism of complete $\sLi$-algebras:
	\begin{align*}
		\Li\big(\Ho{k}{n}\big)\,\widehat{\sqcup}\, \widehat{\sLi}(u, du){}&\stackrel{\varphi}{\longrightarrow}\Li\big(\De{n}\big)\\
		a_I{}&\longmapsto a_I\\
		u{}&\longmapsto a_{[n]}\\
		du{}&\longmapsto \d(a_{[n]})\ .
	\end{align*}
	In the other way around, we want to define a morphism of $\sLi$-algebras as follows
	\begin{align*}
		\Li\big(\De{n}\big){}&\stackrel{\psi}{\longrightarrow} \Li\big(\Ho{k}{n}\big)\,\widehat{\sqcup}\, \widehat{\sLi}(u, du)\\
		a_I{}&\longmapsto a_I\\
		a_{[n]}{}&\longmapsto u\\
		a_{\whk}{}&\longmapsto x\ ,
	\end{align*}
	where $\whk \coloneqq \{0, \ldots, \widehat{k}, \ldots, n\}$ and where $x$ needs to be determined. 
	
	\medskip 
	
	Recall from \cref{eq:DiffMCn} that 
	\begin{align*}
	\d(a_{[n]})=&{}\sum_{l=0}^n (-1)^l a_{\widehat{l}}\\
	&{}- 
	\sum_{\substack{m\geqslant 2\\j\geqslant 0}}
	\sum_{\substack{\tau \in \overline{\RT}_{m+j}\\
			I_1, \ldots, I_m\subseteq[n]\ , I_l\neq \emptyset, \whk\\
			\lambda^{\tau(I_1, \ldots, \whk, \ldots, \whk, \ldots, I_m)}_{[n]}\neq 0}}
	\frac{1}{\lambda^{\tau(I_1, \ldots, \whk, \ldots, \whk, \ldots, I_m)}_{[n]} (m+j)!}
	\, \tau\left(a_{I_1}, \ldots, a_{\whk}, \ldots, a_{\whk}, \ldots, a_{I_m}\right),
	\end{align*}
	where we slightly abuse notation and write $a_{[n]}$ for $u$~.
	
	\medskip
	
	For the composite $\psi\varphi$ to be equal to the identity, by evaluating on $du$ we need the element $x$ to satisfy the  equation 
	\begin{align*}
	du ={}&(-1)^k x+\sum_{l\neq k} (-1)^l a_{\widehat{l}}\\
	&-\sum_{\substack{m\geqslant 2\\j\geqslant 0}} \sum_{\substack{\tau \in \overline{\RT}_{m+j}\\ 
			I_1, \ldots, I_m\subset[n]\ , I_l\neq \emptyset, \whk\\ 
			\lambda^{\tau(I_1, \ldots, \whk, \ldots, \whk, \ldots, I_m)}_{[n]}\neq 0}} \frac{1}{\lambda^{\tau(I_1, \ldots, \whk, \ldots, \whk, \ldots, I_m)}_{[n]} (m+j)!}\, \tau(a_{I_1}, \ldots, x, \ldots, x, \ldots, a_{I_m})\ ,
	\end{align*}
	where $j$ stands for the number of appearances of $x$~.  
	This equation is equivalent to the following fixed-point equation:
	\begin{align}\label{eq:FixPtMorph}
	\begin{split}
	x={}&(-1)^k du-\sum_{l\neq k} (-1)^{k+l} a_{\widehat{l}}\\
	& +(-1)^k \sum_{\substack{j\geqslant 0\\m\geqslant 2}}\sum_{\substack{\tau \in \overline{\RT}_{m+j}\\I_1, \ldots, I_m\subseteq[n]\ , I_l\neq \emptyset, \whk\\ \lambda^{\tau(I_1, \ldots, \whk, \ldots, \whk, \ldots, I_m)}_{[n]}\neq 0}} \frac{1}{\lambda^{\tau(I_1, \ldots, \whk, \ldots, \whk, \ldots, I_m)}_{[n]} (m+j)!}\, \tau(a_{I_1}, \ldots, x, \ldots, x, \ldots, a_{I_m})
	\end{split}
	\end{align}
	in the complete $\sLi$-algebra $\Li\big(\Ho{k}{n}\big)\,\widehat{\sqcup}\, \widehat{\sLi}(u, du)$~. This latter equation is a particular case of the fixed-point equation~\eqref{eqn:FixPt}: we consider the  maps
	\begin{align*}
	&p_0\coloneqq (-1)^k du-\sum_{l\neq k} (-1)^{k+l} a_{\widehat{l}} +(-1)^k 
	\sum_{m\geqslant 2}
	\sum_{\substack{\tau \in \overline{\RT}_{m}\\
			I_1, \ldots, I_m\subseteq[n]\ , I_l\neq \emptyset, \whk\\
			\lambda^{\tau(I_1, \ldots, I_m)}_{[n]}\neq 0}}
	\frac{1}{\lambda^{\tau(I_1, \ldots,  I_m)}_{[n]} m!}
	\, \tau(a_{I_1}, \ldots,  a_{I_m})\ , \\
	&\PP_j(-, \ldots, -)\coloneqq\\
	&\qquad\quad(-1)^k
	\sum_{m\geqslant 2}
	\sum_{\substack{\tau \in \overline{\RT}_{m+j}\\
			I_1, \ldots, I_m\subseteq[n]\ , I_l\neq \emptyset, \whk\\
			\lambda^{\tau(I_1, \ldots, \whk, \ldots, \whk, \ldots, I_m)}_{[n]}\neq 0}}
	\frac{1}{\lambda^{\tau(I_1, \ldots, \whk, \ldots, \whk, \ldots, I_m)}_{[n]} (m+j)!}
	\, \tau(a_{I_1}, \ldots, -, \ldots, -, \ldots, a_{I_m})\ ,
	\end{align*}
	which is well-defined since, for any $x_1, \ldots, x_j$~, the element $\tau(a_{I_1}, \ldots, x_1, \ldots, x_j, \ldots, a_{I_m})$ lives in $\F_{m}$ and since, for any  $m\geqslant 2$~, the number of rooted trees in $\RT_{m+j}$ is finite. 
	It remains to prove that each $P_j$ raises the degree filtration by $1$~, which amounts to show that every non-trivial term appearing in the sum involves at least one $a_I$~, with $I\neq \widehat{k}$~. This is equivalent to proving that for any reduced rooted tree $\tau\in \overline{\RT}$~, we have 
	$\mu_\tau\big(\omega_{\whk}, \ldots, \omega_{\whk}\big)\neq \lambda\, \omega_{[n]}+\cdots$~, with $\lambda\in \k\backslash \{0\}$~. For $n\geqslant 2$~, we have $\omega_{\whk}\omega_{\whk}=0$;  therefore, we get $\mu_\tau\big(\omega_{\whk}, \ldots, \omega_{\whk}\big)=0$~.
	
	\medskip
	
	Now one can apply \cref{prop:FixPtEqua} to get the existence of a solution $x$ to the fixed-point equation~\eqref{eq:FixPtMorph}. Forgetting the differentials, the map $\psi$ is now a well-defined morphism of complete $\sLi$-algebras, which satisfies $\psi\varphi=\id$~. The image under the morphism $\varphi$ of Equation~\eqref{eq:FixPtMorph} produces the same kind of equation
	\begin{align*}
	\varphi(x)={}& \d(a_{[n]})-(-1)^k\sum_{l\neq k} (-1)^l a_{\widehat{l}}\\
	&+(-1)^k 
	\sum_{\substack{m\geqslant 2\\j\geqslant 0}}
	\sum_{\substack{\tau \in \overline{\RT}_{m+j}\\
			I_1, \ldots, I_m\subseteq[n]\ , I_l\neq \emptyset, \whk\\
			\lambda^{\tau(I_1, \ldots, \whk, \ldots, \whk, \ldots, I_m)}_{[n]}\neq 0}}
	\frac{\tau(a_{I_1}, \ldots, \varphi(x), \ldots, \varphi(x), \ldots, a_{I_m})}{\lambda^{\tau(I_1, \ldots, \whk, \ldots, \whk, \ldots, I_m)}_{[n]} (m+j)!}
	\end{align*}
	but in the complete $\sLi$-algebra $\mc^n$ this time. A solution is known, it is $a_{\whk}$~, and by the uniqueness of the solution to the fixed-point equation in \cref{prop:FixPtEqua}, we conclude that 
	$\varphi(x)=a_{\whk}$~. This shows that $\varphi\circ \psi=\id$ and thus that $\psi$ is a chain map. 
\end{proof}

\begin{theorem}\label{thm:KanExt}
	The map
	\begin{align*}
	\Hom_\sSe\left(\De{n}, \R(\g)\right){}&\ \longrightarrow\ \Hom_\sSe\left(\Ho{k}{n}, \R(\g)\right)\times  \g_n\\
	z{}&\ \longmapsto\ \left(z_{|\Ho{k}{n}}, z[n]\right)
	\end{align*}
	is a bijection which is natural 
	in the complete $\sLi$-algebra $\g$ with respect to strict morphisms.
\end{theorem}

This established the bijection \eqref{eq:HornFillers} mentioned at the beginning of this section.

\begin{proof}
	The isomorphism $\mc^n\cong \Li\big(\Ho{k}{n}\big)\,\widehat{\sqcup}\, \widehat{\sLi}(u, du)$
	of \cref{lemm:Fundamental} and the adjunction of \cref{thm:MainAdjunction} induce the following bijections
	\begin{align*}
		\Hom_\sSe\left(\De{n}, \R(\g)\right)\cong{}&\Hom_{\sLialg}\left(\Li(\De{n}), \g\right)\\
		\cong{}& \Hom_{\sLialg}\left(\Li(\Ho{k}{n}), \g\right)\times  \Hom_{\sLialg}\left(\widehat{\sLi}(u, du), \g\right)\\
		\cong{}& \Hom_{\sLialg}\left(\Li(\Ho{k}{n}), \g\right)\times  \g_n\\
		\cong{}& \Hom_\sSe\left(\Ho{k}{n}, \R(\g)\right)\times  \g_n\ ,
	\end{align*}
	which is natural in the complete $\sLi$-algebra $\g$\ . The final bijection from left to right is given by the restriction along the horn inclusion $\Ho{k}{n}\hookrightarrow \De{n}$ together with taking the image of $[n]$~. This concludes the proof. 
\end{proof}

\begin{corollary}
	The simplicial set $\R(\g)$ integrating any complete $\mathcal{s}\L_\infty$-algebra $\g$ is a Kan complex with a canonical algebraic $\infty$-groupoid structure. 
\end{corollary}

\begin{proof}
	This is a direct corollary of \cref{thm:KanExt}: since each $\g_n$ is not empty, 
	every map $\Ho{k}{n} \to \R(\g)$ can be lifted to a map $\De{n} \to \R(\g)$\ . So $\R(\g)$ is a Kan complex. One can however be more precise here by always considering the canonical choice provided by  the element $0\in \g_n$~. Therefore, the Kan complex $\R(\g)$ comes equipped with a canonical choice of filler for every horn, which is the definition of an algebraic $\infty$-groupoid structure. 
\end{proof}

\begin{proposition}
	Under the isomorphism of \cref{thm:isomorphism of models}, the canonical horn fillers given by the element $0\in\g_n$ in the $\infty$-groupoid $\R(\g)$ correspond to the unique thin fillers of \cite[Theorem~5.4]{Getzler09} in $\gamma_\bullet(\g)$\ . 
\end{proposition}

\begin{proof}
	Recall from \cite[Definition~5.2]{Getzler09}  that an $n$-simplex $\alpha\in\gamma_n(\g)$ is called \emph{thin} if
	\[
	\int_{|\Delta^n|}\alpha = 0\ .
	\]
	The isomorphism
	\[
	\R(\g)\xrightarrow{\cong}\gamma_\bullet(\g)
	\]
	sends $x\in\Hom_{\sLialg}(\mc^n,\g)$ to the element
	\[
	\alpha = x(a_{[n]})dt_1\cdots dt_n + \text{lower order terms}\in\gamma_n(\g)\ ,
	\]
	where $x(a_{[n]})\in\g_{-n}$~. Thus, we have
	\[
	\int_{|\Delta^n|}\alpha = \int_{|\Delta^n|}x(a_{[n]})dt_1\cdots dt_n = \tfrac{1}{n!}x(a_{[n]})\ ,
	\]
	which is zero if and only if $x(a_{[n]}) = 0$~. We refer the reader to \cite[Lemma~4.8]{rn17cosimplicial} for more details.
\end{proof}

\subsection{Higher Baker--Campbell--Hausdorff products}

The inverse of the bijection of \cref{thm:KanExt} gives rise to a collection of maps which generalize the classical Baker--Campbell--Hausdorff binary product.

\begin{definition}[Higher Baker--Campbell--Hausdorff products]
	We call  \emph{higher Baker--Camp\-bell--Haus\-dorff products}\index{higher Baker--Campbell--Hausdorff products} the maps
	\begin{align*}
	\Hom_{\sSe}\left({\Ho{k}{n}}, \R(\g)\right) \times \g_n{}&\ \longrightarrow\ \g_{n-1}\\
	\left(x; y\right){}&\ \longmapsto\ \Gamma^n_k(x; y)\coloneqq z\big(\,\whk\,\big)
	\end{align*}
	defined by the evaluation at the cell $\whk$ of $\De{n}$ of the $n$-simplex $z$ inverse to 
	$(x;y)$ under the bijection of \cref{thm:KanExt}.
\end{definition}

When we want to emphasize the full input provided by the $k$-horn $\Ho{k}{n}$ of dimension $n$ we will use the notation $\Gamma^n_k(\{x_{\widehat{j}}\}_{j\neq k}; y)$~, cf.\ \cref{subsec:GeoRep}. In this case, we will sometimes also omit some of the cells, in which case the missing cells are assumed to have value $0$~. 
When $y=0\in \g_n$~, we will drop it from the notation and write $\Gamma_k^n(x)$~.

\medskip

In words, given a horn $x\in\Hom_{\sSe}\left({\Ho{k}{n}}, \R(\g)\right)$ and a filler $y\in\g_n$~, the inverse of the bijection of \cref{thm:KanExt} yields a simplex $z\in\Hom_\sSe\left(\De{n}, \R(\g)\right)$ that we evaluate at $\whk$~. This way, we obtain an element of $\g_{n-1}$ representing the only face of $z$ that was missing at the beginning. For example, the element $\Gamma^2_1(x_{01}, x_{12}; y)\in \g_1$ is the output of the BCH product applied to the following horn.
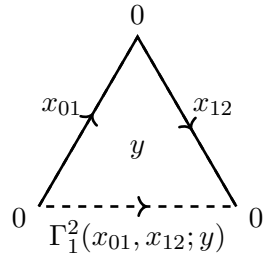
\[
\begin{tikzpicture}

\coordinate (v0) at (210:1.5);
\coordinate (v1) at (90:1.5);
\coordinate (v2) at (-30:1.5);

\draw[line width=1] (v0)--(v1)--(v2);
\draw[dashed, line width=1]  (v0)--(v1)--(v2)--cycle;

\begin{scope}[decoration={
	markings,
	mark=at position 0.55 with {\arrow{>}}},
line width=1
]

\path[postaction={decorate}] (v0) -- (v1);
\path[postaction={decorate}] (v1) -- (v2);
\path[postaction={decorate}] (v0) -- (v2);

\end{scope}


\node at ($(v0) + (30:-0.3)$) {$0$};
\node at ($(v1) + (0,0.3)$) {$0$};
\node at ($(v2) + (-30:0.3)$) {$0$};

\node at ($(v0)!0.5!(v1) + (-30:-0.4)$) {$x_{01}$};
\node at ($(v1)!0.5!(v2) + (30:0.4)$) {$x_{12}$};
\node at ($(v0)!0.5!(v2) + (0,-0.4)$) {$\Gamma^2_1(x_{01}, x_{12}; y)$};

\node at (0,0) {$y$};

\end{tikzpicture}
\]

\begin{remark}
	Strictly speaking, the definition given in \cite[Definition~5.5]{Getzler09} is not the same since the author does not start from the same labeled cells of the $n$-simplex and since the output does not label the same cells. However, it is straightforward to check that both definitions are equivalent.
\end{remark}

The fixed-point \cref{eq:FixPtMorph} in the proof of \cref{lemm:Fundamental} can be solved using \cref{prop:FixPtEqua} to provide us with an explicit formulas for the higher BCH as follows. 

\medskip

We consider the set $\PaPRT$ of \emph{planarly partitioned rooted trees}\index{trees!planarly partitioned rooted}\index{$\PaPRT$}. These are rooted trees with the additional data of a \emph{partition} into sub-trees called \emph{blocks} satisfying
\begin{itemize}
	\item[$\diamond$] each block contains at least one vertex,
	\item[$\diamond$] each block is either a tree with vertices of arity at least $2$ or the single $0$-corolla, and
	\item[$\diamond$] the tree obtained by stripping the partitioned tree of all of its leaves and contracting the blocks to vertices is planar.
\end{itemize}
Notice that the rooted sub-trees inside each block have no planar structure.

\begin{example}
	Here is an example of a planarly partitioned rooted tree $\tau \in \PaPRT$~.
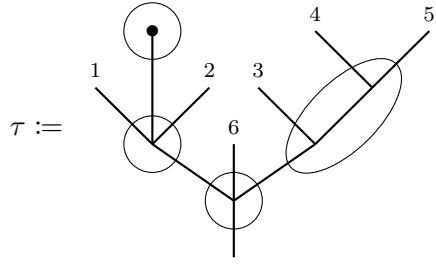
	\[
	\tau \coloneqq \vcenter{\hbox{
		\begin{tikzpicture}
			\def\scale{0.75};
			\pgfmathsetmacro{\diagcm}{sqrt(2)};
			
			\def\xangle{35};
			\pgfmathsetmacro{\xcm}{1/sin(\xangle)};
			
			\coordinate (r) at (0,0);
			\coordinate (v11) at ($(r) + (0,\scale*1)$);
			\coordinate (v21) at ($(v11) + (180-\xangle:\scale*\xcm)$);
			\coordinate (v22) at ($(v11) + (\xangle:\scale*\xcm)$);
			\coordinate (v31) at ($(v22) + (45:\scale*\diagcm)$);
			\coordinate (l1) at ($(v21) + (135:\scale*\diagcm)$);
			\coordinate (l2) at ($(v21) + (0,\scale*2)$);
			\coordinate (l3) at ($(v21) + (45:\scale*\diagcm)$);
			\coordinate (l4) at ($(v22) + (135:\scale*\diagcm)$);
			\coordinate (l5) at ($(v31) + (135:\scale*\diagcm)$);
			\coordinate (l6) at ($(v31) + (45:\scale*\diagcm)$);
			\coordinate (l7) at ($(v11) + (0,\scale*1)$);
			
			\draw[thick] (r) to (v11);
			\draw[thick] (v11) to (v21);
			\draw[thick] (v11) to (v22);
			\draw[thick] (v21) to (l1);
			\draw[thick] (v21) to (l2);
			\draw[thick] (v21) to (l3);
			\draw[thick] (v22) to (l4);
			\draw[thick] (v22) to (v31);
			\draw[thick] (v31) to (l5);
			\draw[thick] (v31) to (l6);
			\draw[thick] (v11) to (l7);
			
			\node[above] at (l1) {$\scriptstyle1$};
			\node at (l2) {$\bullet$};
			\node[above] at (l3) {$\scriptstyle2$};
			\node[above] at (l4) {$\scriptstyle3$};
			\node[above] at (l5) {$\scriptstyle4$};
			\node[above] at (l6) {$\scriptstyle5$};
			\node[above] at (l7) {$\scriptstyle6$};
			
			\draw (v11) circle[radius=\scale*0.5];
			\draw (v21) circle[radius=\scale*0.5];
			\draw (l2) circle[radius=\scale*0.5];
			\draw ($(v22)!0.5!(v31)$) ellipse[x radius=\scale*1.3, y radius=\scale*0.6, rotate=45];
		\end{tikzpicture}
	}}
	\]
\end{example}

Given any planarly partitioned rooted tree $\tau\in \PaPRT$~, we consider the set $\mathrm{Lab}^{[n], k}(\tau)$ of maps
\[
\chi \ : \ \mathrm{L}(\tau) \longrightarrow\left\{I\subseteq[n]\,\mid\, I\neq\emptyset,\widehat{k}\right\}~, 
\]
which amount to label its leaves $\mathrm{L}(\tau)$ with some subsets of $[n]$~. For any pair $(\tau, \chi)$ with $\tau\in\PaPRT$ and $\chi\in\mathrm{Lab}^{[n], k}(\tau)$~, and for any block $\beta$ of $\tau$~, we denote by $\lambda^{\beta(\chi)}_{[n]}$ the coefficient
\[
\lambda^{\beta(\chi)}_{[n]}\coloneqq\lambda^{\beta\left(I_1, \ldots, \whk, \ldots, \whk, \ldots, I_m\right)}_{[n]}
\] 
associated to the tree inside the block $\beta$ with leaves labeled by $I_1, \ldots, I_m$ according to $\chi$ and with the leaves corresponding to the internal edges of $\tau$ labeled by $\whk$~.  By convention, if $\beta$ is the block given by a single vertex of arity $0$ then the coefficient $\lambda^{\beta(\chi)}_{[n]}$ is equal to $1$~.

\begin{example}
	In the case of the block
	\[\beta\coloneqq 
	\vcenter
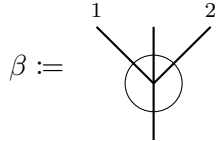
{\hbox{
		\begin{tikzpicture}
			\def\scale{0.75};
			\pgfmathsetmacro{\diagcm}{sqrt(2)};
			
			\coordinate (r) at (0,0);
			\coordinate (v) at ($(r) + (0,\scale*1)$);
			\coordinate (l1) at ($(v) + (135:\scale*\diagcm)$);
			\coordinate (l2) at ($(v) + (0,\scale*1)$);
			\coordinate (l3) at ($(v) + (45:\scale*\diagcm)$);
			
			\draw[thick] (r) to (v);
			\draw[thick] (v) to (l1);
			\draw[thick] (v) to (l2);
			\draw[thick] (v) to (l3);
			
			\node[above] at (l1) {$\scriptstyle1$};
			\node[above] at (l3) {$\scriptstyle2$};
			
			\draw (v) circle[radius=\scale*0.5];
		\end{tikzpicture}}}
	\]
	the coefficient $\lambda^{\beta(\chi)}_{[n]}$ is equal to 
	$\lambda^{\beta\left(I_1,  \whk, I_2\right)}_{[n]}$~.
\end{example}

\medskip

Let $x:\Ho{k}{n}\to\R(\g)$ be a horn and let us write, as usual, $x_I\in\g_{|I|-1}$ for the element of $\g$ associated to the non-degenerate simplex of the horn indexed by $I$~. Let $y\in \g_n$~, that we denote by $x_{[n]}\coloneqq y$~. For any element $z\in \g$~,  we denote by $\ell_\tau(x_{\chi(1)}, \ldots, x_{\chi(p)};  z)$ the element of $\g$ obtained by forgetting the partition, replacing the leaves $l$ by the corresponding elements $x_{\chi(l)}\in\g$ and the vertices of arity $0$ by $z$~, and then applying the structure operations of the algebra at the vertices of the tree with the corresponding arity.

\begin{example}
	In the example mentioned above of the planarly partitioned rooted tree $\tau$~, we get 
	\[\ell_\tau(x_{\chi(1)}, \ldots, x_{\chi(6)};  z)=
	\ell_3\left(
	\ell_3\left(x_{I_1}, z, x_{I_2}\right)
	x_{I_3}, 
	\ell_2\left(x_{I_4}, \ell_2\left(x_{I_5}, x_{I_6}\right)\right)
	\right)~.
	\]
\end{example}

\begin{proposition}\label{prop:ExplicitBCHForm}
	The higher BCH product is given by 
	\begin{equation}\label{Eq:HighBCH}
	\Gamma^n_k(x; y) = 
	\sum_{\substack{\tau\in\PaPRT\\
			\chi\in\mathrm{Lab}^{[n], k}(\tau)}}\ 
	\prod_{\substack{\beta\text{ block of } \tau \\ \lambda^{\beta(\chi)}_{[n]}\neq 0}} 
	\frac{(-1)^k}{\lambda^{\beta(\chi)}_{[n]}[\beta]!}\, 
	\ell_\tau\left(x_{\chi(1)}, \ldots, x_{\chi(p)};  (-1)^k\d y - \sum_{l\neq k}(-1)^{k+l} x_{\widehat{l}}\right), 
	\end{equation}
	where $[\beta]$ is the arity of the tree inside the block $\beta$ and where we write $x_{[n]}$ for $y$~. 
\end{proposition}

\begin{proof}
	This is a straightforward application of the formula 
	\[
	\Gamma^n_k(x; y)=\sum_{k=1}^\infty \sum_{\tau \in \PT^{(k)}_0} \tau^\PP
	\]
	of \cref{prop:FixPtEqua} which gives the solution to the fixed-point equation 
	applied to the analytic function 
	\begin{align*}
	&p_0\coloneqq (-1)^k \d y-\sum_{l\neq k} (-1)^{k+l} x_{\widehat{l}} +(-1)^k 
	\sum_{m\geqslant 2}
	\sum_{\substack{\tau \in \overline{\RT}_{m}\\
			I_1, \ldots, I_m\subseteq[n]\ , I_l\neq \emptyset, \whk\\
			\lambda^{\tau(I_1, \ldots, I_m)}_{[n]}\neq 0}}
	\frac{1}{\lambda^{\tau(I_1, \ldots,  I_m)}_{[n]} m!}
	\, \ell_\tau(x_{I_1}, \ldots,  x_{I_m})\ , \\
	&\PP_j(-, \ldots, -)\coloneqq\\
	&\qquad\quad(-1)^k
	\sum_{m\geqslant 2}
	\sum_{\substack{\tau \in \overline{\RT}_{m+j}\\
			I_1, \ldots, I_m\subseteq[n]\ , I_l\neq \emptyset, \whk\\
			\lambda^{\tau(I_1, \ldots, \whk, \ldots, \whk, \ldots, I_m)}_{[n]}\neq 0}}
	\frac{1}{\lambda^{\tau(I_1, \ldots, \whk, \ldots, \whk, \ldots, I_m)}_{[n]} (m+j)!}
	\, \ell_\tau(x_{I_1}, \ldots, -, \ldots, -, \ldots, x_{I_m})\ .
	\end{align*}
\end{proof}

\begin{remark}
	The above formula shows that the only remaining obstruction to an easily computable formula for the higher BCH products is in the absence of simple formulas for the transferred algebraic structure of $\rmC_\bullet$~, which induces the structure constants $\lambda$~. This structure deserves further study. The first named author has made available a software package for numerical experimentation, see \cite{RN20}.
\end{remark}

Although this expression entails non-trivial combinatorics coming the structure constants and the complex family of trees, in practice we will often only need here the first few terms of this formula plus the fact that the higher terms are obtained by iterated bracketings of elements.  For example, it is straightforward to see that
\begin{equation}\label{eq:BCH first terms}
\Gamma^n_k(x; y) =  (-1)^k\d y - \sum_{l\neq k} (-1)^{k+l} x_{\widehat{l}} + \left(\text{brackets of }y, \d y, \text{and} \ x_I\text{ for }\emptyset,\whk\neq I\subset[n]\right).
\end{equation}

\subsection{New approach to the Baker--Campbell--Hausdorff formula}
In this section, we consider the first horn filler for complete Lie algebras and show that it allows us to recover the classical Baker--Campbell--Hausdorff formula. 
This makes explicit the first layer of the canonical $\infty$-groupoid structure on $\R(\g)$ integrating complete Lie algebras. 
This result was first proved by R.\ Bandiera in his PhD thesis \cite{Bandiera14} but our method is different: we will establish a new characterization of the Baker--Campbell--Hausdorff formula and we will produce an explicit formula.  

\medskip

We start by considering ``classical'' complete Lie algebras, by which we mean complete $\sLie$-algebras $\g$ concentrated in degree $1$~. Any horn $\Ho{1}{2} \to \R(\g)$ is equivalent to the data of two elements $x,y\in \g$\ . 
We consider the element $\Gamma^2_1(x,y)\in \g$ obtained by canonically filling this horn with $0\in \g_1$~.
\[
\begin{tikzpicture}

\coordinate (v0) at (210:1.5);
\coordinate (v1) at (90:1.5);
\coordinate (v2) at (-30:1.5);

\draw[line width=1] (v0)--(v1)--(v2);
\draw[dashed, line width=1]  (v0)--(v1)--(v2)--cycle;

\begin{scope}[decoration={
	markings,
	mark=at position 0.55 with {\arrow{>}}},
line width=1
]

\path[postaction={decorate}] (v0) -- (v1);
\path[postaction={decorate}] (v1) -- (v2);
\path[postaction={decorate}] (v0) -- (v2);

\end{scope}


\node at ($(v0) + (30:-0.3)$) {$0$};
\node at ($(v1) + (0,0.3)$) {$0$};
\node at ($(v2) + (-30:0.3)$) {$0$};

\node at ($(v0)!0.5!(v1) + (-30:-0.4)$) {$x$};
\node at ($(v1)!0.5!(v2) + (30:0.4)$) {$y$};
\node at ($(v0)!0.5!(v2) + (0,-0.4)$) {$\Gamma^2_1(x, y)$};

\node at (0,0) {$0$};

\end{tikzpicture}
\]

\begin{lemma}\label{lemma:gamma gives group}
	The binary product $\Gamma^2_1(-,-)$ is associative, is unital with unit $0\in\g$~, and it satisfies $\Gamma^2_1(x, 0) = \Gamma^2_1(0, x) = x$ for all $x\in\g_1$~. Therefore, it defines a group structure on $\g_1$ which is natural with respect to strict morphisms of complete Lie algebras.
\end{lemma}

\begin{proof}
	We consider the in $\R(\g)$ obtained by joining the two $2$-simplices drawn in solid black on their common edge:
	\[
	\vcenter
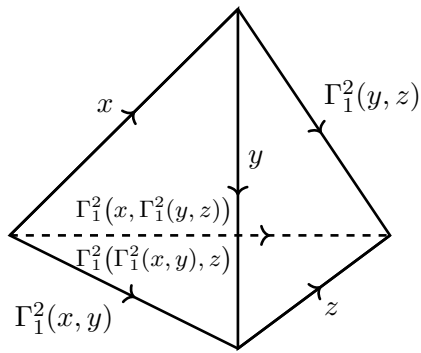
{\hbox{
			\begin{tikzpicture}
			
				\coordinate (a) at (0, 0);
				\coordinate (b) at (3, 3);
				\coordinate (c) at (3, -1.5);
				\coordinate (d) at (5, 0);
				
				\draw[line width=1] (a) to (b) to (c) to (d);
				\draw[line width=1] (b) to (a) to (c) to (d) to (b);
				\draw[line width=1, dashed] (a) to (d);
				
				\node[above left] at ($(a)!0.5!(b)$) {$x$};
				\node[above right] at ($(b)!0.5!(c)$) {$y$};
				\node[below right] at ($(c)!0.5!(d)$) {$z$};
				\node[below left] at ($(a)!0.5!(c)$) {$\Gamma^2_1(x, y)$};
				\node[above right] at ($(b)!0.5!(d)$) {$\Gamma^2_1(y, z)$};
				
				\node[above] at ($(a)!0.38!(d)$) {\Small $\Gamma^2_1\big(x, \Gamma^2_1(y,z)\big)$};
				\node[below] at ($(a)!0.38!(d)$) {\Small $\Gamma^2_1\big(\Gamma^2_1(x,y),z\big)$};
				
				\begin{scope}[decoration={
					markings,
					mark=at position 0.55 with {\arrow{>}}},
				line width=1
				]	
				\path[postaction={decorate}] (a) to (b);
				\path[postaction={decorate}] (a) to (c);
				\path[postaction={decorate}] ($(a)!0.3!(d)$) to (d);
				\path[postaction={decorate}] (b) to (c);
				\path[postaction={decorate}] (b) to (d);
				\path[postaction={decorate}] (c) to (d);
				\end{scope}
			
			\end{tikzpicture}}}
	\]
	where all the vertices and the $2$-faces are labeled by $0$~. We can complete the square to a $3$-horn $a:\Ho{1}{3}\to\R(\g)$ by filling the face of the missing $2$-simplex (i.e.\ the one opposed to the vertex $2$) by $0$~, thus obtaining $\Gamma^2_1\big(x, \Gamma^2_1(y,z)\big)$ on the missing edge. We can further fill this horn by $0$ and we notice that since $\g_2 = 0$ we will have $0$ on the missing face --- the one opposed to the vertex $1$~. In other words, we have $\Gamma^3_1(a) = 0$ But this together with the uniqueness of the fillers implies that if instead we started from the original square and filled the face opposed to the vertex $1$ by $0$ we would obtain the same element on the missing edge, and thus, we obtain
	\[
	\Gamma^2_1\big(x, \Gamma^2_1(y,z)\big) = \Gamma^2_1\big(\Gamma^2_1(x,y),z\big)
	\]
	as desired.
	
	\medskip
	
	The statement about the unit follows immediately from \cref{eq:BCH first terms}. The statement about inverses similarly follows from \cref{eq:BCH first terms} plus uniqueness of fillings.
\end{proof}

\begin{remark}
	The statements of \cref{lemma:gamma gives group} also hold for $\g$ an $\sLie$-algebra not concentrated in degree $1$~, even if we allow the vertices to be given by non-trivial Maurer--Cartan elements (in which case we obtain a groupoid, rather than a group). The arguments are the same, with the exception of the fact that $\Gamma^3_1(a) = 0$ that follows from degree reasons using the formula for higher BCH products together with the fact that we only allow binary brackets.
\end{remark}

In order to show that $\Gamma^2_1$ is actually equal to the Baker--Campbell--Hausdorff product $\BCH$~, we will use a new characterization that we detail below. The Baker--Campbell--Hausdorff formula 
in a complete Lie algebra $\g$ is a universal formula that associates 
an element $\BCH(x,y)\in \g$ to  
two elements $x,y\in\g$ such that
\begin{equation}\label{eq:definition BCH}
e^{\ad_{\BCH(x,y)}} = e^{\ad_x}e^{\ad_y}\ .
\end{equation}
This means that we are looking for the same kind of formula for all couples of elements in any Lie algebra. This problem is solved by looking for a particular element in the free complete (shifted) Lie algebra $\widehat{\sLie}(x,y)$ on the two generators $x$ and $y$ of degree $1$~. The original problem solved by Baker, Campbell and Hausdorff at the beginning of the 20th century was to show that the series 
\begin{equation}\label{eq:definition Bis BCH}
\BCH(x,y)\coloneqq\mathrm{ln}\left(e^x e^y \right)
\end{equation}
in the free complete associative algebra actually lives in $\widehat{\sLie}(x,y)$~. Some 40 years later, Dynkin gave the first explicit formula for it. We refer the reader to \cite{BF12} for an exhaustive historical and mathematical treatment of this problem. It is straightforward to see that $\BCH(x,y)$ satisfies \cref{eq:definition BCH}. 
In this section, we show that \cref{eq:definition BCH} characterizes the $\BCH$ formula \eqref{eq:definition Bis BCH}. 

\begin{proposition}\label{prop:uniqueness of classical BCH}
	There exists a unique element $\BCH(x,y)\in\widehat{\sLie}(x,y)$ satisfying (\ref{eq:definition BCH}).
\end{proposition}

The proof relies on the following lemma. 

\begin{lemma}\label{lemma:Center}\leavevmode
	\begin{enumerate}
		\item The center of the free Lie algebra $\sLie(x,y)$ and the center of the free complete Lie algebra $\widehat{\sLie}(x,y)$ are trivial.
		\item In both algebras, if $z$ is such that $[z,x]=[z,y]=0$~, then $z = 0$~.
	\end{enumerate}
\end{lemma}

\begin{proof}
	In the case of the free Lie algebra $\sLie(x,y)$~, we can apply the Shirshov--Witt theorem \cite{shirshov53} which states that every Lie subalgebra of a 	
	free Lie algebra is again free. Let  $a,b$ be two commuting elements in $\lie(x,y)$~. The Lie subalgebra generated by $a$ and $b$ is free and thus, $a$ and $b$ cannot be linearly independent. Therefore, no element except $0$ can be in the center of ${\lie}(x,y)$~.
	
	\medskip
	
	Let $z\in\sLie(x, y)$ and suppose that $[z,x] = 0$ and $[z,y] = 0$~. We will show that $z$ is in the center of $\sLie(x, y)$~, which implies that $z=0$ by the previous point. Let $u\in\sLie(x, y)$ be any element and remember that we can write it as a finite sum
	\[
	u=\sum_{n=1}^K u_k
	\]
	where $u_k$ is of weight $k$~, i.e.\ it is given by a linear combination of bracketings of exactly $k$ copies of $x$ and $y$~. We will show by induction on $k\geqslant1$ that $[z, u_k] = 0$~. For $k=1$~, we have $u_1 = ax + by$~, for some $a,b\in\k$~, and thus, we have
	\[
	[z,u_1] = a[z,x] + b[z,y] = 0\ .
	\]
	Now suppose that $z$ commutes with any $u_l$ of weight $l$~, for $l<k$~, and let $u_k\in\sLie(x,y)$ be of weight $k$~. We have
	\[
	u_k = \sum_{l=1}^{k-1}a_l[v_l,w_l]
	\]
	for some $a_l\in\k$~, $v_l$ of weight $l$~, and $w_l$ of weight $k-l$ (or linear combinations of such elements, for which the same proof holds). By the Jacobi rule, we have 
	\[
	[z,u_k]=\sum_{l=1}^{k-1}a_l[z,[v_k,w_k]]
	=\sum_{l=1}^{k-1}a_k\big( [[z,v_l],w_l] +[v_l, [z,w_l]]\big)=0 
	\ ,
	\]
	giving the induction step and concluding the proof.
	
	\medskip
	
	In the case of the free complete Lie algebra $\widehat{\sLie}(x,y)$~, let $z=\sum_{k\geqslant 1} z_k$ be an element satisfying $[z, x]=[z,y]=0$~. This time infinite sums are allowed. We have $[z_k, x]=[z_k,y]=0$ holds in $\sLie(x, y)$ for all $k\geqslant1$~, so that the same argument as above implies that $z_k=0$ for all $k\geqslant0$~, and thus that $z=0$~. In particular, the center of $\widehat{\sLie}(x,y)$ is trivial.
\end{proof}

\begin{proof}[Proof of \cref{prop:uniqueness of classical BCH}]	
	We suppose that there exists a second element
	\[
	\widetilde{\BCH}(x,y)\in\widehat{\sLie}(x,y)
	\]
	satisfying (\ref{eq:definition BCH}) and we consider 
	\[
	z\coloneqq\widetilde{\BCH}(x,y) - \BCH(x,y)\ .
	\]
	To show that $z=0$~, it is enough to prove, by induction on $k\geqslant 1$~, that $z_k$ commutes with $x$ and $y$ by \cref{lemma:Center}. We write $u$ for a placeholder for either $x$ or $y$~. 
	By \eqref{eq:definition BCH}, we have
	\begin{equation}\label{eq:compare BCH and BCH tilde}
	e^{\ad_{\BCH(x,y)}}(u) = e^{\ad_{\BCH(x,y)} + \ad_z}(u)\ .
	\end{equation}
	We trivially have
	\[
	\left(e^{\ad_{\BCH(x,y)}}(u)\right)_1 = u = \left(e^{\ad_{\BCH(x,y)} + \ad_z}(u)\right)_1.
	\]
	At the next level, we have
	\begin{align*}
		[\BCH(x,y)_1,u] ={}& \left(e^{\ad_{\BCH(x,y)}}(u)\right)_2\\
		={}& \left(e^{\ad_x}e^{\ad_y}(u)\right)_2\\
		={}&\left(e^{\ad_{\widetilde{\BCH}(x,y)}}(u)\right)_2\\
		={}& \left(e^{\ad_{\BCH(x,y) + z}}(u)\right)_2\\
		={}&[\BCH(x,y)_1,u] + [z_1,u]\ ,
	\end{align*}
	which implies $z_1 = 0$ by \cref{lemma:Center}.  We  suppose now that we have proven that $z_l = 0$ for all $l<k$~. Then the $(k+1)$th level of (\ref{eq:compare BCH and BCH tilde}) gives the equation
	\[
	\left(e^{\ad_{\BCH(x,y)}}(u)\right)_{k+1}=\left(e^{\ad_{\BCH(x,y)}}(u)\right)_{k+1}+ [z_k,u]\ ,
	\]
	which implies that $z_k = 0$~. It follows that $z = 0$~, concluding the proof.
\end{proof}

We now drop the assumption that $\g$ is concentrated in degree $1$~. This corresponds to differential graded Lie algebras after shifting. As at the end of \cref{subsec:HighLSalg}, since the higher operations $\ell_m=0$ vanish for $m\geqslant 3$ in this case, the only planar rooted trees contributing to the gauge action $\lambda\cdot \alpha$ on a Maurer--Cartan element $\alpha$ given in \cref{prop:gaugeformula}  are the ladders 
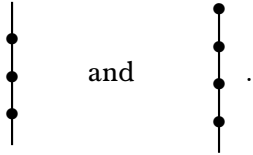
\[\vcenter{\hbox{
		\begin{tikzpicture}
		\node at (0,0.5) {$\bullet$};
		\node at (0,1) {$\bullet$};
		\node at (0,1.5) {$\bullet$};
		\draw[thick] (0,0.1) -- (0,2);
		\end{tikzpicture}}}
\qquad \text{and} \qquad
\vcenter{\hbox{
		\begin{tikzpicture}
		\node at (0,0.5) {$\bullet$};
		\node at (0,1) {$\bullet$};
		\node at (0,1.5) {$\bullet$};
		\node at (0,2) {$\bullet$};
		\draw[thick] (0,0.1) -- (0,2);
		\end{tikzpicture}}}\ .
\]
Thus, we get the formula 
\[
\lambda\cdot {\alpha}=\frac{e^{\ad_\lambda}-\id}{\ad_\lambda}(d\lambda)+e^{\ad_\lambda}(\alpha)\ .\]
for the gauge action. This formula can be simplified using the following \emph{differential trick}, see \cite{DotsenkoShadrinVallette16}.
To any $\sLie$-algebra $\g$~, we associate the central extension 
\[
\g^+\coloneqq\g\oplus\k\delta\ ,
\]
where $\delta$ has degree $0$~, such that $\mathrm{d}\delta=0$~, and where the brackets stay the same on elements of $\g$ and are extended to $\delta$ by
\[
[\delta,\delta] \coloneqq 0\qquad\text{and}\qquad[\delta,x] \coloneqq dx
\]
for $x\in\g$~. To any Maurer--Cartan element $\alpha\in\MC(\g)$~, we associate 
\[
\alpha^+\coloneqq\delta + \alpha\ .
\]
Under this convention, the gauge action in $\g^+$ becomes
\[
\lambda\cdot{\alpha^+} \coloneqq e^{\ad_\lambda}({\alpha^+})\ .
\]

For any another degree $0$ element $\mu\in \g_0$~, we have 
\[
\lambda\cdot(\mu\cdot{\alpha^+}) = e^{\ad_\lambda}e^{\ad_\mu}({\alpha^+}) = e^{\ad_{\BCH(\lambda,\mu)}}({\alpha^+}) = \BCH(\lambda,\mu)\cdot{\alpha^+}\ .
\]
Similarly to \cref{prop:uniqueness of classical BCH}, the following result states that there is a unique universal formula  satisfying the above property. 

\begin{proposition}\label{prop:uniqueness of dg BCH}
	In the free algebra
	\[
	\g=\widehat{\sLie}(\alpha,\lambda,\mu,d\lambda,d\mu)
	\]
	with $|\alpha|=0$~, $|\lambda|=|\mu|=1$~, and
	\[
	\mathrm{d}\alpha \coloneqq -\tfrac{1}{2}[\alpha,\alpha]\ ,\quad\mathrm{d}(\lambda)=d\lambda\ \quad\text{and}\quad\mathrm{d}(\mu)=d\mu\ ,
	\]
	the Baker--Campbell--Hausdorff element $\BCH(\lambda,\mu)$ is the unique element of the complete Lie sub-algebra $\widehat{\sLie}(\lambda, \mu)$ satisfying
	\begin{equation}\label{eq:define BCH dg}
	e^{\ad_\lambda}e^{\ad_\mu}({\alpha^+}) = e^{\ad_{\BCH(\lambda,\mu)}}({\alpha^+})
	\end{equation}
	in $\g^+$~. 
\end{proposition}

\begin{proof}
	It is obvious that the $\BCH$ formula satisfies Equation~\eqref{eq:define BCH dg}. To show that is unique, we proceed in the same way as in the proof of \cref{prop:uniqueness of classical BCH}: let $\widetilde{\BCH}(\lambda, \mu)$ be an element of $\widehat{\sLie}(\lambda, \mu)$ satisfying Equation~\eqref{eq:define BCH dg}. We consider the element $z:=\widetilde{\BCH}(\lambda,\mu)-\BCH(\lambda,\mu)$\ . By the same induction procedure as above, we get that  every weight component $z_k$ of $z$ commutes with $\alpha$ in $\g$~, that is $[z_k,\alpha]=0$~, for $k\geqslant 1$~.
	
	\medskip
	
	To conclude the proof, we now show that any element $u\in \sLie(\lambda, \mu)$ such that $[u, \alpha]=0$ in $\sLie(\alpha, \lambda, \mu)$ is trivial. Let $[u, \alpha]=0$ and suppose that $u\neq 0$~. In this case, the graded Lie subalgebra $\mathfrak{h}\subset \lie(\alpha, \lambda, \mu)$ generated by $u$ and $\alpha$ has global dimension $3$~, with basis $u$~, $\alpha$~, and $[\alpha,\alpha]$~. 
	The Shirshov--Witt theorem for graded Lie algebras of \cite{Mikhalev85, Shtern86} shows that  $\mathfrak{h}$ is also free. Since $|u|=1$ and $|\alpha|=0$~, this free graded Lie algebra admits at least two generators, one of degree $1$ and one of degree $0$~. This implies that its global dimension is at least $4$~, and thus  a contradiction. 
\end{proof}

\begin{theorem}[{\cite[Proposition~5.2.36]{Bandiera14}}]\label{prop:BCH}
	Let $\g$ be a complete $\sLie$-algebra. The canonical horn filler
	\begin{align*}
	\Hom_{\sSe}\left({\Ho{1}{2}}, \R(\g)\right){}&\ \longrightarrow\ \g_1\\
	\left(x_0, x_1, x_2, x_{01}, x_{12}	\right){}&\ \longmapsto\ \Gamma^2_1(x_{01}, x_{12})=\mathrm{BCH}(x_{01}, x_{12})
	\end{align*}
	is equal to the Baker--Campbell--Hausdorff formula.
	
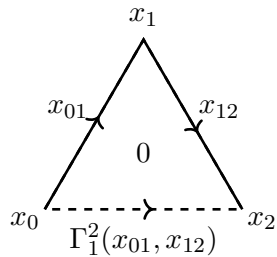
\[
	\begin{tikzpicture}
	
	\coordinate (v0) at (210:1.5);
	\coordinate (v1) at (90:1.5);
	\coordinate (v2) at (-30:1.5);
	
	\draw[line width=1, dashed] (v0)--(v1)--(v2)--cycle;
	\draw[line width=1] (v0)--(v1)--(v2);
	
	\begin{scope}[decoration={
		markings,
		mark=at position 0.55 with {\arrow{>}}},
	line width=1
	]
	
	\path[postaction={decorate}] (v0) -- (v1);
	\path[postaction={decorate}] (v1) -- (v2);
	\path[postaction={decorate}] (v0) -- (v2);
	
	\end{scope}
	
	
	\node at ($(v0) + (30:-0.3)$) {$x_0$};
	\node at ($(v1) + (0,0.3)$) {$x_1$};
	\node at ($(v2) + (-30:0.3)$) {$x_2$};
	
	\node at ($(v0)!0.5!(v1) + (-30:-0.4)$) {$x_{01}$};
	\node at ($(v1)!0.5!(v2) + (30:0.4)$) {$x_{12}$};
	\node at ($(v0)!0.5!(v2) + (0,-0.4)$) {$\Gamma^2_1(x_{01},x_{12})$};
	
	\node at (0,0) {$0$};
	
	\end{tikzpicture}
	\]
\end{theorem}

\begin{proof}
	It is enough to consider the case of the complete Lie algebra
	\[
	\g\coloneqq\widehat{\sLie}(\alpha,\lambda,\mu,d\lambda,d\mu)
	\]
	introduced above in \cref{prop:uniqueness of dg BCH}.
	The data
	$x_0\coloneqq \alpha^+$, 
	$x_{01}\coloneqq \mu$~, 
	$x_1\coloneqq \mu\cdot{\alpha^+}= e^{\ad_\mu}({\alpha^+})$~, 
	$x_{12}\coloneqq \lambda$~, and
	$x_2\coloneqq \lambda\cdot(\mu\cdot{\alpha^+}) = e^{\ad_\lambda}e^{\ad_\mu}({\alpha^+})$
	define a $\Ho{1}{2}$-horn in $\R(\g^+)$~. The canonical horn filler 
	$\Gamma^2_1(\lambda,\mu)$ is a gauge sending $\alpha^+$ to $e^{\ad_\lambda}e^{\ad_\mu}({\alpha^+})$~, that is 
	\begin{equation}
	e^{\ad_{\Gamma^2_1(\lambda,\mu)}}({\alpha^+})=e^{\ad_\lambda}e^{\ad_\mu}({\alpha^+}) \ .
	\end{equation}
	Since the degree of  $\Gamma^2_1(\lambda,\mu)$ is equal to $0$~, it must live in 
	$\widehat{\lie}(\lambda, \mu)$~. It is thus equal to the BCH formula by \cref{prop:uniqueness of dg BCH}.
\end{proof}

Let us denote by $\mathrm{PaPBinRT}$\index{trees!planarly partitioned binary rooted trees}\index{$\mathrm{PaPBinRT}$} the subset of planarly partitioned \emph{binary} rooted trees, given by the trees in $\PaPRT$ that only have vertices of arity $0$ or $2$~. 

\begin{proposition}
	The $\BCH$ formula is equal to 
	\[
	\BCH(x,y)=\Gamma^2_1(x, y) = 
	-\sum_{\substack{\tau\in\mathrm{PaPBinRT}\\
			\chi \ : \ \mathrm{L}(\tau) \to\left\{x,y\right\}}}\ 
	\prod_{\substack{\beta\text{ block of } \tau \\ \lambda^{\beta(\chi)}_{[2]}\neq 0}} 
	\frac{1}{\lambda^{\beta(\chi)}_{[2]}[\beta]!}\, 
	\ell_\tau\left(x, \ldots, y, \ldots, x, \ldots, y;  x+y\right), 
	\]
	where $01$ is identified with $x$ and $12$ with $y$~.
\end{proposition}

\begin{proof}
	This is straightforward application of \cref{prop:ExplicitBCHForm} and  \cref{prop:BCH}. 
\end{proof}

\begin{remark}
	Many terms in the above displayed formula cancel, for instance when two $x$ or two $y$ index the two leaves of a vertex. In the end, one should compare this formula with the tree-wise formulae for the BCH product given in \cite{Kathotia00, DonatellaManetti13}. 
\end{remark}

\begin{corollary}[{\cite[Theorem~5.2.37]{Bandiera14}}]
	In the case of complete Lie algebras concentrated in non-positive degree, the integration functor is isomorphic to the nerve of the Hausdorff group $G\coloneqq (\g_0, \BCH, 0)$:
	\[\R(\g)\cong \mathrm{N}(G)\ .\]
\end{corollary}

\begin{proof}
	This is a direct corollary of \cref{prop:BCH}. 
	Since the associated $\sLie$-algebras is concentrated in homological  degree less or equal to $1$~, the $n$-simplices of its $\infty$-groupoid $\R(\g)$ are made up of geometrical $n$-simplices with all the faces labeled by $0$ except for the edges. The principal edges $0\to 1\to 2\to \cdots \to n-1 \to n$ are labeled by $n$ elements $x_1, x_2, \ldots, x_n$ of $\g$ and other edges are labeled by the corresponding iterations of the $\BCH$-formula, according to \cref{subsec:GeoRep}. 
	Under this interpretation, the face maps are given by applying the $\BCH$-formula to the two elements labeling two consecutive principal edges and the degeneracy maps are given by the duplication of the labeling element of a principal edge. This description coincides with the definition of the nerve of the Hausdorff group.  
\end{proof}

\subsection{Properties of the higher Baker--Campbell--Hausdorff formulas}\label{subsec:properties higher BCH}

We explore some basic properties of the higher BCH formulas that will be useful later when we will study the homotopical behavior of our functors.

\medskip

Recall first that the symmetric group $\S_{n+1}\cong\operatorname{Aut}([n])$ acts on the left on $\Omega_n$ by
$
\sigma\cdot t_i\coloneqq t_{\sigma(i)}$~, which heuristically corresponds to permuting the vertices of the $n$-simplex. This action restricts to an action on $\rmC_n\subseteq\Omega_n$ which is given on the basis elements 
\[
\omega_I = k!\sum_{j=0}^k(-1)^jt_{i_j}dt_{i_0}\cdots\widehat{dt_{i_j}}\cdots dt_{i_k}\ , 
\]
for $\emptyset\neq I=\{i_0,\ldots, i_k\}\subset[n]$~, by 
$
\sigma\cdot \omega_I =(-1)^{\sigma|I}\omega_{\sigma(I)}$~, where the sign is given by the parity of the number of crossings of $\sigma$ restricted to $I$~.

\medskip
 
In the dual picture, writing $a_I\coloneqq \omega_I^\vee$ as usual, there is a right action 
$a_I \cdot \sigma = (-1)^{\sigma|I}a_{\sigma^{-1}(I)}$~. It induces a right action of $\S_{n+1}$ on $\mc^n$ and thus, a left action of $\S_{n+1}$ on the $n$-simplices $\R(\g)_n$ of the integration functor. 

\begin{proposition}\label{prop:BCH symmetry}
	Let $\g$ be a complete $\sLi$-algebra.	Any  $x\in\Hom_{\sSe}({\Ho{k}{n}}, \R(\g))$~,  $y\in\g_n$~, and $\sigma\in \S_{n+1}$ satisfy 
	\[
	\sigma\cdot x\in\Hom_{\sSe}\left({\Ho{\sigma(k)}{n}}, \R(\g)\right)
	\]
	and
	\[
	\Gamma^n_{\sigma(k)}\left(\sigma\cdot x;(-1)^{\sigma}y\right) = (-1)^{\sigma\mid\widehat{k}}\,\Gamma^n_k(x;y)\ ,
	\]
	where the second sign is given by the parity of the number of crossings of $\sigma$ restricted to $\widehat{k}$~.
\end{proposition}

\begin{proof}
	The first part of the statement immediately follows from the fact that the action of $\sigma$ maps bijectively $\Ho{k}{n}$ to $\Ho{\sigma(k)}{n}$~. 
	Regarding the second part, the uniqueness of the horn fillers of \cref{thm:KanExt} amounts to 
	\[
	\Gamma^n_k\left((h^n_k)^*z; z_{[n]}\right) = z_{\widehat{k}}~, 
	\]
	where $h^n_k:\Ho{k}{n}\xhookrightarrow{}\Delta^n$ is the inclusion of the $k$-horn into the full $n$-simplex and where $z\in\R(\g)_n$~. Gluing the filler $y$ and the result of the BCH product to $x$~, we obtain a full simplex $z\in\R(\g)_n$~.  The simplex $\sigma\cdot z$ has on each face $I$ the element on the face $\sigma^{-1}(I)$ of $x$ multiplied by the sign of $\sigma$ ``restricted to $I$''~. Thus, we conclude the proof by
	\begin{align*}
	\Gamma^n_{\sigma(k)}\left(\sigma\cdot x; (-1)^{\sigma}y\right)= (\sigma \cdot z)_{\widehat{k}}
	=(-1)^{\sigma\mid\widehat{k}}\,\sigma\cdot z_{\widehat{k}}
	=(-1)^{\sigma\mid\widehat{k}}\, \Gamma^n_k(x;y)\ .
	\end{align*}
\end{proof}

\begin{example}\label{ex:SymBCH}
	Consider the $2$-horn $\Ho{1}{2}$ in $\R(\g)$ with $0$ at the vertices, $\lambda\in\g_1$ on the interval $01$~, and $\mu\in\g_1$ on $12$~, and let $\rho\in\g_2$~. Applying the permutation $\sigma=(02)$~, we obtain 
	\[
	\Gamma^2_1(-\mu, -\lambda; -\rho) = -\Gamma^2_1(\lambda, \mu; \rho)\ .
	\]
	Notice that if $\g$ is a Lie algebra then $\rho=0$ and this recovers the symmetry of the  the classical BCH formula.
\end{example}

\begin{lemma}\label{lemma:first terms of BCH}
	Let $\g$ be a complete $\sLi$-algebra and let $x_0\in \MC(\g)$~, $x_{01}\in(\F_p\g)_1$~, $x_{12}\in(\F_q\g)_1$~, and $y\in(\F_r\g)_2$~. Consider the following horn:
	\[
	x = \vcenter
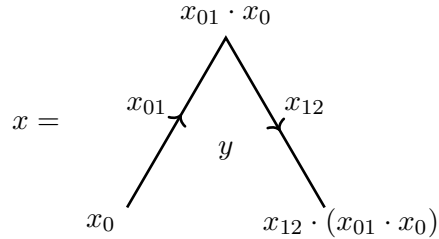
{\hbox{
	\begin{tikzpicture}
	
	\coordinate (v0) at (210:1.5);
	\coordinate (v1) at (90:1.5);
	\coordinate (v2) at (-30:1.5);
	
	\draw[line width=1] (v0)--(v1)--(v2);
	
	\begin{scope}[decoration={
		markings,
		mark=at position 0.55 with {\arrow{>}}},
	line width=1
	]
	
	\path[postaction={decorate}] (v0) -- (v1);
	\path[postaction={decorate}] (v1) -- (v2);
	
	\end{scope}
	
	
	\node at ($(v0) + (30:-0.4)$) {$x_0$};
	\node at ($(v1) + (0,0.3)$) {$x_{01}\cdot x_0$};
	\node at ($(v2) + (-30:0.4)$) {$x_{12} \cdot (x_{01}\cdot x_0)$};
	
	\node at ($(v0)!0.5!(v1) + (-30:-0.45)$) {$x_{01}$};
	\node at ($(v1)!0.5!(v2) + (30:0.45)$) {$x_{12}$};
	\node at (0:0) {$y$};
	\end{tikzpicture}
	}}
	\]
	The  BCH product satisfies 
	\[
	\Gamma^2_1(x;y)\equiv x_{01} + x_{12} - dy \mod \F_{\min(p,q,r) + 1}\g\ .
	\]
\end{lemma}

\begin{proof}
	This follows immediately from the fact that all of terms which appear between the parentheses on the right-hand side of \cref{eq:BCH first terms} contain at least one bracket of two elements among $x_{01}, x_{02}, y$~, and $dy$~. Thus, we have
	\[
	\Gamma^2_1(x;y) - x_{01} - x_{02} + dy \in \F_{2\min(p,q,r)}\g\subseteq \F_{\min(p,q,r)+1} \g
	\]
	since $p,q,r\geqslant 1$~. 
\end{proof}

The following examples of computation will be needed in \cref{subsec:Berglund}.

\begin{lemma}\label{lem:ArbresDansBCH}
	Let $\g$ be a complete $\sLi$-algebra, let $x:\Ho{n}{k}\to\R(\g)$ be a horn whose vertices are trivial, that is $x_i=0$ for $0\leqslant i\leqslant n$~, and let $y\in\g_n$ be any degree $n$ element.
	\begin{enumerate}
		\item The formula for $\Gamma^n_k(x;y)$ given in \cref{Eq:HighBCH} can only contain trees with vertices of arity up to $n$~.
		\item In the case $n=2$~, the term $\Gamma^2_k(x;y) - \Gamma^2_k(x)$ is equal to a series of composites of binary operations $\ell_2$ with inputs in $dy$ and in elements labeling the edges of $x$~. Moreover, in each of these composite of operations the element $dy$ appears at least once.
		\item In the case $n>2$~, the term $\Gamma^n_k(x;y) - \Gamma^n_k(x)$ is equal to a series of composites of binary operations $\ell_2$ with exactly one input $dy$ and the rest among the elements labeling the edges of $x$~. 
	\end{enumerate}
	For all $n\geqslant 2$~, if the elements of $x$ are all closed, then $\Gamma^n_k(x;y) - \Gamma^n_k(x)$ is exact.		
\end{lemma}

\begin{proof}
	All of the statements hold by degree reasons by \cref{prop:ExplicitBCHForm} as follows. 
	We know that $\Gamma^n_k(x;y)$ must yield an element of degree $n-1$~, that every 
	operations $\ell_m$ labeling the vertices have degree $-1$ and arity greater or equal to 2, and that the inputs at the leaves have degree greater or equal to $1$ since $x_i=0$~, for $0\leqslant i\leqslant n$~.
	\begin{enumerate}
		\item The degree of an element obtained from a tree is at least equal to the sum of the arities of all of its vertices minus the number of its vertices. Thus, any tree involving a vertex of arity at least $n+1$ will yield an element of degree at least $n$~, so such trees can not appear in the final result.
		\item For $n=2$~, all trees must be binary by the previous point. Furthermore, all of the inputs of the trees must have degree $1$ since otherwise the degree of the resulting element would be at least $2$~, and thus, the element $y$ can not appear. Finally, it is clear from \cref{Eq:HighBCH} that every term of $\Gamma^2_k(x;y)$ containing only elements labeling the edges of the horn will also appear in $\Gamma^2_k(x)$ and thus not appear in $\Gamma^2_k(x;y) - \Gamma^2_k(x)$~.
		\item For $n>2$~, once again any appearance of $y\in\g_n$ would yield an element of degree at least $n$~, so that $y$ can not be present in the result. As before, any term not involving $dy$ will cancel, so that it must appear at least once in every term. But since $|dy| = n-1 \geqslant 2$~, if $dy$ appears more than once then the result would have degree higher than $n-1$~, which is not possible. It follows that $dy$ needs to appear exactly once in each term. Thus, all of the other terms labeling the leaves and coming from $x$ must have degree $1$~. The degree of the composite of operations applied to $dy$ and these terms will be equal to
		\[
		n - 1 + \#\text{leaves} -1- \#\text{vertices}=n-1\ ,
		\]
		so that the tree must have one more leaf than vertices, which forces it to be binary.
	\end{enumerate}
	To see that $\Gamma^n_k(x;y) - \Gamma^n_k(x)$ is closed, it is enough to notice that since all trees involved are binary the differential $d$ distributes over them, and the appearances of $dy$ in each term makes them exact.		
\end{proof}

\begin{lemma}\label{subsect:computation for Berglund}
	Let $\g$ be a complete $\sLi$-algebra and let $x\in Z_1(\g)$ be closed. For any  $y\in\g_2$~, the Baker--Campbell--Hausdorff formula $\Gamma^2_1(0,x;y)$ obtained by filling the horn
	
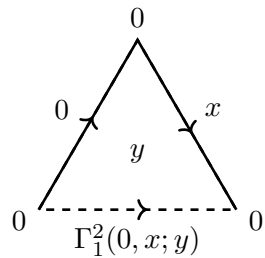
\[
	\begin{tikzpicture}
	
	\coordinate (v0) at (210:1.5);
	\coordinate (v1) at (90:1.5);
	\coordinate (v2) at (-30:1.5);
	
	\draw[line width=1, dashed] (v0)--(v1)--(v2)--cycle;
	\draw[line width=1] (v0)--(v1)--(v2);
	
	\begin{scope}[decoration={
		markings,
		mark=at position 0.55 with {\arrow{>}}},
	line width=1
	]
	
	\path[postaction={decorate}] (v0) -- (v1);
	\path[postaction={decorate}] (v1) -- (v2);
	\path[postaction={decorate}] (v0) -- (v2);
	
	\end{scope}
	
	
	\node at ($(v0) + (30:-0.3)$) {$0$};
	\node at ($(v1) + (0,0.3)$) {$0$};
	\node at ($(v2) + (-30:0.3)$) {$0$};
	
	\node at ($(v0)!0.5!(v1) + (-30:-0.4)$) {$0$};
	\node at ($(v1)!0.5!(v2) + (30:0.4)$) {$x$};
	\node at ($(v0)!0.5!(v2) + (0,-0.4)$) {$\Gamma^2_1(0,x;y)$};
	
	\node at (0,0) {$y$};
	
	\end{tikzpicture}
	\]
	satisfies 
	$\Gamma^2_1(0,x;y)-x \in B_1(\g)$~, that is there exists $z\in \g_2$ such that  $\Gamma^2_1(0,x;y)-x=\d z$~. 
	Conversely, if we are given $z\in\g_2$~, there exists an element $y\in\g_2$  such that
	$
	\Gamma^2_1(0,x;y)-x  =  \d z$~. 
\end{lemma} 

\begin{proof}
	The first point follows from Point~(2) of \cref{lem:ArbresDansBCH}  and the fact that $\Gamma^2_1(0, x) = x$~.
	
	\medskip
	
	Regarding the second point, we have to solve the equation $\Gamma^2_1(0,x;y)-x = dz$ where $y$ is unknown. The left-hand side of this equation is given by
	\[
	\Gamma^2_1(0,x;y)-x = -dy  + \sum_{m\geqslant 1}\PP_m\left(dy^{\otimes m}\right),
	\]
	where $\PP_m$ is the operator obtained by considering the sum of all terms of $\Gamma^2_1(0,x;y)-x$~, except $-dy$ for $m=1$~, with $dy$ appearing exactly $m$ times, and taking the corresponding leaves as inputs. Notice that these are all binary trees and all of the other leaves are decorated by $x$~. Since $x$ is closed and since $d$ distributes over binary trees, we have
	\[
	\Gamma^2_1(0,x;y)-x = d\left(-y  + \sum_{m\geqslant 1}\PP_m\left(y\otimes dy^{\otimes (m-1)}\right)\right).
	\]
	Therefore, it is enough to solve the formal fixed-point equation
	\begin{equation}\label{eq:fixed pt computation}
	y = -z  + \sum_{m\geqslant 1}\PP_m\left(y\otimes dy^{\otimes (m-1)}\right).
	\end{equation}
	In order to apply \cref{prop:FixPtEqua}, we need to show that each of the operation $\PP_m$ on the right-hand side raises the filtration degree by $1$~. But any binary tree that does not have any $x$ decorating their leaves will have a vertex giving $\ell_2(dy, dy) = 0$ (by symmetry since $dy$ has degree $1$) in $\Gamma^2_1(0,x;y)-x$~, and thus, it can be dropped from the formulas, so that the statement follows. Thus, an application of \cref{prop:FixPtEqua} yields a solution $y$ to \cref{eq:fixed pt computation} and completes the proof.
\end{proof}
\section{Homotopy theory}\label{sect:homotopy theory}

The purpose of this section is to establish the main homotopical properties for the adjoint pair of functors $\Li$ and $\R$ governing higher Lie theory. First, using the properties of the higher BCH formulas, we are able to give a short proof of a functorial version of Berglund's Hurewicz type theorem, which  computes the homotopy groups of the algebraic $\infty$-groupoid $\R(\g)$~. Then, we prove the homotopy invariance property of the integration functor $\R$ with respect to filtered $\infty_\pi$-morphisms and, on a similar vein, we study when deformations problems embed into one another. Finally, we settle a new model category structure on complete $\sLi$-algebras for which $\Li$ and $\R$ form a Quillen adjunction.

\subsection{Berglund's Hurewicz type theorem}\label{subsec:Berglund}
The algebraic $\infty$-groupoid $\R(\g)$ of a complete $\sLi$-algebras share a particularly rare feature: their homotopy groups are computable as homology groups associated to $\g$~. In this section, we give another, simpler proof of this result, originally due to Berglund \cite{ber15} on the level of Getzler spaces, by using the particular properties of the integration functor $\R$ established in this article. In fact, we establish a version of this result that is functorial with respect to $\infty_\pi$-morphisms.

\medskip

On the level of $\R(\g)$~, the 
relation between the homotopy groups and the homology groups is given
by the Berglund map
\begin{align*}
\mathcal{B}_n\ :\ {H}_n(\g^\alpha){}&\ \stackrel{\cong}{\longrightarrow}\ \pi_n\left(\R(\g), \alpha\right)\\
[u]{}&\ \longmapsto\ \left[u\otimes \omega_{[n]}+\sum_{i=0}^n \alpha\otimes \omega_i\right] ,
\end{align*}
where $\g^\alpha$ denotes the complete $\sLi$-algebra twisted by the Maurer--Cartan element $\alpha$~, see \cref{prop:TwiProc}, and 
where we implicitly use the canonical isomorphism of \cref{thm:isomorphism of models} to identify $\R(\g)$ with $\MC(\g\widehat{\otimes} \mathrm{C}_\bullet)$~.

\begin{theorem}[{\cite[Theorem~1.1]{ber15}}]\label{thm:Berglund}
	For any complete $\sLi$-algebra $\g$ and any Maurer--Cartan element $\alpha \in \MC(\g)$~, the Berglund map is well-defined and gives a group isomorphism
	\[
	\pi_n(\R(\g), \alpha)\cong {H}_n(\g^\alpha)\ , \quad \text{for}\  \ n\geqslant 1\ ,
	\]
	where the group structure on the right-hand side is given by the Baker--Campbell--Hausdorff formula  for $n=1$  and by the sum  for $n\geqslant 2$~. Moreover, it is  natural with respect to $\infty_\pi$-morphisms: for any $\infty_\pi$-morphism $f\colon \g \rightsquigarrow\h$~, the following diagram commutes
	\[
	\vcenter{\hbox{
			\begin{tikzpicture}
			\node (a) at (0,0){$H_n(\g^\alpha)$};
			\node (b) at (5,0){$\pi_n(\R(\g),\alpha)$};
			\node (c) at (0,-2){$H_n\left(\h^{\upsilon^*f(\alpha)}\right)$};
			\node (d) at (5,-2){$\pi_n(\R(\h), \upsilon^*f(\alpha))$};
			
			\draw[->] (a) to node[above]{$\scriptstyle{\mathcal{B}_n(\g)}$} (b);
			\draw[->] (a) to node[left]{$\scriptstyle{H_n\big((\upsilon^*f)^\alpha_1\big)}$} (c);
			\draw[->] (b) to node[right]{$\scriptstyle{\pi_n(\R(f)^\alpha)}$} (d);
			\draw[->] (c) to node[above]{$\scriptstyle{\mathcal{B}_n(\h)}$} (d);
			\end{tikzpicture}
	}}
	\]
	where $\upsilon^*f$ is the $\infty$-morphism associated to $f$~. 
\end{theorem}

\begin{proof}
	Let us first treat the case $\alpha=0$~. 
	We  notice that simplicial maps $\sigma : \De{n} \to \R(\g)$ satisfying 
	the following commutative diagram
	\[
	\vcenter{\hbox{
			\begin{tikzpicture}
			\node (a) at (0,0){$\partial\De{n}$};
			\node (b) at (2,0){$\De{n}$};
			\node (c) at (0,-1.5){$*$};
			\node (d) at (2,-1.5){$\R(\g)$};
			
			\draw[right hook->] (a) to (b);
			\draw[->] (a) to (c);
			\draw[->] (b) to node[right]{$\scriptstyle\sigma$} (d);
			\draw[->] (c) to node[above]{$\scriptstyle0$} (d);
			\end{tikzpicture}
	}}
	\]
	are in natural bijection with the $n$-cycles ${Z}_n(\g)$ of $\g$~.
	\cref{thm:isomorphism of models} shows that such  maps $\sigma$ are in one-to-one correspondence with  Maurer--Cartan elements  $u\otimes \omega_{[n]}\in \MC \big(\g\widehat{\otimes} \mathrm{C}_n\big)$~, where $u\in \g_n$~. Since the top form $\omega_{[n]}$ squares to $0$ in the commutative algebra $\Omega_n$~, for any $n\geqslant 1$~, the Maurer--Cartan equation for $u\otimes \omega_{[n]}$ is just $du\otimes \omega_{[n]} = d(u\otimes \omega_{[n]})=0$~, see the end of the proof of \cref{lemm:Fundamental} for more details.
	
	\medskip
	
	In order to conclude that the Berglund map $\mathcal{B}_n$ is well-defined and bijective, it is enough to show that to two $n$-simplices $\sigma, \tau : \De{n} \to \R(\g)$ as above are homotopic relative to the boundary $\partial \De{n}$ if and only if their associated cycles $u, v \in {Z}_n(\g)$ are homologous, i.e.\ $v-u=dw$~, for a $w\in \g_{n+1}$~.
	
	\medskip
	
	The cases $n=1$ and $n\geqslant2$ are based on the same idea, but must be proven in a slightly different way. We begin with the case $n=1$~, where we will rely on the results of \cref{subsect:computation for Berglund}. Suppose that $\sigma$ and $\tau$ are homotopic relative to $\partial\Delta^1$~. Graphically, this means
	
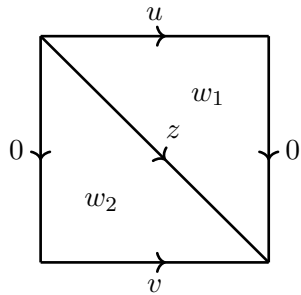
\[
	\begin{tikzpicture}
	\begin{scope}[decoration={
		markings,
		mark=at position 0.55 with {\arrow{>}}},
	line width=1
	] 
	\coordinate (a) at (0,0);
	\coordinate (b) at (3,0);
	\coordinate (c) at (0,-3);
	\coordinate (d) at (3,-3);
	
	\draw[postaction={decorate}] (a) to node[above=0.075]{$u$} (b);
	\draw[postaction={decorate}] (b) to node[right=0.075]{$0$} (d);
	\draw[postaction={decorate}] (a) to node[left=0.075]{$0$} (c);
	\draw[postaction={decorate}] (c) to  node[below=0.075]{$v$} (d);
	\draw[postaction={decorate}] (a) to node[above right]{$z$} (d);
	
	\node at (.8,-2.2){$w_2$};
	\node at (2.2,-.8){$w_1$};
	\end{scope}
	\end{tikzpicture}
	\]
	The first part of \cref{subsect:computation for Berglund} and the symmetry property of \cref{prop:BCH symmetry} show that we have
	$\Gamma^2_1(u,0;w_1)-u=dw_1$~, and thus
	\[
	[u] = \left[\Gamma^2_1(u,0;w_1)\right] = [z] = \left[\Gamma^2_1(0,v;w_2)\right] = [v]\ .
	\]
	Conversely, suppose that $v = u + d z$~. By the second part of \cref{subsect:computation for Berglund}, there exists a $w\in\g_2$ such that
	\[
	\Gamma^2_1(0,u;w)= v\ .
	\]
	Thus, we get a homotopy by 
	
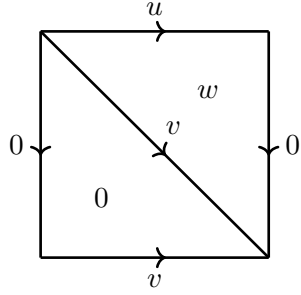
\[
	\begin{tikzpicture}
	\begin{scope}[decoration={
		markings,
		mark=at position 0.55 with {\arrow{>}}},
	line width=1
	] 
	\coordinate (a) at (0,0);
	\coordinate (b) at (3,0);
	\coordinate (c) at (0,-3);
	\coordinate (d) at (3,-3);
	
	\draw[postaction={decorate}] (a) to node[above=0.075]{$u$} (b);
	\draw[postaction={decorate}] (b) to node[right=0.075]{$0$} (d);
	\draw[postaction={decorate}] (a) to node[left=0.075]{$0$} (c);
	\draw[postaction={decorate}] (c) to  node[below=0.075]{$v$} (d);
	\draw[postaction={decorate}] (a) to node[above right]{$v$} (d);
	
	\node at (.8,-2.2){$0$};
	\node at (2.2,-.8){$w$};
	\end{scope}
	\end{tikzpicture}
	\]
	For the case $n\geqslant2$~, notice first that an element
	\[
	z\otimes \omega_{[n+1]}+u\otimes \omega_{\hat{k}}+v\otimes \omega_{\hat{l}}\in\g\otimes \mathrm{C}_{n+1}\ ,
	\]
	for $k\neq l$~, is a Maurer--Cartan element if and only if 
	\[
	du =dv=0\qquad\text{and}\qquad (-1)^n dz+(-1)^{k}u+(-1)^{l}v=0\ .
	\]
	This comes from the fact that any product in $\Omega_{n+1}$ of two forms among $\{\omega_{[n+1]}, \omega_{\hat{k}}, \omega_{\hat{l}}\}$ vanishes; so the Maurer--Cartan equation reduces to 
	\[
	d \left(z\otimes \omega_{[n+1]}+u\otimes \omega_{\hat{k}}+v\otimes \omega_{\hat{l}} \right)=0\ .
	\]
	From right to left, suppose that there exists $z\in \g_{n+1}$ such that $v-u=(-1)^n dz$~. We construct a homotopy $\mathrm{H} : \De{n}\times \De{1} \to \R(\g)$ from $\sigma$ to $\tau$ as follows. The prismatic decomposition  of $\De{n}\times \De{1}$ is a union a $n+1$ copies of $\De{n+1}$ given respectively by 
	\begin{align*}
	\langle 0\rangle&\coloneqq \langle 0 0, 0 1, 1 1, 2 1, \ldots, n  1\rangle\\
	& \ \ \vdots\\
	\langle i\rangle&\coloneqq\langle 00, 10, \ldots,  (i-1) 0, i 0, i  1, (i+1) 1, 
	\ldots, n 1\rangle\\
	& \ \ \vdots \\
	\langle n\rangle&\coloneqq\langle 0 0, 0 1, \ldots n  0, n  1\rangle\ .
	\end{align*}
	The image of $\langle 0\rangle$ under $\mathrm{H}$ is equal to $z\otimes \omega_{[n+1]}+ u \otimes \omega_{\hat{0}} + v \otimes \omega_{\hat{1}}$ and the image of $\langle i\rangle$ under $H$ is equal to 
	$v \otimes \omega_{\hat{i}} + v \otimes \omega_{\hat{i+1}}$~, for any $1\leqslant i \leqslant n$~. The above remark shows that that this assignment defines a simplicial map, which is a homotopy from $\sigma$ to $\tau$ relative to the boundary $\partial \De{n}$~. 
	
	\medskip
	
	In the other way round, let $\mathrm{H} : \De{n}\times \De{1} \to \R(\g)$ be a homotopy relative to the boundary $\partial \De{n}$ from $\sigma$ to $\tau$~. Let us denote by $\theta_i\in \g\widehat{\otimes} \mathrm{C}_{n+1}$ the image of $\langle i \rangle$ under $\mathrm{H}$~. Since the boundary is sent to $0$~, we are left with 
	\begin{align*}
	\theta_0={}& z_0 \otimes \omega_{[n+1]}+ v \otimes \omega_{\hat{0}} + \left(u+(-1)^n dz_0\right)\otimes \omega_{\hat{1}}\\
	\theta_1={}&z_1 \otimes \omega_{[n+1]}+ (u+(-1)^n d z_0) \otimes \omega_{\hat{1}} + \left(u+(-1)^n(d z_0-dz_1)\right)\otimes \omega_{\hat{2}}\\
	& \ \ \vdots \\
	\theta_{n-1}={}&z_{n-1} \otimes \omega_{[n+1]}+  \left(u+(-1)^n(dz_0-dz_1+\cdots \pm dz_{n-2})\right) \otimes \omega_{\widehat{n-1}} + \\
	{}& +\left(u+(-1)^n(dz_0-dz_1+\cdots \pm dz_{n-2})\right)\otimes \omega_{\hat{n}}\\
	\theta_{n}={}&w_{n} \otimes \omega_{[n+1]}+ \left(u+(-1)^n(dz_0-dz_1+\cdots \pm dz_{n-2})\right) \otimes \omega_{\hat{n}} +\\
	{}&+ (\underbrace{u+(-1)^n(dz_0-dz_1+\cdots \pm dz_{n}}_{=v})\otimes \omega_{\widehat{n+1}}\ ,
	\end{align*}
	which proves that $u$ and $v$ are homologous.
	
	\medskip
	The general case of any Maurer--Cartan element $\alpha\in \MC(\g)$ is treated with the results of \cref{prop:TwiProc}. The argument mentioned above applies to the twisted complete $\sLi$-algebra $\g^\alpha$ and its Maurer--Cartan element $0$~. So we have a well-defined bijection
	\begin{align*}
	H_n(\g^\alpha){}&\ \stackrel{\cong}{\longrightarrow}\ \pi_n\left(\R(\g^\alpha), 0\right)\\
	[u]{}&\ \longmapsto\ \left[u\otimes \omega_{[n]}\right] 
	\end{align*}
	It is enough to compose it with the bijection $\pi_n(\R(\g^\alpha), 0)\cong\pi_n(\R(\g), \alpha)$ 
	of \cref{cor:change base point} to conclude that the Berglund map is always well-defined and bijective. 
	
	\medskip
	
	Regarding the compatibility with the group structures, we treat in detail only  the case $\alpha=0$~. 
	The general case is obtained by the same arguments as above since the bijection $\pi_n(\R(\g^\alpha), 0)\cong\pi_n(\R(\g), \alpha)$ comes from an isomorphism of simplicial sets $\R(\g^\alpha)\cong\R(\g)$~, see \cref{thm:isom translation by alpha},  and is thus a group morphism. 
	Let $u,v \in \g_n$ be two representatives of two homology classes of $H_n(\g)$ and let $\sigma,\tau : \De{n}\to \R(g)$ be two representatives of the  corresponding images under the Berglund map, i.e.\ $u\otimes \omega_{[n]}$ and $v\otimes \omega_{[n]}$~. A representative for the product of $[\sigma]$ and $[\tau]$ is given by the unique horn filler $\De{n+1} \to \R(\g)$ corresponding to the horn $\Ho{n}{n+1} \to \R(\g)$ given by
	\[u \otimes \omega_{\widehat{n-1}}+v \otimes \omega_{\widehat{n+1}}\ , 
	\]
	where the top cell is filled by $0$~. We claim that this horn filler is equal to 
	\[u \otimes \omega_{\widehat{n-1}}+
	(u+v) \otimes \omega_{\widehat{n}}+
	v \otimes \omega_{\widehat{n+1}}\ , 
	\]
	for $n\geqslant 2$~. Since the product of any pair coming from $\{\omega_{\widehat{n-1}}, \omega_{\widehat{n}}, \omega_{\widehat{n+1}}\}$ vanishes in  $\Omega_{n+1}$~, this latter satisfies the Maurer--Cartan equation
	in $\g\widehat{\otimes} \rmC_{n+1}$~, which is given by 
	\[d\left(u \otimes \omega_{\widehat{n-1}}+
	(u+v) \otimes \omega_{\widehat{n}}+
	v \otimes \omega_{\widehat{n+1}}\right)=0\ . 
	\]
	We conclude by the uniqueness of horn fillers, see \cref{thm:KanExt}. 
	For $n=1$~, we get 
	\[u \otimes \omega_{\widehat{0}}+
	\BCH(u,v) \otimes \omega_{\widehat{1}}+
	v \otimes \omega_{\widehat{2}}\ ,
	\]
	by \cref{prop:BCH}.
	This proves that the Berglund map is morphism of groups: abelian for $n\geqslant 2$ and with the BCH formula as product for $n=1$~.
	
	\medskip
	
	Finally, regarding the  naturality with respect to $\infty_\pi$-morphisms, we have to prove that the following diagram is commutative:
	\[
	\vcenter{\hbox{
			\begin{tikzpicture}
			\node (a) at (0,0){$H_n(\g^\alpha)$};
			\node (b) at (4.5,0){$\pi_n(\R(\g), \alpha)$};
			\node (c) at (9,0){$\pi_n\big(\R\big(\g^\alpha\big), 0\big)$};
			\node (d) at (0,-2){$H_n\big(\h^{\upsilon^*f(\alpha)}\big)$};
			\node (e) at (4.5,-2){$\pi_n(\R(\h), \upsilon^*f(\alpha))$};
			\node (f) at (9,-2){$\pi_n\big(\R\big(\h^{\upsilon^*f(\alpha)}\big), 0\big)$};
			
			\draw[->] (a) to node[above]{$\scriptstyle{\mathcal{B}_n(\g)}$} (b);
			\draw[->] (b) to node[above]{$\scriptstyle{\cong}$} (c);
			\draw[->] (d) to node[above]{$\scriptstyle{\mathcal{B}_n(\h)}$} (e);
			\draw[->] (e) to node[above]{$\scriptstyle{\cong}$} (f);
			\draw[->] (a) to node[left]{$\scriptstyle{H_n\big((\upsilon^*f)^\alpha_1\big)}$} (d);
			\draw[->] (c) to node[right]{$\scriptstyle{\pi_n(\R(f)^\alpha)}$} (f);
			\end{tikzpicture}
	}}
	\]
	On elements, this means
	\[
	\vcenter{\hbox{
			\begin{tikzpicture}
			\node (a) at (0,0){$[u]$};
			\node (b) at (4.5,0){$\left[u\otimes \omega_{[n]}+\sum_{i=0}^n \alpha\otimes \omega_i\right]$};
			\node (c) at (9,0){$\left[u\otimes \omega_{[n]}\right]$};
			\node (d) at (0,-2){$[v]$};
			\node (e) at (4.5,-2){$\left[v\otimes \omega_{[n]}+\sum_{i=0}^n \upsilon^*f(\alpha)\otimes \omega_i\right]$};
			\node (f) at (9,-2){$\left[v\otimes \omega_{[n]}\right]$};
			
			\draw[|->] (a) to (b);
			\draw[|->] (b) to (c);
			\draw[|->] (d) to (e);
			\draw[|->] (e) to (f);
			\draw[|->] (a) to (d);
			\draw[|->] (c) to  (f);
			\end{tikzpicture}
	}}
	\]
	where
	\[
	v=\sum_{m\geqslant 0}\sum_{\substack{\tau \in \PaRT_{m+1} \\|\tau|=0}}\tfrac{1}{m!}f_\tau(\alpha, \ldots, \alpha, u)~.
	\]
	On the right-hand side of the diagram, let us denote the image of $u\otimes \omega_{[n]}$ under $\R(f)^\alpha$ by $w\otimes \omega_{[n]}$~. 
	By unwinding the definitions of \cref{subsec:FuncinftyPiMorph,sec:TwisProcSimRep}, it is straightforward to see that 
	\[
	w=\sum_\tau \frac{1}{l!m!\lambda_{\tau,m,l}}f_\tau\big(\underbrace{\alpha,\ldots,\alpha}_{m\text{ times}}, \underbrace{u,\ldots,u}_{l\text{ times}}\big)\ ,
	\]
	where the sum runs over all combinations of $\tau\in\PaRT_{m+l}$ and integers $m, l$~,  such that
	\[
	\gamma_{C_n}\big(\tau\otimes\underbrace{\alpha\otimes\cdots\otimes\alpha}_{m\text{ times}}\otimes \underbrace{u\otimes\cdots\otimes u}_{l\text{ times}}\big) = \lambda_{\tau,m,l}\omega_{[n]}
	\]
	in the transferred $\Omega\Bar\com$-algebra structure on $C_n$~, with $\lambda_{\tau,m,l}\neq0$~, where $1=\omega_0+\cdots+\omega_n$~. Since
	\[
	i_n\big(\omega_{[n]}\big)i_n\big(\omega_{[n]}\big) = 0\qquad\text{and}\qquad h_ni_n\big(\omega_{[n]}\big) = 0\ ,
	\]
	it follows that only the partitioned rooted trees for which exactly one vertex is contained in each block can appear, they can appear only with $l=1$~, and their result will be $\omega_{[n]}$~, with coefficient $1$~. 
	This proves that $w=v$ and concludes the proof.
\end{proof}

\begin{remark}
	We actually expect the Berglund isomorphism to respect a richer structure than just the group structure. Both sides of the isomorphism have an $\sLi$-algebra structure, or at least something very close to it: the homology of $\g$ carries a proper $\sLi$-algebra structure obtained through the homotopy transfer theorem from the structure of $\g$ \cite[Section~10.3]{LodayVallette12},  while the homotopy groups of $\R(\g)$ can be endowed with the higher Whitehead--Samelson products, see \cite{Porter65}. Both structures depend on some choices and we expect that they can be made in a coherent fashion so that the Berglund isomorphism lifts at least to an $\infty_\pi$-isomorphism between them.
\end{remark}

\subsection{Homotopy invariance of the integration functor}
We now show that the functor $\R$ is homotopy invariant, in the sense that it sends $\infty_\pi$-quasi-isomorphisms, or equivalently zig-zags of quasi-isomorphisms of complete  $\sLi$-algebras, to weak equivalences of simplicial sets. This result is analogous to the main result of \cite{DolgushevRogers15}, which shows a similar homotopy invariance for the functor $\MC_\bullet$ with respect to $\infty$-morphisms. The strategy of the proof is the same as in \cite[Section~3]{DolgushevRogers15}. However, the tools developed previously in this paper, like the higher BCH products, allow us to significantly simplify the proof. 

\medskip

Recall that the first component $f_| \colon \g \to \h$ of an $\infty_\pi$-morphism, where $|$ stands for the trivial tree, is a chain map; the same holds true for the first component $f_1\colon \g \to \h$ of an $\infty$-morphism. 

\begin{definition}[Filtered $\infty_\pi$-quasi-isomorphism]
	A \emph{filtered $\infty_\pi$-quasi-isomorphism} $f:\g\rightsquigarrow\h$ is an $\infty_\pi$-quasi-isomorphism such that its first component $f_{|}$ induces quasi-isomorphisms
	\[
	f_{|}^{(n)}:\g/\F_n\g\stackrel{\sim}{\longrightarrow}\h/\F_n\h
	\]
	for each $n\geqslant 1$~. The analogous definition for $\infty$-quasi-isomorphisms obtained by replacing $f_{|}$ by $f_1$ yields the notion of a \emph{filtered $\infty$-quasi-isomorphism}.
\end{definition}

\begin{remark}
	Here, as in the rest of the present paper, all $\infty_\pi$-morphisms are assumed to be compatible with the filtrations. Being a \emph{filtered} $\infty_\pi$-quasi-isomorphism is an additional compatibility between the fact of being an $\infty_\pi$-\emph{quasi-isomorphism} and the filtrations.
\end{remark}

Since $\left(\upsilon^*f\right)_1=f_|$~,  an $\infty_\pi$-morphism $f$ is a filtered $\infty_\pi$-quasi-isomorphism if and only if $\upsilon^*f$ is a filtered $\infty$-quasi-isomorphism.

\begin{theorem}[Homotopy invariance]\label{thm:HoInvariance}
	Any filtered $\infty_\pi$-quasi-isomorphism $f : \g\stackrel{\sim}{\rightsquigarrow} \h$ of complete $\sLi$-algebras induces a weak equivalence of simplicial sets $\R(f):\R(\g)\stackrel{\sim}{\to} \R(\h)$~.
\end{theorem}

\begin{proof}
	The proof is divided in two parts. Let us first deal with $\pi_n(\R(f))$~, for $n\geqslant 1$~. 
	In order to be able apply \cref{thm:Berglund}, we need to prove that for any 
	Maurer--Cartan element $\alpha\in\MC(\g)$~, the twisted morphism $(\upsilon^* f)^\alpha \colon \g^\alpha \rightsquigarrow \h^{\upsilon^*f(\alpha)}$ is an $\infty$-quasi-isomorphism. To prove this, we apply Eilenberg--Moore Comparison Theorem \cite[Theorem~5.5.11]{WeibelBook} to the chain map $(\upsilon^* f)_1^\alpha \colon \g^\alpha \to \h^{\upsilon^*f(\alpha)}$ with respect to the underling complete and exhaustive filtrations. At the $0$th page, we get
	\[
	E^0_{p,q}\big((\upsilon^* f)_1^\alpha\big)= \mathrm{gr}(\upsilon^* f)_1 \colon 
	\F_p \g_{p+q}/ \F_{p+1} \g_{p+q}\to \F_p \h_{p+q}/ \F_{p+1} \h_{p+q}~.
	\]
	A direct application of the 5-lemma shows that this is a quasi-isomorphism. So the spectral sequence converges and implies that $(\upsilon^* f)_1^\alpha$ is a quasi-isomorphism. Given this, \cref{thm:Berglund}  shows that 
	\[
	\pi_n(\R(f))\ \colon\ \pi_n(\R(\g),\alpha) \stackrel{\cong}{\longrightarrow} \pi_n(\R(\h),\upsilon^*f(\alpha))
	\] 
	is an isomorphism for any $\alpha\in \MC(\g)$ and any $n\geqslant 1$~. 
	
	\medskip
	
	The second part of the proof deal with the $0$th homotopy group. For better readability we split this further into two parts: injectivity in \cref{lemma: pi0f injective} and surjectivity in \cref{lemma: pi0f surjective}, see below.
\end{proof}

In order to prove these two lemmas, we will need the following result. 

\begin{proposition}\label{lemma:sequence of gauges}
	Let $\g$ be a complete $\sLi$-algebra and suppose we have two sequences $\{\alpha_k\}_{k\geqslant 1}$ in $\MC(\g)$ and $\{\lambda_k\}_{k\geqslant 2}$ in $\g_1$ such that:
	\begin{enumerate}
		\item $\alpha_k\in\F_k\g$~,
		\item $\lambda_{k+1}\in\F_k\g$~, and
		\item $\alpha_{k+1} = \lambda_{k+1}\cdot \alpha_k$~.
	\end{enumerate}
	We consider the sequence $\{\zeta_k\}_{k \geqslant 2}$ in $\g_1$ defined by $\zeta_2 \coloneqq \lambda_2$ and inductively by
	
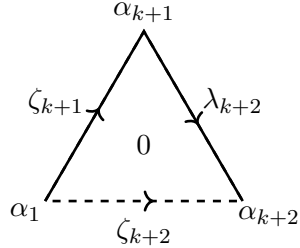
\[
	\begin{tikzpicture}
	
	\coordinate (v0) at (210:1.5);
	\coordinate (v1) at (90:1.5);
	\coordinate (v2) at (-30:1.5);
	
	\draw[line width=1, dashed] (v0)--(v1)--(v2)--cycle;
	\draw[line width=1] (v0)--(v1)--(v2);
	
	\begin{scope}[decoration={
		markings,
		mark=at position 0.55 with {\arrow{>}}},
	line width=1
	]
	
	\path[postaction={decorate}] (v0) -- (v1);
	\path[postaction={decorate}] (v1) -- (v2);
	\path[postaction={decorate}] (v0) -- (v2);
	
	\end{scope}
	
	
	\node at ($(v0) + (30:-0.3)$) {$\alpha_1$};
	\node at ($(v1) + (0,0.25)$) {$\alpha_{k+1}$};
	\node at ($(v2) + (-30:0.4)$) {$\alpha_{k+2}$};
	
	\node at ($(v0)!0.5!(v1) + (-30:-0.4)$) {$\zeta_{k+1}\ \ \ $};
	\node at ($(v1)!0.5!(v2) + (30:0.4)$) {$\ \ \ \lambda_{k+2}$};
	\node at ($(v0)!0.5!(v2) + (0,-0.4)$) {$\zeta_{k+2}$};
	
	\node at (0,0) {$0$};
	
	\end{tikzpicture}
	\]
	Then the sequence $\{\zeta_k\}_{k\geqslant 2}$ converges to an element $\zeta\in \g_1$ such that $\zeta \cdot \alpha_1 =0$~. 
\end{proposition}

\begin{proof}
	Since for all $k\geqslant 1$ we have $\alpha_k\in\F_k\g$~, it follows that  $\lim_{k\to\infty}\alpha_k = 0$~. By \cref{prop:ExplicitBCHForm}, we know that $\zeta_{k+2}$ is given by $\zeta_{k+1} + \lambda_{k+2}$ plus a sum of trees involving $\alpha_1, \alpha_{k+1}, \alpha_{k+2}, \zeta_{k+1}$~, and $\lambda_{k+2}$~, see \cref{eq:BCH first terms}. We will show that every one of these trees gives rise to an element in $\F_{k+1}\g$~.
	
	\medskip
	
	Let us consider such a tree $\tau$~. Since the output needs to be of degree $1$ and since the vertices are labeled by the operations $\ell_m$ of degree $-1$~, it must have $|\tau|+1$ elements of degree $1$ labeling its leaves, that is either $\zeta_{k+1}$ or $\lambda_{k+2}$~. Suppose that no $\lambda_{k+2}$ appear, then there exists one vertex with at least two leaves labeled by $\zeta_{k+1}$~. But this is impossible since $i_2(\omega_{01})^2=0$ in $\Omega_2$~. So at least one leaf is labeled by $\lambda_{k+2}$~, which implies that the output of this composite lives in $\F_{k+1} \g$~.
	
	\medskip
	
	Thus, we have  $\zeta_{k+2} - \zeta_{k+1}\in\F_{k+1}\g$~, which shows that the sequence $\{\zeta_k\}_{k\geqslant 2}$ converges. Since $\zeta_k\cdot \alpha_1 = \alpha_k$~, by construction,  and since $\alpha_k$ converges to $0$ as $k\to\infty$~, it follows that $\zeta\coloneqq \lim_{k\to\infty}\zeta_k$ is a gauge from $x_1$ to $0$~, by continuity of the gauge action. 
\end{proof}

\begin{lemma}\label{lemma: pi0f injective}
	The induced map
	\[
	\pi_0\R(f) \ \colon\  \pi_0\R(\g)\hookrightarrow\pi_0\R(\h)
	\]
	is injective.
\end{lemma}

\begin{proof}
	After \cref{cor:pi0}, we implicitly identify $\pi_0\R(\g)$ with the gauge equivalence classes of Maurer--Cartan elements of $\g$~. Let $\alpha,\beta\in\MC(\g)$ be two Maurer--Cartan elements, and let us denote by $\widetilde{\alpha}\coloneqq \upsilon^*f(\alpha), \widetilde{\beta}\coloneqq \upsilon^*f(\beta)$ their image in $\h$ under $\upsilon^*f$~. Suppose there exists a gauge $\widetilde{\mu}\in\h_{1}$ from $\widetilde{\alpha}$ to $\widetilde{\beta}$~. We want to show that there exists a gauge in $\g_1$ from $\alpha$ to $\beta$~.
	
	\medskip
	
	We will prove this by inductively constructing several sequences, depicted in orange in the picture below: 
	\[\{\alpha_k\}_{k\geqslant 1} \ \text{in}\  \g_0\ , \quad  \{\lambda_k\}_{k\geqslant 2} \  \text{in} \ \g_1\ , \quad \text{and} \ \{\widetilde{\mu}_k\}_{k\geqslant 1}\ \text{in} \ \h_1~, 
	\]
	such that
	\[\alpha_1\coloneqq \alpha\ , \quad \alpha_{k+1}=\lambda_{k+1}\cdot\alpha_k
	\ , \quad \widetilde{\mu}_1\coloneqq \widetilde{\mu}
	\ , \quad 
	\widetilde{\beta}=\widetilde{\mu}_{k}\cdot \widetilde{\alpha}_k
	\]
	with $\beta-\alpha_k\in\F_k\g$~, $\lambda_{k+1}\in\F_k\g$~ , and $\widetilde{\mu}_k\in\F_k\h$~.
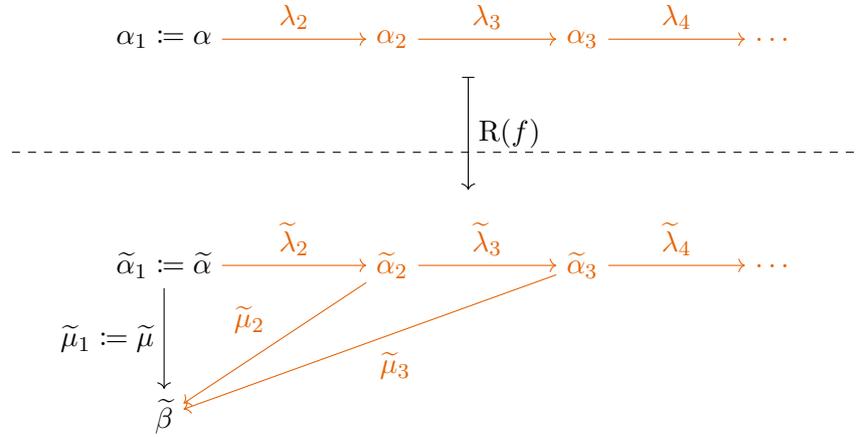
	\begin{center}
		\begin{tikzpicture}
		\node (a) at (0,0) {$\alpha_1\coloneqq\alpha$};
		\node[red] (b) at (3,0) {$\alpha_2$};
		\node[red] (c) at (5.5,0) {$\alpha_3$};
		\node[red] (d) at (8,0) {$\ldots$};
		
		\draw[->,red] (a) to node[above]{$\lambda_2$} (b);
		\draw[->,red] (b) to node[above]{$\lambda_3$} (c);
		\draw[->,red] (c) to node[above]{$\lambda_4$} (d);

		\node (A) at (0,-3) {$\widetilde{\alpha}_1\coloneqq\widetilde{\alpha}$};
		\node[red] (B) at (3,-3) {$\widetilde{\alpha}_2$};
		\node[red] (C) at (5.5,-3) {$\widetilde{\alpha}_3$};
		\node[red] (D) at (8,-3) {$\ldots$};
		
		\draw[->,red] (A) to node[above]{$\widetilde{\lambda}_2$} (B);
		\draw[->,red] (B) to node[above]{$\widetilde{\lambda}_3$} (C);
		\draw[->,red] (C) to node[above]{$\widetilde{\lambda}_4$} (D);
		
		\node (x) at (0,-5) {$\widetilde{\beta}$};
		\draw[->] (A) to node[left]{$\widetilde{\mu}_1\coloneqq\widetilde{\mu}$} (x);
		\draw[->,red] (B) to node[above left]{$\widetilde{\mu}_2$} (x);
		\draw[->,red] (C) to node[below right]{$\widetilde{\mu}_3$} (x);
		
		\draw[|->] (4,-.5) to node[right]{$\R(f)$} (4,-2);
		\draw[dashed] (-2,-1.5) to (9.3,-1.5);
		\end{tikzpicture}
	\end{center}
	Then the sequence of intermediate gauges $\lambda_k$ will produce to a single gauge linking $\alpha$ to $\beta$~. The gauges $\widetilde{\mu}_k$ are only auxiliary, but will play a fundamental role in the induction. The core of the proof is essentially a step-by-step application of obstruction theory.

	\medskip
	
	For the actual proof, the step $k=1$ is obvious. Suppose now that we have constructed the sequences up to $k$~. Then we have
	\begin{align*}
	d(\beta-\alpha_k) ={}& -\sum_{m\geqslant2}\textstyle{\frac{1}{m!}}\big(\ell_m(\beta,\ldots,\beta) - \ell_m(\alpha_k,\ldots,\alpha_k)\big)\\
	={}&  -\sum_{m\geqslant2}\textstyle{\frac{1}{m!}}
	\big(\ell_m(\beta-\alpha_k,\beta, \ldots,\beta)+\cdots + \ell_m(\alpha_k,\ldots,\alpha_k, \beta-\alpha_k)\big)\in\F_{k+1}\g\ ,
	\end{align*}
	and the same is also true for its image $\widetilde{\beta} - \widetilde{\alpha}_k$ under $\upsilon^* f$:
	\begin{align*}
	d\big(\upsilon^*f(\beta-\alpha_k)\big) ={}& 
	\sum_{m\geqslant 1}\textstyle{\frac{1}{m!}} d\left((\upsilon^*f)_m(\beta-\alpha_k, \ldots, \beta-\alpha_k)\right)\\
	={}&
	\underbrace{f_|(d(\beta-\alpha_k))}_{\in \F_{k+1}\h}+\sum_{m\geqslant 2}\textstyle{\frac{1}{m!}} d\big(\underbrace{(\upsilon^*f)_m(\beta-\alpha_k, \ldots, \beta-\alpha_k)}_{\in \F_{k+1}\h}\big)\in  \F_{k+1}\h~.
	\end{align*}

	Therefore, they induce closed elements
	\[
	\beta-\alpha_k\in Z(\F_k\g/\F_{k+1}\g)\qquad\text{and}\qquad\widetilde{\beta} - \widetilde{\alpha}_k\in Z(\F_k\h/\F_{k+1}\h)
	\]
	and we can consider their homology classes. We can see that these classes are obstructions to the possibility of gauging $\alpha_k$ to an element $\alpha_{k+1}$ that agrees with $\beta$ up to the $k$th filtration degree of $\g$~, and similarly for the elements in $\h$~. Indeed, we have the gauge $\widetilde{\mu}_k$ from $\widetilde{\alpha}_k$ to $\widetilde{\beta}$ in $\h$~. This implies that
	\[
	\widetilde{\beta} \equiv \widetilde{\alpha}_k + d\widetilde{\mu}_k  \mod\F_{k+1}\h~,
	\]
	and thus that
	\[
	\widetilde{\beta} - \widetilde{\alpha}_k \in B(\F_k\h/\F_{k+1}\h)\ .
	\]
	Notice that here it is fundamental that $\widetilde{\mu}_k\in\F_k\h$ as the existence of a random gauge from $\widetilde{\alpha}_k$ to $\widetilde{\beta}$ would not have been enough. 
	The spectral sequence argument given in the proof of \cref{thm:HoInvariance} shows that  $f_|$ induces a quasi-isomorphism $\F_k\g/\F_{k+1}\g\stackrel{\sim}{\to}\F_k\h/\F_{k+1}\h$~. This implies that there exists some $\lambda'\in\F_k\g$  such that $\beta-\alpha_k = d\lambda'$ in $\F_k\g/\F_{k+1}\g$~. Moreover, since $d(\widetilde{\mu}_k-f_1(\lambda')) = 0$ in $\F_k\h/\F_{k+1}\h$ and $f_1$ is a filtered quasi-isomorphism, there is $\theta\in\F_k\g$ which is closed in $\F_k\g/\F_{k+1}\g$ and $\sigma\in\F_k\h$ that are such that
	\[
	\widetilde{\mu}_k - f_1(\lambda') = f_1(\theta) + d\sigma
	\]
	in $\F_k\h/\F_{k+1}\h$~. We define $\lambda_{k+1}\coloneqq\lambda' + \theta \in \F_k \g$ and $\alpha_{k+1}\coloneqq\lambda_{k+1}\cdot\alpha_{k}$~.
	
	\medskip
	
	Finally, we define $\widetilde{\mu}_{k+1}$ by filling the horn
	\[
	\begin{tikzpicture}
	
	\coordinate (v0) at (210:1.5);
	\coordinate (v1) at (90:1.5);
	\coordinate (v2) at (-30:1.5);
	
	\draw[line width=1, dashed] (v0)--(v1)--(v2)--cycle;
	\draw[line width=1] (v2)--(v0)--(v1);
	
	\begin{scope}[decoration={
		markings,
		mark=at position 0.55 with {\arrow{>}}},
	line width=1
	]
	
	\path[postaction={decorate}] (v0) -- (v1);
	\path[postaction={decorate}] (v1) -- (v2);
	\path[postaction={decorate}] (v0) -- (v2);
	
	\end{scope}
	
	
	\node at ($(v0) + (30:-0.3)$) {$\widetilde{\alpha}_{k}$};
	\node at ($(v1) + (0,0.3)$) {$\widetilde{\alpha}_{k+1}$};
	\node at ($(v2) + (-30:0.3)$) {$\widetilde{\beta}$};
	
	\node at ($(v0)!0.5!(v1) + (-30:-0.4)$) {$\widetilde{\lambda}_{k+1}\ \ \ $};
	\node at ($(v1)!0.5!(v2) + (30:0.4) + (0.35,0)$) {$\widetilde{\mu}_{k+1}$};
	\node at ($(v0)!0.5!(v2) + (0,-0.4)$) {$\widetilde{\mu}_k$};
	
	\node at (0,0) {$-\sigma$};
	
	\end{tikzpicture}
	\]
	This way, we have
	\[
	\widetilde{\mu}_{k+1} \equiv  \widetilde{\mu}_k - \widetilde{\lambda}_k -d\sigma \mod \F_{k+1}\h\ ,
	\]
	and since $\widetilde{\mu}_k = f_|(\lambda_{k+1}) + d\sigma$~, we have $\widetilde{\mu}_{k+1} \in \F_{k+1}\h$ as desired.
	
	\medskip
	
	To conclude the proof, we notice  that
	\[
	\lim_{k\to\infty}\alpha_{k} = \beta\ 
	\]
	and we apply  \cref{lemma:sequence of gauges} to the twisted $\sLi$-algebra $\g^\beta$: this implies that there exists a gauge from $\alpha$ to $\beta$~. 
\end{proof}

\begin{lemma}\label{lemma: pi0f surjective}
	The induced map
	\[
	\pi_0\R(f) : \pi_0\R(\g)\twoheadrightarrow\pi_0\R(\h)
	\]
	is surjective.
\end{lemma}

\begin{proof}
	Again using \cref{cor:pi0},  we implicitly identify $\pi_0\R(\g)$ with the gauge equivalence classes of Maurer--Cartan elements of $\g$~. 	
	We need to prove that for every $\widetilde{\alpha}\in\MC(\h)$ there exists some $\beta\in\MC(\g)$ such that $\widetilde{\beta}\coloneqq \upsilon^*f(\beta)$ and $\widetilde{\alpha}$ are gauge equivalent in $h$~.
	
	\medskip
	
	The structure of the proof is as follows. Given a Maurer--Cartan element $\widetilde{\alpha}\in\MC(\h)$~, we will construct a sequence of elements $\{\beta_k\}_{k\geqslant 1}$ in $\g_0$ such that
	\[
	\beta_k\in\MC(\g/\F_{k+2}\g)\qquad\text{and}\qquad\beta_{k+1}-\beta_k\in\F_{k+1}\g\ ,
	\]
	together with two sequences of elements $\big\{\widetilde{\lambda}_k\big\}_{k\geqslant 1}$ in $\h_{1}$ 
	and $\big\{\widetilde{\alpha}_k\big\}_{k\geqslant 1}$ in $\MC(\h)$ 
	satisfying 
	\[
	\widetilde{\alpha}_k=\widetilde{\lambda}_k\cdot\widetilde{\alpha} \ , \quad 
	\widetilde{\alpha}_k - \widetilde{\beta}_k\in\F_{k+1}\h \ ,
	\quad\text{and}\quad
	\widetilde{\lambda}_{k+1}-\widetilde{\lambda}_k\in\F_{k+1}\h\ .
	\]
	Let us already conclude the proof from that. Taking the limit, we obtain a Maurer--Cartan element
	\[
	\beta\coloneqq\lim_{k\to\infty}\beta_k\in\MC(\g)\ .
	\]
	By continuity of $\upsilon^*f$~, the sequence $\widetilde{\beta}_k=\upsilon^*f(\beta_k)$ tends to $\widetilde{\beta}\coloneqq \upsilon^*f(\beta)$ in $\h$~. So does the sequence $\widetilde{\alpha}_k$ since $\widetilde{\alpha}_k - \widetilde{\beta}_k\in\F_{k+1}\h$~. Since $\widetilde{\lambda}_{k+1}-\widetilde{\lambda}_k\in\F_{k+1}\h$~,  the sequence $\widetilde{\lambda}_k$ converges to an element denoted by $\widetilde{\lambda}$~. And since the gauge action is continuous, the passage to the limit of 
	$\widetilde{\alpha}_k=\widetilde{\lambda}_k\cdot\widetilde{\alpha}$ gives 
	$\widetilde{\beta}=\widetilde{\lambda}\cdot\widetilde{\alpha}$~.
	
	\medskip
	
	To start the induction, we notice that
	\[
	d\widetilde{\alpha} = -\sum_{m\geqslant2}{\textstyle \frac{1}{m!}}\ell_m(\widetilde{\alpha},\ldots,\widetilde{\alpha})\in\F_2\h\ ,
	\]
	so that $\widetilde{\alpha}\in Z(\h/\F_2\h)$~. Since $f_|$ induces a quasi-isomorphism
	$\g/\F_2\g\stackrel{\sim}{\to}\h/\F_2\h$~,  there exist $\beta_0\in\g_0$~, such that 
	$d\beta_0\in \F_2 \g$~, 
	and $\widetilde{v} \in\h_{1}$ such that
	\[
	\widetilde{\alpha} \equiv f_|(\beta_0) + d\widetilde{v} \mod \F_2\h~.
	\]
	The cycle condition on $\beta_1$ is equivalent to $\beta_1\in \MC(\g/\F_2\g)$~. Setting $\widetilde{\lambda}_0\coloneqq 0$~, we get $\widetilde{\alpha}_0=\widetilde{\alpha}$~. Obviously, $\widetilde{\alpha}_0 - \widetilde{\beta}_0\in\F_{1}\h$ since $\h=\F_1\h$~.
	
	\medskip
	
	Now suppose that we have done our construction up to step $k$~. Denoting 
	$\widetilde{\alpha}_k - \widetilde{\beta}_k =\chi$ with $\chi\in \F_{k+1}\h$~,  we have
	\begin{align*}
	d\big(\widetilde{\alpha}_k -{}& \widetilde{\beta}_k\big) \equiv -\sum_{m\geqslant2}{\textstyle \frac{1}{m!}}\left(\ell_m(\widetilde{\alpha}_k,\ldots,\widetilde{\alpha}_k) - \ell_m\big(\widetilde{\beta}_k,\ldots,\widetilde{\beta}_k\big) \right)\mod  \F_{k+2}\h\\
	\equiv{}& -\sum_{m\geqslant2}{\textstyle \frac{1}{m!}}
	\left(\ell_m(\widetilde{\alpha}_k,\ldots,\widetilde{\alpha}_k) - 
	\ell_m\big(\widetilde{\alpha}_k-\chi,\ldots,\widetilde{\alpha}_k-\chi\big) \right)\mod  \F_{k+2}\h\\
	\equiv{}&0 \mod \F_{k+2}\h~,
	\end{align*}
	that is $d\big(\widetilde{\alpha}_k - \widetilde{\beta}_k\big) \in \F_{k+2}\h$~, 
	where in the first line we used the fact that $\beta_k\in\MC(\g/\F_{k+2}\g)$ and thus $\widetilde{\beta}_k \in\MC(\h/\F_{k+2}\h)$~. Therefore, we have
	\[
	\widetilde{\alpha}_k - \widetilde{\beta}_k\in Z(\F_{k+1}\h/\F_{k+2}\h)\ .
	\]
	The spectral sequence argument given in the proof of \cref{thm:HoInvariance} shows that  $f_|$ induces a quasi-isomorphism $\F_{k+1}\g/\F_{k+2}\g\stackrel{\sim}{\to}\F_{k+1}\h/\F_{k+2}\h$~. This implies that there exist $x\in\F_{k+1}\g$~,  $\widetilde{y}\in \F_{k+1}\h_1$~, and $z\in \F_{k+2}\h_0$ such that
	\[
	dx\in\F_{k+2}\g\qquad\text{and}\qquad\widetilde{\alpha}_k - \widetilde{\beta}_k = f_|(x) + d\widetilde{y}+z~.
	\]
	We define 
	\[\beta'\coloneqq\beta_k + x\ , \quad \widetilde{\alpha}'\coloneqq (-\widetilde{y})\cdot\widetilde{\alpha}_k\ , \quad \text{and}\quad \widetilde{\lambda}'\coloneqq\Gamma^2_1(y,\widetilde{\lambda}_k)~.\]
	\[
	\begin{tikzpicture}
	
	\coordinate (v0) at (210:1.5);
	\coordinate (v1) at (90:1.5);
	\coordinate (v2) at (-30:1.5);
	
	\draw[line width=1] (v0)--(v1)--(v2);
	\draw[dashed, line width=1]  (v0)--(v2);
	
	\begin{scope}[decoration={
		markings,
		mark=at position 0.55 with {\arrow{>}}},
	line width=1
	]
	
	\path[postaction={decorate}] (v0) -- (v1);
	\path[postaction={decorate}] (v1) -- (v2);
	\path[postaction={decorate}] (v0) -- (v2);
	
	\end{scope}
	
	
	\node at ($(v0) + (30:-0.3)$) {$\widetilde{\alpha}$};
	\node at ($(v1) + (0,0.3)$) {$\widetilde{\alpha}_k$};
	\node at ($(v2) + (-30:0.3)$) {$\widetilde{\alpha}'$};
	
	\node at ($(v0)!0.5!(v1) + (-30:-0.4)$) {$\widetilde{\lambda}_k\ $};
	\node at ($(v1)!0.5!(v2) + (30:0.4)$) {$-\widetilde{y}$};
	\node at ($(v0)!0.5!(v2) + (0,-0.4)$) {$\widetilde{\lambda}'$};
	
	\node at (0,0) {$0$};
	
	\end{tikzpicture}
	\]
	Recall that that  curvature of an element if  the left-hand side of the Maurer--Cartan equation, that we denote by 
	\[\Xi(\xi)\coloneqq d(\xi)+\sum_{m\geqslant 2} {\textstyle \frac{1}{m!}}\ell_m(\xi, \ldots, \xi)~.\]
	Since the curvature $\Xi(\beta_k)$ of $\beta_k$ and $dx$ are in $\F_{k+2}\g$ and since $x\in\F_{k+1}\g$~, it follows that the curvature $\Xi(\beta')$ of $\beta'$ is in $\F_{k+2}\g$~. This  gives
	\begin{align*}
	d(\Xi(\beta'))={}&
	d\Bigg(d\beta' + \sum_{m\geqslant2}{\textstyle \frac{1}{m!}}\ell_m(\beta',\ldots,\beta')\Bigg) \\ ={}&-\sum_{p\geqslant1}{\textstyle \frac{1}{p!}}\ell_{p+1}\left(d\beta' + \sum_{q\geqslant2}{\textstyle\frac{1}{q!}}\ell_q(\beta',\ldots,\beta'),\beta',\ldots,\beta'\right)\in\F_{k+3}\g\ .
	\end{align*}
	Therefore,
	\[
	\Xi(\beta')=d\beta' + \sum_{m\geqslant2}{\textstyle \frac{1}{m!}}\ell_m(\beta',\ldots,\beta')\in 
	Z(\F_{k+2}\g/\F_{k+3}\g)\ .
	\]
	We will now show that this element is in fact a boundary. On one hand, we have
	\begin{align*}
	d\widetilde{\beta}'{}&+ \sum_{m\geqslant2}\textstyle{\frac{1}{m!}}\ell_m\left(\widetilde{\beta}',\ldots,\widetilde{\beta}'\right) \equiv\\
	\equiv{}& d\left(\sum_{m\geqslant1}{\textstyle \frac{1}{m!}}(\upsilon^*f)_m(\beta',\ldots,\beta')\right) + \sum_{\substack{m\geqslant2
			\\ n_1,\ldots,n_m\geqslant1}} {\textstyle \frac{1}{m!n_1!\cdots n_m!}}
	\ell_m\big((\upsilon^*f)_{n_1}, \ldots, (\upsilon^*f)_{n_m}\big)(\beta',\ldots,\beta')\\
	\equiv{}& \sum_{p\geqslant1} {\textstyle \frac{1}{(p-1)!}}(\upsilon^*f)_{p}\left(d\beta' + \sum_{q\geqslant 2}{\textstyle \frac{1}{q!}}\ell_q(\beta',\ldots,\beta'),\beta',\ldots,\beta'\right)\\
	\equiv{}& f_|\left(d\beta' + \sum_{q\geqslant2}{\textstyle \frac{1}{q!}}\ell_q(\beta',\ldots,\beta')\right) \mod \F_{k+3}\h\ ,
	\end{align*}
	where, in the last line, we used the fact that the curvature of $\beta'$ is in $\F_{k+2}\g$~. On the other hand, we have
	\begin{align*}
	\widetilde{\beta}' \equiv{}& \widetilde{\beta}_k + f_|(x) + f_\tau\big(x, {\beta}_k\big) \mod \F_{k+3}\h\\
	\equiv{}& \widetilde{\alpha}_k -d\widetilde{y} -z + f_\tau\big(x, {\beta}_k\big) \mod \F_{k+3}\h\ ,
	\end{align*}
	where $\tau$ stands for the partitioned rooted tree 
	$\vcenter{\hbox{
			\begin{tikzpicture}
			\def\scale{0.3};
			\pgfmathsetmacro{\diagcm}{sqrt(2)};
			
			\coordinate (r) at (0,0);
			\coordinate (v1) at (0,\scale*1);
			\coordinate (l1) at ($(v1) + (135:\scale*\diagcm)$);
			\coordinate (l3) at ($(v1) + (45:\scale*\diagcm)$);
			
			\draw[thick] (r) to (v1);
			\draw[thick] (v1) to (l1);
			\draw[thick] (v1) to (l3);
			
			\draw (v1) circle[radius=\scale*0.65];
\end{tikzpicture}}}$~. 
This implies 
\begin{align*}
d\widetilde{\beta}' +&\ \sum_{m\geqslant2}{\textstyle \frac{1}{m!}}\ell_m\big(\widetilde{\beta}',\ldots,\widetilde{\beta}'\big) \equiv\\
\equiv&\ d\left(\widetilde{\alpha}_k -z+ f_\tau(x,\beta_k)\right) + \sum_{m\geqslant2}{\textstyle \frac{1}{m!}}\ell_m(\widetilde{\alpha}_k-d\widetilde{y},\ldots,\widetilde{\alpha}_k-d\widetilde{y}) \mod \F_{k+3}\h\\
\equiv&\ d\widetilde{\alpha}_k + \sum_{m\geqslant2}{\textstyle \frac{1}{m!}}\ell_m(\widetilde{\alpha}_k,\ldots,\widetilde{\alpha}_k) 
+ d\big(-z+f_\tau(x,\beta_k)\big) - \ell_2(d\widetilde{y},\widetilde{\alpha}_k) \mod \F_{k+3}\h\\
\equiv&\ d\left(-z+f_2(x,\beta_k) + \ell_2(\widetilde{y},\widetilde{\alpha}_k)\right) \mod \F_{k+3}\h ~.
\end{align*}
In the second line, we  used the fact that $z\in \F_{k+2}\h$~, in the third line, we used the facts that $\widetilde{\alpha}_k$ is a Maurer--Cartan element, and that $d\widetilde{y}\in\F_{k+1}\h$~, and, in the last line, we used the fact that $d\widetilde{\alpha}_k\in \F_2\h$~. Therefore, the cycle
\[
f_|\left(d\beta' + \sum_{m\geqslant2}{\textstyle \frac{1}{m!}}\ell_m(\beta',\ldots,\beta')\right)
\]
of $\F_{k+2}\h/\F_{k+3}\h$ is a boundary, and since $f_|$ is a filtered quasi-isomorphism, it follows that there exists $w\in\F_{k+2}\g$ such that
\[
d\beta' + \sum_{m\geqslant2}{\textstyle \frac{1}{m!}}\ell_m(\beta',\ldots,\beta') = dw
\]
in $\F_{k+2}\g/\F_{k+3}\g$~. To conclude, we set
\[
\beta_{k+1}\coloneqq\beta'-w=\beta_k+x-w \ , \quad 
\widetilde{\lambda}_{k+1}\coloneqq\widetilde{\lambda}'\ , 
\quad\text{and}\quad
\widetilde{\alpha}_{k+1}\coloneqq \widetilde{\alpha}'=\widetilde{\lambda}_{k+1}\cdot\widetilde{\alpha}~. 
\]
From what we have seen above, we get
\begin{align*}
\Xi(\beta_{k+1})\equiv \Xi(\beta'-w)\equiv\Xi(\beta')-dw\equiv0
\mod \F_{k+3}\g~,
\end{align*}
so $\beta_{k+1}\in \MC(\g/\F_{k+3}\g)$~. We also have $\beta_{k+1}-\beta_k=x-w\in \F_{k+1}\g$~. 
Since $x\in \F_{k+1}\g$~, $\widetilde{y}\in \F_{k+1}\h$~, $w\in\F_{k+2}\g$~, and  
$z\in \F_{k+2}\h$~,  we obtain 
\begin{align*}
\widetilde{\alpha}_{k+1} - \widetilde{\beta}_{k+1}\equiv-dy +\widetilde{\alpha}_k-\widetilde{\beta}_k 
-f_|(x) \equiv z \equiv0 \mod \F_{k+2}\h~, 
\end{align*}
that is $\widetilde{\alpha}_{k+1} - \widetilde{\beta}_{k+1}\in \F_{k+2}\h$~. 
By the BCH formula \eqref{eq:BCH first terms}, we get that 
$\widetilde{\lambda}_{k+1}-\widetilde{\lambda}_k=\widetilde{\lambda}'-\widetilde{\lambda}_k$ is equal to $y$ plus 
a sum of composite of operations of $\h$ applied at least one $y$ each time, by the same argument as in the proof of \cref{lemma:sequence of gauges}. This proves that $\widetilde{\lambda}_{k+1}-\widetilde{\lambda}_k \F_{k+1}\h$~, 
which concludes the proof. 
\end{proof}

\subsection{Inclusions of deformation problems}
An interesting problem to consider is to ask when does a deformation problem ``injects'' into another one. 
Such a line of thought was used for instance in \cite{CPRNW19} to prove 
that two commutative algebras are related by a zig-zag of quasi-isomorphisms if and only if it is the case of their underlying associative algebras. The following result is a generalization of \emph{loc. cit.}\ to the level of $\sLi$-algebras. 

\begin{definition}[$\infty_\pi$-retract]
	An $\infty_\pi$-morphism $f:\g\rightsquigarrow\h$ of complete $\sLi$-algebras is an \emph{$\infty_\pi$-retract} if for each $\alpha\in\MC(\g)$ there exists a linear map $\mathcal{s} \colon \h\to\g$ such that
	\begin{enumerate}
		\item $\mathcal{s}$ respects the filtrations,
		\item $\mathcal{s}(\upsilon^*f)^\alpha_1=\id_\g$
		, and
		\item\label{pt3 def retraction} if $\widetilde{x}\in\h$ is such that $d^{\widetilde{\alpha}}\widetilde{x}\in\F_k\h$ then
		\[
		\mathcal{s}\left(d^{\widetilde{\alpha}}\widetilde{x}\right) \equiv d^{\alpha}\mathcal{s}(\widetilde{x}) \mod\F_{k+1}\g\ ,
		\]
		where $\widetilde{\alpha}\coloneqq \upsilon^*f(\alpha)$~.
	\end{enumerate}
\end{definition}

\begin{proposition}\label{prop:InclPi0}
	If $f:\g\rightsquigarrow\h$ is an $\infty_\pi$-retraction,  then the map
	\[
	\pi_0(\R(f))\ \colon \ \pi_0(\R(\g))\hookrightarrow\pi_0(\R(\h))
	\]
	is injective.
\end{proposition}

\begin{proof}
	The pattern of the proof is similar to the one of \cref{lemma: pi0f injective}. Recall that  
	\cref{cor:pi0} gives $\pi_0\R(\g)\cong \mathcal{MC}(\g)$~, so we have to prove that if 
	$\alpha, \beta\in \MC(\g)$ are such that $\widetilde{\alpha}\coloneqq \upsilon^*f(\alpha)$ and 
	$\widetilde{\beta}\coloneqq \upsilon^*f(\beta)$ are gauge equivalent in $\h$~, then they are also gauge equivalent in $\g$~. By noticing that the property of being an $\infty_\pi$-retract is preserved by the twisting procedure, we can always reduce to the case where $\beta = 0$ by considering the twisted $\sLi$-algebra $\g^\beta$~. 
	
	\medskip
	
	The geometric idea 	
	behind the proof of injectivity on the $0$th homotopy groups 
	\[
	\pi_0\R(f) \ \colon\  \pi_0\R(\g)\cong \mathcal{MC}(\g)\xhookrightarrow{\quad}\pi_0\R(\h)\cong \mathcal{MC}(\h)
	\]
	is given by the following drawing.
	
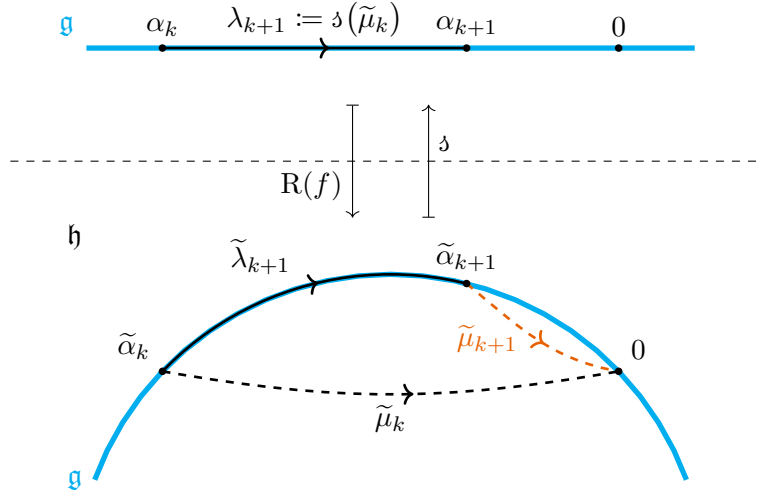
\[
	\begin{tikzpicture}
	\def\bd{4};
	\def\angle{75}
	
	\def\xxn{-3};
	\def\xxnp{1};
	\def\xzero{3};
	
	\pgfmathsetmacro{\r}{\bd/sin(\angle)};
	\def\cy{-\r-1.5};
	
	\pgfmathsetmacro{\alphatxn}{acos(\xxn/\r)};
	\pgfmathsetmacro{\alphatxnp}{acos(\xxnp/\r)};
	\pgfmathsetmacro{\alphaz}{acos(\xzero/\r)};
	\pgfmathsetmacro{\alphamid}{(\alphatxn+\alphatxnp)/2};
	
	\begin{scope}[decoration={
		markings,
		mark=at position 0.55 with {\arrow{>}}},
	line width=2
	]
	
	\node[above left, cyan] at (-\bd,1.5) {$\g$};
	
	\coordinate (xn) at (\xxn,1.5);
	\coordinate (xn+1) at (\xxnp,1.5);
	\coordinate (zero) at (\xzero,1.5);
	
	\draw[cyan] (-\bd,1.5)--(\bd,01.5);
	
	\node at (xn)[circle, fill, inner sep=1pt]{};
	\node at (xn)[above]{$\alpha_k$};
	\node at (xn+1)[circle, fill, inner sep=1pt]{};
	\node at (xn+1)[above]{$\alpha_{k+1}$};
	\node at (zero)[circle, fill, inner sep=1pt]{};
	\node at (zero)[above]{$0$};
	
	\draw[line width=1, postaction={decorate}] (xn)--node[above]{$\lambda_{k+1}\coloneqq \mathcal{s}\big(\widetilde{\mu}_k\big)$}(xn+1);
	
	\end{scope}
	
	\draw[dashed] ({-\bd-1},0)--({\bd+1},0);
	
	\draw[|->] (-0.5,0.75) to node[below left]{$\R(f)$} (-0.5, -0.75);
	\draw[|->] (0.5,-0.75) to node[above right]{$\mathcal{s}$} (0.5, 0.75);
	
	\begin{scope}[decoration={
		markings,
		mark=at position 0.55 with {\arrow{>}}},
	line width=2
	]
	
	\coordinate (txn) at ($(\alphatxn:\r) + (0,\cy)$);
	\coordinate (txnp) at ($(\alphatxnp:\r) + (0,\cy)$);
	\coordinate (tz) at ($(\alphaz:\r) + (0,\cy)$);
	
	\draw[cyan, domain=20:160] plot ({\r*cos(\x)}, {\r*sin(\x) + \cy});
	
	\draw[dashed, red, line width=1, postaction={decorate}] (txnp) to[out=-40, in=170] node[left]{$\widetilde{\mu}_{k+1}\ $}(tz);
	
	\node at (txn)[circle, fill, inner sep=1pt]{};
	\node at (txn)[above left]{$\widetilde{\alpha}_k$};
	\node at (txnp)[circle, fill, inner sep=1pt]{};
	\node at (txnp)[above]{$\widetilde{\alpha}_{k+1}$};
	\node at (tz)[circle, fill, inner sep=1pt]{};
	\node at (tz)[above right]{$0$};
	
	\draw[line width=1, domain=\alphatxn:\alphatxnp, postaction={decorate}] plot ({\r*cos(\x)}, {\r*sin(\x) + \cy});
	
	\node[above left] at ($(\alphamid:\r) + (0,\cy)$) {$\widetilde{\lambda}_{k+1}$};
	
	\draw[dashed, line width=1, postaction={decorate}] (txn) to[out=-10, in=190] node[below]{$\widetilde{\mu}_k$}(tz);
	
	\pgfmathsetmacro{\cc}{cos(20)};
	\pgfmathsetmacro{\ss}{sin(20)};
	\node[left, cyan] at ($(-\r*\cc, \r*\ss + \cy)$) {$\g$};
	
	\node[left] at ({-\r*\cc},-1) {$\h$};
	
	\end{scope}
	\end{tikzpicture}
	\]
	The upper half of the drawing represents $\g$~, and the bottom half represents $\h$ with the image of $\g$ under $\R(f)$ drawn in blue. Whenever an element denoted by a letter is present in the upper half of the drawing, it is mapped under $\R(f)$ to the element denoted by the same letter and a tilde in the bottom half; the only element not following this convention are the $\widetilde{\mu}$'s which are only present in the bottom half. 
	
	\medskip
	
	Let us start with a Maurer--Cartan element $\alpha_1\in\MC(\g)$ such that $\widetilde{\alpha}_1\coloneqq\upsilon^*f(\alpha_1)$ is gauge equivalent to $0$ via a gauge $\widetilde{\mu}_1\in\h_1$~. By induction, we construct the following sequences 
	\[\lambda_{k+1}\coloneqq \mathcal{s}\big(\widetilde{\mu}_k\big)\ , \quad \alpha_{k+1}\coloneqq \lambda_k\cdot \alpha_k\ , \quad \widetilde{\alpha}_{k+1}\coloneqq\upsilon^*f(\alpha_{k+1})\ ,\]
	$ \widetilde{\lambda}_{k+1}\in\h_1$ as the element satisfying
	\[
	\R(f)_1(\alpha_k\otimes\omega_0 + \lambda_{k+1}\otimes\omega_{01} + \alpha_{k+1}\otimes\omega_1) = 
	\widetilde{\alpha}_k\otimes\omega_0 + \widetilde{\lambda}_{k+1}\otimes\omega_{01} + \widetilde{\alpha}_{k+1}\otimes\omega_1\ ,
	\]
	and $\widetilde{\mu}_{k+1}\in\h_1$ by the BCH product given by filling the horn
	
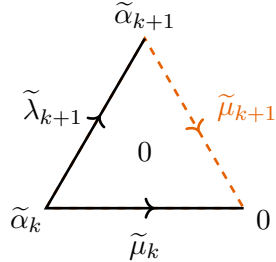
\[
	\begin{tikzpicture}
	
	\coordinate (v0) at (210:1.5);
	\coordinate (v1) at (90:1.5);
	\coordinate (v2) at (-30:1.5);
	
	\draw[line width=1, red, dashed] (v0)--(v1)--(v2)--cycle;
	\draw[line width=1] (v2)--(v0)--(v1);
	
	\begin{scope}[decoration={
		markings,
		mark=at position 0.55 with {\arrow{>}}},
	line width=1
	]
	
	\path[postaction={decorate}] (v0) -- (v1);
	\path[postaction={decorate}, red] (v1) -- (v2);
	\path[postaction={decorate}] (v0) -- (v2);
	
	\end{scope}
	
	
	\node at ($(v0) + (30:-0.3)$) {$\widetilde{\alpha}_k$};
	\node at ($(v1) + (0,0.3)$) {$\widetilde{\alpha}_{k+1}$};
	\node at ($(v2) + (-30:0.3)$) {$0$};
	
	\node at ($(v0)!0.5!(v1) + (-30:-0.4) + (-0.1,0)$) {$\widetilde{\lambda}_{k+1}\ \ $};
	\node[red] at ($(v1)!0.5!(v2) + (30:0.4) + (0.35,0)$) {$\widetilde{\mu}_{k+1}$};
	\node at ($(v0)!0.5!(v2) + (0,-0.4) + (0,-0.1)$) {$\widetilde{\mu}_k$};
	
	\node at (0,0) {$0$};
	
	\end{tikzpicture}
	\]
	by zero, or in other words $\widetilde{\mu}_{k+1}\coloneqq\Gamma^2_0\big(\widetilde{\lambda}_{k+1}, \widetilde{\mu}_{k}\big)$~.
	
	\medskip
	
	We will show by induction on $k\geqslant 1$~, that $\alpha_k, \lambda_{k+1}\in\F_k\g$~. As an auxiliary step, we will also prove  that $d\widetilde{\mu}_k\in\F_k\h$~, for $k\geqslant 1$~. The base case $k=1$ is trivial. We assume that the result hold up to $k$ and we prove it for $k+1$~.
	
	\medskip
	
	By construction, we have $\widetilde{\mu}_k\cdot\widetilde{\alpha}_k = 0$ and, by assumption, we get $\widetilde{\alpha}_k\in\F_k\h$~; this implies 
	\[
	0=\widetilde{\mu}_k\cdot\widetilde{\alpha}_k \equiv d\widetilde{\mu}_k \mod \F_k\h\ ,
	\]
	so that $d\widetilde{\mu}_k\in\F_k\h$~.
	
	\medskip
	
	It is straightforward to check from \cref{prop:Diffmc1} that $\widetilde{\lambda}_{k+1}\equiv f_{|}(\lambda_{k+1})\mod\F_{k+1}\h$~; furthermore, by the BCH product \cref{eq:BCH first terms} and the definition of $\widetilde{\mu}_{k+1}$~, we have
	\[
	\widetilde{\mu}_{k+1} \equiv \widetilde{\mu}_k - \widetilde{\lambda}_{k+1} \equiv \widetilde{\mu}_k - f_{|}\mathcal{s}\big(\widetilde{\mu}_k\big) \mod\F_{k+1}\h\ .
	\]
	Applying $\mathcal{s}$ on both sides, we obtain 
	\[
	\lambda_{k+2} = \mathcal{s}\big(\widetilde{\mu}_{k+1}\big)\equiv \mathcal{s}\big(\widetilde{\mu}_k\big) - \mathcal{s}f_{|}\mathcal{s}\big(\widetilde{\mu}_k\big) = 0 \mod\F_{k+1}\g\ ,
	\]
	showing that $\lambda_{k+2}\in\F_{k+1}\g$~.
	
	\medskip
	
	Finally, by construction, we have $\widetilde{\mu}_k\cdot\widetilde{\alpha}_k = 0$~, which implies 
	\[
	0\equiv\widetilde{\alpha}_k + d\widetilde{\mu}_k\equiv f_{|}(\alpha_k)+d\widetilde{\mu}_k\mod\F_{k+1}\h\ .
	\]
	Applying $\mathcal{s}$ on both sides, we obtain 
	\[
	0 \equiv 
	\alpha_k + \mathcal{s}\big(d\widetilde{\mu}_k\big) \equiv 
	\alpha_k + d\mathcal{s}\big(\widetilde{\mu}_k\big) = \alpha_k + d\lambda_{k+1} \equiv \lambda_{k+1}\cdot \alpha_k\equiv \alpha_{k+1}\mod\F_{k+1}\g\ ,
	\]
	where in the second equivalence we used the fact that $d\widetilde{\mu}_k\in\F_k\h$ together with point (\ref{pt3 def retraction}) of the definition of an $\infty_\pi$-retraction. This shows that $\alpha_{k+1}\in\F_{k+1}\g$~.
	
	\medskip
	
	We find ourselves in the situation of \cref{lemma:sequence of gauges}. Therefore,  $\alpha_1$ is gauge equivalent to $0$~.	
\end{proof}

The condition such that an $\infty_\pi$-morphism $f:\g\rightsquigarrow\h$ induces injections on the higher homotopy groups
\[
\pi_n(\R(f)):\pi_n(\R(\g), \alpha)\hookrightarrow\pi_n(\R(\h),\upsilon^*f(\alpha))
\]
at a base point $\alpha\in\MC(\g)$ and for some $n\geqslant1$ are obvious from \cref{thm:Berglund}: one just needs to require that
\[
H_n\left((\upsilon^*f)^\alpha_1\right)\, \colon\,  H_n(\g^\alpha)\hookrightarrow 
H_n\Big(\h^{\upsilon^*f(\alpha)}\Big)
\]
is injective.

\begin{remark}
	The above result includes the case of differential graded Lie algebras and strict morphisms covered by  of \cite[Theorem~1.7]{CPRNW19}.
\end{remark}

\subsection{Model category structure}\label{subsec:ModCat}
Let us first recall the classical model category structure on the category $\sSe$ of simplicial sets due to Kan and Quillen. We refer the reader to the monographs \cite{Hovey99, GoerssJardine09} for more details. 

\begin{theorem}\cite{Quillen67}\label{thm:QuillenSet}
	The following three classes of maps form a model category structure on the category $\sSe$ of simplicial sets: 
	\begin{itemize}
		\item[$\diamond$] a morphism $f  \colon  X_\bullet\to Y_\bullet$ of simplicial sets is a \emph{weak equivalence} if it induces a weak homotopy equivalence of topological spaces $|f|\ :\ |X_\bullet| \xrightarrow{\sim} |Y_\bullet|$ between the associated geometric realization,		
		\item[$\diamond$] a morphism $f  :  X_\bullet\hookrightarrow Y_\bullet$ of simplicial sets is a \emph{cofibration} if it is made up of injective maps 
		$f_n  :  X_n\hookrightarrow Y_n$  
		in each simplicial degree, and	
		\item[$\diamond$] a morphism $f  :  X_\bullet\twoheadrightarrow Y_\bullet$ of simplicial sets is a \emph{fibration} if it satisfies the right lifting property with respect to  (acyclic) horn inclusions:
		\[
		\vcenter{\hbox{
				\begin{tikzpicture}
				\node (a) at (0,0) {$\Ho{k}{n}$};
				\node (b) at (2.5,0) {$X_\bullet$};
				\node (c) at (0,-2) {$\De{n}$};
				\node (d) at (2.5,-2) {$Y_\bullet$};
				
				\draw[->] (a) to (b);
				\draw[right hook->] (a) to node[below, sloped]{$\sim$} (c);
				\draw[->] (b) to node[right]{$f$} (d);
				\draw[->, dashed] (c) to node[above left]{$\exists$} (b);
				\draw[->] (c) to (d);
				\end{tikzpicture}
		}}
		\]
	\end{itemize}
	This model category structure is cofibrantly generated by the following classes of maps:
	\begin{itemize}
		\item[$\diamond$] the boundary inclusions $\partial \De{n} \hookrightarrow \De{n}$ are the generating cofibrations and		
		\item[$\diamond$] the horn inclusions $\Ho{k}{n}  \stackrel{\sim}{\hookrightarrow} \De{n}$ are the generating acyclic cofibrations. 
	\end{itemize}
\end{theorem}

We now transfer this cofibrantly generated model category structure to complete $\sLi$-al\-ge\-bras under the 
$\sLi$-algebra functor $\Li$~.

\begin{theorem}\label{thm:MConSLi}
	The following three classes of maps form a model category structure on the category of complete $\sLi$-algebras:
	\begin{itemize}
		\item[$\diamond$] a morphism $f  :  \g \to \h$ of complete $\sLi$-algebras is a weak equivalence if 
		it induces a bijection 
		\[\mathcal{MC}(f): \mathcal{MC}(\mathcal{\g})\xrightarrow{\cong}\mathcal{MC}(\h)\]
		between the associated moduli spaces of Maurer--Cartan elements and isomorphisms 
		\[{H}_n\left(f^\alpha\right) : {H_n}\left(\g^\alpha\right) \xrightarrow{\cong} {H_n}\big(\h^{f(\alpha)}\big)\ , \]
		for any Maurer--Cartan element $\alpha \in \MC(\g)$ and any $n\geqslant 1$~,
		
		\item[$\diamond$] a morphism $f  :  \g\twoheadrightarrow \h$ of complete $\sLi$-algebras is a fibration  if it is surjective in all degrees $n\geqslant2$~, and		
		\item[$\diamond$] a morphism $f  :  \g \hookrightarrow \h$ of complete $\sLi$-algebras is a cofibration if 
		it satisfies the right lifting property with respect to acyclic fibrations. 
		
	\end{itemize}
	This model category structure is cofibrantly generated by the following classes of maps:
	\begin{itemize}
		\item[$\diamond$] the boundary inclusions $\Li(\partial \De{n}) \hookrightarrow \Li(\De{n})$ are the generating cofibrations, 
		
		\item[$\diamond$] the horn inclusions $\Li(\Ho{k}{n})  \hookrightarrow \Li(\De{n})$ are the generating acyclic cofibrations. 
	\end{itemize}
\end{theorem}

\begin{proof}
	Let us first  make the class of  fibrations explicit. In a candidate for a transferred cofibrantly generated model category structure under  a left adjoint functor, the class of fibrations is the class of maps $f  :  \g\to \h$ whose image $\R(f)$ under the right adjoint functor is a fibration of simplicial sets, that is the ones which  satisfy the right lifting property with respect to the acyclic horn inclusions. We use the same arguments as in the proof of \cref{thm:KanExt}. First, we consider the equivalent problem on the level of complete $\sLi$-algebras under the adjunction of \cref{thm:MainAdjunction}:
	\[
	\vcenter{\hbox{
		\begin{tikzpicture}
			\node (a) at (0,0) {$\Li\big(\Ho{k}{n}\big)$};
			\node (b) at (2.5,0) {$\g$};
			\node (c) at (0,-2) {$\Li\big(\De{n}\big)$};
			\node (d) at (2.5,-2) {$\h$};
			
			\draw[->] (a) to (b);
			\draw[right hook->] (a) to node[below, sloped]{$\sim$} (c);
			\draw[->] (b) to node[right]{$f$} (d);
			\draw[->, dashed] (c) to node[above left]{$\exists$} (b);
			\draw[->] (c) to (d);
		\end{tikzpicture}
	}}
	\]
	Then, using \cref{lemm:Fundamental}, we conclude that the maps for which such lifts exist are the morphisms $f  :  \g \to \h$ of complete $\sLi$-algebras which are surjective in degrees $n\geqslant2$~. 
	
	\medskip
	
	Let us now make weak equivalences explicit. Again, in a candidate for a transferred cofibrantly generated model category structure under  a left adjoint functor, the class of weak equivalences is the class of maps $f  :  \g\to \h$ whose image $\R(f)$ under the right adjoint functor is a weak equivalence of simplicial sets. \cref{cor:pi0} asserts that there is a natural bijection $\pi_0(\R(\g))\cong \mathcal{MC}(g)$ and \cref{thm:Berglund} asserts that there are natural isomorphisms 
	$\pi_n(\R(\g), \alpha)\cong \mathrm{H}_n(\g^\alpha)$~,  for any $\alpha \in \MC(\g)$ and any $n\geqslant 1$~, which concludes the proof of this point. 
	
	\medskip
	
	We are now ready to prove that this is actually a model category structure. To do so, we use Quillen's transfer theorem under a left adjoint functor \cite[Section~II.4]{Quillen67}. The version used here follows from  \cite[Section~2.6]{BergerMoerdijk03}. We need to prove the following four facts. 
	\begin{description}
		\item[\rm (1) \it  The category of complete $\sLi$-algebras is complete and cocomplete] this result 
		was pro\-ved in \cref{prop:CatCocomp}. 
		
		\item[\rm (2) \it 
		The sets of generating (acyclic) cofibrations satisfies the small object argument]
		since the sets of complete $\sLi$-algebras 
		\[\{\Li(\partial \De{n})\}_{n\geqslant 0}\quad \text{and} \quad \{\Li(\Ho{k}{n})\}_{n\geqslant 2,\ k=0,\ldots, n}\]
		are countable, they satisfy the small object argument with respect to the generating (acylic) cofibrations
		\[\left\{\Li(\partial \De{n}) \hookrightarrow \Li(\De{n})\right\}_{n\geqslant 0} \quad 
		\text{and}\quad 
		\left\{\Li(\Ho{k}{n})  \hookrightarrow \Li(\De{n})\right\}_{n\geqslant 2,\ k=0,\ldots, n}~.
		\]
		
		\item[\rm (3) \it Every complete $\sLi$-algebra is fibrant] this is a direct corollary of 
		\cref{lemm:Fundamental}, see also \cref{thm:KanExt}. 
		
		\item[\rm (4) \it The category of complete $\L_\infty$-algebra has functorial good path objects] we consider the fol\-lo\-wing com\-posite of morphisms for any complete $\sLi$-al\-ge\-bra $\g$:
		\[
		\vcenter{\hbox{
			\begin{tikzpicture}
				\node (a) at (0,0){$\g$};
				\node (b) at (3,0){$\g \,\widehat{\otimes}\, \Omega_1$};
				\node (c) at (6,0){$\g \times \g$};
				\node (d) at (0,-0.75){$x$};
				\node (e) at (3,-0.75){$x\otimes1$};
				\node (f) at (3,-1.5){$x\otimes P(t)$};
				\node (g) at (6,-1.5){$\big(P(0)x, P(1)x\big)$};
				\node (h) at (3,-2.25){$x\otimes Q(t)dt$};
				\node (i) at (6,-2.25){$0\ ,$};
				
				\draw[->] (a) to node[above=-0.05]{$\sim$} (b);
				\draw[->>] (b) to (c);
				\draw[|->] (d) to (e);
				\draw[|->] (f) to (g);
				\draw[|->] (h) to (i);
			\end{tikzpicture}
		}}
		\]
		where we denote $t\coloneq t_1$ in $\Omega_1\cong \k[t, dt]$~. 
		The first map is a filtered quasi-isomorphism of complete $\sLi$-algebras, so it is a weak equivalence  by \cref{thm:HoInvariance}.
		The second map is an epimorphism and thus a fibration. So all together this construction forms a functorial path object. 
	\end{description}
\end{proof}

\begin{remark}
The model category structure of \cite[Theorem~5.2.28]{Bandiera14}, see also \cite[Theorem~3.1]{BFMT18} on complete $\sLie$-algebras is a model sub-category of the present one. 
\end{remark}

\begin{proposition}\label{prop:QuillAdjWE}
	The adjoint pair of functors $\Li\dashv \R$ forms a Quillen adjunction. The $\sLi$-algebra functor $\Li$ and the integration $\R$ preserve the respective weak equivalences. 
\end{proposition}

\begin{proof}
	The first point is a direct corollary of \cref{thm:MConSLi} since the right adjoint functor $\R$ preserves weak equivalences and fibrations by construction. The $\sLi$-algebra functor $\Li$ preserves weak equivalences by Ken Brown's lemma \cite[Lemma~1.1.12]{Hovey99}. 
\end{proof}

\begin{proposition}\label{prop:C1PATH}
	The complete $\sLi$-algebra $\g\, \widehat{\otimes}^\pi \mathrm{C}_1$ is a functorial good path object  for 
	complete $\sLi$-algebras $\g$~. 
\end{proposition}

\begin{proof}
	The arguments are the same as above for the good path object
	$\g\, \widehat{\otimes}\, \Omega_1$ applied this time to the maps
	\[
	\vcenter{\hbox{
			\begin{tikzcd}[column sep=0.6cm, row sep=0cm]
			\g \arrow[r, "\sim"] & \g \,\widehat{\otimes}^\pi \mathrm{C}_1 
			\arrow[r, twoheadrightarrow]
			& \g \times \g
			\end{tikzcd}}}\]
	defined respectively by 
	$x\mapsto x\otimes (\omega_0+\omega_1)$ and by 
	$x\otimes \omega_0 \mapsto (x, 0)$~, 
	$x\otimes \omega_1 \mapsto (0, x)$~, 
	$x\otimes \omega_{01}\mapsto 0$~.
	The formula for the $\Cobar\Bar\com$-algebra structure given in \cref{subsec:HighLSalg} shows that the latter map  is indeed a morphism of $\sLi$-algebras. 
\end{proof}

Using this cylinder, a (right) homotopy between two morphisms $f, f'\colon \g \to \h$ of complete $\sLi$-algebras amount to the data of map $\g\to \h$ of degree $1$ satisfying some relations, that can be made explicit using \cref{prop:Diffmc1}. 

\begin{proposition}\label{prop:mc1Cylin}
	The complete $\sLi$-algebra $\mc^1$ is a very good cylinder of $\mc^0$~. 
\end{proposition}

\begin{proof}
	The cosimplicial maps of $\mc^\bullet$ provide us with 
	\[
	\mc^0 \vee \mc^0 \cong \widehat{\sLi}(a_0, a_1) \longrightarrow \mc^1\cong \widehat{\sLi}(a_0, a_1, a_{01}) \longrightarrow \mc^0
	\cong\widehat{\sLi}(a_0)~,  \]
	which are respectively given by 
	$a_0 \mapsto a_0$~, $a_1 \mapsto a_1$ and by 
	$a_0 \mapsto a_0$~, $a_1 \mapsto a_0$~, $a_{01} \mapsto 0$~. 
	The map $\mc^1\to\mc^0$ is equal to $\Li\left(\De{1} \to \De{0}\right)$; so it is a weak equivalence by \cref{prop:QuillAdjWE}. 
	The map $\mc^1\to\mc^0$ is equal to $\Li\left(\De{1} \to \De{0}\right)$; so it is a weak equivalence by \cref{prop:QuillAdjWE}. 
	It is obviously an epimorphism, so a fibration. 
	The first map is a generating cofibration, so a cofibration. 
\end{proof}
\section{Rational models for spaces}\label{sec: rational models}

In this final section, we  explore how the adjoint pair of functors $\Li \dashv \R$ interacts with rational homotopy theory. After a brief recollection on the subject, we  prove, under mild, standard assumptions, that the composite $\R\Li$ gives, after removing an extra point, the Bousfield--Kan $\QQ$-completion and thus a rationalization in the nilpotent case.

\medskip

Although some results should hold over more general fields of characteristic $0$~, in this section we fix $\k=\mathbb{Q}$~.

\subsection{Recollections on rational homotopy theory}
We give a brief reminder of rational models, $\QQ$-completion, and rationalizations of simplicial sets in the context of commutative algebras, mostly following \cite{Sullivan77, BousfieldKan, BousfieldGugenheim}. 

\medskip

Let $\ucomnalg$ denote the category of unital commutative algebras concentrated in non-positive degrees and let $\ucomnalg_*$ the category of such algebras which are \emph{augmented}. Recall that we always work in the homological degree convention. 

\begin{proposition}[{\cite[Theorem~4.3]{BousfieldGugenheim}}]
The categories $\ensuremath{\mathrm{uCom}_{\leqslant 0}\text{-}\, \mathsf{alg}}$
 and $\ensuremath{\mathrm{uCom}_{\leqslant 0}\text{-}\, \mathsf{alg}}_*$ carry a closed model category structure 
 where the weak equivalences are the quasi-isomorphisms and where the fibrations are the degree-wise surjections. 
\end{proposition}

\begin{proof}
We refer to \cite[Section~4]{BousfieldGugenheim} for a detailed proof. Let us just mention that the model category structure on $\ucomnalg_*$ is obtained as the slice category of $\ucomnalg$  over the algebra $\k$~.
\end{proof}

In the model category $\ucomnalg$~, all commutative algebras are fibrant and the cofibrant ones are the retracts  of  the Sullivan algebras, which are quasi-free unital commutative algebras with a suitable filtration on the space of generators, see \cite[Chapter~12]{RHT}. 

\medskip

Similarly, the category $\sSe_*$ of \emph{pointed} simplicial sets is endowed with a model category structure from \cref{thm:QuillenSet} as the coslice category of $\sSe$ under the point. In the model category $\sSe$ (resp. $\sSe_*$), all (pointed) simplicial sets are cofibrant and the fibrant ones are the (pointed) Kan complexes. 

\medskip

Considering the opposite category $\ucomnalg^\mathrm{op}$, one can apply \cref{lem:AdjCosimpliObj} to the simplicial commutative algebra $\Omega_\bullet$ to produce the following  pair of  adjoint functors  whose left adjoint has domain in simplicial sets, see \cite{Sullivan77}. 

\begin{definition}[Spatial realization and algebra of polynomial differential forms]\label{def:SimpRepAPL}
The \emph{spatial realization}\index{spatial realization} is defined as the functor
\begin{align*}
\langle-\rangle\ :\ \ucomnalg^\mathrm{op}&\ \longrightarrow\ \sSe\\
A{}&\ \longmapsto\ 
\Hom_{\ucomnalg}(A,\Omega_\bullet) 
\end{align*}
and the functor  of the \emph{algebra of  polynomial differential forms}
is defined by 
\begin{align*}
\APL\ :\ \sSe&\ \longrightarrow\ \ucomnalg^\mathrm{op}\\
X_\bullet {}&\ \longmapsto\ \Ran_{\mathrm{Y}^{\mathrm{op}}} {\Omega_\bullet} (X_\bullet)\cong \Hom_{\sSe}(X_\bullet, \Omega_\bullet)\cong \lim_{\mathsf{E}(X_\bullet)} \Omega_\bullet\ .
\end{align*}
\end{definition}

Any base point $*\to X_\bullet$ gives an augmentation
\[\APL(X_\bullet)\longrightarrow \APL(*)\cong\k\] and, dually, any augmentation 
$A\to \k$ gives a base point 
\[*\cong \langle \k \rangle \longrightarrow \langle A\rangle~.\]
So the above two functors restrict to a pair of contravariant adjoint functors that we denote with the same notations:
\[
\hbox{
	\begin{tikzpicture}
	\def\upshift{0.185}
	\def\downshift{0.15}
	\pgfmathsetmacro{\midshift}{0.005}
	
	\node[left] (x) at (0, 0) {$\APL \ \ :\ \ \sSe_*$};
	\node[right] (y) at (2, 0) {$\ucomnalg_*^\mathrm{op}\ \ :\ \ \langle-\rangle\ .$};
	
	\draw[-{To[left]}] ($(x.east) + (0.1, \upshift)$) -- ($(y.west) + (-0.1, \upshift)$);
	\draw[-{To[left]}] ($(y.west) + (-0.1, -\downshift)$) -- ($(x.east) + (0.1, -\downshift)$);
	
	\node at ($(x.east)!0.5!(y.west) + (0, \midshift)$) {\scalebox{0.8}{$\perp$}};
	\end{tikzpicture}}
\]

\begin{lemma}\cite[Section~8]{BousfieldGugenheim}\label{lem:QuillenAdj}
The two adjunctions $\APL\dashv\langle-\rangle$ described above are Quillen adjunctions. 
\end{lemma}

\begin{proof}
Let us remind  that working in the opposite categories,  the roles of fibrations and cofibrations in the model structure are inverted.
The left adjoint functor $\APL$ sends cofibrations (degree-wise injections) to fibrations (degree-wise surjections) by construction, and 
the right adjoint functor $\langle-\rangle$ sends cofibrations to fibrations by \cite[Lemma~8.2]{BousfieldGugenheim}.
\end{proof}

As a consequence, they induce adjunctions between the respective homotopy categories, that we understand here to have the same objects but with homotopy classes of morphisms from a cofibrant replacement of the domain to a  fibrant replacement of the source. 
The right derived functor of $\RR \langle-\rangle$ is defined by applying $\langle-\rangle$ to a cofibrant replacement of a unital commutative algebra. For the left derived functor of $\APL$~, there is no need to consider a cofibrant replacement since all (pointed) simplicial sets are already cofibrant. So we still denote this derived functor by the same notation. 
\[
\hbox{
	\begin{tikzpicture}
	\def\upshift{0.185}
	\def\downshift{0.15}
	\pgfmathsetmacro{\midshift}{0.005}
	
	\node[left] (x) at (0, 0) {$\APL \ \ :\ \ \mathsf{ho}(\sSe_*)$};
	\node[right] (y) at (2, 0) {$\mathsf{ho}(\ucomnalg_*^\mathrm{op})\ \ :\ \ \RR\langle-\rangle\ .$};
	
	\draw[-{To[left]}] ($(x.east) + (0.1, \upshift)$) -- ($(y.west) + (-0.1, \upshift)$);
	\draw[-{To[left]}] ($(y.west) + (-0.1, -\downshift)$) -- ($(x.east) + (0.1, -\downshift)$);
	
	\node at ($(x.east)!0.5!(y.west) + (0, \midshift)$) {\scalebox{0.8}{$\perp$}};
	\end{tikzpicture}}
\]

In order the get a Quillen equivalence, that is an equivalence on the level of the homotopy categories, we need to restrict ourselves to the following setting, for which for refer to \cite[Section~9.2]{BousfieldGugenheim}. Recall that a group $G$ is \emph{nilpotent} if its \emph{lower central series}
\[
G=\Gamma_1G\supseteq\Gamma_2G\supseteq\cdots,
\]
with $\Gamma_{q+1}G$ the subgroup of $G$ generated by the commutators $[G,\Gamma_qG]$~, terminates after a finite number of steps. A $G$-module $M$ is \emph{nilpotent} if its \emph{lower central series}
\[
M=\Gamma_1 M\supseteq\Gamma_2 M\supseteq\cdots,
\]
with $\Gamma_{q+1}M$ the sub-$G$-module generated by the elements of the form $gm-m$~, for $g\in G$ and $m\in\Gamma_qM$~, terminates after a finite number of steps. 

\begin{definition}[Nilpotent simplicial set]
A connected simplicial set $X_\bullet\in\sSe_*$ is \emph{nilpotent} if its fundamental group $\pi_1 X_\bullet$ is nilpotent and if the $\pi_1 X_\bullet$-modules  $\pi_n X_\bullet$ are nilpotent for all $n\geqslant2$~.
\end{definition}

A group $G$ is \emph{uniquely divisible} if the equation $x^r=g$ has a unique solution $x\in G$ for any $g\in G$ and any integer $r\geqslant1$~. Notice that a uniquely divisible \emph{abelian} group is essentially the same as a vector space over $\mathbb{Q}$ since one can define scalar multiplication in a natural way. 

\begin{definition}[Rational simplicial set]
A nilpotent simplicial set $X_\bullet\in\sSe_*$ is \emph{rational} if the following equivalent conditions are satisfied.
\begin{enumerate}
	\item The homotopy groups $\pi_n X_\bullet$ are uniquely divisible, for any $n\geqslant1$~. 
	\item The homology groups $H_n(X_\bullet;\mathbb{Z})$ are uniquely divisible, for any  $n\geqslant1$~. 
\end{enumerate}
\end{definition}

For more details about this equivalence in the nilpotent case, we refer the reader to \cite[Chapter~V, Proposition~3.3]{BousfieldKan}. It can be shown that these two conditions are not equivalent if one drops the nilpotence assumption. 

\begin{definition}[Finite type simplicial set]
A  rational  simplicial set $X_\bullet \in\sSe$ is of \emph{finite type} if its rational homology groups 
 $H_n(X_\bullet;\mathbb{Q})$ are finite dimensional $\mathbb{Q}$-vector spaces, for $n\geqslant1$~.
\end{definition}

If $X_\bullet$ is a connected simplicial set in $\sSe$~, then we say that $X_\bullet$ is nilpotent or rational, if it is so for any possible choice of base point. 
We denote by $\fnQsp$ (resp. $\fnQsp_*$) the full subcategory of the homotopy category $\ho(\sSe)$ 
(resp. $\ho(\sSe_*)$)
with objects the rational (resp. pointed) simplicial sets of finite type. 

\begin{definition}[Finite type commutative algebra]
A \emph{finite type} unital commutative algebra $A$ in $\ucomnalg$ (resp. $\ucomnalg_*$) is 
a homologically connected algebra, i.e. its unit induces an isomorphism $H_0(\eta):\k\to H_0(A)$~, whose 
 homology groups $H_n(X_\bullet;\mathbb{Q})$ are finite dimensional $\mathbb{Q}$-vector spaces, for $n\leqslant1$~.
\end{definition}

We denote by $\fQalg$ (resp. $\fQalg_*$) 
 the full subcategory of of the homotopy category  $\ho(\ucomnalg)$ generated by the finite type (resp. augmented) unital commutative algebras. 

\begin{theorem}[{Sullivan--de Rham equivalence theorem, \cite[Theorem~9.4]{BousfieldGugenheim}}]
	The adjunction $\APL \dashv \mathbb{R}\langle -\rangle$ restricts to an equivalence of categories between 
	$\fnQsp$ and $\fQalg$~, and similarly between $\fnQsp_*$ and $\fQalg_*$~. 
\end{theorem}

Bousfield--Kan introduced in \cite[Chapter~I]{BousfieldKan} a \emph{$\QQ$-completion functor} 
\[\QQ_\infty : \sSe_* \longrightarrow \sSe_*\]
which lands in Kan complexes and which satisfies the key property that any morphism $f : X_\bullet \to Y_\bullet$ induces an isomorphism $\widetilde{H}_n(f, \QQ)$ on the level of the rational reduced homology groups, for any $n\geqslant 0$~, if and only if $\QQ_\infty(f)$ is a homotopy equivalence. This functor comes equipped with a natural transformation from the identity functor, that is 
\[X_\bullet \longrightarrow \QQ_\infty(X_\bullet)~.\]

\begin{theorem}[{\cite[Theorem~12.2]{BousfieldGugenheim}}]\label{thm:Sullivan=BK}
For pointed connected finite type simplicial sets, the unit 
	\[
	X_\bullet\longrightarrow \mathbb{R}\langle\APL(X_\bullet)\rangle
	\]
of the 
$\APL \dashv \mathbb{R}\langle -\rangle$ adjunction is equivalent to the $\QQ$-completion. 
\end{theorem}

\begin{definition}[Rationalization]
A \emph{rationalization} (or \emph{$\mathbb{Q}$-localization}) of a nilpotent simplicial set $X_\bullet$ is a rational simplicial set $X^\QQ_\bullet$ equipped with a morphism $r : X_\bullet\to X^\QQ_\bullet$ satisfying either of the following two  equivalent conditions.
\begin{enumerate}
	\item It induces isomorphisms $\pi_n r : \pi_n X_\bullet \otimes \QQ \xrightarrow{\cong} \pi_n X_\bullet^\QQ$~, for any $n\geqslant1$~, where the left-hand side stands for the Malcev completion \cite{Malcev} in the case $n=1$~.
		\item It induces isomorphisms $H_n(r) : H_n(X_\bullet, \QQ)  \xrightarrow{\cong} H_n\big(X_\bullet^\QQ\big)$~, for any $n\geqslant1$~.
\end{enumerate}
\end{definition}

For more details about this equivalence in the nilpotent case, we refer the reader to \cite[Proposition~3.2, Chapter V]{BousfieldKan}. 

\begin{theorem}[{Sullivan--de Rham rationalization \cite{Sullivan77}, \cite[Theorem~11.2]{BousfieldGugenheim}}]
For pointed connected nilpotent finite type simplicial sets $X_\bullet$~, the unit 
	\[
	X_\bullet\longrightarrow \mathbb{R}\langle\APL(X_\bullet)\rangle
	\]
of the $\APL \dashv \mathbb{R}\langle -\rangle$ adjunction is a functorial rationalization.
\end{theorem}

\begin{proof}
This follows from \cref{thm:Sullivan=BK} and \cite[Proposition~3.1]{BousfieldKan} which asserts that the $\QQ$-completion provides us with a rationalization for nilpotent simplicial sets. 
\end{proof}

\subsection{Dupont--Sullivan adjunction}
\cref{lem:AdjCosimpliObj} applied to the simplicial $\Cobar \Bar \com$-algebra $\rmC_\bullet$ of \cref{subsec:MCcosimpsLI} produces the following  adjoint pair of  functors: 
\[
\hbox{
	\begin{tikzpicture}
	\def\upshift{0.185}
	\def\downshift{0.15}
	\pgfmathsetmacro{\midshift}{0.005}
	
	\node[left] (x) at (0, 0) {$\CPL \ \ :\ \ \sSe$};
	\node[right] (y) at (2, 0) {$\Omega\Bar\com\text{-}\,\mathsf{alg}^\mathrm{op}\ \ :\ \ \langle-\rangle_\infty~,$};
	
	\draw[-{To[left]}] ($(x.east) + (0.1, \upshift)$) -- ($(y.west) + (-0.1, \upshift)$);
	\draw[-{To[left]}] ($(y.west) + (-0.1, -\downshift)$) -- ($(x.east) + (0.1, -\downshift)$);
	
	\node at ($(x.east)!0.5!(y.west) + (0, \midshift)$) {\scalebox{0.8}{$\perp$}};
	\end{tikzpicture}}
\]
where 
\[\langle A \rangle_\infty\coloneqq\Hom_{\Omega\Bar\com\text{-}\,\mathsf{alg}}(A,\rmC_\bullet) \quad \text{and}\quad 
\CPL(X_\bullet)\coloneqq \lim_{\mathsf{E}(X_\bullet)}\rmC_\bullet\ .\]

For clarity in this section, let us denote by $\Omega_\bullet^{\mathrm{ch}}$ and $\rmC_\bullet^{\mathrm{ch}}$ the respective underlying simplicial chain complexes. 
Limits of algebras are given by the limits of the underlying chain complexes, for example the underlying chain complex of 
$\APL(X_\bullet)$ is given by $\lim_{\mathsf{E}(X_\bullet)}\Omega_\bullet^{\mathrm{ch}}$ and the underlying chain complex of $\CPL(X_\bullet)$ is given by $\lim_{\mathsf{E}(X_\bullet)}\rmC_\bullet^{\mathrm{ch}}$~. 
Since Dupont's contraction is compatible with the simplicial structures of $\Omega_\bullet^{\mathrm{ch}}$ and $\rmC_\bullet^{\mathrm{ch}}$ by \cref{prop:DupontContr}, it induces a contraction on the level of the limits over $\mathsf{E}(X_\bullet)$~:
\begin{equation}\label{eq:ContraX}
\vcenter{\hbox{
	\begin{tikzpicture}
	\def\upshift{0.075}
	\def\downshift{0.075}
	\pgfmathsetmacro{\midshift}{0.005}
	
	\node[left] (x) at (0, 0) {$\lim_{\mathsf{E}(X_\bullet)}\Omega_\bullet^{\mathrm{ch}}$};
	\node[right=1.5 cm of x] (y) {$\lim_{\mathsf{E}(X_\bullet)}\rmC_\bullet^{\mathrm{ch}}$};
	
	\draw[-{To[left]}] ($(x.east) + (0.1, \upshift)$) -- node[above]{\mbox{\tiny{$p(X_\bullet)$}}} ($(y.west) + (-0.1, \upshift)$);
	\draw[-{To[left]}] ($(y.west) + (-0.1, -\downshift)$) -- node[below]{\mbox{\tiny{$i(X_\bullet)$}}} ($(x.east) + (0.1, -\downshift)$);
	
	\draw[->] ($(x.south west) + (0, 0.1)$) to [out=-160,in=160,looseness=5] node[left]{\mbox{\tiny{$h(X_\bullet)$}}}
	 ($(x.north west) - (0, 0.1)$);
	\end{tikzpicture}}}.
\end{equation}

\begin{proposition}\label{prop: structures on CPL}\leavevmode
\begin{enumerate}
	\item Two $\Cobar\Bar\com$-algebra structures on $\lim_{\mathsf{E}(X_\bullet)}\rmC_\bullet^{\mathrm{ch}}$ obtained by the limit $\CPL(X_\bullet)= \lim_{\mathsf{E}(X_\bullet)}\rmC_\bullet$ in the category of 
$\Cobar\Bar\com$-algebras and by the homotopy transfer theorem \cref{thm:HTT} applied to the contracting homotopy \eqref{eq:ContraX} and to the $\Cobar\Bar\com$-algebra structure on $\APL(X_\bullet)=\lim_{\mathsf{E}(X_\bullet)}\Omega_\bullet$ are equal. 	
		
	\item\label{pt:DS square} For any morphism $f : X_\bullet \to Y_\bullet$ of simplicial sets, the following diagram commutes in the category of $\Omega\Bar\com$-algebras.
	\[
	\vcenter{\hbox{
		\begin{tikzpicture}
			\node (a) at (0,0) {$\APL(Y_\bullet)$};
			\node (b) at (3,0) {$\APL(X_\bullet)$};
			\node (c) at (0,-2) {$\CPL(Y_\bullet)$};
			\node (d) at (3,-2) {$\CPL(X_\bullet)$};
			
			\draw[->] (a) to node[above]{$\scriptstyle{\APL(f)}$} (b);
			\draw[->] (c) to node[left]{$\scriptstyle{i(X_\bullet)}$} (a);
			\draw[->] (c) to node[above]{$\scriptstyle{\CPL(f)}$} (d);
			\draw[->] (d) to node[right]{$\scriptstyle{i(Y_\bullet)}$} (b);
		\end{tikzpicture}
	}}
	\]
\end{enumerate}
\end{proposition}

\begin{proof}
Regarding the first point, let $x:\Delta^n\to X_\bullet$ be any $n$-simplex of $X_\bullet$~. We denote by 
	\[
	\pi^{\rmC}_x\ \colon \  \lim_{\mathsf{E}(X_\bullet)} \rmC_\bullet^{\mathrm{ch}} \longrightarrow \rmC_n^{\mathrm{ch}}
	\]
	 the canonical projection map coming from the limit over the category of elements of 
	 $X_\bullet$~.  
	 It will be enough to show that $\pi^{\rmC}_x$ is a morphism of $\Cobar\Bar\com$-algebras if we endow the left-hand side with the structure transferred from $\APL(X)$ and the right-hand side with the structure transferred from $\Omega_n$~. 
If this is the case, by the universal property of limits, there exists a unique map of $\Cobar\Bar\com$-algebras from 
$\lim_{\mathsf{E}(X_\bullet)} \rmC_\bullet^{\mathrm{ch}}$ with the transferred structure to $\CPL(X)$~. Since the underlying chain complex of the latter is also $\lim_{\mathsf{E}(X_\bullet)} \rmC_\bullet^{\mathrm{ch}}$~, the universal property of limits on the chain level shows that this map is the identity. Therefore, the two structures are equal.
	
	\medskip
	
	In order to prove that $\pi^{\rmC}_x$
	is a morphism of $\Cobar\Bar\com$-algebras, we use the explicit form of the homotopy transfer theorem given below \cref{thm:HTT} in terms of twisting morphisms. If we denote by 
	\[
	\varphi\in\Tw\left(\Bar\com, \End_{\APL(X_\bullet)}\right),
	\]
	the twisting morphism corresponding to the $\Cobar\Bar\com$-algebra structure on $\APL(X_\bullet)$~, then 
 the twisting morphism  $\varphi^{\mathrm{tr}}$corresponding to the $\Cobar\Bar\com$-algebra structure transferred on the limit
$\lim_{\mathsf{E}(X_\bullet)} \rmC_\bullet^{\mathrm{ch}}$ is equal to 
	\[
	\varphi^{\mathrm{tr}}= \operatorname{VdL}\circ\T^c(s\varphi)\circ\widetilde{\Delta}_{\Bar\com}\ ,
	\]
	where
	\[
	\widetilde{\Delta}_{\Bar\com}:\Bar\com\longrightarrow\T^c(\Bar\com)
	\]
	is the monadic decomposition map of the cooperad $\Bar\com$ and $\VdL$ is the Van der Laan map.
	Even more explicitly, let $\tau \in \RT_m$ be a rooted tree, viewed as a basis element $\Bar\com$~,  and let  $x_1, \ldots, x_m$ be $m$ elements of $\lim_{\mathsf{E}(X_\bullet)} \rmC_\bullet^{\mathrm{ch}}$~. The element $\varphi^{\mathrm{tr}}(\tau)(x_1, \ldots, x_m)$ is obtained by labeling the vertices of $\tau$ with the iterated commutative product of $\APL(X_\bullet)$~, the internal edges by the contracting homotopy $h(X_\bullet)$~, the root by $p(X_\bullet)$~,  and by applying this composite of operations to $i(X_\bullet)(x_1), \ldots, i(X_\bullet)(x_m)$~.  The canonical projection
	\[
	\pi^{\Omega}_x\ \colon \  \APL(X_\bullet) \cong \lim_{\mathsf{E}(X_\bullet)} \Omega_\bullet\longrightarrow \Omega_n
	\]
	is a morphism of commutative algebras which commutes with $\pi^\rmC_x$ and the various maps of the contracting homotopies. This implies 
	\[\pi_x^\rmC\varphi^{\mathrm{tr}}(\tau)(x_1, \ldots, x_m)=
	\varphi^{\mathrm{tr}}_{\rmC_n}(\tau)\left(\pi_x^\rmC(x_1), \ldots,\pi_x^\rmC(x_m)\right)~, 
	\]
	where $\varphi^{\mathrm{tr}}_{\rmC_n}$ denotes the transferred $\Cobar\Bar\com$-algebra structure on $\rmC_n$~. 
	So $\pi_x^{\rmC}$ is a morphism of $\Cobar\Bar\com$-algebras, and thus, that the two structures are equal.
	
	\medskip
	
	Regarding the second point, it is  a direct consequence of the universal property of the limit since Dupont's contraction is simplicial.
\end{proof}

\begin{remark}
Notice that Point~(1) and \cite[Theorem~11.4.9]{LodayVallette12} imply that there exist a zig-zag of quasi-isomorphisms of $\Cobar \Bar \com$-algebras 
\[\APL(X_\bullet) 
\stackrel{\sim}{\longleftarrow}\cdot \stackrel{\sim}{\longrightarrow} \cdots 
\stackrel{\sim}{\longleftarrow}\cdot \stackrel{\sim}{\longrightarrow} 
\CPL(X_\bullet)  ~, \]
that is $\CPL(X_\bullet)$ is a ``homotopy commutative model'' for $X_\bullet$~. 
\end{remark}

\begin{remark}
	Point (\ref{pt:DS square}) of \cref{prop: structures on CPL} can be significantly strengthened as follows, but we will not use it in the sequel. The maps $i(X_\bullet)$ and $p(X_\bullet)$ can be extended to $\infty_{\iota_\com}$-quasi-isomorphisms $i_\infty(X_\bullet)$ and $p_\infty(X_\bullet)$ via canonical formulas 
	\cite[Section~10.3]{LodayVallette12}
	 and similarly for $Y_\bullet$~. If $f:X_\bullet\to Y_\bullet$ is a morphism of simplicial sets, then the squares
	\[
	\vcenter{\hbox{
			\begin{tikzpicture}
			\node (a) at (0,0) {$\APL(Y_\bullet)$};
			\node (b) at (3,0) {$\APL(X_\bullet)$};
			\node (c) at (0,-2) {$\CPL(Y_\bullet)$};
			\node (d) at (3,-2) {$\CPL(X_\bullet)$};
			
			\draw[->] (a) to node[above]{$\scriptstyle{\APL(f)}$} (b);
			\draw[->,decorate,decoration={zigzag,segment length=4,amplitude=.9,post=lineto,post length=2pt}] (c) to node[left]{$\scriptstyle{i_\infty(X_\bullet)}$} (a);
			\draw[->] (c) to node[above]{$\scriptstyle{\CPL(f)}$} (d);
			\draw[->,decorate,decoration={zigzag,segment length=4,amplitude=.9,post=lineto,post length=2pt}] (d) to node[right]{$\scriptstyle{i_\infty(Y_\bullet)}$} (b);
			\end{tikzpicture}
	}}
	\qquad\text{and}\qquad
	\vcenter{\hbox{
		\begin{tikzpicture}
			\node (a) at (0,0) {$\APL(Y_\bullet)$};
			\node (b) at (3,0) {$\APL(X_\bullet)$};
			\node (c) at (0,-2) {$\CPL(Y_\bullet)$};
			\node (d) at (3,-2) {$\CPL(X_\bullet)$};
			
			\draw[->] (a) to node[above]{$\scriptstyle{\APL(f)}$} (b);
			\draw[->,decorate,decoration={zigzag,segment length=4,amplitude=.9,post=lineto,post length=2pt}] (a) to node[left]{$\scriptstyle{p_\infty(X_\bullet)}$} (c);
			\draw[->] (c) to node[above]{$\scriptstyle{\CPL(f)}$} (d);
			\draw[->,decorate,decoration={zigzag,segment length=4,amplitude=.9,post=lineto,post length=2pt}] (b) to node[right]{$\scriptstyle{p_\infty(Y_\bullet)}$} (d);
		\end{tikzpicture}
	}}
	\]
	commute in the category of $\Omega\Bar\com$-algebras, i.e.\ the $\infty_\pi$-quasi-isomorphisms $i_\infty(-)$ and $p_\infty(-)$ are natural. In particular, this is true for $x:\De{n}\to Y_\bullet$ the inclusion of any $n$-simplex, for which $\APL(x)=\pi^\Omega_x$ and $\CPL(x)=\pi^{\rmC}_x$~. The proof is left to the reader since it follows easily by unwinding the definitions and using the facts:
	\[\APL(f)i(Y_\bullet) = i(X_\bullet)\CPL(f)\ , \ \CPL(f)p(Y_\bullet) = p(X_\bullet)\APL(f)\ ,\
	\APL(f)h(X_\bullet) = h(X_\bullet)\APL(f)~.\]
\end{remark}

We equip the category of $\Cobar \Bar \com$-algebras with its projective model category structure where 
the weak equivalences are the quasi-isomorphisms and where the fibrations are the degree-wise surjections, see  \cite{Hinich97}.

\begin{corollary}\label{cor:QuillenAdjCPL}
The  pair of adjoint functors $\CPL\dashv\langle-\rangle_\infty$ is  a Quillen adjunction.
\end{corollary}

\begin{proof}
The left adjoint functor $\CPL$ sends cofibrations (degree-wise injections) to fibrations (de\-gree-wise surjections) by construction.
It also sends and acyclic cofibrations to acyclic fibrations by Point~(2) of \cref{prop: structures on CPL} since $\APL$ sends acyclic cofibrations to acyclic fibrations 
and since the chain maps $i(X_\bullet)$ and $i(Y_\bullet)$ are both quasi-isomorphisms. 
\end{proof}

\subsection{Getting rid of artificial base points}

A quick look at the $\sLi$-algebra functor $\Li$ reveals that it creates an ``artificial'' canonical ``base point'' in the sense that $\Li(X_\bullet)$ contains the Maurer--Cartan element $0$~, for any simplicial set $X_\bullet$~. For example, 
\[
\Li(*) \cong \mc_0
\]
should intuitively represent a single point, but it has two Maurer--Cartan elements: its generator and the canonical Maurer--Cartan element $0$~.

\medskip

In order to work in the correct context for rational homotopy theory, we consider pointed spaces rather than just spaces.  One can get rid of the added point by working as follows. 

\begin{definition}[Reduced $\sLi$-algebra functor]
The \emph{reduced $\sLi$-algebra functor} 
\[
\widetilde{\mathfrak{L}}:\sSe_*\longrightarrow\sLialg
\]
is defined by the coequalizer
\begin{center}
	\begin{tikzpicture}
	\node (a) at (0,0) {$\mathfrak{L}(*) \cong \mc^0$};
	\node (b) at (3.5,0) {$\mathfrak{L}(X_\bullet)$};
	\node (c) at (6.5,0) {$\widetilde{\mathfrak{L}}(X_\bullet)\ .$};
	
	\path[->] ($(a.east)+(0,.07)$) edge node[above]{$\mathfrak{L}(x)$} ($(b.west)+(0,.07)$);
	\path[->] ($(a.east)-(0,.07)$) edge node[below]{$0$} ($(b.west)-(0,.07)$);
	
	\path[->] (b) edge (c);
	\end{tikzpicture}
\end{center}
where $x$ stands for the base point $
*\to X_\bullet$~. 
\end{definition}

Recall that colimits of $\Bar \com$-coalgebras are given by colimits of the underlying chain complexes. 
So, if we denote the cokernel on the chain level by 
\[\k\cong C(*) \xrightarrow{C(x)} C(X_\bullet) \twoheadrightarrow \widetilde{C}(X_\bullet)~,\]
then the reduced $\sLi$-algebra functor is isomorphic to 
\[\widetilde{\Li}(X_\bullet)\cong 
\sLi\big(\widetilde{C}(X_\bullet)\big)
~,\]
with the differential coming from the $\Bar \com$-coalgebra structure on 
$C(X_\bullet)\cong \lim_{\mathsf{E}(X_\bullet)}\rmC^\bullet$~. 

\begin{definition}[Pointed integration functor]
The \emph{pointed integration functor} 
\[
\widetilde{\R}:\sLialg\longrightarrow\sSe_*
\]
is defined by $\widetilde{\R}(\g)\coloneqq \R(\g)$ with the canonical base point 
\[
*\cong\R(0)\xrightarrow{\R(0)}\R(\g)~.
\]
\end{definition}

\begin{lemma}\label{lemma:get rid of additional point}
	The pair of functors $\widetilde{\Li}$ and $\widetilde{\R}$ are adjoint 
\[
\hbox{
	\begin{tikzpicture}
	\def\upshift{0.185}
	\def\downshift{0.15}
	\pgfmathsetmacro{\midshift}{0.005}
	
	\node[left] (x) at (0, 0) {$\widetilde{\Li}\ \ :\ \ \sSe_*$};
	\node[right] (y) at (2, 0) {$\sLialg\ \ :\ \ \widetilde{\R}\ .$};
	
	\draw[-{To[left]}] ($(x.east) + (0.1, \upshift)$) -- ($(y.west) + (-0.1, \upshift)$);
	\draw[-{To[left]}] ($(y.west) + (-0.1, -\downshift)$) -- ($(x.east) + (0.1, -\downshift)$);
	
	\node at ($(x.east)!0.5!(y.west) + (0, \midshift)$) {\scalebox{0.8}{$\perp$}};
	\end{tikzpicture}}
\]
\end{lemma}

\begin{proof}
	Let
	\[
	*\stackrel{x}{\longrightarrow}X\in\sSe_*
	\]
	be a pointed space. By definition, the data of a morphism of complete $\sLi$-algebras $\widetilde{f} : \widetilde{\Li}(X_\bullet)\to\g$ is equivalent to the data of  a morphism $f : \Li(X_\bullet)\to\g$ such that
	\[
	\vcenter{\hbox{
		\begin{tikzpicture}
			\node (a) at (0,0){$\Li(*)$};
			\node (b) at (2.5,0){$\Li(X_\bullet)$};
			\node (c) at (0,-1.5){$0$};
			\node (d) at (2.5,-1.5){$\g$};
			
			\draw[->] (a) to node[above]{$\Li(x)$} (b);
			\draw[->] (a) to (c);
			\draw[->] (b) to node[right]{${f}$} (d);
			\draw[->] (c) to (d);
		\end{tikzpicture}
	}}
	\]
	Through the original adjunction $\Li\dashv \R$~, this is in unique natural correspondence with the commutative  square
	\[
	\vcenter{\hbox{
		\begin{tikzpicture}
			\node (a) at (0,0){$*$};
			\node (b) at (2.5,0){$X_\bullet$};
			\node (c) at (0,-1.5){$\R(0)\cong*$};
			\node (d) at (2.5,-1.5){$\R(\g)$};
			
			\draw[->] (a) to node[above]{$x$} (b);
			\draw[->] (a) to (c);
			\draw[->] (b) to node[right]{$g$} (d);
			\draw[->] (c) to (d);
		\end{tikzpicture}
	}}
	\]
	proving that the two functors are adjoint.
\end{proof}

\subsection{Homotopy Lie rational models}
We are now ready to prove that the functors of higher Lie theory faithfully recover the  rational homotopy type of spaces. 

\begin{theorem}\label{thm: rationalization}
For $X_\bullet$ a pointed connected finite simplicial set, the unit 
	\[
	X_\bullet \longrightarrow \widetilde{\R}\widetilde{\Li}(X_\bullet)
	\]
	of the adjunction $\widetilde{\Li}\dashv\widetilde{\R}$ is homotopy equivalent to the $\QQ$-completion $\QQ_\infty(X_\bullet)$ of Bousfield--Kan. In particular, it is a rationalization for these simplicial sets that are nilpotent.
\end{theorem}

\begin{proof}	
Let $X_\bullet$ be a pointed connected nilpotent simplicial set. 
If we denote the base point by $x \colon * \to X_\bullet$~, the unital commutative algebra $\APL(X_\bullet)$ is augmented by $\APL(x)\colon  \APL(X_\bullet) \to \APL(*)\cong\QQ$~. We denote by $\overline{\mathrm{A}}_{\mathrm{PL}}(X_\bullet)$ its kernel and we equip it with the induced commutative algebra structure. Since $X_\bullet$ is connected, the algebra $\APL(X_\bullet)$ is homologically connected and so it admits a cofibrant resolution 
that is a  Sullivan algebra $(S(V), \d)$~, where $V$ is concentrated in negative degree, by \cite[Proposition~12.1]{RHT}. We can compute the right derived functor as	
\begin{align*}
	\mathbb{R}\langle\APL(X_\bullet)\rangle&\simeq\Hom_{\ucomnalg}\left(S(V), \Omega_\bullet\right)
	\cong \Hom_{\com_{\leqslant 0}\text{-}\mathsf{alg}}\left(\bar{S}(V), \Omega_\bullet\right)\\&
	\cong \Hom_{\com\text{-}\mathsf{alg}}\left(\bar{S}(V), \Omega_\bullet\right).
\end{align*}
Let us recall the  commutative triangle of made up of two  canonical twisting morphisms and one morphism of operads 
	\begin{center}
		\begin{tikzpicture}

		\node (a) at (2,2){$\Omega\Bar\com$};
		\node (b) at (0,1){$\Bar\com$};
		\node (c) at (2,0){$\com$};
		
		\draw[->] (b) -- node[above left]{$\iota_\com$} (a);
		\draw[->] (a) -- node[above, sloped]{$\sim$} (c);
		\draw[->] (b) -- node[below left]{$\pi_\com\!$} (c);
		\end{tikzpicture}, 
	\end{center}
which induces the following bar-cobar adjunctions by \cite[Section~11.3]{LodayVallette12}:
\[
\vcenter{\hbox{
	\begin{tikzpicture}
	\def\upshift{0.185}
	\def\downshift{0.15}
	\pgfmathsetmacro{\midshift}{0.005}
	
	\node[left] (x) at (0, 0) {$\Cobar_{\pi_\com}\ \ :\ \ \mathsf{conil}\ \Bar\com\text{-}\mathsf{coalg}$};
	\node[right] (y) at (2, 0) {$\com\text{-}\mathsf{alg}\ \ :\ \ \Bar_{\pi_\com}\ ,$};
	
	\draw[-{To[left]}] ($(x.east) + (0.1, \upshift)$) -- ($(y.west) + (-0.1, \upshift)$);
	\draw[-{To[left]}] ($(y.west) + (-0.1, -\downshift)$) -- ($(x.east) + (0.1, -\downshift)$);
	
	\node at ($(x.east)!0.5!(y.west) + (0, \midshift)$) {\scalebox{0.8}{$\perp$}};
	
	\node[left] (x2) at (0, -1) {$\Cobar_{\iota_\com}\ \ :\ \ \mathsf{conil}\ \Bar\com\text{-}\mathsf{coalg}$};
	\node[right] (y2) at (2, -1) {$\Cobar\Bar\com\text{-}\mathsf{alg}\ \ :\ \ \Bar_{\iota_\com}\ .$};
	
	\draw[-{To[left]}] ($(x2.east) + (0.1, \upshift)$) -- ($(y2.west) + (-0.1, \upshift)$);
	\draw[-{To[left]}] ($(y2.west) + (-0.1, -\downshift)$) -- ($(x2.east) + (0.1, -\downshift)$);
	
	\node at ($(x2.east)!0.5!(y2.west) + (0, \midshift)$) {\scalebox{0.8}{$\perp$}};
	\end{tikzpicture}}}
\]
The two bar constructions agree on commutative algebras.

\medskip

The category $\com\text{-}\mathsf{alg}$ of differential $\mathbb{Z}$-graded commutative algebras admits a projective model category structure where the weak equivalences are the quasi-isomorphisms and where the fibrations are the degree-wise surjections, by  \cite{Hinich97}. The arguments of Theorem~2.6-(1) and Theorem~2.9-(1) of \cite{Vallette20} show that the counit 
\[\Cobar_{\pi_\com} \Bar_{\pi_\com} \overline{\mathrm{A}}_{\mathrm{PL}}(X_\bullet) \stackrel{\sim}{\longrightarrow} \overline{\mathrm{A}}_{\mathrm{PL}}(X_\bullet)\]
of the $\Cobar_{\pi_\com}\dashv \Bar_{\pi_\com}$-adjunction is a cofibrant resolution. The formula \cite[Section~11.3.2]{LodayVallette12} for this counit shows that is degree-wise surjective, that is a fibration, so there exists a quasi-isomorphism of cofibrant commutative algebras 
\[\bar{S}(V) \stackrel{\sim}{\longrightarrow} \Cobar_{\pi_\com} \Bar_{\pi_\com} \overline{\mathrm{A}}_{\mathrm{PL}}(X_\bullet)~.   \]

\medskip
With the same formulas, one can consider an adjunction similar to $\APL\dashv\langle-\rangle$ with the category $\ucomnalg$ being replaced  by the category 
$\com\text{-}\mathsf{alg}$~. Such an adjunction is again a Quillen adjunction: the left adjoint functor sends cofibrations to fibrations, by construction, and 
acyclic cofibrations to acylic fibrations, by \cref{lem:QuillenAdj}.  Therefore, its right adjoint functor sends quasi-isomorphisms between cofibrant commutative algebras to weak equivalences of simplicial sets, by 
Ken Brown's lemma \cite[Lemma~1.1.12]{Hovey99}. Here, this implies
	\[
	\mathbb{R}\langle\APL(X_\bullet)\rangle 
	\simeq \Hom_{\com\text{-}\mathsf{alg}}\left(\bar{S}(V), \Omega_\bullet\right)
	\simeq \Hom_{\com\text{-}\mathsf{alg}}\left(\Cobar_{\pi_\com} \Bar_{\pi_\com} \overline{\mathrm{A}}_{\mathrm{PL}}(X_\bullet), \Omega_\bullet\right) ~.
	\]
Using the two adjoint pairs of functors and the property that the   above two bar constructions coincide on commutative algebras, we get 
\begin{align*}
	\mathbb{R}\langle\APL(X_\bullet)\rangle & \simeq 
	\Hom_{\com\text{-}\mathsf{alg}}\left(\Cobar_{\pi_\com} \Bar_{\pi_\com} \overline{\mathrm{A}}_{\mathrm{PL}}(X_\bullet), \Omega_\bullet\right)\\
&\cong 
	\Hom_{\Bar\com\text{-}\mathsf{coalg}}\left(\Bar_{\pi_\com} \overline{\mathrm{A}}_{\mathrm{PL}}(X_\bullet), \Bar_{\pi_\com}  \Omega_\bullet\right) \\
		&\cong 
	\Hom_{\Bar\com\text{-}\mathsf{coalg}}\left(\Bar_{\iota_\com} \overline{\mathrm{A}}_{\mathrm{PL}}(X_\bullet), \Bar_{\iota_\com}  \Omega_\bullet\right)\\ &\cong 
	\Hom_{\Cobar\Bar\com\text{-}\mathsf{alg}}\left(\Omega_{\iota_\com} \Bar_{\iota_\com} \overline{\mathrm{A}}_{\mathrm{PL}}(X_\bullet), \Omega_\bullet\right) ~. 
\end{align*}
Theorem~2.7 of \cite{rnw17}, which is the analogue of \cref{thm:MC elements of convolution algebras} in the non-necessarily complete case, gives 
\[		\mathbb{R}\langle\APL(X_\bullet)\rangle  \simeq 
\Hom_{\Cobar\Bar\com\text{-}\mathsf{alg}}\left(\Omega_{\iota_\com} \Bar_{\iota_\com} \overline{\mathrm{A}}_{\mathrm{PL}}, \Omega_\bullet\right)
\cong \MC\left(\Hom^{\iota_\com}\left( \Bar_{\iota_\com} \overline{\mathrm{A}}_{\mathrm{PL}}, \Omega_\bullet\right)\right).
\]
Using it again and the first step of the proof of \cite[Theorem~4.1]{rnw18}, see Point~(1) of page 9 of \emph{loc.\ cit.}, we have
\begin{align*}
 \mathbb{R}\langle\APL(X_\bullet)\rangle  &\simeq
\MC\left(\Hom^{\iota_\com}\left( \Bar_{\iota_\com} \overline{\mathrm{A}}_{\mathrm{PL}}(X_\bullet), \Omega_\bullet\right)\right)
\simeq 
\MC\left(\Hom^{\iota_\com}\left( \Bar_{\iota_\com} \overline{\mathrm{A}}_{\mathrm{PL}}(X_\bullet), \rmC_\bullet\right)\right)\\
&\cong 
\Hom_{\Cobar\Bar\com\text{-}\mathsf{alg}}\left(\Omega_{\iota_\com} \Bar_{\iota_\com} \overline{\mathrm{A}}_{\mathrm{PL}}(X_\bullet), \rmC_\bullet\right).
 \end{align*}
Point~(1) of \cref{prop: structures on CPL} provides us with an $\infty_{\iota_\com}$-quasi-isomorphism 
\[ {\mathrm{C}}_{\mathrm{PL}}(X_\bullet)
\stackrel{\sim}{\rightsquigarrow}
{\mathrm{A}}_{\mathrm{PL}}(X_\bullet)~,
\]
obtained as an extension of $i(X_\bullet)$~, for instance under the tree-wise formula of \cite[Section~10.3.5]{LodayVallette12}. The definition of $\CPL(X_\bullet)$ in terms of a limit shows that there exists a unique morphism
$\bar{\imath}(X_\bullet)$ which makes the following diagram commutative
\[
\vcenter{\hbox{
	\begin{tikzpicture}
		\node (a) at (0,0) {$\overline{\mathrm{A}}_{\mathrm{PL}}(X_\bullet)$};
		\node (b) at (3.5,0) {$\APL(X_\bullet)$};
		\node (c) at (7,0) {$\APL(*)\cong \QQ$};
		\node (d) at (0,-2) {$\overline{\mathrm{C}}_{\mathrm{PL}}(X_\bullet)$};
		\node (e) at (3.5,-2) {$\CPL(X_\bullet)$};
		\node (f) at (7,-2) {$\CPL(*)\cong \QQ$};
		
		\draw[right hook->] (a) to (b);
		\draw[->] (b) to node[above]{$\scriptstyle{\APL(x)}$} (c);
		\draw[right hook->] (d) to (e);
		\draw[->] (e) to node[above]{$\scriptstyle{\CPL(x)}$} (f);
		\draw[->] (d) to node[left]{$\scriptstyle{\bar{\imath}(X_\bullet)}$} (a);
		\draw[->] (e) to node[left]{$\scriptstyle{i(X_\bullet)}$} (b);
		\draw[->] (f) to node[right]{$\scriptstyle{i(*)}$} (c);
	\end{tikzpicture}
}}
\]
Replacing $i(X_\bullet)$ by $\bar{\imath}(X_\bullet)$ in the tree-wise formula for the $\infty_{\iota_\com}$-morphism extending $i(X_\bullet)$~, one gets an $\infty_{\iota_\com}$-quasi-isomorphism extending 
$\bar{\imath}(X_\bullet)$~:
\[
\overline{\mathrm{C}}_{\mathrm{PL}}(X_\bullet)
\stackrel{\sim}{\rightsquigarrow}
\overline{\mathrm{A}}_{\mathrm{PL}}(X_\bullet)~,
\]
that induces a quasi-isomorphism
\[\Omega_{\iota_\com} \Bar_{\iota_\com} \overline{\mathrm{A}}_{\mathrm{PL}}(X_\bullet)
\stackrel{\sim}{\longrightarrow}
\Omega_{\iota_\com} \Bar_{\iota_\com}\overline{\mathrm{C}}_{\mathrm{PL}}(X_\bullet)~,
\]
 of cofibrant $\Cobar\Bar\com$-algebras by \cite[Proposition~4.4]{Vallette20} and the arguments used above (cobar is cofibrant). By Ken Brown's lemma applied to the Quillen adjunction of \cref{cor:QuillenAdjCPL}, we get 
\begin{align*}
 	\mathbb{R}\langle\APL(X_\bullet)\rangle & \simeq 
	\Hom_{\Cobar\Bar\com\text{-}\mathsf{alg}}\left(\Omega_{\iota_\com} \Bar_{\iota_\com} \overline{\mathrm{A}}_{\mathrm{PL}}(X_\bullet), \rmC_\bullet\right)\\&\simeq 
		\Hom_{\Cobar\Bar\com\text{-}\mathsf{alg}}\left(\Omega_{\iota_\com} \Bar_{\iota_\com} \overline{\mathrm{C}}_{\mathrm{PL}}(X_\bullet), \rmC_\bullet\right). 
\end{align*}
Using Theorem~2.7 of \cite{rnw17} one more time, we reach 
\[ 	\mathbb{R}\langle\APL(X_\bullet)\rangle \simeq 	
\MC\left(\Hom^{\iota_\com}\left( \Bar_{\iota_\com} \overline{\mathrm{C}}_{\mathrm{PL}}(X_\bullet), \rmC_\bullet\right)\right).\]
For any $n\geqslant 0$~,  since $\rmC_n$ is finite dimensional, we have and isomorphism 
\[\Hom\left( \Bar_{\iota_\com} \overline{\mathrm{C}}_{\mathrm{PL}}(X_\bullet), \rmC_n\right)\cong 
\Hom\left(\rmC^n,  \left(\Bar_{\iota_\com} \overline{\mathrm{C}}_{\mathrm{PL}}(X_\bullet)\right)^\vee\right), 
\]
which preserves the respective $\sLi$-algebra structures: the one induced by the twisting morphism 
$\iota_\com \colon \Bar \com \to \Cobar\Bar\com$~,
on the left-hand side, and the one induced by linear dual twisting morphism 
$
\pi \colon\Bar \Cobar \com^\vee \to \Cobar \com^\vee\cong \sLi
$~, 
on the right-hand side. 
This implies the following weakly equivalent form 
\begin{align*}
 	\mathbb{R}\langle\APL(X_\bullet)\rangle \simeq {}& 	
\MC\left(\Hom^{\pi}\left(\rmC^\bullet,  \left(\Bar_{\iota_\com} \overline{\mathrm{C}}_{\mathrm{PL}}\right)^\vee\right)\right)\\
\cong {}&
\hom_{\,\sLialg}\left(
\hatCobar_\pi \mathrm{C}^\bullet ,  \left(\Bar_{\iota_\com} \overline{\mathrm{C}}_{\mathrm{PL}}(X_\bullet)\right)^\vee
\right)=\R\left(\left(\Bar_{\iota_\com} \overline{\mathrm{C}}_{\mathrm{PL}}(X_\bullet)\right)^\vee\right).
\end{align*}
	Recall now that $\CPL(X_\bullet)$ is given by the right Kan extension 
	\[
	\CPL(X_\bullet)\coloneqq\lim_{\mathsf{E}(X_\bullet)}\rmC_\bullet \ ,
	\]
in the category of $\Cobar\Bar \com$-algebras.	Since $X_\bullet$ is a finite simplicial set, and since the chain complexes $\rmC_n$ are finite dimensional, for any $n\geqslant 1$, then $\CPL(X_\bullet)$ is also finite dimensional. 
From this property, we get 
\[\left(\Bar_{\iota_\com}\CPL(X_\bullet)\right)^\vee\cong \widehat{\Omega}_\pi \left(\CPL(X_\bullet)^\vee\right)
\cong
\widehat{\Omega}_\pi \left(\colim_{\mathsf{E}(X_\bullet)}\rmC^\bullet\right)
\ .
 \]
Using the fact that the complete cobar construction preserves colimits by \cref{CobarColimits}, we obtain 
\[\left(\Bar_{\iota_\com}\CPL(X_\bullet)\right)^\vee\cong
\colim_{\mathsf{E}(X_\bullet)} \widehat{\Omega}_\pi \rmC^\bullet\cong
\colim_{\mathsf{E}(X_\bullet)} \mc^\bullet\cong \Li(X_\bullet)~.
\]
Since limits of $\Cobar \Bar \com$-algebras are given by limits of the underlying chain complexes, the underlying chain complex of $\CPL(X_\bullet)$ is isomorphic to $C(X_\bullet)^\vee$~. 
Under the notation 
\[\k\cong C(*) \xrightarrow{C(x)} C(X_\bullet) \twoheadrightarrow \widetilde{C}(X_\bullet)~,\]
for the chain level cokernel, 
the underlying chain complex of $\overline{\mathrm{C}}_{\mathrm{PL}}(X_\bullet)$ is isomorphic to $\widetilde{C}(X_\bullet)^\vee$~. 
The isomorphism mentioned above is given on the chain level by
\[\left(\Bar \com\left( C(X_\bullet)^\vee\right)\right)^\vee\cong \sLi(C(X_\bullet))\ ,\]
which restricts to 
\[\left(\Bar \com\left( \widetilde{C}(X_\bullet)^\vee\right)\right)^\vee\cong \sLi\left(\widetilde{C}(X_\bullet)\right),\]
that is 
\[\left(\Bar_{\iota_\com} \overline{\mathrm{C}}_{\mathrm{PL}}(X_\bullet)\right)^\vee\cong \widetilde{\Li}(X_\bullet)~,\]
when one considers the respective algebraic structures and differentials. 
	Putting it all together, we finally obtain 
	\[
	\mathbb{R}\langle\APL(X_\bullet)\rangle\simeq\widetilde{\R}\widetilde{\mathfrak{L}}(X_\bullet)\ .
	\]
	The last point about rationalization follows from \cite[Proposition~3.1]{BousfieldKan}.
\end{proof}

\begin{remark}
	\cref{thm: rationalization} is a generalization of \cite[Theorem~2.11]{BFMT18} from Lie algebras to  $\sLi$-algebras. Notice that the present proof uses different methods. 
\end{remark}

\begin{corollary}\label{Cor:LastbutnotLeast}
Let $X_\bullet$ be a pointed connected nilpotent finite simplicial set $X_\bullet$. 
\begin{enumerate}
\item 
The Malcev completion of its fundamental group is isomorphic to 
	\[
	\widehat{\pi}_1(X_\bullet)\cong H_1\big(\widetilde{\Li}(X_\bullet)\big)~,
	\]
equipped with the Baker--Campbell--Hausdorff product. 
\item Its rational homotopy groups are isomorphic to 
	\[
	{\pi}_n X_\bullet \otimes \QQ \cong H_n\big(\widetilde{\Li}(X_\bullet)\big)~,
	\]
for any $n\geqslant 2$~. 
\end{enumerate}
\end{corollary}

\begin{proof}
Recall by \cite[Chapter~V, Proposition~3.1-(ii)]{BousfieldKan} that  the fundamental group of the Bousfield--Kan $\QQ$-completion computes the Malcev completion of the fundamental group of $X_\bullet$ and that the  higher homotopy groups of $\QQ_\infty(X_\bullet)$ are isomorphic to the rational higher homotopy groups of $X_\bullet$~. We conclude by \cref{thm: rationalization} and Berglund's Hurewicz theorem \cref{thm:Berglund}.
\end{proof}

\appendix

\section{Formal equations}\label{sec:app}

In this appendix, we provide explicit tree-wise formulas for fixed-point equations and formal differential equations. 

\subsection{Formal fixed-point equations}\label{Sec:FixedPtEq}
Let $(V, \mathrm{F})$ be a  complete chain complex, that is 
\[V=\F_0 V=\F_1 V\supset \F_2 V\supset \cdots\]
satisfying 
$V\cong \lim_{k\in \NN} V/\F_k V$~.

\begin{definition}[Analytic function]
	A (formal) \emph{analytic function}\index{formal!analytic function} $\PP : V \to V$ is an application of the form 
	\[\PP(x)=\sum_{m=0}^{\infty}\PP_m\left(x^{\otimes m}\right),\] 
	with $\PP_m \in \F_1 \Hom\left(V^{\otimes m}, V\right)$~, that is a linear map satisfying
	\[
	\PP_m(\F_{i_1} V, \ldots,\F_{i_m} V)\subset \F_{i_1+\cdots+i_m+1} V\ .
	\]
\end{definition}

We use the notation $p_0\coloneq\PP_0(1)\in \F_1 V$~, with $1$ the unit of $\k\cong V^{\otimes 0}$~. 

\medskip

We consider the set $\PT$ of \emph{planar rooted trees}\index{trees!planar rooted}\index{$\PT$} with vertices of arity greater or equal to $0$~. 
The subset made up of trees with $m$ leaves is denoted by $\PT_m$~. The number of vertices of a planar rooted tree $\tau$ is  denoted by $|\tau|$ and the subset of trees with $k$ vertices is denoted by $\PT^{(k)}$~. 
By convention, the set of planar rooted trees of weight $0$ is made up of the trivial arity 1 tree: 
$\PT^{(0)}\coloneq \{\,|\,\}$~.

\begin{example}
	Here is an example of a  planar rooted tree with  3 leaves and 4 vertices. 
	\[
	\vcenter
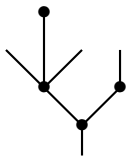
{\hbox{
	\begin{tikzpicture}
		\def\scale{0.5};
		\pgfmathsetmacro{\diagcm}{sqrt(2)};
		
		\coordinate (r) at (0,0);
		\coordinate (v1) at ($(r) + (0,\scale*0.8)$);
		\coordinate (v21) at ($(v1) + (135:\scale*\diagcm)$);
		\coordinate (v22) at ($(v1) + (45:\scale*\diagcm)$);
		\coordinate (l1) at ($(v21) + (135:\scale*\diagcm)$);
		\coordinate (l2) at ($(v21) + (45:\scale*\diagcm)$);
		\coordinate (l3) at ($(v22) + (0,\scale*1)$);
		\coordinate (v4) at ($(v21) + (0,\scale*2)$);
		
		\draw[thick] (r) to (v1);
		\draw[thick] (v1) to (v21);
		\draw[thick] (v1) to (v22);
		\draw[thick] (v21) to (l1);
		\draw[thick] (v21) to (v4);
		\draw[thick] (v21) to (l2);
		\draw[thick] (v22) to (l3);
		
		\node at (v1) {$\bullet$};
		\node at (v21) {$\bullet$};
		\node at (v22) {$\bullet$};
		\node at (v4) {$\bullet$};
		
	\end{tikzpicture}}}
	\]
\end{example}

In this section, we only consider the set $\PT_0$ of planar rooted trees without leaves.

\begin{example}
	The sets $\PT^{(4)}_0$ of planar rooted trees without leaves and with 4 vertices is
	\[
	\PT^{(4)}_0=\left\{
	\vcenter{\hbox{
		\begin{tikzpicture}
			\def\scale{0.5};
			\pgfmathsetmacro{\diagcm}{sqrt(2)};
			\coordinate (r) at (0,0);
			\coordinate (v1) at ($(r) + (0,\scale*0.8)$);
			\coordinate (v2) at ($(v1) + (0,\scale*1)$);
			\coordinate (v3) at ($(v2) + (0,\scale*1)$);
			\coordinate (v4) at ($(v3) + (0,\scale*1)$);
			\draw[thick] (r) to (v1);
			\draw[thick] (v1) to (v2);
			\draw[thick] (v2) to (v3);
			\draw[thick] (v3) to (v4);
			\node at (v1) {$\bullet$};
			\node at (v2) {$\bullet$};
			\node at (v3) {$\bullet$};
			\node at (v4) {$\bullet$};
		\end{tikzpicture}}}, 
	\vcenter{\hbox{
		\begin{tikzpicture}
			\def\scale{0.5};
			\pgfmathsetmacro{\diagcm}{sqrt(2)};
			\coordinate (r) at (0,0);
			\coordinate (v1) at ($(r) + (0,\scale*0.8)$);
			\coordinate (v2) at ($(v1) + (0,\scale*1)$);
			\coordinate (v3) at ($(v2) + (135:\scale*\diagcm)$);
			\coordinate (v4) at ($(v2) + (45:\scale*\diagcm)$);
			\draw[thick] (r) to (v1);
			\draw[thick] (v1) to (v2);
			\draw[thick] (v2) to (v3);
			\draw[thick] (v2) to (v4);
			\node at (v1) {$\bullet$};
			\node at (v2) {$\bullet$};
			\node at (v3) {$\bullet$};
			\node at (v4) {$\bullet$};
		\end{tikzpicture}}}, 
	\vcenter{\hbox{
		\begin{tikzpicture}
			\def\scale{0.5};
			\pgfmathsetmacro{\diagcm}{sqrt(2)};
			\coordinate (r) at (0,0);
			\coordinate (v1) at ($(r) + (0,\scale*0.8)$);
			\coordinate (v2) at ($(v1) + (135:\scale*\diagcm)$);
			\coordinate (v3) at ($(v1) + (45:\scale*\diagcm)$);
			\coordinate (v4) at ($(v2) + (0,\scale*1)$);
			\draw[thick] (r) to (v1);
			\draw[thick] (v1) to (v2);
			\draw[thick] (v1) to (v3);
			\draw[thick] (v2) to (v4);
			\node at (v1) {$\bullet$};
			\node at (v2) {$\bullet$};
			\node at (v3) {$\bullet$};
			\node at (v4) {$\bullet$};
		\end{tikzpicture}}}, 
	\vcenter{\hbox{
		\begin{tikzpicture}
			\def\scale{0.5};
			\pgfmathsetmacro{\diagcm}{sqrt(2)};
			\coordinate (r) at (0,0);
			\coordinate (v1) at ($(r) + (0,\scale*0.8)$);
			\coordinate (v2) at ($(v1) + (135:\scale*\diagcm)$);
			\coordinate (v3) at ($(v1) + (45:\scale*\diagcm)$);
			\coordinate (v4) at ($(v3) + (0,\scale*1)$);
			\draw[thick] (r) to (v1);
			\draw[thick] (v1) to (v2);
			\draw[thick] (v1) to (v3);
			\draw[thick] (v3) to (v4);
			\node at (v1) {$\bullet$};
			\node at (v2) {$\bullet$};
			\node at (v3) {$\bullet$};
			\node at (v4) {$\bullet$};
		\end{tikzpicture}}}, 
	\vcenter{\hbox{
		\begin{tikzpicture}
			\def\scale{0.5};
			\pgfmathsetmacro{\diagcm}{sqrt(2)};
			\coordinate (r) at (0,0);
			\coordinate (v1) at ($(r) + (0,\scale*0.8)$);
			\coordinate (v2) at ($(v1) + (135:\scale*\diagcm)$);
			\coordinate (v3) at ($(v1) + (45:\scale*\diagcm)$);
			\coordinate (v4) at ($(v1) + (0,\scale*1)$);
			\draw[thick] (r) to (v1);
			\draw[thick] (v1) to (v2);
			\draw[thick] (v1) to (v3);
			\draw[thick] (v1) to (v4);
			\node at (v1) {$\bullet$};
			\node at (v2) {$\bullet$};
			\node at (v3) {$\bullet$};
			\node at (v4) {$\bullet$};
	\end{tikzpicture}}}
	\right\}\ .
	\]
\end{example}  

To any planar rooted tree $\tau \in \PT_0^{(k)}$~, we associate the element $\tau^\PP\in \F_k V$ obtained by labeling the vertices of arity $m$ with $\PP_m$ and by considering the global evaluation. 

\begin{example}
	The planar tree 
	\[\tau\coloneqq\vcenter
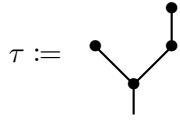
{\hbox{
			\begin{tikzpicture}
			\def\scale{0.5};
			\pgfmathsetmacro{\diagcm}{sqrt(2)};
			\coordinate (r) at (0,0);
			\coordinate (v1) at ($(r) + (0,\scale*0.8)$);
			\coordinate (v2) at ($(v1) + (135:\scale*\diagcm)$);
			\coordinate (v3) at ($(v1) + (45:\scale*\diagcm)$);
			\coordinate (v4) at ($(v3) + (0,\scale*1)$);
			\draw[thick] (r) to (v1);
			\draw[thick] (v1) to (v2);
			\draw[thick] (v1) to (v3);
			\draw[thick] (v3) to (v4);
			\node at (v1) {$\bullet$};
			\node at (v2) {$\bullet$};
			\node at (v3) {$\bullet$};
			\node at (v4) {$\bullet$};
			\end{tikzpicture}}}
	\]
	produces the element $\tau^\PP=\PP_2(p_0, \PP_1(p_0))\in \F_4 V$~.
\end{example}

\begin{proposition}\label{prop:FixPtEqua}
	The fixed-point equation 
	\begin{eqnarray}\label{eqn:FixPt}
	x=\PP(x)=p_0+\sum_{m=1}^{\infty}\PP_m\left(x^{\otimes m}\right).
	\end{eqnarray}
	associated to any analytic function $\PP$ admits the following  unique solution:
	\[
	x=\sum_{k=1}^\infty \sum_{\tau \in \PT^{(k)}_0} \tau^\PP\ .
	\]
\end{proposition}

\begin{proof}
	Suppose that there exists a solution and let us write it as a series  $x=\sum_{k=1}^\infty x_k$~, with $x_k\in \F_k V$~. The projection of Equation~\eqref{eqn:FixPt} onto $V/\F_{k+1} V$~, for $k\geqslant 1$~, gives 
	\begin{equation}\label{eqn:partial}\tag{$\ast$}
	x_1+\cdots+x_{k}=p_0+\PP_1(x_1+\cdots+x_{k-1})+\cdots+\PP_{k-1}(x_1, \ldots, x_1)\ .
	\end{equation}
	By induction on $k\geqslant 1$~, this shows that the  projection of $x$ onto $V/\F_{k+1} V$ is unique, which proves that $x$ is unique since $(V, \F)$ is complete. 
	In the other way round, let us consider 
	\[
	x_k\coloneqq\sum_{\tau \in \PT^{(k)}_0} \tau^\PP \in \F_k V \qquad \text{and}
	\qquad x\coloneqq\sum_{k=1}^\infty x_k\ .
	\]
	We now prove, by induction on $k\geqslant 1$~, that $x$ satisfies the projection \eqref{eqn:partial} of  Equation~\eqref{eqn:FixPt} onto $V/\F_{k+1} V$~. 
	For $k=1$~, the set $\PT^{(1)}_0$ only contains the arity $0$ corolla 
$\vcenter{\hbox{\begin{tikzpicture}[scale=0.5]
	\draw[thick] (0,0) -- (0,0.6);
	\fill (0,0.6) circle [radius=4pt];
	\end{tikzpicture}}}$~, 
	so $x_1=p_0$~. Suppose now that the result holds true for $k$~, and let us prove it for $k+1$:
	\begin{align*}
	p_0+\PP_1(x_1+\cdots+x_{k})+{}&\cdots+\PP_{k}(x_1, \ldots, x_1)=\\
	={}&p_0+\PP_1(x_1+\cdots+x_{k-1})+\cdots+\PP_{k-1}(x_1, \ldots, x_1)\\
	&+\PP_1(x_k)+\PP_2(x_1, x_{k-1})+\cdots+\PP_2(x_{k-1}, x_1)+\cdots\\
	&+\PP_{k+1}(x_1, \ldots, x_1) \\
	={}&x_1+\cdots+x_{k} + \sum_{\tau \in \PT^{(k+1)}_0} \tau^\PP\\
	={}&x_1+\cdots+x_{k+1}\ .
	\end{align*}
\end{proof}

\subsection{Formal differential equations}\label{App:FormDE}

Let $\k[[t]]$ be the ring of formal power series in the variable $t$~. If $V$ is a (potentially graded) vector space, we denote
\[
V[[t]]\coloneqq V\otimes\k[[t]]\ ,
\]
the complete graded vector space of the formal power series in $t$ with coefficients in $V$~, where the $t$-adic filtration is defined by 
$\F_k V[[t]]\coloneqq t^kV[[t]]$~.
We consider the \emph{differentiation in $t$ operator}
\[
\frac{d}{dt}:V[[t]]\longrightarrow V[[t]]
\]
acting by
\[
\frac{d}{dt}\left(\sum_{k\geqslant0}v_k t^k\right)\coloneqq \sum_{k\geqslant1}k v_k t^{k-1}\ .
\]
Given a linear map
\[
f:V^{\otimes m}\longrightarrow V\ ,
\]
its canonical $\k[[t]]$-linear extension
\[
f:V[[t]]^{\otimes m}\longrightarrow V[[t]] \ .
\]

\begin{definition}[Formal differential equation]\label{def:FDE}
	Let $V$ be a vector space. For each $m, k\geqslant0$ let
	\[
	f_m^k:V^{\otimes m}\longrightarrow V
	\]
	be linear maps, and suppose that for any $k$~, there are finitely many non-zero $f_m^k$. A \emph{formal differential equation}\index{formal!differential equation} in $V$ is an equation of the form
	\begin{equation}\label{eq:formal differential equation}
	\frac{d}{dt}v(t) = \sum_{m,k\geqslant0} f_m^k\big(\underbrace{v(t),\ldots,v(t)}_{m\text{\ times}}\big)\, t^k
	\end{equation}
	in $V[[t]]$~.
\end{definition}

\begin{remark}
	The condition that, for any $k$~, there are finitely many non-zero $f_m^k$ is necessary in order for Equation~\eqref{eq:formal differential equation} to make sense. However, it can sometimes be relaxed. For example, if $V$ itself is a complete chain complex and the $f_m^k$ are suitably filtered maps, then we do not have to make any finiteness assumption.
\end{remark}

Notice that for any $v(t)\in V[[t]]$~, the value $v(0)\in V$ is well-defined. This in not true for evaluation at any other $t\in\k$ if one does not make any further assumptions.

\begin{lemma}
	Any equation of type \eqref{eq:formal differential equation} admits a unique solution for any fixed initial value at $t=0$~.
\end{lemma}

\begin{proof}
	This is a straightforward induction: one proves that the $k$th coefficient is completely determined by the coefficients up to the $(k-1)$th, together with the fact that the initial value fixes the $0$th coefficient.
\end{proof}

\begin{remark}
	One can also prove this lemma by imitating formally the proof of the classical Cauchy--Lipschitz theorem to obtain a formal fixed-point equation
	\[
	v(t) = v_0 + \int_0^t\sum_{m,k\geqslant0}f_m^k(v(s),\ldots,v(s))\, s^k \,ds
	\]
	in the complete  vector space $V[[t]]$ and then conclude by \cref{prop:FixPtEqua}.
\end{remark}

We now give directly an explicit formula for the solution of a formal differential equation.

\medskip
Denote by $\wPT$ the set of \emph{weighted  planar rooted trees}, that is planar rooted trees 
where each vertex can  have arity greater of equal to $0$ and is labeled an integer greater or equal to $1$~. 
We denote by $K(\tau)$ the total weight of a weighted planar rooted tree, which is defined as the sum of all the weights labeling its vertices. The subset of trees of weight $k$ is denoted by $\wPT^{[k]}$~. 
By convention, the only tree of weight $0$ is the trivial tree: $\wPT^{[0]}\coloneq \{|\}$~. 
We denote by $\mathrm{c}_m^{[k]}$ the corolla of weight $k$ with $m$ leaves.

\begin{example}
	Here is an example of a weighted  planar rooted tree of total weight equal to $7$~.
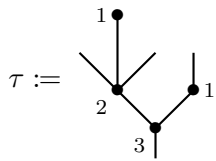
	\[
	\tau\coloneqq\vcenter{\hbox{
			\begin{tikzpicture}
			\def\scale{0.5};
			\pgfmathsetmacro{\diagcm}{sqrt(2)};
			
			\coordinate (r) at (0,0);
			\coordinate (v1) at ($(r) + (0,\scale*0.8)$);
			\coordinate (v21) at ($(v1) + (135:\scale*\diagcm)$);
			\coordinate (v22) at ($(v1) + (45:\scale*\diagcm)$);
			\coordinate (l1) at ($(v21) + (135:\scale*\diagcm)$);
			\coordinate (l2) at ($(v21) + (45:\scale*\diagcm)$);
			\coordinate (l3) at ($(v22) + (0,\scale*1)$);
			\coordinate (v4) at ($(v21) + (0,\scale*2)$);
			
			\draw[thick] (r) to (v1);
			\draw[thick] (v1) to (v21);
			\draw[thick] (v1) to (v22);
			\draw[thick] (v21) to (l1);
			\draw[thick] (v21) to (v4);
			\draw[thick] (v21) to (l2);
			\draw[thick] (v22) to (l3);
			
			\node at (v1) {$\bullet$};
			\node at (v21) {$\bullet$};
			\node at (v22) {$\bullet$};
			\node at (v4) {$\bullet$};
			
			\node[below left] at (v1) {$\scriptstyle3$};
			\node[below left] at (v21) {$\scriptstyle2$};
			\node[right] at (v22) {$\scriptstyle1$};
			\node[left] at (v4) {$\scriptstyle1$};
			
			\end{tikzpicture}}}
	\]
\end{example}

Since any non-trivial rooted tree $\tau = \mathrm{c}_m^{[l]}\circ(\tau_1,\ldots,\tau_m)\in \wPT$  can be decomposed as  the grafting of $m$ trees $\tau_1,\ldots,\tau_m\in\wPT$ onto the corolla $\mathrm{c}_m^{[l]}$~, we can construct some integer valued functions on $\wPT$ recursively. For instance, the \emph{exponential weight function}  is defined by 
\[
W(\, |\, ) \coloneq 1\ ,\qquad W\left(\mathrm{c}_m^{[k]}\right) \coloneq k\ ,\qquad W(\tau) \coloneq K(\tau)\prod_{i=1}^m W(\tau_i)\ .
\]
The exponential weight of the above displayed tree $\tau$ is equal to 
$W(\tau)=6$~.

\medskip

For any family $f=\left\{f_m^k : V^{\otimes m}\to V\right\}_{m,k\geqslant 0}$ as in \cref{def:FDE}, 
a weighted planar rooted tree $\tau\in\wPT^{(k)}$ induces a function
\[
\tau^f\ :\ V^{\otimes m}\longrightarrow V\otimes \k\, t^k\subset V[[t]]
\]
defined by 
\[
|^f(v)\coloneqq v\ ,\qquad \left(c_m^{[k]}\right)^f\coloneqq t^k f_{m}^{k-1}\ ,\qquad 
\tau^f\coloneqq t^l f_{m}^{l-1}\circ\left(\tau_1^f,\ldots,\tau^f_m\right).
\]
We denote by 
\[
\tau^f(v_0)\coloneqq\tau^f(v_0,\ldots,v_0)
\]
the evaluation on an element $v_0\in V$~.

\begin{proposition}\label{prop:formula for formal ODE}
	The element of $V[[t]]$ given by
	\[
	v(t) \coloneq  \sum_{\tau\in \wPT}{\textstyle \frac{1}{W(\tau)}}\tau^f(v_0)
	\]
	is the unique solution for the formal differential equation (\ref{eq:formal differential equation}) with initial value $v(0) = v_0$~.
\end{proposition}

\begin{proof}
	As explained above, the summand of exponent $k$ of the formal variable $t$ in $v(t)$ is given the sum of weighted rooted planar trees of total weight $k$~:
	\[v(t)= \sum_{k=0}^\infty \sum_{\tau\in \wPT^{(k)}}{\textstyle\frac{1}{W(\tau)}}\tau^f(v_0)\ .\]
	Since there are finitely many non-zero $f_m^k$ for any fixed $k$~, the coefficient of $t^k$ is  a finite sum and thus the element $v(t)$ is well-defined. 
	Finally, decomposing all non-trivial rooted trees as $\tau = \mathrm{c}_m^{[l]}\circ(\tau_1,\ldots,\tau_m)$ we get: 
	\begin{align*}
	\frac{d}{dt}v(t) =&\ \frac{d}{dt}\left(\sum_{\tau\in \wPT}{\textstyle\frac{1}{W(\tau)}}\tau^f(v_0)\right)
	=\frac{d}{dt}\left(v_0 +\sum_{k=1}^\infty \sum_{\tau\in \wPT^{[k]}}{\textstyle\frac{1}{W(\tau)}}\tau^f(v_0)\right)\\
	=&\sum_{k=1}^\infty \sum_{\tau\in \wPT^{[k]}}{\textstyle\frac{k}{W(\tau)}}f_{m}^{l-1}\left(\tau^f_1(v_0),\ldots,\tau^f_m(v_0)\right)t^{l-1}
	\\
	=&\sum_{k=1}^\infty \sum_{\tau\in \wPT^{[k]}}{\textstyle\frac{1}{\prod_{i=1}^m F(\tau_i)}}f_{m}^{l-1}\left(\tau^f_1(v_0),\ldots,\tau^f_m(v_0)\right)t^{l-1}
	\\
	=&\sum_{k=1}^\infty \sum_{\tau\in \wPT^{[k]}} f_{m}^{l-1}\left({\textstyle\frac{1}{F(\tau_1)}} \tau^f_1(v_0),\ldots,{\textstyle\frac{1}{F(\tau_m)}}\tau^f_m(v_0)\right)t^{l-1}
	\\
	=&\ \sum_{m,k\geqslant0}f_{m}^{k}\left(\sum_{\tau_1\in \wPT}{\textstyle \frac{1}{W(\tau_1)}}\tau^f_1(v_0),\ldots,\sum_{\tau_n\in \wPT}{\textstyle \frac{1}{W(\tau_m)}}\tau^f_m(v_0)\right)t^k\\
	=&\ \sum_{m,k\geqslant0}f_{m}^{k}\big(v(t),\ldots,v(t)\big)t^k\ .
	\end{align*}
	The fact that we have the correct initial value is obvious from the fact that the only term which does not involve the formal variable $t$ is $|^f(v_0) = v_0$~.
\end{proof}

\begin{remark}
	The material presented here has close relations with the B-series studied in numerical analysis, cf. e.g. \cite[Section~38]{Butcher} and \cite[Chapter~III]{HLW10}.
\end{remark}


\bibliographystyle{alpha}
\bibliography{bib}

\end{document}